\documentclass[10pt,reqno]{amsproc}

\usepackage[sc]{mathpazo}\renewcommand{\mathbf}{\mathbold}
\normalfont
\usepackage[T1]{fontenc}

\usepackage[margin=1in]{geometry}
\usepackage{comment}
\usepackage[small,bf]{caption}
\usepackage{scalerel,nicefrac}
\usepackage[all,cmtip]{xy}
\usepackage{tikz-cd}

\usepackage{rotating, setspace}

\usepackage{bbm}
\usepackage{listings}
\usepackage{geometry}
\usepackage{amsmath,amsthm,amssymb}
\usepackage{hyperref}
\usepackage{cleveref}
\usepackage{color}
\usepackage{mathtools}
\usepackage{tikz-cd}
\usepackage{subfigure}
\usepackage{shuffle}
\usepackage{mathrsfs}
\usepackage{multicol}
\usepackage{thmtools}
\usepackage[utf8]{inputenc}

\makeatletter
\renewcommand{\@secnumfont}{\bfseries}
\makeatother

\hypersetup{
    colorlinks=true,
    linkcolor=blue,
    filecolor=blue,      
    }

\makeatletter
\def\l@subsection{\@tocline{2}{0pt}{2.5pc}{5pc}{}}
\makeatother
\numberwithin{equation}{section}

\newcommand{\ps}[1]{\mkern-.25mu\mathbin{(\mkern-3.5mu({#1})\mkern-3.5mu)}}

\newcommand{\N}{\mathbb{N}}

\newcommand{\R}{\mathbb{R}}
\newcommand{\C}{\mathbb{C}}

\newcommand{\E}{\mathbb{E}}

\newcommand{\bx}{\mathbf{x}}
\newcommand{\by}{\mathbf{y}}

\newcommand{\bs}{\mathbf{s}}
\newcommand{\bt}{\mathbf{t}}

\newcommand{\bE}{\mathbf{E}}

\newcommand{\cP}{\mathcal{P}}
\newcommand{\cA}{\mathcal{A}}
\newcommand{\cB}{\mathcal{B}}
\newcommand{\cW}{\mathcal{W}}
\newcommand{\cD}{\mathcal{D}}
\newcommand{\cK}{\mathcal{K}}
\newcommand{\hcA}{\widehat{\cA}}
\newcommand{\hcB}{\widehat{\cB}}

\DeclareMathOperator{\cmCh}{\mathbf{Ch}}

\newcommand{\bsC}{\textsf{\textbf{C}}}
\newcommand{\bsD}{\textsf{\textbf{D}}}
\newcommand{\cmG}{\mathbf{G}}
\newcommand{\cmg}{{\boldsymbol{\mathfrak{g}}}}
\newcommand{\cmh}{{\boldsymbol{\mathfrak{h}}}}
\newcommand{\cmA}{\mathbf{A}}
\newcommand{\cmH}{\mathbf{H}}
\newcommand{\mult}{\odot}
\newcommand{\cmb}{\delta} %
\newcommand{\gt}{\vartriangleright}

\newcommand{\gtd}{\gtrdot}
\newcommand{\ltd}{\lessdot}
\newcommand{\dg}{\textsf{\textbf{D}}}

\newcommand{\hA}{\hat{A}}
\newcommand{\bF}{\mathbf{F}}
\newcommand{\bS}{\mathbf{S}}

\newcommand{\oT}{\overline{T}}

\newcommand{\hH}{\widehat{H}}

\newcommand{\hcmA}{\hat{\cmA}}
\newcommand{\bT}{\mathbf{T}}
\newcommand{\bB}{\mathbf{B}}
\newcommand{\hT}{\widehat{T}}

\newcommand{\tbx}{\widetilde{\bx}}
\newcommand{\tby}{\widetilde{\by}}

\newcommand{\pcmH}{\widehat{\pcmH}}

\newcommand{\ucona}{\zeta}
\newcommand{\uconc}{Z}

\newcommand{\V}{V}
\newcommand{\units}{\times}

\newcommand{\hs}{\hat{s}}
\newcommand{\ts}{\tilde{s}}

\newcommand{\oE}{\overline{E}}

\newcommand{\orho}{\sigma}
\newcommand{\CTRL}{W}

\newcommand{\ctr}[2]{\omega^{#1}_{#2}}
\newcommand{\inc}[2]{\Delta_{#1, #2}}
\newcommand{\CTR}[3]{W^{#1}_{#2}(#3)}
\newcommand{\INC}[3]{D^{#1}_{#2}(#3) }

\newcommand{\hotimes}{\widehat{\otimes}}

\newcommand{\ffact}[2]{(#1)_{#2}!}
\newcommand{\fbinom}[3]{\binom{#1}{#2}_{\mkern-5.5mu #3}}

\newcommand{\andd}{\quad \text{and} \quad}

\newcommand{\fg}{\mathfrak{g}}

\newcommand{\fh}{\mathfrak{h}}
\newcommand{\GL}{\text{GL}}

\newcommand{\cmB}{\mathbf{B}}

\newcommand{\ox}{\overline{x}}
\newcommand{\oy}{\overline{y}}
\newcommand{\oX}{\overline{X}}
\newcommand{\oY}{\overline{Y}}

\newcommand{\rrX}{\mathbb{X}}
\newcommand{\rX}{\boldsymbol{X}}
\newcommand{\rx}{\boldsymbol{x}}
\newcommand{\gr}[1]{(#1)}
\newcommand{\pgr}[1]{\{#1\}}

\newcommand{\orrX}{\overline{\rrX}}
\newcommand{\orX}{\overline{\rX}}
\newcommand{\orx}{\overline{\rx}}

\newcommand{\orrY}{\overline{\rrY}}
\newcommand{\orY}{\overline{\rY}}
\newcommand{\ory}{\overline{\ry}}

\newcommand{\hrrX}{\widehat{\rrX}}

\newcommand{\trrX}{\widetilde{\rrX}}
\newcommand{\trX}{\widetilde{\rX}}
\newcommand{\trx}{\widetilde{\rx}}

\newcommand{\rrY}{\mathbb{Y}}
\newcommand{\rY}{\boldsymbol{Y}}
\newcommand{\ry}{\boldsymbol{y}}
\newcommand{\trrY}{\widetilde{\rrY}}
\newcommand{\trY}{\widetilde{\rY}}
\newcommand{\try}{\widetilde{\ry}}

\newcommand{\dya}{\mathbb{D}}

\newcommand{\tf}{\widetilde{f}}

\newcommand{\cmFXA}{\mathbf{FXA}}

\newcommand{\intind}{\tau}
\newcommand{\sig}{S}
\newcommand{\ssig}{R}
\newcommand{\hssig}{\widehat{R}}

\newcommand{\cona}{\alpha}
\newcommand{\conb}{\beta}
\newcommand{\conc}{\gamma}
\newcommand{\con}{{\boldsymbol{\omega}}}

\newcommand{\tcona}{\widetilde{\cona}}

\newcommand{\tconc}{\widetilde{\conc}}

\newcommand{\tucona}{\widetilde{\zeta}}
\newcommand{\tuconc}{\widetilde{Z}}

\newcommand{\gridmult}{\boxplus}
\newcommand{\tgridmult}{\widetilde{\gridmult}}

\newcommand{\hmult}{\odot_h}
\newcommand{\vmult}{\odot_v}
\newcommand*{\bighmult}[2]{{\bigodot_{#1}^{#2} }\,_{\mkern-2.5mu h} \, }
\newcommand*{\bigvmult}[2]{{\bigodot_{#1}^{#2} }\,_{\mkern-2.5mu v}  \, }

\newcommand{\smm}{s_-}
\newcommand{\spp}{s_+}
\newcommand{\sss}{s_*}
\newcommand{\tmm}{t_-}
\newcommand{\tpp}{t_+}
\newcommand{\tss}{t_*}
\newcommand{\smp}{\smm, \spp}
\newcommand{\tmp}{\tmm, \tpp}
\newcommand{\smpi}[1]{\smm^{#1}, \spp^{#1}}
\newcommand{\tmpi}[1]{\tmm^{#1}, \tpp^{#1}}

\newcommand{\ppart}{\mathcal{P}}
\newcommand{\dpart}{\mathcal{D}}

\newcommand{\surfaces}{\mathsf{Surfaces}}
\newcommand{\paths}{\mathsf{Paths}}
\newcommand{\concat}{\star}
\newcommand{\bSigma}{{\boldsymbol{\Sigma}}}
\newcommand{\bdy}{\partial}
\newcommand{\sC}{\mathsf{C}}

\newcommand{\ovrho}{\overline{\rho}}

\newcommand*{\pmat}[1]{\begin{pmatrix}#1\end{pmatrix}}

\DeclareMathOperator{\sab}{sab}
\DeclareMathOperator{\Lie}{Lie}
\DeclareMathOperator{\FL}{FL}
\DeclareMathOperator{\Pf}{Pf}
\DeclareMathOperator{\id}{id}
\DeclareMathOperator{\im}{im}

\DeclareMathOperator{\op}{op}

\DeclareMathOperator{\spann}{span}
\DeclareMathOperator{\FXA}{FXA}
\DeclareMathOperator{\FBi}{FBi}
\DeclareMathOperator{\Mat}{Mat}
\DeclareMathOperator{\Ch}{Ch}
\DeclareMathOperator{\hCh}{\widehat{Ch}}

\DeclareMathOperator{\ad}{ad}
\DeclareMathOperator{\DGF}{DGF}
\DeclareMathOperator{\RP}{RP}

\DeclareMathOperator{\RS}{RS}
\DeclareMathOperator{\Aut}{Aut}
\DeclareMathOperator{\LCS}{LCS}
\DeclareMathOperator{\diam}{diam}
\DeclareMathOperator{\HPF}{HPF}
\DeclareMathOperator{\VPF}{VPF}
\DeclareMathOperator{\SF}{SF}
\DeclareMathOperator{\Law}{Law}
\DeclareMathOperator{\PBi}{PBi}

\newtheorem{counter}{Counter}[section]
\newtheorem{lemma}[counter]{Lemma}

\newtheorem{proposition}[counter]{Proposition}
\newtheorem{theorem}[counter]{Theorem}
\newtheorem{corollary}[counter]{Corollary}

\theoremstyle{definition}
\newtheorem{definition}[counter]{Definition}
\newtheorem{example}[counter]{Example}
\newtheorem{remark}[counter]{Remark}

\title{The Surface Signature and Rough Surfaces}
\author{Darrick Lee}
\date{\today}

\begin{document}

\maketitle

\begin{abstract}
    \small
    Parallel transport, or path development, provides a rich characterization of paths which preserves the underlying algebraic structure of concatenation.
    The path signature is universal among such maps: any (translation-invariant) parallel transport factors uniquely through the path signature. 
    Furthermore, the path signature is a central object in the theory of rough paths, which provides an integration theory for highly irregular paths.
    A fundamental result is Lyons' extension theorem, which allows us to compute the signature of rough paths, and in turn provides a way to compute parallel transport of arbitrarily irregular paths.

    In this article, we consider the notion of surface holonomy, a generalization of parallel transport to the higher dimensional setting of surfaces parametrized by rectangular domains, which preserves the higher algebraic structures of horizontal and vertical concatenation. 
    Building on work of Kapranov, we introduce the surface signature, which is universal among surface holonomy maps with respect to continuous 2-connections.
    Furthermore, we introduce the notion of a rough surface and prove a surface extension theorem, which allows us to compute the signature of rough surfaces.
    By exploiting the universal property of the surface signature, this provides a method to compute the surface holonomy of arbitrarily irregular surfaces.
\end{abstract}

\clearpage
$\,$\vspace{10pt}
\tableofcontents

\clearpage
\section{Introduction}

\subsection{Paths} \label{ssec:paths}

Continuous paths are ubiquitous in both pure and applied mathematics, modelling sequential phenomena such as stochastic processes or empirical time series data. Consider the space of parametrized piecewise smooth paths
\begin{align}\label{eq:paths}
    \paths^\infty(\V) \coloneqq \bigcup_{s_1 \leq s_2} C^\infty([s_1, s_2], \V). 
\end{align}
into a finite dimensional Hilbert space $\V$. While this space is inherently nonlinear, as there is no natural way to add paths parametrized on different intervals, paths are equipped with a concatenation structure if their parametrization and boundaries coincide. Suppose $s_1 \leq s_2 \leq s_3$. Two paths $x \in C^\infty([s_1, s_2], \V)$ and $y \in C^\infty([s_2, s_3], \V)$ are \emph{composable} if $x_{s_2} = y_{s_2}$, and we define their concatenation
\begin{align} \label{eq:intro_path_concatenation}
    x \concat y \in C^\infty([s_1, s_3], \V) \quad \text{by} \quad (x \concat y)_s \coloneqq \left\{
        \begin{array}{ll}
            x_{s} & : s \in [s_1, s_2] \\
            y_{s} & : s \in [s_2, s_3].
        \end{array}
    \right.
\end{align}
This concatenation structure highlights the fundamental sequential nature of paths: since $x$ is parametrized on $[s_1, s_2]$, it occurs before $y$, which is parametrized on $[s_2, s_3]$. \medskip

We often need to work with either measures on path spaces (such as for stochastic processes), or functions on path spaces (such as for time series classification). One way to study measures and functions on $\paths^\infty(\V)$ is to build descriptions of paths $F: \paths^\infty(\V) \to A_0$ valued in a unital associative algebra $(A_0, \cdot, 1)$ which preserve the concatenation structure,
\begin{align}
    F(x \concat y) = F(x) \cdot F(y).
\end{align}

One example is the notion of \emph{path development} or \emph{(translation-invariant) parallel transport}\footnote{We refer to parallel transport with respect to translation-invariant connections as \emph{path development}, which is the common terminology in the analysis literature. We use the terms \emph{holonomy} and \emph{parallel transport} interchangeably, as is common in the higher geometry literature.}. Given a linear map $\cona \in L(\V, A_0)$ which we call a \emph{connection}, we define the \emph{parallel transport with respect to $\cona$}, denoted $F^\cona: \paths^\infty(\V) \to A_0$, as the solution of the differential equation for $F^\cona_t(x): [s_1, s_2] \to A_0$ given by
\begin{align} \label{eq:intro_parallel_transport}
    \frac{dF^\cona_s(x)}{dt} = F^\cona_s(x) \cdot \cona\left( \frac{dx}{ds}\right), \quad F^{\cona}_{s_1}(x) = 1,
\end{align}
with $F^\cona(x) \coloneqq F^{\cona}_{s_2}(x)$, which preserves this concatenation structure. Parallel transport maps provide a rich description of path spaces up to translation of both the parametrization and the paths. In other words, if we append the parametrization to the path\footnote{To be precise, for a path $x$, we append the parametriation by considering $\overline{x}_t = (t, x_t)$ valued in $\R \oplus \V$; thus, connections are valued in $L(\R \oplus \V, A_0)$.}, they
\begin{enumerate}
    \item \textbf{separate paths:} given two distinct paths $x, y \in C^\infty([0,1],\V)$, there exists an algebra $A_0$ and a connection $\cona$ such that $F^\cona(\ox) \neq F^\cona(\oy)$;
    \item are \textbf{analytically universal}\footnote{We use the term \emph{analytically universal} to refer to density of functions, in order to distinguish this from \emph{universal} which refers to the category theoretic meaning.}\textbf{:} linear functionals $\langle \ell, F^\cona(\overline{\cdot})\rangle: \cK \to \R$ are dense in $C(\cK, \R)$ equipped with the uniform topology, where $\cK \subset C^\infty([0,1],\V)$ is a compact subset; and
    \item are \textbf{characteristic:} given two distinct measures $\mu, \nu$ valued on the compact subset $\cK$, there exists an algebra $A_0$ and a connection $\cona$ such that $\E_{x \sim \mu}[F^\cona(\ox)] \neq \E_{y \sim \nu}[F^\cona(\oy)]$. 
\end{enumerate}

The \emph{path signature} is the universal path development map. It can be defined explicitly for a path $x: [\smp] \to \V$ as a sequence of iterated integrals valued in a group $G_0 \subset T_0\ps{\V} \coloneqq \prod_{n=0}^\infty \V^{\otimes n}$ embedded in formal power series of tensors
\begin{align} \label{eq:intro_path_signature}
    S: \paths^\infty(\V) \to G_0, \quad S(x) = 1 + \sum_{n=1}^\infty \int_{\Delta^n(\smp)} dx_{s_1} \cdots dx_{s_n},
\end{align}
where $\Delta^n(\smp) \coloneqq \{ \smm \leq s_1 < \ldots < s_n \leq \spp\}$ and $\cdot$ denotes the tensor product in $T_0\ps{\V}$. The parameter $n$ denotes the \emph{level} of the tensor algebra and path signature. The path signature satisfies a well-known universal property.

\begin{theorem}[Informal] \label{thm:intro_path_signature_universal}
    Let $(A_0, \cdot, 1)$ be a unital Banach algebra, and $\cona \in L(\V, A_0)$ be a bounded linear map. Then, there exists a unique map $\tcona: G_0 \to A_0$ such that 
    \begin{align}\nonumber
        F^\cona(x) = \tcona(S(x)).
    \end{align}
\end{theorem}
In turn, the path signature itself can separate paths, and satisfies the analytic universality and characteristicness properties above. \medskip

So far, our discussion has been limited to smooth paths, and a natural question is whether such constructions are well-defined for lower regularity paths, such as $\rho$-H\"older paths,
\begin{align} \label{eq:intro_holder_paths}
    C^\rho([\smp], \V) \coloneqq \left\{ x \in C([\smp], \V) \, : \, \|x\|_\rho \coloneqq \sup_{(s_1, s_2) \in \Delta^2(\smp)} \frac{\|x_{s_1} - x_{s_2}\|}{|s_1 - s_2|^\rho} < \infty\right\},
\end{align}
for $\rho \in (0,1]$. For $\rho \in (\frac{1}{2}, 1]$, we can use the Young integral~\cite{young_inequality_1936} to compute the iterated integrals in~\Cref{eq:intro_path_signature} to define the path signature~\cite{lyons_differential_2007}. Furthermore, this boundary of $\rho = \frac12$ is sharp: for $\rho \in (0, \frac{1}{2}]$, there does not exist an integration theory that continuously extends Riemann integration (see~\cite[Section 1.5]{lyons_differential_2007}). \medskip

However, the theory of rough paths~\cite{lyons_differential_1998} allows us to go beyond the Young regime by considering lifted paths valued in the tensor algebra of the state space $\V$. In particular, for $\rho \in (0,1]$, a \emph{(weakly geometric) $\rho$-rough path}\footnote{We only work with weakly geometric rough paths in this article, and will assume all rough paths are weakly geometric throughout.} is a \emph{multiplicative functional}
\begin{align} \label{eq:intro_multiplicative_functional}
    \rx : \Delta^2(\smp) \to G_0^{\gr{\leq n_\rho}} \quad \text{which satisfies} \quad  \rx_{s_1, s_2} \cdot \rx_{s_2, s_3} = \rx_{s_1, s_3},
\end{align}
where $n_\rho \coloneqq \lfloor \rho^{-1} \rfloor$, and such that certain regularity conditions at each level. Here, $G_0^{\gr{\leq n}}$ denotes the truncation of $G_0$ at level $n$. We denote the space of $\rho$-rough paths by $\RP^\rho$.\medskip

For a path $y \in C^\rho([\smp], \V)$ in the Young regime ($\rho > \frac12$), we can construct a multiplicative functional $\ry: \Delta^2(\smp) \to G_0^{\gr{\leq n}}$ for any $n \in \N$ via the path signature
\begin{align} \label{eq:intro_young_mult_func}
    \ry^{\gr{k}}_{s_1, s_2} \coloneqq S^{\gr{k}}(y|_{[s_1, s_2]}).
\end{align}
However, for a path $x \in C^\rho([\smp], \V)$ beyond the Young regime ($\rho \leq \frac12$), the path signature is no longer well-defined. Instead, we \emph{postulate} the path signature of $x$ up to level $k = n_\rho$ in such a way that the multiplicativity condition in~\Cref{eq:intro_multiplicative_functional} holds. Lyons' extension theorem~\cite{lyons_differential_1998} (stated formally in~\Cref{thm:path_extension}) shows that there exist a unique multiplicative extension of rough paths.

\begin{theorem}[Informal~\cite{lyons_differential_1998}] \label{thm:intro_informal_path_extension}
    Let $\rho \in (0,1]$ and $\rx: \Delta^2 \to G_0^{\gr{\leq n_\rho}}$ be a $\rho$-rough path. Then, there exists a unique multiplicative functional $\trx: \Delta^2 \to G_0$ which satisfies certain regularity conditions and $\trx^{\gr{k}} = \rx^{\gr{k}}$ for all $k \in 0, 1, \ldots, n_\rho$.
\end{theorem}

For the path $y$ in the Young regime, we can define a corresponding $\rho$-rough path $\ry: \Delta^2 \to G_0^{\gr{\leq 1}} = \V$ as the path itself, $\ry_{s_1, s_2} = y_{s_2} - y_{s_1}$. We can apply the extension theorem to obtain $\try: \Delta \to G_0$, and by uniqueness, this is equivalent to the multiplicative functional defined via the path signature in~\Cref{eq:intro_young_mult_func}. This justifies the definition of the path signature of $\rho$-rough paths as
\begin{align} \label{eq:intro_rough_path_signature}
    S: \RP^\rho \to G_0, \quad \text{where} \quad S(\bx) \coloneqq \trx_{0,1},
\end{align}
where $\trx: \Delta^2 \to G_0$ is the unique extension from~\Cref{thm:intro_informal_path_extension}. In fact, we can define a \emph{$\rho$-rough path metric} $d_\rho$ such that $(\RP^\rho, d_\rho)$ is a complete metric space, and the path signature in~\Cref{eq:intro_rough_path_signature} becomes continuous with respect to this metric. Furthermore, this allows us to compute the parallel transport of $\rho$-rough paths.

\begin{theorem}[Informal] \label{eq:intro_rough_path_parallel_transport}
    Let $(A_0, \cdot, 1)$ be a unital Banach algebra, and $\cona \in L(\V, A_0)$ be a bounded linear map. Then, parallel transport of $\rho$-rough paths with respect to $\cona$, denoted $F^\cona :\RP^\rho \to A_0$ is a continuous map (with respect to $d_\rho$) defined by 
    \begin{align}\nonumber
        F^\cona(\rx) \coloneqq \tcona(S(\rx)),
    \end{align}
    where $\tcona$ is the map induced by the unique map $\tcona: G_0 \to A_0$ from~\Cref{thm:intro_path_signature_universal}.
\end{theorem}

Thus, the theory of rough paths allows us to work with paths of arbitrary $\rho$-H\"older regularity, at the cost of postulating additional data in terms of low level signature terms. In fact, both parallel transport and the path signature can still be used to characterize (probability measures of) rough paths, and approximate functions in the space of rough paths~\cite{chevyrev_characteristic_2016,chevyrev_signature_2022,cuchiero_global_2023}. \medskip

\subsection{Surfaces} \label{ssec:surfaces}
Motivated by the success of rough paths and signature methods in stochastic analysis and machine learning, we consider the extension of these constructions to the 2-dimensional setting of surfaces. Similar to paths, continuous surfaces are pervasive in mathematics, with random fields and image data being the natural analogues of our prior examples of stochastic processes and time series data. Consider the space of parametrized piecewise smooth surfaces,
\begin{align} \label{eq:surface}
    \surfaces(\V) \coloneqq \bigcup_{s_1 \leq s_2, \, t_1 \leq t_2} C^\infty([s_1, s_2] \times [t_1, t_2], \V).
\end{align}
This space of surfaces is equipped with a partially defined horizontal concatenation operation. If $X \in C^\infty([s_1, s_2] \times [t_1, t_2], \V)$ and $Y \in C^\infty([s_2, s_3] \times [t_1, t_2], \V)$ such that $X_{s_2, t} = Y_{s_2, t}$ for all $t \in [t_1, t_2]$, we say that $X$ and $Y$ are \emph{horizontally composable} and we define 
\begin{align} \label{eq:intro_surf_hconcat}
    X \concat_h Y \in C^\infty([s_1, s_3]\times[t_1, t_2], \V) \quad \text{by} \quad (X \concat_h Y)_{s,t} \coloneqq \left\{
        \begin{array}{ll}
            X_{s,t} & : s \in [s_1, s_2] \\
            Y_{s,t} & : s \in [s_2, s_2].
        \end{array}
    \right.
\end{align}
Similarly, if $Z \in C^\infty([s_1, s_2] \times[t_2, t_3])$ such that $X_{s, t_2} = Z_{s, t_2}$ for all $s \in [s_1, s_2]$, we say that $X$ and $Z$ are \emph{vertically composable} and we define
\begin{align} \label{eq:intro_surf_vconcat}
    X \concat_v Z \in C^\infty([s_1, s_2] \times[t_1, t_3], \V) \quad \text{by} \quad (X \concat_v Z)_{s,t} \coloneqq \left\{
        \begin{array}{ll}
            X_{s,t} & : t \in [t_1, t_2] \\
            Z_{s,t} & : t \in [t_2, t_3].
        \end{array}
    \right.
\end{align}
Both horizontal and vertical concatenation are associative by definition when the operations are well-defined. Furthermore, these two operations are compatible via the \emph{interchange law}: for $X, Y, Z, W \in \surfaces(\V)$ which are appropriately composable, in particular with $X,Y,Z$ defined as above, and $W \in C^\infty([s_2, s_3] \times [t_2, t_3], \V)$ such that $W_{s, t_2} = Y_{s,t_2}$ for all $s \in [s_2, s_3]$ and $W_{s_2,t} = Z_{s_2, t}$ for all $t \in [t_2, t_3]$, 
\begin{align} \label{eq:intro_interchange}
    (X \concat_h Y ) \concat_v (Z \concat_h W) = (X \concat_v Z) \concat_h (Y \concat_v W).
\end{align}

\begin{figure}[!h]
    \includegraphics[width=\linewidth]{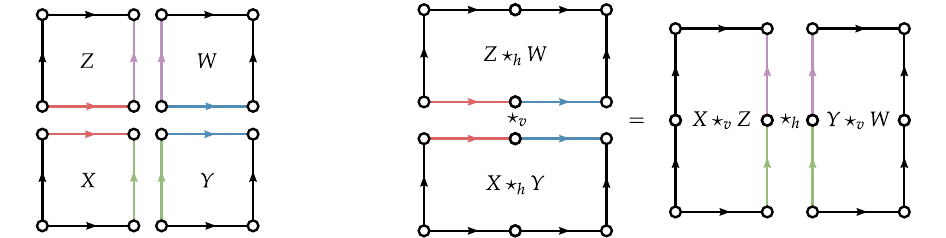}
\end{figure}

We wish to consider the notion of a structure-preserving map for surfaces $\bF$, which is valued in an algebraic structure equipped with two algebraic operations $\hmult$ and $\vmult$ which respect the interchange law, such that horizontal concatenation and vertical concatenation are preserved,
\begin{align} \label{eq:intro_preserves_concatenation}
    \bF(X \concat_h Y) = \bF(X) \hmult \bF(Y), \quad \bF(X \concat_v W) = \bF(X) \vmult \bF(W).
\end{align}
The simplest such structure would be to equip a set $M$ with these two algebraic operations. However, by the classical Eckmann-Hilton argument, this results in a trivial abelian structure.
\begin{theorem}[Eckmann-Hilton]
    Let $M$ be a set equipped with two unital monoid structures, $(M,\hmult, e_h)$ and $(M,\vmult, e_v)$.
    If the interchange law  
    \begin{align}\nonumber
        (a \hmult b) \vmult (c \hmult d) = (a \vmult c) \hmult (b \vmult d) 
    \end{align} 
    holds for all $a,b,c,d \in M$, then both monoid structures are equal, associative, and commutative, in particular,  $a\hmult b= b \hmult a = a \vmult b =b \vmult a $ for all $a,b \in M$.
\end{theorem}
This motivates the use of higher algebraic structures from higher category theory. In particular, we use the notion of a \emph{double group} as our desired codomain. A double group $\dg(\cmG)$ is constructed from two ordinary groups $\cmG = (G_1, G_0)$ which are equipped with some additional algebraic structure to form a \emph{crossed module of groups}. In this article, we will primarily work with groups which are embedded in associative algebras, and consider double groups constructed from a \emph{crossed module of associative algebras} $\cmA = (A_1, A_0)$. Elements of $\dg(\cmA) \subset A_0^4 \times A_1$ are arranged in a square: with four elements of $A_0$ on each of the boundary edges, and one element of $A_1$ in the interior. Our desired structure preserving map for surfaces will be a map 
\begin{align} \label{eq:intro_sh_functor}
    \bF: \surfaces^\infty(\V) \to \dg(\cmA),
\end{align}
which is made up of two components:
\begin{itemize}
    \item a \textbf{path component} $F: \paths^\infty(\V) \to A_0$, which is used to map the boundaries of a surface to the boundaries of a square in $\dg(\cmA)$ valued in $A_0$, and
    \item a \textbf{surface component} $H: \surfaces^\infty(\V) \to A_1$, which maps a surface to the interior of a square in $\dg(\cmA)$. 
\end{itemize}
In particular, we can visualize the map $\bF$ applied to a surface $X \in \surfaces^\infty(\V)$ with boundary paths $x,y,z,w \in \paths^\infty(\V)$ as follows. 

\begin{figure}[!h]
    \includegraphics[width=\linewidth]{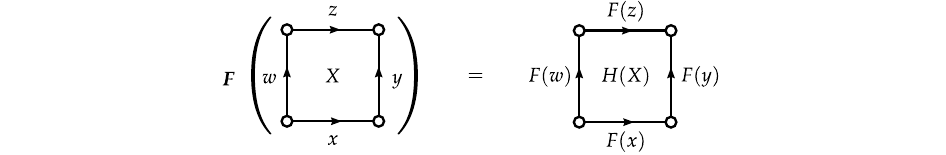}
\end{figure}   

Double groups are equipped with two composition operators, $\hmult$ and $\vmult$. Similar to the case of surfaces, these operations are only well defined when the appropriate boundaries coincide, and these operations satisfy the interchange law (see~\Cref{eq:intro_interchange}). Thus, we say that the map $\bF$ is a \emph{functor} if it preserves the concatenation structure; in other words,~\Cref{eq:intro_preserves_concatenation} holds. \medskip

\emph{Surface holonomy} is a generalization of parallel transport of paths to the setting of surfaces developed in the mathematical physics community~\cite{baez_higher_2006,schreiber_smooth_2011,breen_differential_2005, martins_surface_2011,faria_martins_crossed_2016, yekutieli_nonabelian_2015}, and yields such a functor. A \emph{2-connection} valued in a crossed module of algebras $\cmA = (A_1, A_0)$ consists of an ordinary connection $\cona \in L(\V, A_0)$ along with a linear map $\conc \in L(\Lambda^2 \V, A_1)$, where $\Lambda^2 \V$ is the exterior power\footnote{If $\V$ has a basis $\{e_i\}_{i=1}^d$, then $\{e_i \wedge e_j\}_{i < j}$ forms a basis for $\Lambda^2 \V$.} of $\V$, which satisfies a \emph{fake-flatness condition} (\Cref{def:2connection}). Given a 2-connection $(\cona, \conc)$, the \emph{surface holonomy functor with respect to $(\cona, \conc)$} is a functor $\bF^{\cona, \conc}: \surfaces^\infty(\V) \to \dg(\cmA)$ which consists of a path component $F^\cona: \paths^\infty(\V) \to A_0$ given by parallel transport defined in~\Cref{eq:intro_parallel_transport}, and a surface component $H^{\cona, \conc}: \surfaces^\infty(\V) \to A_1$ which is also given by a differential equation, which will be defined later (\Cref{def:sh}). We provide a more detailed background on surface holonomy in~\Cref{sec:surface_holonomy}.\medskip

Similar to parallel transport of paths, surface holonomy provides a rich, structured description of the space of surfaces up to translation. If we append the parametrization of surfaces\footnote{For a surface $X$, we append the parametrization by considering $\oX_{s,t} = (s,t, X_{s,t})$ valued in $\R^2 \oplus \V$.}, surface holonomy maps
\begin{enumerate}
    \item \textbf{separates parametrized surfaces:} given two distinct surfaces $X, Y \in C^\infty([0,1]^2, \V)$, there exists a 2-connection $(\cona, \conc)$ such that $H^{\cona, \conc}(\oX) \neq H^{\cona, \conc}(\oY)$;
    \item are \textbf{analytically universal:} the span of exponentials of linear functionals $\exp( i \langle \ell, H^{\cona, \conc}(\overline{\cdot})\rangle) : \cK \to \C$ is dense in $C(\cK, \C)$ equipped with the uniform topology, where $\cK \subset C^\infty([0,1]^2, \V)$ is a compact subset; and 
    \item are \textbf{characteristic:} given two distinct measures $\mu, \nu$ valued on the compact subset $\cK$, there exists a 2-connection $(\cona, \conc)$ such that $\Law_{X \sim \mu} H^{\cona, \conc}(\oX) \neq \Law_{Y \sim \nu} H^{\cona, \conc}(\oY)$. 
\end{enumerate}
This was shown in~\cite{lee_random_2023}, where surface holonomy is formulated in terms of Lie groups and Lie algebras, but can be equivalently formulated in terms of associative algebras (see~\Cref{apxsec:matrix_sh}).

\subsection{Contributions}

In this article, we continue the generalization of the path space constructions and extend the three main theorems in~\Cref{ssec:paths}. Our main contributions are:
\begin{enumerate}
    \item An explicit formulation of the \emph{surface signature} in terms of associative algebras, which allows us to show that this surface signature is the universal surface holonomy map, generalizing~\Cref{thm:intro_path_signature_universal}.
    \item An extension theorem for \emph{rough surfaces}, generalizing~\Cref{thm:intro_informal_path_extension}, which allows us to compute the surface signature for arbitrary $\rho$-H\"older surfaces, provided that certain lower level signature terms are provided, in analogy with the extension theorem for rough paths. Combined with the universal property of the surface signature, this allows us to compute surface holonomy for rough surfaces, generalizing~\Cref{eq:intro_rough_path_parallel_transport}.
\end{enumerate}

\subsubsection{Surface Signature.} \label{sssec:intro_surface_signature}
In~\cite{kapranov_membranes_2015}, Kapranov proposed a \emph{universal translation invariant connection} on a finite dimensional vector space $\V$ in terms of crossed modules of Lie algebras, and its associated higher holonomy\footnote{In fact, Kapranov also considers the more general setting of higher holonomy for $n$-dimensional maps such as $[0,1]^n \to \V$, but we exclusively focus on the case of $n=2$ in this article.}. We call the surface holonomy with respect to this universal 2-connection the surface signature. \medskip

Our primary contribution to the study of the surface signature in~\Cref{sec:surface_signature} is a reformulation of the construction in terms of associative algebras, which allows us to explicitly define the surface signature via Picard iterations. By considering a free crossed module of associative algebras, we introduce the higher analogue of the tensor algebra as a quotient
\begin{align}\nonumber
    T_1(\V) \coloneqq \left.\bigoplus_{n=2}^\infty \, \bigoplus_{k=0}^{n-2} \V^{\otimes k} \otimes \Lambda^2 \V \otimes \V^{\otimes(n-k-2)}\right/\Pf,
\end{align}
where $\Pf$ is a linear subspace which we call the \emph{Peiffer ideal}. The parameter $n$ denotes the \emph{level} of elements in $T_1(\V)$, and coincides with the level of the tensor algebra. We show in~\Cref{ssec:relationship_kapranov} that Kapranov's universal 2-connection embeds into our free crossed module of algebras, and this allows us to define the \emph{surface signature} $\hssig: \surfaces^\infty(\V) \to G_1$ into the \emph{group-like elements} $G_1 \subset T_1\ps{\V}$ in the completion of $T_1(\V)$, as
\begin{align}\nonumber
    \hssig_{s,t}(X) = 1 + \sum_{n=1}^\infty \int_{\bs \in [0,1]^n} \int_{\bt \in \Delta^n} \Big(S(x^{s_1,t_1}) \gt dX_{s_1, t_1} \wedge dX_{s_1, t_1}\Big) * \ldots * \Big(S(x^{s_n,t_n}) \gt dX_{s_n, t_n} \wedge dX_{s_n, t_n}\Big) 
\end{align}
for a surface $X \in C^\infty([0,1]^2, \V)$, where $*$ is the product in $T_1(\V)$, and $\gt$ is an action of $G_0$ on $T_1\ps{\V}$. Here, we interpret $dX_{s,t} \wedge dX_{s,t}$ as a differential 2-form valued in $\Lambda^2 \V \subset T_1\ps{\V}$, and $x^{s,t} :[0, s+t] \to \V$ is the \emph{$(s,t)$-tail path} of $X$, defined by
\begin{align}\nonumber
    x^{s,t}_u \coloneqq \left\{ \begin{array}{ll}
        X_{0,u} & : u \in [0,t] \\
        X_{u-t, t} & : u \in [t, t+s].
    \end{array}\right.
\end{align}
In particular, the pair $\cmG = (G_1, G_0)$ of group-like elements in $T_1\ps{\V}$ and $T_0\ps{\V}$ respectively, forms a crossed module of groups, and we obtain the \emph{surface signature functor}
\begin{align} \label{eq:intro_surface_signature_functor}
    \bS: \surfaces^\infty(\V) \to \dg(\cmG),
\end{align}
with path component given by the path signature $S$, and the surface component given by the surface signature $\hssig$.
Moreover, the surface signature is the universal surface holonomy map in the following sense. This is stated formally and proved in~\Cref{thm:surface_signature_universal}.
\begin{theorem}[Informal] \label{thm:intro_surface_signature_universal}
    Let $(\cona, \conc)$ be a continuous 2-connection valued in a crossed module of Banach algebras $\cmA = (A_1, A_0)$. Then, there exists a unique map $\tconc : G_1 \to A_1$ such that for any $X \in \surfaces^\infty(\V)$,
    \begin{align}
        H^{\cona, \conc}(X) = \tconc\left(\hssig(X)\right).
    \end{align}
\end{theorem}

\subsubsection{Rough Surfaces}\label{sssec:intro_rough_surfaces}
We will now go beyond the smooth setting from Kapranov~\cite{kapranov_membranes_2015}, and consider \emph{$\rho$-H\"older surfaces} defined on rectangles $Q = [\smp]\times[\tmp] \subset \R^2$,
\begin{align} \label{eq:intro_holder_surfaces}
    C^\rho(Q, \V) \coloneqq \left\{ X \in C(Q, \V) \, : \, \|X\|_\rho \coloneqq \sup_{\substack{(s_1, s_2) \in \Delta^2(\smp) \\(t_1, t_2)\in \Delta^2(\tmp)}} \frac{\|X_{s_1,t_1} - X_{s_2, t_2}\|}{|(s_1,t_1) - (s_2, t_2)|^\rho} < \infty\right\}.
\end{align}

Z\"ust~\cite{zust_integration_2011} introduced the generalization of 1D Young integration to H\"older forms, which we state in the 2-dimensional setting.
\begin{theorem}{\cite{zust_integration_2011}} \label{thm:zust}
    Let $Q = [\smp] \times [\tmp] \subset \R^2$ be a rectangle. Suppose $f \in C^{\rho_f}(Q, \R)$ and $g^i \in C^{\rho_i}(Q, \R)$ for $i = 1,2$, where $\ovrho = \rho_f + \rho_1 + \rho_2 > 2$. Then, the integral $\int_Q f \, dg^1 \wedge dg^2$ is well-defined and satisfies
    \begin{align}\nonumber
        \left|\int_Q f_{s,t} \, dg^1_{s,t} \wedge dg^2_{s,t}\right| \leq C \diam(Q)^{\ovrho} \|f\|_{\rho_f} \cdot \|g^1\|_{\rho_1} \cdot \|g^2\|_{\rho_2},
    \end{align}
    where $C> 0$ is a constant which depends on $\rho_f$, $\rho_1$, and $\rho_2$. 
\end{theorem}

\Cref{thm:zust} suggests that we can compute the surface signature for $\rho$-H\"older surfaces when $\rho > \frac{2}{3}$, which we call the \emph{Z\"ust regime}. However, this application is not immediate. Instead, we move on to considering rough surfaces, and we will see that this is a corollary of the surface extension theorem. 

\begin{remark}
    We can also consider surfaces $X$ with additional regularity conditions on the 2D increments
    \begin{align}\nonumber
        \square_{s_1, s_2; t_1, t_2}[X] \coloneqq X_{s_1, t_1} - X_{s_1, t_2} - X_{s_2, t_1} + X_{s_2, t_2},
    \end{align}
    which has a different 2D generalization of Young integration~\cite{towghi_multidimensional_2002,harang_extension_2021}, which holds for \emph{rectangular} $\rho$-H\"older surfaces when $\rho > \frac{1}{2}$, which we call the \emph{Young regime}. Our focus in this article is on $\rho$-H\"older surfaces \emph{without} additional rectangular regularity conditions, but we discuss the surface signature for surfaces in the Young regime in~\Cref{ssec:surface_extension_young_zust} and~\Cref{apxsec:surface_extension_rectangular}.
\end{remark}

We define a \emph{$\rho$-rough surface} to be a \emph{multiplicative double group functional} $\rrX: \Delta^2 \times \Delta^2 \to \dg(\cmG^{\gr{\leq N_\rho}})$, where $N_\rho = \lfloor 2\rho^{-1}\rfloor$, which satisfies
\begin{align}\nonumber
    \rrX_{s_1, s_2; t_1, t_2} \hmult \rrX_{s_2, s_3; t_1, t_2} = \rrX_{s_1, s_3; t_1, t_2} \andd \rrX_{s_1, s_2; t_1, t_2} \hmult \rrX_{s_1, s_2; t_2, t_3} = \rrX_{s_1, s_2; t_1, t_3},
\end{align}
and some additional regularity conditions. If $Y \in C^\infty([0,1]^2, \V)$ is a smooth surface, we can construct a $\rrY : \Delta^2 \times \Delta^2 \to \dg(\cmG^{\gr{\leq n}})$ for any $n \in \N$ via the surface signature functor from~\Cref{eq:intro_surface_signature_functor},
\begin{align} \label{eq:intro_mult_dgf_smooth_surface}
    \rrY_{s_1, s_2; t_1, t_2}^{\gr{k}} \coloneqq \bS^{\gr{k}}(Y|_{[s_1, s_2] \times[t_1, t_2]}).
\end{align}
Thus, for an arbitrary $\rho$-rough surface $\rrX$, we \emph{postulate} the surface signature up to level $N_\rho$. Our main result is the following extension theorem for rough surfaces, stated formally in~\Cref{thm:surface_extension}.
\begin{theorem}[Informal] \label{thm:intro_surface_extension}
    Let $\rho \in (0,1]$ and $\rrX: \Delta^2 \times \Delta^2 \to \dg(\cmG^{\gr{\leq N_\rho}})$ be a $\rho$-rough surface. Then, there exists a unique multiplicative double group functional $\trrX: \Delta\times \Delta \to \dg(\cmG)$ which satisfies certain regularity conditions and $\trrX^{\gr{k}} = \rrX^{\gr{k}}$ for $k = 0,1 \ldots, N_\rho$. 
\end{theorem}
In the smooth setting, we show that this unique extension coincides with the functional defined via the surface signature as in~\Cref{eq:intro_mult_dgf_smooth_surface} in~\Cref{cor:smooth_extension_surface_signature}. Thus, we define the surface signature of $\rho$-rough surfaces as
\begin{align}\nonumber
    \bS: \RS^\rho \to \dg(\cmG), \quad \text{by} \quad \bS(\rrX) \coloneqq \trrX_{0,1;0,1},
\end{align}
where $\trrX$ is the unique extension from~\Cref{thm:intro_surface_extension}. For a surface $Y \in C^\rho([0,1]^2, \V)$ in the Z\"ust regime ($\rho > \frac{2}{3}$), we can define a $\rho$-rough surface $\rrY : \Delta^2 \times \Delta^2 \to \dg(\cmG^{\gr{\leq 2}})$ from the surface itself (\Cref{prop:dgf_zust_case}), and thus compute the surface signature without additional information. Furthermore, we define a $\rho$-rough surface metric $d_\rho$ such that $(\RS^\rho, d_\rho)$ is a complete metric space (\Cref{prop:rough_surface_complete}) and the extension is continuous with respect to this metric (\Cref{thm:cont_extension}). Finally, this allows us to compute the surface holonomy of $\rho$-rough surfaces, which is stated formally in~\Cref{thm:surface_holonomy_rough}.

\begin{theorem}[Informal]
    Let $(\cona, \conc)$ be a continuous 2-connection valued in a crossed module of Banach algebras $\cmA = (A_1, A_0)$. Surface holonomy of $\rho$-rough surfaces, denoted $H^{\cona, \conc} : \RS^\rho \to A_1$ is a continuous map defined by
    \begin{align}\nonumber
        H^{\cona, \conc}(\rrX) \coloneqq \tconc \left( \hssig(\rrX)\right),
    \end{align}
    where $\tconc: G_1 \to A_1$ is the unique map from~\Cref{thm:intro_surface_signature_universal}.
\end{theorem}

\subsection{Related Work}
This article extends the notion of the path signature~\cite{chen_integration_1958} and the extension theorem for rough paths~\cite{lyons_differential_1998}. The rough path extension theorem can also be formulated as \emph{sewing lemma}~\cite{gubinelli_controlling_2004,feyel_non-commutative_2008,gerasimovics_non-autonomous_2021}, and our generalization can be viewed as a 2D nonabelian sewing lemma. Various aspects of these two generalizations have previously been studied.

The notion of a rough sheet was introduced in~\cite{chouk_rough_2014} as a generalization of controlled rough paths, and along with~\cite{chouk_skorohod_2015} studied change of variable formulas for Brownian sheets in this context. However, this theory leads to \emph{``... a very complex algebraic structure which is not well understood at the moment''}~\cite{chouk_rough_2014}. From a more geometric perspective,~\cite{alberti_integration_2023,stepanov_towards_2021} studied integration of rough differential forms via sewing lemmas, motivated by Z\"ust's generalization of the 1D Young integral. Furthermore, abelian multiparameter sewing lemmas have been studied in both the deterministic~\cite{harang_extension_2021} and stochastic~\cite{bechtold_multiparameter_2023} settings. However, these works do not approach the problem by generalizing the notion of the signature itself.

Recently, several notions of a multiparameter signature have been developed. In~\cite{giusti_topological_2022}, a topological/geometric signature based on Jacobian minors was developed for maps from arbitrary cubical domains $[0,1]^d$, while in~\cite{diehl_two-parameter_2022}, a discrete variant is introduced based on discrete 2D increments, motivated by the algebraic perspective of quasisymmetric functions. In~\cite{diehl_signature_2024}, a combination of the two above approaches is considered by developing a continuous 2D signature based on continuous 2D increments. Furthermore, these notions have recently been applied to image and texture classification in~\cite{zhang_two-dimensional_2022}.

While these notions of multiparameter signatures share some analytic and algebraic properties of the path signature, they all fail to be multiplicative with respect to horizontal/vertical concatenation, and are only studied in the Lipschitz setting. Surface holonomy was originally developed as a generalization of parallel transport to a higher dimensional setting~\cite{baez_higher_2006,schreiber_smooth_2011,breen_differential_2005, martins_surface_2011,faria_martins_crossed_2016, yekutieli_nonabelian_2015}. In~\cite{lee_random_2023} it was shown that surface holonomy provides a multiplicative approach to studying functions and measures on the space of surfaces in the Young regime. 

\begin{remark}
    While this article was being prepared, the author was notified of~\cite{chevyrev_multiplicative_2024} which independently achieved similar results as this article. Both of these articles are motivated by Kapranov's notion of a universal translation-invariant surface holonomy~\cite{kapranov_membranes_2015}, but focus on different and complementary aspects of this problem.
    
    In terms of algebraic aspects,~\cite{chevyrev_multiplicative_2024} works primarily in the Lie algebraic setting introduced by Kapranov and develops the notion of a 2D Magnus expansion for surfaces, analogous to the notion of the \emph{log signature} of a path. Here, we develop an approach based on associative algebras, and show that Kapranov's crossed module of Lie algebras embeds into our construction (\Cref{prop:inclusion_kapranov}). We define the surface signature in terms of associative algebras, analogous to the \emph{signature} of a path.

    In terms of analytic aspects, both articles provide a 2D nonabelian sewing lemma (or extension theorem) which allows us to consider the log signature (in~\cite{chevyrev_multiplicative_2024}) and signature (here) of surfaces with arbitrary $\rho$-H\"older regularity, provided certain (log-) signature terms are postulated. In this article, we explicitly consider the factorial decay of the terms in the extension theorem, which ensures that the surface signature of a rough surface is a bounded object. This allows us to prove a global universal property of the surface signature in the smooth setting, and we use this to define surface holonomy with respect to bounded $2$-connections in the rough setting. In~\cite{chevyrev_multiplicative_2024}, studying this decay is not possible in the Lie algebra setting, and thus can only consider a local universal property for smooth surfaces parametrized on sufficiently small domains. However, this allows~\cite{chevyrev_multiplicative_2024} to use a simpler approach to proving the extension theorem, which does not use a section $s$ of the boundary map $\delta$, which is required to obtain the desired bounds here.
\end{remark}

{\small
\textbf{Acknowledgements.}
The author would like to thank Chad Giusti, Vidit Nanda, and Harald Oberhauser for several helpful discussions at the beginning of this project.
Additionally, the author would like to thank Ilya Chevyrev for the suggestion to consider the universal locally $m$-convex algebras from~\cite{chevyrev_characteristic_2016}.
The author would also like to thank the organizers, Kurusch Ebrahimi-Fard and Fabian Harang, of the \emph{Signatures of Paths and Images} workshop, where this work was first presented. 
DL was supported by the Hong Kong Innovation and Technology Commission (InnoHK Project CIMDA).} 

\section{The Path Signature} \label{sec:path_signature}

In this section, we review the construction of the path signature from the perspective of universal properties. In particular, our aim is to show that the path signature is the universal parallel transport map. We will introduce some related concepts along the way and discuss basic properties. We work exclusively with piecewise smooth paths $\paths^\infty(\V)$ equipped with the composition from~\Cref{eq:intro_path_concatenation}. We will generalize these constructions in~\Cref{sec:surface_signature} to the case of surfaces, where we mimic the structure of this section.\medskip

\textbf{Universal Connection.} Let $\V$ be a finite dimensional vector space and $T_0(\V)$ be the free associative algebra (tensor algebra),
\begin{align}\nonumber
    T_0(\V) \coloneqq \bigoplus_{n=0}^\infty \V^{\otimes n}.
\end{align}
Let $(A_0, \cdot, 1)$ be a unital associative algebra and $\cona \in L(\V, A_0)$. Then, by the universal property of the tensor algebra, there exists a unique algebra morphism $\tcona: T_0(\V) \to A_0$ such that
\[
    \begin{tikzcd}
        V \ar [r, "\cona"] \ar[d, swap,"\zeta"] & A_0 \\
        T_0(\V) \ar[ur, swap, dashed, "\tcona"]&
    \end{tikzcd}
\]
commutes. Here, $\zeta \in L(V, T_0(\V))$ is the canonical inclusion of $V \hookrightarrow T_0(\V)$. In the geometric language of~\Cref{ssec:paths}, we call $\zeta$ the \emph{universal (translation-invariant) connection}\footnote{All connections we consider are translation-invariant, and we will simply call them connections.}. \medskip

\textbf{Free Lie Algebra.}
We can also consider this universal connection from the perspective of Lie algebras. Let $\fg_0$ be the free Lie algebra with respect to $\V$. The free Lie algebra also satisfies an analogous universal property. Let $\fh$ be a Lie algebra, and $\cona \in L(V, \fh)$. Then, there exists a unique Lie algebra morphism $\tcona: \fg_0 \to \fh$ such that
\[
    \begin{tikzcd}
        V \ar [r, "\cona"] \ar[d, swap,"\zeta"] & \fh \\
        \fg_0 \ar[ur, swap, dashed, "\tcona"]&
    \end{tikzcd}
\]
commutes. Here, $\zeta \in L(\V, \fg_0)$ is the canonical inclusion of $V \hookrightarrow \fg_0$. However, we note that $T_0(\V)$ is the universal enveloping algebra of $\fg_0$, and we can define an inclusion of Lie algebras
\begin{align}\nonumber
    \iota_0 : \fg_0 \hookrightarrow \Lie(T_0(\V)),
\end{align}
where $\Lie(T_0(\V))$ is the Lie algebra on $T_0(\V)$ equipped with the commutator bracket. \medskip

\textbf{Norms.}
Suppose $\V$ is a finite dimensional Hilbert space with orthonormal basis $\{e_1, \ldots, e_d\}$, and extend the Hilbert structure to $\V^{\otimes n}$ by taking $\{e_{i_1} \cdots e_{i_n}\}_{i_j \in [d]}$ to be an orthonormal basis. Note that we have the property that for $a^{\gr{n}} \in V^{\otimes n}$ and $b^{\gr{k}} \in V^{\otimes k}$, we have 
\begin{align}\nonumber
    \|a^{\gr{n}} \cdot b^{\gr{k}}\| \leq \| a^{\gr{n}}\| \cdot \|b^{\gr{k}}\|.
\end{align}

\textbf{Truncation.}
The \emph{level $n$ truncated tensor algebra} is
\begin{align}\nonumber
    T^{\gr{\leq n}}_0(\V) \coloneqq \bigoplus_{k=0}^n V^{\otimes k},
\end{align}
and equipped with the norms above, $T^{\gr{\leq n}}_0(\V)$ is a Banach algebra. We denote the truncation map by $\pi_0^{\gr{n}} : T_0(\V) \to T^{\gr{\leq n}}_0(\V)$.\medskip

\textbf{Group-Like Elements.} We can also consider truncations of the Lie algebra $\fg_0$. We define the lower central series of $\fg_0$ recursively as
\begin{align}\nonumber
    \LCS_1(\fg_0) \coloneqq \fg_0 \andd \LCS_r(\fg_0) \coloneqq [\fg_0, \LCS_{r-1}(\fg_0)].
\end{align}
By definition, the elements of $\LCS_n(\fg_0)$ are all the elements of at least level $n$.
Then, we define the $n$-truncated nilpotent Lie algebras $\fg_0^{\gr{\leq n}} $by
\begin{align} \label{eq:truncated_fg0}
   \fg_0^{\gr{\leq n}} \coloneqq \fg_0/\LCS_{n+1}(\fg_0).
\end{align}
Let $T^{\gr{\leq n}, 0}_0(\V),\, T^{\gr{\leq n}, 1}_0(\V) \subset T^{\gr{\leq n}}_0(\V)$ be the elements with constant term $0$ and $1$ respectively, and define the exponential $\exp: T^{\gr{\leq n}, 0}_0(\V) \to T^{\gr{\leq n}, 1}_0(\V)$ and logarithm $\log: T^{\gr{\leq n}, 1}_0(\V) \to T^{\gr{\leq n}, 0}_0(\V)$ maps by 
\begin{align} \label{eq:explog0}
\exp(x) \coloneqq \sum_{k=0}^n \frac{x^{\cdot k}}{k!} \andd \log(y) \coloneqq \sum_{k=1}^n (-1)^{k+1} \frac{y^{\cdot k}}{k}
\end{align}
respectively. Then, the \emph{level $n$ group-like elements} is the Lie group associated to $\fg_0^{\gr{\leq n}}$ defined by
\begin{align} \label{eq:truncated_G0}
    G_0^{\gr{\leq n}} \coloneqq \{ \exp(x) \in T_0^{\gr{\leq n}} \, : \, x \in \fg_0^{\gr{\leq n}} \}.
\end{align}

\textbf{Completion.}
Let $T_0\ps{\V}$ be the completion of $T_0(\V)$ as formal noncommutative power series
\begin{align}\nonumber
    T_0\ps{\V} \coloneqq \prod_{n=0}^\infty \V^{\otimes n}.
\end{align}
However, we will require a topology our algebras in order to discuss analytic properties. Following~\cite[Section 2]{chevyrev_characteristic_2016}, we equip $T_0(\V)$ with the topology generated by the norms
\begin{align} \label{eq:plambda_norm}
    p_\lambda(x) \coloneqq \sum_{n=0}^\infty \lambda^n \|x^{\gr{n}}\|,
\end{align}
for all $\lambda > 0$, where $x = (x^{\gr{n}})_{n=0}^\infty$ with $x^{\gr{n}} \in V^{\otimes n}$. The completion of $T_0(\V)$ with respect to this topology is the locally $m$-convex algebra denoted $E_0(\V)$, and can be expressed as~\cite[Corollary 2.5]{chevyrev_characteristic_2016}
\begin{align} \label{eq:E0_characterization}
    E_0(\V) = \left\{ a \in T\ps{\V} \, : \, \sum_{k=0}^\infty \lambda^k \|a^{\gr{k}}\| < \infty \, \text{ for all } \, \lambda > 0\right\}.
\end{align}
In particular, this has the following universal property.
\begin{corollary}[\cite{chevyrev_characteristic_2016}]\label{cor:E0_universal}
    Let $A_0$ be a Banach algebra, and $\cona \in L(\V, A_0)$ be continuous. Then, there exists a unique continuous map $\tcona: E_0(\V) \to A_0$ such that $\cona = \tcona \circ \zeta$, where $\zeta \in L(\V, E_0(\V))$ is the inclusion.
\end{corollary}
\noindent We define the \emph{group-like elements} to be 
\begin{align} \label{eq:G0_def}
    G_0 \coloneqq \left\{ x \in E_0(\V) \, : \, \pi^{\gr{n}}(x) \in G^{\gr{\leq n}}_0 \, \text{ for all } n \in \N\right\}.
\end{align}

\textbf{The Path Signature.}
We can consider parallel transport of $x \in C^\infty([\tmp], \V)$ with respect to $\zeta$ using~\Cref{eq:intro_parallel_transport} as
\begin{align}\nonumber
    \frac{d S_t(x)}{dt} = S_t(x) \cdot \frac{dx_t}{dt}, \quad S_{\tmm}(x) = 1,
\end{align}
where we use $\cdot$ to denote the tensor product. The solution is called the \emph{path signature of $x$}, denoted by $S(x) \coloneqq S_{\tpp}(x)$, and defines a map
\begin{align}\nonumber
    S: \paths^\infty(\V) \to T_0\ps{\V}.
\end{align}
In fact, due to the factorial decay, the path signature is valued in $E_0(\V)$, and furthermore, can be shown to be valued in the group-like elements. 
\begin{corollary}
    For any $x \in \paths^\infty(\V)$, we have $S(x) \in G_0$.
\end{corollary}
By considering Picard iterations, the path signature can be expressed as the infinite sum of iterated integrals, $S(x) = (S^{\gr{n}}(x))_{n=0}^\infty$, where $S^{\gr{0}}(x) = 1$ and 
\begin{align}\nonumber
    S^{\gr{n}}(x) = \int_{\Delta^n(\tmp)} dx_{t_1} \cdot dx_{t_2} \cdots dx_{t_n} \in \V^{\otimes n},
\end{align}
where 
\begin{align}\nonumber
    \Delta^n(\tmp) \coloneqq \{ \tmm \leq t_1 < \ldots < t_n \leq \tpp\}.
\end{align}
Furthermore, the signature preserves the compositional structure of paths, which is often called \emph{Chen's identity}: if $x \in C^\infty([t_1, t_2], \V)$ and $y \in C^\infty([t_2, t_3], \V)$ are composable, then
\begin{align}\nonumber
    S(x \concat y) = S(x) \cdot S(y).
\end{align}

\textbf{Universal Parallel Transport.}
The path signature can be viewed as the \emph{universal parallel transport map on $\V$}, in the following sense.

\begin{theorem} \label{thm:universal_path_signature}
    Let $(A_0, \cdot, 1)$ be a Banach algebra, and $\cona \in L(\V, A_0)$ be continuous. Let $\tcona: E_0(\V) \to A_0$ be the unique algebra morphism from~\Cref{cor:E0_universal}. Then, the parallel transport map $F^\cona : \paths^\infty(\V) \to A_0$ from~\Cref{eq:intro_parallel_transport} satisfies
    \begin{align}\nonumber
        F^{\cona}(x) = \tcona(S(x)).
    \end{align}
\end{theorem}
\begin{proof}
    By considering Picard iterations of~\Cref{eq:intro_parallel_transport}, we obtain
    \begin{align}\nonumber
        F^\cona(x) = 1 + \sum_{n=1}^\infty \int_{\Delta^n} \cona(dx_{t_1}) \cdots \cona(dx_{t_n}) = 1 + \sum_{n=1}^\infty \tcona( S^{\gr{n}}(x)) = \tcona(S(x)).
    \end{align}
\end{proof}

\section{Surface Holonomy} \label{sec:surface_holonomy}

In this section, we provide an introduction to the notion of \emph{surface holonomy}, a generalization of parallel transport of paths to surfaces. As discussed in~\Cref{ssec:surfaces}, the space of piecewise smooth surfaces
\begin{align}\nonumber
    \surfaces(\V) \coloneqq \bigcup_{s_1 \leq s_2, \, t_1 \leq t_2} C^\infty([s_1, s_2] \times [t_1, t_2], \V).
\end{align}
parametrized by rectangular domains is equipped with two distinct partially-defined algebraic operations: horizontal and vertical concatenation defined in~\Cref{eq:intro_surf_hconcat} and~\Cref{eq:intro_surf_vconcat} respectively. This structure can be encoded in a structure called a \emph{double category}. However, we aim to minimize category-theoretic notions in the main text, and discuss this in~\Cref{apxsec:categorical_structure} for the interested reader. 

Surface holonomy is a map on the space of surfaces which is compatible with both horizontal and vertical concatenation. 
The codomain of the surface holonomy is a \emph{double group}, an algebraic structure equipped with partially defined horizontal and vertical compositions, similar to the structure of surfaces. 
Double groups can be encoded using structures called \emph{crossed modules of groups}, which are convenient for computational purposes. We will then consider \emph{crossed modules of algebras} and their associated double groups. Finally, we provide an introduction to surface holonomy valued in algebras. For a more detailed discussion of surface holonomy, we refer the reader to~\cite{lee_random_2023,faria_martins_crossed_2016,martins_surface_2011}.

\subsection{Crossed Modules and Double Groups}      

In this section, we provide the necessary background on crossed modules and double groups. In particular, we will only define double groups with respect to a crossed module. The interested reader is referred to~\Cref{apxsec:categorical_structure}, \cite[Section 4 and Appendix C]{lee_random_2023} and~\cite{brown_nonabelian_2011} for further details on the category theoretic picture. 

\begin{definition}
    \label{def:GCM}
        A \emph{crossed module of groups},
        \begin{align*}
            \cmG = \left(\cmb: G_1\rightarrow G_0,\quad \gt: G_0 \rightarrow \Aut(G_1) \right)
        \end{align*}
        is given by two groups $(G_0, \cdot), (G_1, *)$, a group morphism $\cmb: G_1 \rightarrow G_0$ and a left action of $G_0$ on $G_1$ (denoted elementwise by $g \gt : G_1 \rightarrow G_1$ for $g \in G$) which is a group morphism, such that
        \begin{enumerate}
            \item[(CM1)] \textbf{(First Peiffer relation)} $\cmb(g \gt E) = g \cdot \cmb(E) \cdot g^{-1}$; for $g \in G_0$ and $E \in G_1$
            \item[(CM2)] \textbf{(Second Peiffer relation)} $\cmb(E_1) \gt (E_2) = E_1 * E_2 *E_1^{-1}$; for $E_1, E_2 \in G_1$.
        \end{enumerate}
        A \emph{crossed module of Lie groups} is the same as above, except $G$ and $H$ are Lie groups, and all morphisms are smooth.
        Given another crossed module $\cmH = (\cmb: H_1 \to H_0, \gt)$, a \emph{morphism of crossed modules} $f = (f_1, f_0) : \cmG \to \cmH$ consists of group homomorphisms $f_0: G_0 \to H_0$ and $f_1 : G_1 \to H_1$ such that for all $g \in G_0$ and $E \in G_1$, we have
        \begin{align}\nonumber
            \delta \circ f_1(E) = f_0 \circ \delta(E) \andd f_1(g \gt E) = f_0(g) \gt f_1(E)
        \end{align}
\end{definition}

Given a crossed module of groups, we define the associated notion of a double group.

\begin{definition}
    Let $\cmG = (\cmb: G_1 \to G_0, \gt)$ be a crossed module of groups. The \emph{double group\footnote{More precisely, these are the \emph{squares} or \emph{2-morphisms} of a double group.} of $\cmG$} is
    \begin{align} \label{eq:dg_squares}
        \dg(\cmG) \coloneqq \left\{ S = (x,y,z,w, E) \in G_0^4 \times G_1 \, : \, \cmb(E) = x \cdot y \cdot z^{-1} \cdot w^{-1} \right\},
    \end{align}
    where the elements are arranged in a square as follows.
    \begin{figure}[!h]
        \includegraphics[width=\linewidth]{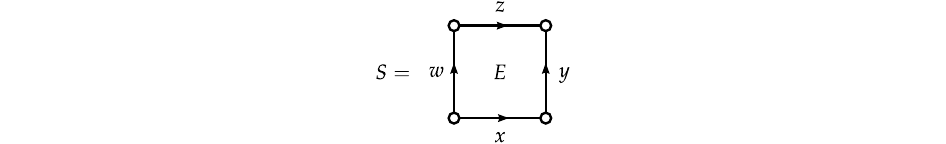}
    \end{figure}       

    Let $S = (x,y,z,w, E), S' = (x', y', z', y, E') \in \dg(\cmG)$. We say that $S$ and $S'$ are \emph{horizontally composable} if $y = w'$, and define the \emph{horizontal composition} to be 
    \begin{align} \label{eq:dg_hmult}
         S \hmult S' = (x\cdot x', y', z \cdot z', w, (x \gt E') * E).
    \end{align}
    Similarly, we say $S$ and $S'$ are \emph{vertically composable} if $z = x'$, and define \emph{vertical composition} to be
    \begin{align} \label{eq:dg_vmult}
        S \vmult S' = (x, y\cdot y', z', w \cdot w', E * (w \gt E')).
    \end{align}
    Furthermore, the \emph{interchange law} holds: for appropriately composable squares $X, Y, Z, W \in \dg(\cmG)$,
    \begin{align} \label{eq:dg_interchange}
        (X \hmult Y) \vmult (Z \hmult W) = (X \vmult Z) \hmult (Y \vmult W).
    \end{align}
    Let $e_0 \in G_0$ and $e_1 \in G_1$ denote the respective group identities. Then, for $x \in G_0$, we define the horizontal and vertical identities by
    \begin{align}\nonumber
        1^h_x \coloneqq (e_0, x, e_0, x, e_1) \andd 1^v_x \coloneqq (x, e_0, x, e_0, e_1),
    \end{align}
    which satisfy
    \begin{align}\nonumber
        1^h_x \hmult S = S, \quad S \hmult 1^h_x  = S, \quad 1^v_x \vmult S = S, \andd S \vmult 1^v_x = S,
    \end{align}
    whenever $S$ is appropriately composable with the identity. 
    Finally, for a square $S = (x,y,z,w,E) \in \dg(\cmG)$, the horizontal and vertical inverses are given by
    \begin{align}\nonumber
        S^{-h} \coloneqq (x^{-1}, w, z^{-1}, y, x^{-1} \gt E^{-1}) \andd S^{-v} \coloneqq (z, y^{-1}, x, w^{-1}, w^{-1} \gt E^{-1}),
    \end{align}
    which satisfies
    \begin{align}\nonumber
        S \hmult S^{-h} = 1^h_w, \quad S^{-h} \hmult S^h = y, \quad S \vmult S^{-v} = 1^v_x, \andd S^{-v} \vmult S = 1^v_y.
    \end{align}
\end{definition}

There are two main consistency properties that one must check. The first is that the boundary condition in the definition of squares in~\Cref{eq:dg_squares} holds for horizontal and vertical compositions, and the second is that the interchange law holds. These properties hold due to the first and second Peiffer relations respectively, and further details can be found in~\cite[Section 6.6]{brown_nonabelian_2011}.

\subsection{Crossed Modules of Algebras}
In this section, we provide a brief introduction to crossed modules of associative algebras, for further details, see~\cite{wagemann_crossed_2021,faria_martins_crossed_2016}.
We start with the definition of an algebra bimodule, see~\cite[Section 1.1.5]{loday_algebraic_2012}
\begin{definition}
    Let $(A, \cdot)$ be an algebra. An $A$-bimodule is a vector space $M$ equipped with a left and right action of $A$ on $M$, denoted $\gtd$ and $\ltd$ respectively such that for any $a_1, a_2 \in A$ and $m \in M$, we have
    \begin{align}\nonumber
        (a_1 \cdot a_2) \gtd m = a_1 \gtd (a_2 \gtd m), \quad m \ltd (a_1 \cdot a_2) = (m \ltd a_1) \ltd a_2, \andd (a_1 \gtd m) \ltd a_2 = a_1 \gtd (m \ltd a_2).
    \end{align}
    If $A$ is a unital algebra with unit $1$, then we also require
    \begin{align}\nonumber
        1 \gtd m = m \andd m \ltd 1 = m.
    \end{align}
\end{definition}

\begin{definition}
    A \emph{crossed module of algebras}
    \begin{align}\nonumber
        \cmA = (\delta: A_1 \to A_0, \gtd, \ltd)
    \end{align}
    consists of a unital algebra $(A_0, \cdot, 1)$ and a non-unital algebra $(A_1, *)$ equipped with left and right actions, $\gtd$ and $\ltd$, of $A_0$ on $A_1$ such that $A_1$ is an $A_0$-bimodule. These actions are compatible with the product on $A_1$, such that for any $a \in A_0$ and $E,E' \in A_1$, we have
    \begin{align}\nonumber
        a \gtd (E* E') = (a \gtd E) * E' \andd (E* E') \ltd a = E * (E' \ltd a),
    \end{align}
    and the left and right actions are fully interchangeable,
    \begin{align}\nonumber
        (E \ltd a) * (a' \gtd E') = E * ((a \cdot a') \gtd E') = (E \ltd (a\cdot a')) * E'.
    \end{align}
    Furthermore, there is an algebra morphism $\delta: A_1 \to A_0$ such that
    \begin{enumerate}
        \item \textbf{(First Peiffer Identity)} $\delta (a \gtd E) = a \cdot \delta(E)$ and $\delta(E \ltd a) = \delta(E) \cdot a$ for any $a \in A_0$ and $E \in A_1$, and
        \item \textbf{(Second Peiffer Identity)} $\delta(E) \gtd E' = E* E'$ and $ E \ltd \delta(E') = E*E'$ for any $E,E' \in A_1$. 
    \end{enumerate}
    Given another crossed module of algebras $\bB = (\cmb: B_1 \to B_0, \gtd, \ltd)$ a \emph{morphism of crossed modules of algebras} $f = (f_1, f_0) : \cmA \to \cmB$ consists of algebra morphisms $f_0 : A_0 \to B_0$ and $f_1 : A_1 \to B_1$ such that for all $a \in A_0$ and $E \in A_1$, 
    \begin{align}\nonumber
        \delta \circ f_1(E) = f_0 \circ d(E), \quad f_1(a \gtd E) = f_0(a) \gtd f_1(E), \andd f_1(E \ltd a) = f_1(E) \ltd f_0(a).
    \end{align}
\end{definition}

In this definition, $A_1$ is assumed to be non-unital; however, we will often require both algebras in the crossed module to be unital. We can simply adjoin a unit to $A_1$ (which we also denote this as $1$, and generally do not distinguish between the identity element of $A_0$ and $A_1$) and define
\begin{align}\nonumber
    \hA_1 \coloneqq \R \oplus A_1.
\end{align}
Multiplication in $\hA_1$ is defined for $(\lambda_1, a_1), (\lambda_2, a_2) \in \hA_1$ by
\begin{align}\nonumber
    (\lambda_1, a_1) * (\lambda_2, a_2) = (\lambda_1 \lambda_2, \lambda_1 a_2 + \lambda_2 a_1 + a_1 * a_2).
\end{align}
We can then extend $\delta$ by $\delta(1) = 1$. However, there is no canonical way to define the action $a \gtd 1$, such that the boundary map satisfies (by the first Peiffer identity)
\begin{align}\nonumber
    \delta(a \gtd 1) = a \cdot \delta(1) = a.
\end{align}
Instead, by considering the group of units of a unital algebra, we obtain the structure required for surface holonomy. For a unital algebra $\hA$, we let $\hA^{\units}$ denote the group of invertible elements in $\hA$. Consider a crossed module of algebras $\cmA = (\delta: A_1 \to A_0, \gtd, \ltd)$. We define an action $\gt$ of $A_0^{\units}$ on $\hA_1$ by
\begin{align} \label{eq:unit_action}
    a \gt (\lambda, b) \coloneqq (\lambda, a \gtd b \ltd a^{-1}).
\end{align}

\begin{proposition}{\cite[Proposition 8]{faria_martins_crossed_2016}} \label{prop:cma_invertible}
    Let $\cmA = (\delta: A_1 \to A_0, \gtd, \ltd)$ be a crossed module of algebras. Then, the action $\gt$ defined in~\Cref{eq:unit_action} of $A_0^{\units}$ on $\hA_1$ acts by unital algebra isomorphisms. Furthermore,
    \begin{align}\nonumber
        \hcmA^{\units} \coloneqq (\delta: \hA_1^{\times} \to A_0^{\times}, \gt)
    \end{align}
    is a crossed module of groups. 
\end{proposition}

\subsection{Surface Holonomy}

In this section, we will consider algebra-valued surface holonomy. This has previously been considered in~\cite{faria_martins_crossed_2016}, which was called \emph{bare holonomy}, but our formulation is slightly different since we use a cubical formulation rather than a globular one. 

\begin{definition} \label{def:2connection}
    Let $\V$ be a vector space and $\cmA = (\delta : A_1 \to A_0, \gtd, \ltd)$ be a crossed module of algebras. A \emph{(translation-invariant, algebra-valued, fake-flat) 2-connection on $\V$ valued in $\cmA$} is a pair $(\cona, \conc)$ of linear maps $\cona \in L(\V, A_0)$ and $\conc \in L(\Lambda^2 \V, A_1)$ such that the \emph{fake-flatness condition}
    \begin{align} \label{eq:fake_flatness}
        \delta \circ \conc = [\cona, \cona]_0
    \end{align}
    holds, where the bracket $[\cdot, \cdot]_0$ is the commutator in $A_0$. 
\end{definition}

Now, we will define the notion of surface holonomy for algebra-valued 2-connections.
\begin{definition} \label{def:sh}
    Let $(\cona, \conc)$ be a 2-connection valued in a crossed module $\cmA$. Let $X \in C^\infty([\smp]\times[\tmp], \V)$ be a smooth surface, and we define the differential equation for $\hH^{\cona, \conc}_{s,t}(X) : [\smp]\times[\tmp] \to \hA_1^{\units}$ by
    \begin{align}\nonumber
        \frac{\partial \hH^{\cona, \conc}_{s,t}(X)}{\partial t} = \hH^{\cona, \conc}_{s,t}(X) * \int_{\smm}^s F^{\cona}(x^{s',t}) \gt \conc \left( \frac{\partial X_{s',t}}{\partial s'}, \frac{\partial X_{s',t}}{\partial t}\right) \, ds', \quad \quad \hH^{\cona, \conc}_{s,\tmm}(X) = 1_1,
    \end{align}
    where $x^{s,t} = [0, (s- \smm) + (t - \tmm)] \to \V$ is the \emph{$(s,t)$-tail path} of $X$ defined by
    \begin{align} \label{eq:tail_path}
        x^{s,t}_u \coloneqq \left\{ \begin{array}{ll}
            X_{u+\smm,\tmm} & : u \in[0, s-\smm] \\
            X_{s, (u-s+\smm) + \tmm} & : u \in [s- \smm, (s-\smm) + (t- \tmm)].
        \end{array}\right.
    \end{align}
    Let $\hH^{\cona, \conc}(X) \coloneqq \hH^{\cona, \conc}_{\spp, \tpp}(X)$ be the \emph{surface holonomy of $X$ with respect to $(\cona, \conc)$}.
\end{definition}

This definition only defines the surface holonomy of a surface as an element in $\hA_1^\units$. However, in order to consider horizontal and vertical composition structures, we define the \emph{surface holonomy functor} into $\dg(\cmA^\units)$. 

\begin{definition} \label{def:sh_functor}
    Let $(\cona, \conc)$ be a 2-connection valued in a crossed module $\cmA$. The \emph{surface holonomy functor} $\bF^{\cona, \conc} : \surfaces^\infty(\V) \to \dg(\cmA^\units)$ is defined for a surface $X \in \surfaces^\infty(\V)$ with boundary paths $x,y,z,w \in \paths^\infty(\V)$ (see the figure after~\Cref{eq:intro_sh_functor}) by
    \begin{align}\nonumber
        \bF^{\cona, \conc}(X) \coloneqq (F^\cona(x), \, F^\cona(y), \, F^\cona(z), \, F^{\cona}(w), \, \hH^{\cona,\conc}(X)).
    \end{align}
\end{definition}

The fact that the surface holonomy functor is indeed valued in the double group (it satisfies the boundary condition in~\Cref{eq:dg_squares}) and preserves horizontal/vertical compositions is well-known~\cite{martins_surface_2011,faria_martins_crossed_2016,schreiber_smooth_2011,baez_higher_2006,yekutieli_nonabelian_2015}. However, due to the importance of multiplicativity in the present article, we will provide a proof of this in our setting, which is adapted from~\cite[Lemmas 2.23, 2.25]{martins_surface_2011}.

\begin{proposition} \label{prop:sh_multiplicative}
    Let $X \in C^\infty([s_1, s_2] \times [t_1, t_2], \V)$, $Y \in C^\infty([s_2, s_3] \times [t_1, t_2], \V)$ and $Z \in C^\infty([s_1, s_2] \times [t_2, t_3], \V)$ such that $X$ and $Y$ are horizontally composable, and $X$ and $Z$ are vertically composable. Let $u = \bdy_b(X)$ and $v = \bdy_l(X)$ be the bottom and left boundary paths of $X$ respectively. Then,
    \begin{align}
        \hH^{\cona, \conc}(X \concat_h Y) &= (F^\cona(u) \gt \hH^{\cona, \conc}(Y)) * \hH^{\cona, \conc}(X) \label{eq:sh_vfunc}\\
        \hH^{\cona, \conc}(X \concat_v Z) &= \hH^{\cona, \conc}(X) * (F^\cona(v) \gt \hH^{\cona, \conc}(Z)). \label{eq:sh_hfunc}
    \end{align}
\end{proposition}
\begin{proof}
    We will show both of these properties by showing that both sides satisfy the same differential equations. For the vertical composition, let
    \begin{align}\nonumber
        B_t = \left\{ \begin{array}{ll}
            \hH^{\cona, \conc}_{s_2, t}(X)  & t \in [t_1, t_2] \\
            \hH^{\cona, \conc}(X) * (F^\cona(v) \gt \hH_{s_2, t}^{\cona, \conc}(Z)) & t \in [t_2, t_3]
        \end{array}\right.
    \end{align}
    be the right hand side of~\Cref{eq:sh_hfunc} as a function of $t$. For $t \in [t_1, t_2]$, we have
    \begin{align}\nonumber
        \frac{\partial B_t}{\partial t} = \frac{\partial \hH^{\cona, \conc}_t(X \concat_v Z)}{\partial t}
    \end{align}
    by definition. Then for $t \in [t_2, t_3]$, we have
    \begin{align}\nonumber
        \frac{\partial B_t}{\partial t} &= \hH^{\cona, \conc}(X) * \left(F^\cona(v) \gt \frac{\hH_{s_2, t}^{\cona, \conc}(Z)}{\partial t}\right) \\
        &= \hH^{\cona, \conc}(X) * \left(F^\cona(v) \gt \hH_{s_2, t}^{\cona, \conc}(Z)\right) *  \int_{s_1}^{s_2} \left( F^\cona(v) \gt F^{\cona}(z^{s,t}) \gt \conc \left( \frac{\partial Z_{s,t}}{\partial s}, \frac{\partial Z_{s,t}}{\partial t}\right) \, ds\right) \nonumber\\
        & = \frac{\partial \hH^{\cona, \conc}_t(X \concat_v Z)}{\partial t}\nonumber
    \end{align}
    where we use the fact that $G_0$ acts by group homomorphisms in the second line, and the fact that $v \concat z^{s,t}$ is the tail path in the differential equation for $\hH^{\cona, \conc}_t(X \concat_v Z)$. Thus, $B_t = \hH^{\cona, \conc}_t(X \concat_v Z)$ since they satisfy the same differential equation.\medskip

    For the horizontal case, let 
    \begin{align}\nonumber
        C_t = (F^\cona(u) \gt \hH_{s_3, t}^{\cona, \conc}(Y)) * \hH_{s_2, t}^{\cona, \conc}(X)
    \end{align}
    be the right hand side of~\Cref{eq:sh_vfunc} as a function of $t$. Then, we have
    \begin{align}\nonumber
        \frac{\partial C_t}{\partial t}  =& \left(F^\cona(u) \gt \frac{\partial \hH_{s_3, t}^{\cona, \conc}(Y)}{\partial t}\right) * \hH_{s_2, t}^{\cona, \conc}(X) + (F^\cona(u) \gt \hH_{s_3, t}^{\cona, \conc}(Y)) * \frac{\partial \hH_{s_2, t}^{\cona, \conc}(X)}{\partial t}\\
         =& \left(F^\cona(u) \gt \hH_{s_3, t}^{\cona, \conc}(Y)\right) * \left(\int_{s_2}^{s_3}F^\cona(u \concat y^{s,t})  \gt \conc \left( \frac{\partial Y_{s,t}}{\partial s}, \frac{\partial Y_{s,t}}{\partial t}\right) ds\right) * \hH_{s_2, t}^{\cona, \conc}(X) \nonumber\\
        & + (F^\cona(u) \gt \hH_{s_3, t}^{\cona, \conc}(Y)) * \hH_{s_2, t}^{\cona, \conc}(X) * \int_{s_1}^{s_2}F^\cona(x^{s,t})  \gt \conc \left( \frac{\partial X_{s,t}}{\partial s}, \frac{\partial X_{s,t}}{\partial t}\right) ds.\nonumber
    \end{align}
    By the second Peiffer identity, the first term in the sum is equivalent to
    \begin{align}\nonumber
        \left(F^\cona(u) \gt \hH_{s_3, t}^{\cona, \conc}(Y)\right) * \hH_{s_2, t}^{\cona, \conc}(X) * \left(\int_{s_2}^{s_3}\cmb(\hH_{s_2, t}^{\cona, \conc}(X)) \gt F^\cona(u \concat y^{s,t})  \gt \conc \left( \frac{\partial Y_{s,t}}{\partial s}, \frac{\partial Y_{s,t}}{\partial t}\right) ds\right).
    \end{align}
    Then, we have $\cmb(\hH_{s_2, t}^{\cona, \conc}(X)) = F^\cona(\bdy X|_{[s_1, s_2]\times [t_1, t]})$, and since $\bdy X|_{[s_1, s_2]\times [t_1, t]} \concat y^{s,t}$ is exactly the tail path in the differential equation for $\hH^{\cona, \conc}(X \concat_h Y)$, we have
    \begin{align}\nonumber
        \frac{\partial C_t}{\partial t} = \frac{\partial \hH^{\cona, \conc}_t(X \concat_h Y)}{\partial t},
    \end{align}
    so $C_t = \hH^{\cona, \conc}_t(X \concat_h Y)$.
\end{proof}

In particular, this shows that for appropriately composable surfaces $X,Y,Z \in \surfaces^\infty(\V)$, we have
\begin{align}\nonumber
    \bF^{\cona, \conc}(X \concat_h Y) = \bF^{\cona, \conc}(X) \hmult \bF^{\cona, \conc}(Y) \andd \bF^{\cona, \conc}(X \concat_v Z) = \bF^{\cona, \conc}(X) \vmult \bF^{\cona, \conc}(Z).
\end{align}

\begin{remark} \label{rem:thin_homotopy}
    Another important property of surface holonomy is invariance with respect to \emph{thin homotopy}~\cite{martins_surface_2011,faria_martins_crossed_2016,schreiber_smooth_2011,baez_higher_2006,lee_random_2023}. This is analogous to the property that parallel transport of paths and the path signature are invariant with respect to \emph{tree-like equivalence}. In fact, for smooth paths, the notions of thin homotopy and tree-like equivalence coincide~\cite{meneses_thin_2021,tlas_holonomic_2016}. However, because thin homotopy invariance is not a main focus of this article, we refer the interested reader to the above references for discussion on this property. 
\end{remark}

\section{The Surface Signature} \label{sec:surface_signature}

We will now consider the universal (translation-invariant) surface holonomy functor, which was originally proposed by Kapranov~\cite{kapranov_membranes_2015}. We begin by formulating the universal 2-connection in terms of a free crossed module of associative algebras $\bT$, following the procedure discussed for the path signature in~\Cref{sec:path_signature}. This procedure for constructing the universal 2-connection differs slightly from that of Kapranov, who formulates it in terms of a free crossed module of Lie algebras, $\cmg^{\sab}$. We show that these two constructions are equivalent via an inclusion of crossed modules of Lie algebras,
\begin{align}
    \cmg^{\sab} \hookrightarrow \Lie(\bT).
\end{align}
This inclusion allows us to formulate a decomposition of $\bT$, yielding properties of the algebra structure which will be crucial to the extension theorem in the following section. We will then discuss norms on the free crossed module $\bT$, along with truncations and completions, analogous to truncation and completion of the tensor algebra used for path signatures. By leveraging the algebra structure, we define the \emph{surface signature} as an infinite series of iterated integrals. Finally, we show that the surface signature is universal, in the sense that surface holonomy for bounded translation-invariant 2-connections uniquely factors through the surface signature.

\subsection{Construction of the Universal 2-Connection} \label{ssec:construct_univ2con}
We begin by defining the universal 2-connection as the canonical inclusion map which arises from a free construction, analogous to the setting of the path signature in~\Cref{sec:path_signature}. \medskip

\subsubsection{Free Crossed Modules of Associative Algebras}
Our goal in this section is to develop the crossed module analogue of the tensor algebra $T_0(\V)$ via a free construction. 
Because crossed modules of algebras consist of algebra bimodules, we begin with the notion of a free $A$-bimodule.

\begin{lemma}[\cite{loday_algebraic_2012}] \label{lem:fbi_universal}
    Let $(A, \cdot, 1)$ be a unital algebra, and $W$ a vector space. The \emph{free $A$-bimodule over $W$} is $\FBi(A,W) \coloneqq A \otimes W \otimes A$, and satisfies the universal property that for any $A$-bimodule $M$ and a linear map $f: W \to M$, there exists a unique morphism of $A$-bimodules $\tf: \FBi(A,W) \to M$ such that
    \[
        \begin{tikzcd}
            W \ar [r, "f"] \ar[d, swap,"\iota"] & M \\
            \FBi(A,W) \ar[ur, swap, dashed, "\tf"]&
        \end{tikzcd}
    \]
    where $\iota(w) = 1 \otimes w \otimes 1$ is the inclusion.
\end{lemma}

\begin{lemma} \label{lem:fbi_functorial}
    Let $A$ and $B$ be unital algebras and $W$ a vector space. Suppose we have an algebra morphism $f: A \to B$, then there exists a unique morphism of bimodules $\tf: \FBi(A,W) \to \FBi(B, W)$ such that
    \begin{align}\nonumber
        \iota_B = \tf \circ \iota_A \andd \tf(a_1 \gtd P \ltd a_2) = f(a_1) \gtd \tf(P) \ltd f(a_2).
    \end{align}
\end{lemma}
\begin{proof}
    These two properties determine the morphism since the first condition implies that we must have $\tf(1_A \otimes w \otimes 1_A) = 1_B \otimes w \otimes 1_B$; then since $\FBi(A,W)$ is freely generated by the $A$-bimodule action on $W$, the second condition determines the morphism. 
\end{proof}

Next, we require the notion of a free crossed module of algebras.
\begin{definition} \label{def:fxa}
    Let $A_0$ be an associative algebra, $W$ be a vector space and $\delta_0: W \to A_0$ be a linear map. The \emph{free crossed module of algebras with respect to $\delta_0$}, denoted
    \begin{align}\nonumber
        \cmFXA(\delta_0) = (\delta: \FXA_1(\delta_0) \to A_0, \gtd, \ltd),
    \end{align}
    along with the linear inclusion $\iota: W \to \FXA_1$ such that $\delta_0 = \delta \circ \iota$,
    has the following universal property: for any crossed module of algebras $\cmA = (\delta^A : A_1 \to A_0, \gtd_A, \ltd_A)$ with a linear map $\conc: W \to A_1$ such that $\delta_0 = \delta_A \circ \conc$, there exists a unique algebra map $\tconc : \FXA_1 \to A_1$ such that the following diagram commutes
    \[
    \begin{tikzcd}
        W \ar[r, "\iota"] \ar[dr, swap, "\conc"] & \FXA_1(\delta_0) \ar[r, "\delta"] \ar[d, dashed, "\tconc"] & A_0 \ar[d, equal] \\
        & A_1 \ar[r, "\delta^A"] & A_0
    \end{tikzcd}
    \]
\end{definition}

An explicit model for the free crossed module of algebras is provided in~\cite[Lemma 3.2.6]{wagemann_crossed_2021}, and for completeness, we provide a proof of the following result in~\Cref{apx:fxa}.
\begin{proposition} \label{prop:fxa_construction}
    Let $A_0$ be a unital associative algebra, $W$ be a vector space and $\delta_0: W \to A_0$ be a linear map. Define the linear map $\delta: A_0 \otimes W \otimes A_0$ by
    \begin{align}\nonumber
        \delta(a_1 \otimes w \otimes a_2) = a_1 \otimes \delta_0(w) \otimes a_2
    \end{align}
    We define the \emph{Peiffer ideal} $\Pf \subset A_0 \otimes W \otimes A_0$ by
    \begin{align}\nonumber
        \Pf \coloneqq \spann \{ \delta(E) \gtd F - E \ltd \delta(F) \, : \, E, F \in A_0 \otimes W \otimes A_0\}.
    \end{align}
    Then, we define
    \begin{align}\nonumber
        \FXA_1(\delta_0) \coloneqq (A_0 \otimes W \otimes A_0)/ \Pf
    \end{align}
    The $A_0$-bimodule structure on $A_0 \otimes W \otimes A_0$ and the map $\delta$ is well-defined on $\FXA_1(\delta_0)$, and we define the product $*$ on $\FXA_1(\delta_0)$ by
    \begin{align} \label{eq:fxa_product}
        P * Q \coloneqq \delta(P) \gtd Q = P \ltd \delta(Q).
    \end{align}
    Then, $\cmFXA(\delta_0) = (\delta: \FXA_1(\delta_0) \to A_0, \gtd_A, \ltd_A)$ is a free crossed module of algebras.
\end{proposition}

Furthermore, this free construction is functorial, where the proof is in~\Cref{apx:fxa}.

\begin{proposition} \label{prop:fxa_functorial}
    Let $W$ be a vector space, $A_0, B_0$ be associative algebras, and $\delta^A_0: W \to A_0$ and $\delta^B_0: W \to B_0$ be linear maps. Suppose $f_0: A_0 \to B_0$ such that $\delta^B_0 = f \circ \delta^A_0$. Then, there exists a morphism of crossed modules of algebras $ (\FXA_1(f_0), f_0) : \cmFXA(\delta^A_0) \to \cmFXA(\delta^B_0)$ such that the following commutes
    \[
    \begin{tikzcd}
        W \ar[r, "\iota^A"] \ar[dr, swap, "\iota^B"] & \FXA_1(\delta^A_0) \ar[r, "\delta^A"] \ar[d, dashed, "\FXA_1(f_0)"] & A_0 \ar[d, "f_0"] \\
        & \FXA_1(\delta^B_0) \ar[r, "\delta^B"] & B_0
    \end{tikzcd}
    \]
\end{proposition}

\subsubsection{Universal 2-Connection}
Now, we will construct the universal 2-connection. We start with the universal 1-connection $\ucona \in L(\V, T_0(\V))$, and consider the \emph{curvature} $\kappa \in L(\Lambda^2 \V, T_0(\V))$ defined by
\begin{align}\nonumber
    \kappa(v_1 \wedge v_2) \coloneqq [\ucona(v_1), \ucona(v_2)] = v_1 \cdot v_2 - v_2 \cdot v_1 \in V^{\otimes 2} \subset T(\V)
\end{align}
The higher analogue of the tensor algebra $T_0(\V)$, is the free crossed module with respect to the curvature,
\begin{align}\nonumber
    T_1(\V) \coloneqq \FXA_1(\kappa) = \Big(T_0(\V) \otimes \Lambda^2 V \otimes T_0(\V)\Big) /\Pf.
\end{align}
For $E, F \in T_1(\V)$, the product is defined by
\begin{align} \label{eq:bT_product}
    E * F = \delta(E) \cdot F = E \cdot \delta(F).
\end{align}
Note that the boundary map is defined on the free $T_0(\V)$-bimodule $\delta: T(\V) \otimes \Lambda^2 V \otimes T(\V) \to T(\V)$ by
\begin{align} \nonumber
    \delta(a \otimes v \otimes b) = a \cdot \kappa(v) \cdot b,
\end{align}
and this is well-defined on $T_1(\V)$ by~\Cref{prop:fxa_construction}. We will denote this crossed module by
\begin{align} \label{eq:bT_def}
    \bT(\V) \coloneqq (\delta: T_1(\V) \to T_0(\V), \gtd, \ltd),
\end{align}
where the inclusion $\uconc \in L(\Lambda^2 \V, T_1)$ given by
\begin{align}\nonumber
    \uconc(v_1 \wedge v_2) = 1 \otimes (v_1 \wedge v_2) \otimes 1
\end{align}
is the \emph{universal 2-connection}. 

\begin{proposition} \label{prop:univ_2con_alg}
    Let $(\zeta, Z)$ be the 2-connection valued in $\bT(V)$ defined above. For any 2-connection $(\cona, \conc)$ valued in a crossed module of algebras $\cmA = (\delta: A_1 \to A_0)$, there exists a unique morphism of crossed modules of algebras $(\tcona, \tconc): \bT(\V) \to \cmA$ such that
    \begin{align}\nonumber
        \cona = \tcona \circ \zeta \andd \conc = \tconc \circ Z.
    \end{align}
\end{proposition}
\begin{proof}
    First, consider the free crossed module with resepct to $\delta^A_0 = \delta^A \circ \gamma : \Lambda^2 \V \to A_0$ to obtain
    \begin{align}\nonumber
        \cmFXA(\delta^A_0) = (\delta^F: \FXA_1(\delta^A_0) \to A_0),
    \end{align} 
    along with the inclusion $\iota_A: \Lambda^2 \V \to \FXA_1(\delta^A_0)$. Given the linear map $\conc \in L(\Lambda^2 \V, A_1)$, there exists a unique map $\overline{\conc}: \FXA_1(\delta^A_0) \to A_1$ such that $\conc = \overline{\conc} \circ \iota_A$ by the universal property of free crossed modules of algebras in~\Cref{def:fxa}. Next, by the universal property of the free algebra $T_0(\V)$, there exists a unique map $\tcona : T_0(\V) \to A_0$ such that $\cona = \tcona \circ \zeta$. Then, by the functoriality of the free crossed module construction from~\Cref{prop:fxa_functorial}, we obtain the following commutative diagram,
    \[
        \begin{tikzcd}
            & T_1(\V) \ar[r,"\delta^T"] \ar[d, "\FXA_1(\tcona)", dashed] & T_0(\V) \ar[d, "\tcona", dashed] \\
            \Lambda^2 \V \ar[ru, "\uconc"] \ar[r, "\iota_A"] \ar[dr, swap, "\conc"] & \FXA_1(\delta_0) \ar[r, "\delta^F"] \ar[d, "\overline{\conc}", dashed] & A_0 \ar[d,equal]  \\
            & A_1 \ar[r, "\delta^A"] & A_0
        \end{tikzcd}
    \]
    where the dashed arrows are unique maps given by free constructions. Then, we define $\tconc = \FXA_1(\tcona) \circ \overline{\conc} : T_1(\V) \to A_1$, and since $(\tconc, \tcona): \bT(\V) \to \cmA$ is a composition of two morphisms of crossed modules, it is also a morphism of crossed modules. 
\end{proof}

\subsection{Relationship with Kapranov's Construction and Decomposition of \texorpdfstring{$T_1(\V)$}{T1(V)}} \label{ssec:relationship_kapranov}

In this section, we will show how our construction of the universal 2-connection coincides with Kapranov's construction in~\cite{kapranov_membranes_2015}. We begin with an overview of Kapranov's construction. Here, we require the notion of a crossed module of Lie algebras.

\begin{definition}
    A \emph{crossed module of Lie algebras}
    \begin{align}\nonumber
        \cmg = (\delta: \fg_1 \to \fg_0, \gt)
    \end{align}
    consists of Lie algebras $(\fg_0, [\cdot, \cdot]_0)$ and $(\fg_1, [\cdot, \cdot]_1)$ equipped with a Lie algebra action of $\fg_0$ on $\fg_1$ by derivations; in other words, it satisfies
    \begin{align}\nonumber
        x \gt [E, F]_1 = [x \gt E, F]_1 + [E, x \gt F]_1 \andd [x, y]_0 \gt E = x \gt (y \gt E) - y \gt (x \gt E)
    \end{align}
    for all $x,y \in \fg_0$ and $E, F \in \fg_1$. Furthermore, there is a morphism of Lie algebras $\delta: \fg_1 \to \fg_0$ such that
    \begin{enumerate}
        \item \textbf{(First Peiffer Identity)} $\delta(x \gt E) = [x, \delta(E)]_0$ for all $x \in \fg_0$ and $E \in \fg_1$; and
        \item \textbf{(Second Peiffer Identity)} $\delta(E) \gt E' = [E, E']_1$ for all $E, E' \in \fg_1$. 
    \end{enumerate}
    Suppose $\cmh = (\delta : \fh_1 \to \fh_0, \gt)$ is another crossed module of Lie algebras. A \emph{morphism between crossed module of Lie algebras} $f = (f_0, f_1) : \cmg \to \cmh$ consists of two morphisms of Lie algebras $f_0 : \fg_0 \to \fh_0$ and $f_1 : \fg_1 \to \fh_1$ such that for all $x \in \fg_0$ and $E \in \fg_1$, 
    \begin{align} \label{eq:cmla_morphism}
        \delta \circ f_1(E) = f_0 \circ \delta(E)  \andd f_1 ( x \gt E) = f_0(x) \gt f_1(E).
    \end{align}
\end{definition}

\subsubsection{Kapranov's Construction}
Let $\V$ be a finite dimensional vector space. Kapranov~\cite{kapranov_membranes_2015} begins by constructing a graded vector space
\begin{align}\nonumber
    \Lambda^\bullet \V \coloneqq \bigoplus_{k=1}^\infty \Lambda^k \V,
\end{align}
where elements in $\Lambda^k\V$ are of degree $k-1$. Next, we construct the free graded Lie algebra $\fg_\bullet \coloneqq \FL(\Lambda^\bullet V)$, where $\fg_0 = \FL(\V)$ is the usual free Lie algebra of $\V$, and $\fg_1$ is freely generated as a vector space by the elements of the form
\begin{align}\nonumber
    [u_1, \ldots [u_{k}, v \wedge w]\ldots],
\end{align}
where $u_1, \ldots, u_k, v, w \in \V$. Next, one can define a differential $\delta: \fg_1 \to \fg_0$ by
\begin{align}\nonumber
    \delta([u_1, \ldots [u_{k}, v \wedge w]\ldots]) =  [u_1, \ldots [u_{k}, [v,w]]\ldots],
\end{align}
which is a derivation of Lie algebras. In fact, a differential can be defined on the entire $\fg_\bullet$ to construct a differential graded Lie algebra, though we only require the first two degrees. Kapranov then introduces a \emph{semiabelianization} procedure to truncate the (differential) graded Lie algebra into a crossed module of Lie algebras. In particular, we define the equivalence relation $[\delta(E), F] \sim [E, \delta(F)]$ for $E,F \in \fg_1$, and define the Lie algebra
\begin{align}\nonumber
    \fg_1^{\sab} \coloneqq \fg_1/\sim \quad \text{with Lie bracket} \quad [E, F]_{\sab} \coloneqq [\delta(E), F] = [E, \delta(F)].
\end{align}
Kapranov shows in~\cite[Example 1.3.2, Proposition 1.3.6]{kapranov_membranes_2015} that
\begin{align}\nonumber
    \cmg^{\sab} = (\delta: \fg_1^{\sab} \to \fg_0, \gt),
\end{align}
where the action is defined by $x \gt E \coloneqq [x, E]$ for $x \in \fg_0$ and $E \in \fg_1^{\sab}$, is a crossed module of Lie algebras. In fact, one can show that this is the free crossed module of Lie algebras with respect to the map $\delta_0: \Lambda^2 \V \to \fg_0$ defined by $\delta_0(v \wedge w) = [v, w]$ (see the construction of free crossed modules of Lie algebras in~\cite[Section 3.6]{cirio_categorifying_2012}).
Kapranov's universal 2-connection is defined in terms of crossed modules of Lie algebras, given by the linear inclusion maps $\tucona \in L(\V, \fg_0)$ and $\tuconc \in L(\Lambda^2 \V, \fg_1^{\sab})$.\medskip

\subsubsection{Relationship with Kapranov's Construction}
Now, we discuss how Kapranov's construction coincides with our construction in terms of algebras in~\Cref{ssec:construct_univ2con}. We begin by considering the universal enveloping algebra~\cite[Section 21a]{felix_rational_2001} of the free Lie algebra $\fg_\bullet$, which is a differential graded algebra denoted by $U_\bullet \coloneqq U(\fg_\bullet)$. Note that $U_0 = T(\V)$ and $U_1 = \oT_1(\V) \coloneqq \FBi(T(\V), \Lambda^2 \V)$, which is the free $T(\V)$-bimodule. There is a natural linear embedding $\overline{\iota}_1: \fg_1 \hookrightarrow \oT_1(\V)$. For $x \in \fg_0$ and $E \in \fg_1$, we have
\begin{align} \label{eq:iota1_fbi}
    \iota_1([x, E]) = x\cdot E - E \cdot x
\end{align}
where $[\cdot, \cdot]$ is the bracket in $\fg_\bullet$ and $\cdot$ is the product in $U_\bullet$.
Once we quotient $\oT_1(\V)$ by the Peiffer subspace $\Pf$, the kernel of the linear map $\iota_1: \fg_1 \rightarrow T_1(\V)$ is generated by
\begin{align}\nonumber
    [\delta(E), F] - [E, \delta(F)]  &= \delta(E) \cdot F - F \cdot \delta(E) - E \cdot \delta(F) + \delta(F) \cdot E \\
    &= ( \delta(E) \cdot F - E \cdot \delta(F)) + (\delta(F) \cdot E - F \cdot \delta(E)),
\end{align}
where $E, F \in \fg_1$. Thus, we obtain an injective linear map $\iota_1: \fg_1^{\sab} \to T_1(\V)$. 
Furthermore, if we consider the commutator bracket $[\cdot, \cdot]_1$ on $T_1(\V)$ and using the definition of the product in~\Cref{eq:bT_product}, we have
\begin{align}\nonumber
    \iota_1([E, F]_{\sab}) = \iota_1([\delta(E), F]) = (\delta(E) \cdot F - F \cdot \delta(E))  = (E * F - F * E) = [\iota_1(E), \iota_1(F)]_1.
\end{align}
And therefore, $\iota_1$ is an inclusion of Lie algebras of $\fg_1^{\sab}$ into $\Lie(T_1(\V))$, where $\Lie(A)$ is the Lie algebra of an algebra $A$ equipped with the commutator bracket. In fact, applying the Lie functor to a crossed module of algebras yields a crossed module of Lie algebras.

\begin{proposition}{\cite[Lemma 6]{faria_martins_crossed_2016}}
    Let $\cmA = (\delta: A_1 \to A_0, \gtd, \ltd)$ be a crossed module of algebras. Then, 
    \begin{align}\nonumber
        \Lie(\cmA) = (\delta: \Lie(A_1) \to \Lie(A_0), \gt),
    \end{align}
    where $\Lie(A_0)$ and $\Lie(A_1)$ are the Lie algebras induced by the commutator brackets on $A_0$ and $A_1$ respectively, and the action is defined for $X \in \Lie(A_0)$ and $E \in \Lie(A_1)$ by
    \begin{align}\nonumber
        X \gt E \coloneqq X \gtd E - E \ltd X,
    \end{align}
    is a crossed module of Lie algebras. 
\end{proposition}

Furthermore, we have a natural inclusion of Lie algebras $\iota_0 : \fg_0 \to \Lie(T_0(\V))$. In fact, this induces a morphism at the level of crossed modules.

\begin{proposition} \label{prop:inclusion_kapranov}
    Let $\bT(\V) = (\delta: T_1(\V) \to T_0(\V), \gtd, \ltd)$ be the free crossed module in~\Cref{eq:bT_def}. Then, the inclusions $\iota = (\iota_0, \iota_1) : \cmg^{\sab} \to \Lie(\bT(\V))$ is a morphism of crossed modules of Lie algebras.
\end{proposition}
\begin{proof}
    In our discussion above, we have already shown that $\iota_0$ and $\iota_1$ are morphisms of Lie algebras, so we are just left to show that the two conditions in~\Cref{eq:cmla_morphism} hold. To distinguish the Lie brackets, we let
    \begin{itemize}
        \item $[\cdot, \cdot]$ denote the underlying Lie bracket of the differential graded Lie algebra $\fg_\bullet$, which coincides with the Lie bracket of $\fg_0$;
        \item $[\cdot, \cdot]_0$ denote the commtuator bracket in $T_0(\V)$; and
        \item $[\cdot, \cdot]_1$ denote the commutator bracket in $T_1(\V)$.
    \end{itemize}
    
    Let $x \in \fg_0$ and $E \in \fg_1^{\sab}$. Then,
    \begin{align}\nonumber
        \iota_1( x \gt E) = \iota_1([x,E]) = (x\cdot E - E \cdot x) = \iota_0(x) \gt \iota_1(E).
    \end{align} 
    Next, we must show that the following commutes.
    \[
        \begin{tikzcd}
            \fg_1^{\sab} \ar[r, "\delta_\fg"] \ar[d, swap, "\iota_1"] & \fg_0 \ar[d, "\iota_0"] \\
            \Lie(T_1(\V)) \ar[r, "\delta_T"] & \Lie(T_0(\V)).
        \end{tikzcd}    
    \]
    Let $(v \wedge w)_{\fg_1^{\sab}} \in \Lambda^2 \V \subset \fg_1^{\sab}$. Then,
    \begin{align} \label{eq:kapranov_cmla_morphism_generator}
        \iota_0 \circ \delta_\fg((v \wedge w)_{\fg_1^{\sab}}) = \iota_0([v, w]) = v\cdot w - w \cdot v = \delta_T ((v \wedge w)_{T_1}) = \delta_T \circ \iota_1((v \wedge w)_{\fg_1^{\sab}}),
    \end{align}
    where $(v \wedge w)_{T_1} \in \Lambda^2 \V \subset \Lie(T_1(\V))$, and we note that $\iota_1((v \wedge w)_{\fg_1^{\sab}}) = (v \wedge w)_{T_1}$. Next, by~\cite[Proposition 1.5.1]{kapranov_membranes_2015}, the Lie algebra $\fg_1^{\sab}$ is spanned by elements of the form $[x_1, \ldots, [x_k, E] \ldots]$, where $x_1, \ldots x_k \in \V$ and $E \in \Lambda^2 \V$. Then, we have
    \begin{align}\nonumber
        \iota_0 \circ \delta_\fg([x_1, \ldots, [x_k, E] \ldots]) &= \iota_0( [x_1, \ldots, [x_k, \delta_\fg(E)] \ldots])  \quad \text{(first Peiffer identity of $\cmg$)}\\
        & = [x_1, \ldots, [x_k, \iota_0 \circ \delta_\fg(E)]_0 \ldots]_0  \quad \text{(by definition of $\iota_0$)}\nonumber\\
        & = [x_1, \ldots, [x_k, \delta_T \circ \iota_1 (E)]_0 \ldots]_0  \quad \text{(by~\Cref{eq:kapranov_cmla_morphism_generator})}\nonumber\\
        & = \delta_T([x_1, \ldots, [x_k, \iota_1 (E)]_1 \ldots]_1 ) \quad \text{(first Peiffer identity of $\Lie(\bT)$)} \nonumber\\
        & = \delta_T \circ \iota_1([x_1, \ldots, [x_k, E] \ldots] ) \quad \text{(by definition of $\iota_1$ in~\Cref{eq:iota1_fbi})}\nonumber
    \end{align}
\end{proof}

\medskip
\subsubsection{Decomposition of $T_1(\V)$}
Now, we will use this relationship to Kapranov's construction to provide a decomposition of $T_1(\V)$, which will help us better understand its algebra structure. Kapranov shows in~\cite[Proposition 2.5.1]{kapranov_membranes_2015} that $\fg_1^{\sab}$ decomposes as 
\begin{align} \label{eq:kapranov_decomposition}
    \fg_1^{\sab} = [\fg_0, \fg_0] \oplus \fg_1^{\sab, \ker},
\end{align}
where $[\fg_0, \fg_0]$ is the commutant of $\fg_0$, and $\fg_1^{\sab, \ker} = \ker(\delta: \fg_1^{\sab} \to)$ is an abelian Lie algebra. 
We briefly recall part of the argument. First, we note that the image of the crossed module boundary is $\im(\delta: \fg_1^{\sab} \to \fg_0) = [\fg_0, \fg_0]$. By the Shirshov-Witt theorem~\cite[Theorem 2.5]{reutenauer_free_1993}, any Lie subalgebra of a free Lie algebra is free, thus $[\fg_0, \fg_0]$ is free. Then, there exists a Lie algebra section $\hs: [\fg_0, \fg_0] \to \fg_1^{\sab}$. Indeed, let $L \subset [\fg_0, \fg_0]$ be a vector space of homogeneous Lie algebra generators of $[\fg_0, \fg_0]$, and choose a graded linear section $s: L \to \fg_1^{\sab}$ such that $\delta \circ s = \id_L$. Then, define $\hs: [\fg_0, \fg_0] \to \fg_1^{\sab}$ to be the unique Lie algebra morphism obtained via the universal property
\[
    \begin{tikzcd}
        L \ar[r, "s"] \ar[d] & \fg_1^{\sab} \\
        \left[\fg_0, \fg_0\right] \ar[ru, dashed, swap, "\hs"] & 
    \end{tikzcd}
\]
Then, $\delta\circ \hs ([E,F]) = [\delta\circ \hs(E), \delta\circ \hs(F)]$, and since $\delta \circ \hs = \id$ on the Lie algebra generators, this is a section. \medskip

We wish to repeat this argument for $\bT(\V)$ at the level of algebras. Because $[\fg_0, \fg_0]$ is a Lie subalgebra of $\fg_0$, there is an inclusion of algebras $U([\fg_0, \fg_0]) \hookrightarrow T_0(\V)$ of the universal enveloping algebras. Consider the same graded linear section $s: L \to \fg_1^{\sab}$ of Lie algebra generators, and let $\ts_0 = \iota_1 \circ s : L \to T_1(\V)$ be the linear section into $T_1(\V)$. Then, using the fact that $U([\fg_0, \fg_0]) \cong T_0(L)$ is a free algebra, we use the universal property of $T_0(L)$, 
\[
    \begin{tikzcd}
        L \ar[r, "\tilde{s}_0"] \ar[d] & T_1(\V) \\
        T_0(L) \ar[ru, dashed, swap, "\ts"]
    \end{tikzcd}
\]
to extend the linear section to an algebra morphism $\ts : T_0(L) \to T_1(\V)$. Then for $\ell_1, \ldots, \ell_k \in L$, we have
\begin{align}\nonumber
    \delta \circ \ts (\ell_1 \cdots \ell_k) = (\delta \circ \ts(\ell_1)) \cdots (\delta \circ \ts(\ell_k)) = \ell_1 \cdots \ell_k
\end{align}
so $\delta \circ \ts = \id$, and $\ts$ is a section of algebras. Now we define
\begin{align} \label{eq:T1s}
    T_1^s \coloneqq \ts(U([\fg_0, \fg_0])) \subset T_1,
\end{align}
where we use the superscript $s$ to emphasize the fact that this subspace depends on our initial choice of linear $s: L \to \fg_1^{\sab}$. Furthermore, because the original section $s$ is graded, the section $\ts$ is well-defined level-wise, and we define
\begin{align}\nonumber
    T_1^{s, \gr{n}}(\V) \coloneqq \ts(U^{\gr{n}}([\fg_0, \fg_0])) \subset T_1^{\gr{n}}.
\end{align}

\begin{remark}
    In general, $\im(\delta: T_1(\V) \to T_0(\V)) \neq  U([\fg_0, \fg_0])$. For instance, let $u,v, w \in \V$, and consider $u \cdot (v \wedge w) \in T_1^{\gr{3}}(\V)$. Then, $\delta(u \cdot (v \wedge w)) = u \cdot (v \cdot w - w \cdot v) \notin U([\fg_0, \fg_0])$. Indeed, this is the case because $u$ and $[v,w]$ can be chosen as part of a linear basis for $\fg_0$, but $u$ will not be part of a linear basis for $[\fg_0, \fg_0]$, and by the Poincare-Birkhoff-Witt basis~\cite[Theorem 0.2]{reutenauer_free_1993} for $[\fg_0, \fg_0]$, we have $u \cdot [v,w] \notin U([\fg_0, \fg_0])$. Due to these elements, an analogous argument to obtain an algebra section for $\im(\delta: T_1(\V) \to T_0(\V))$ without reference to free Lie algebras and universal enveloping algebras may not be possible (see~\cite[Theorem 3.4]{cohn_subalgebras_1964}).
\end{remark}

\begin{proposition} \label{prop:T1_decomposition}
    Let $\bT(\V) = (\delta: T_1(\V) \to T_0(\V), \gtd, \ltd)$ be the free crossed module. Let $T_1^{\ker}(\V) \coloneqq \ker(\delta)$, and $T_1^s$ be defined in~\Cref{eq:T1s} with respect to a linear section $s: L \to \fg_1^{\sab}$. Then, the subspace $T_1^{\ker}(\V) \oplus T_1^s(\V) \subset T_1(\V)$ is independent of the choice of section $s$. Furthermore, we have the following multiplicative properties
    \begin{align} \label{eq:T1s_mult_prop}
        T_1^s(\V) * T_1^s(\V) \subset T_1^s(\V) \andd T_1^s(\V) * T_1^{\ker}(\V) = T_1^{\ker}(\V) * T_1^s(\V) = T_1^{\ker}(\V) * T_1^{\ker}(\V) = 0.
    \end{align}
\end{proposition}
\begin{proof}
    First, let $E, F \in T_1(\V)$ such that $\delta(E) = \delta(F)$; then $\delta(E - F) = 0$, so $E-F \in T_1^{\ker}$. Thus given two sections $s, s' : L \to \fg_1^{\sab}$, we have
    \begin{align}\nonumber
        T_1^{s'}(\V) \subset T_1^s(\V) \oplus T_1^{\ker}(\V),
    \end{align}
    so the subspace $T_1^s(\V) \oplus T_1^{\ker}(\V)$ is independent of the choice of section. \medskip

    Next, fix a section $s$, and let $\ts(E), \ts(F) \in T_1^s(\V)$. Then, since $\ts$ is an algebra morphism, we have
    \begin{align}\nonumber
        \ts(E) * \ts(F) = \ts(E \cdot F) \in T_1^s(\V)
    \end{align}
    so $T_1^s(\V) * T_1^s(\V) \subset T_1^s(\V)$. Now, for the other three properties in~\Cref{eq:T1s_mult_prop}, let $E \in T_1^{\ker}(\V)$ and $F \in T_1(\V)$. Then, we have
    \begin{align}\nonumber
        E*F = (\delta(E) \cdot F) = 0 \andd F * E = (F \cdot \delta(E)) = 0.
    \end{align}
\end{proof}

\subsection{Norms, Truncation, and Completion}

In this section, we discuss norms on the free crossed module $\bT(\V)$, as well as notions of truncation and completion. In particular, we will considering the structure of crossed modules of algebras with topological structure. 

\begin{definition}
    A \emph{crossed module of Banach algebras (resp.~locally $m$-convex algebras)} is a crossed module of algebras,
    \begin{align}\nonumber
        \cmA = (\delta: A_1 \to A_0, \gtd, \ltd)
    \end{align}
    such that $A_0$ and $A_1$ are Banach algebras (resp.~locally $m$-convex algebras), $\delta$ is a continuous algebra morphism, and the actions $\gtd$ and $\ltd$ are short (resp.~continuous).
\end{definition}
In particular, in the Banach algebra setting, for $a \in A_0$ and $E \in A_1$, we have
\begin{align}\nonumber
    \|a \gtd E\|_1 \leq \|a\|_0 \cdot \|E\|_1 \andd \|E \ltd a\|_1 \leq \|a\|_0 \cdot \|E\|_1.
\end{align}
Furthermore, we define the \emph{group-like elements} in the truncation and completion of $\bT(\V)$, which leads to crossed modules of groups which are embedded in the invertible elements (\Cref{prop:cma_invertible}).

\begin{definition} \label{def:continuous_connection}
    Let $\cmA$ be a crossed module of Banach or locally $m$-convex algebras. A 2-connection $(\cona, \conc)$ valued in $\cmA$ is \emph{continuous} if both $\cona \in L(\V, A_0)$ and $\conc \in L(\Lambda^2 \V, A_1)$ are continuous. 
\end{definition}

\subsubsection{Norms} \label{sssec:norms}
We begin at the level of bimodules prior to taking the quotient by the Peiffer ideal. We note that the free $T_0(\V)$-bimodule has the form
\begin{align}\nonumber
    \oT_1(\V) \coloneqq \FBi(T_0(\V), \Lambda^2 \V) = T_0(\V) \otimes \Lambda^2 \V \otimes T_0(\V) = \bigoplus_{n=2}^\infty \, \bigoplus_{k=0}^{n-2} \V^{\otimes k} \otimes \Lambda^2 \V \otimes \V^{\otimes(n-k-2)}.
\end{align}
Here, we have decomposed this free bimodule in terms of the grading given by the total degree of tensors, where elements of $\V$ have degree $1$ and elements of $\Lambda^2 \V$ have degree $2$. We will refer to this grading $n$ as the \emph{level} of the tensors. We define the level $(n,k)$ subspace and level $n$ subspace respectively by 
\begin{align}\nonumber
    \oT_1^{\gr{n,k}}(\V) = \V^{\otimes k} \otimes \Lambda^2 \V \otimes \V^{\otimes(n-k-2)} \andd \oT_1^{\gr{n}}(\V) \coloneqq \bigoplus_{k=0}^{n-2} \oT_1^{\gr{n,k}}(\V).
\end{align}

Now, we suppose that $\V$ is a finite dimensional Hilbert space with orthonormal basis $\{e_1, \ldots, e_d\}$. Then, there is a natural Hilbert space structure on $\Lambda^2 \V$ given by the orthonormal basis $\{ e_{i,j} \coloneqq e_i \wedge e_j\}_{i < j}$. We can further extend this Hilbert space structure to $\oT_1^{\gr{n,k}}(\V)$ by defining
\begin{align} \label{eq:oT1_basis}
    e_{q_1} \cdots e_{q_k} \cdot e_{i,j} \cdot e_{q_{k+1}} \cdots e_{q_{n-2}}
\end{align}
for $q_1, \ldots, q_{n-2} \in [d]$ and $i < j \in [d]$ to be an orthonormal basis. Furthermore, we give $V^{\otimes n}$ the usual Hilbert space structure by defining $\{e_{q_1} \cdots e_{q_n}\}_{q_i \in [d]}$ to be an orthonormal basis. 
We use $\|\cdot\|$ to denote the norm on all $\oT_1^{\gr{n,k}}(\V)$ and $V^{\otimes n}$. 
These choices of inner products have the property that for any $u \in \oT^{\gr{n,k}}_1(\V)$, $v \in \V^{\otimes m}$, and $w \in \V^{\otimes p}$, we have 
\begin{align} \label{eq:norm_properties}
    \|u \cdot v\| = \|u\| \cdot \|v\|, \quad \|v \cdot u\| = \|u\| \cdot \|v\|, \andd \|v \cdot w\| = \|v\|\cdot  \|w\|.
\end{align}
The boundary map $\delta_{\gr{n,k}} : \oT_1^{\gr{n,k}}(V) \to V^{\otimes n}$ is defined on the basis vectors in~\Cref{eq:oT1_basis} by
\begin{align}\nonumber
    \delta_{\gr{n,k}}(e_{q_1} \cdots e_{q_k} \cdot e_{i,j} \cdot e_{q_{k+1}} \cdots e_{q_{n-2}}) = e_{q_1} \cdots e_{q_k} \cdot (e_i \cdot e_j - e_j \cdot e_i) \cdot e_{q_{k+1}} \cdots e_{q_{n-2}},
\end{align}
and thus has norm $\|\delta_{\gr{n,k}}\|_{\op} = \sqrt{2}$.
We extend the Hilbert space structure via the direct sum to $\oT_1^{\gr{n}}$, and the boundary map $\delta_{\gr{n}} : \oT_1^{\gr{n}} \to T_0$ given by $\delta_{\gr{n}} \coloneqq \sum_{k=0}^{n-2} \delta_{\gr{n,k}}$ has the same norm.

The level $n$ closed Peiffer subspace $\Pf_n \subset \oT_{1}^{\gr{n}}(\V)$ is
\begin{align}\nonumber
    \Pf_n \coloneqq \spann \{ \delta E \cdot F - E \cdot \delta F \, : \, E \in \oT_1^{\gr{n_1}}, \, F \in \oT_1^{\gr{n_2}}, \, n_1 + n_2 = n\}.
\end{align}
Because we are working with Hilbert spaces, we can define the quotient by the orthogonal complement
\begin{align}\nonumber
    T_1^{\gr{n}}(\V) \coloneqq \oT_1^{\gr{n}}(\V)/\Pf_n \cong \Pf_n^{\perp},
\end{align}
where $\Pf_n^{\perp}$ denotes the orthogonal complement of $\Pf_n$. We will use the $(\cdot)^{\perp}$ notation to denote the projection map from $T_1$ to $\Pf^{\perp}$. In particular, the product in $T_1$ is defined on $E, F \in T_1$ to be
\begin{align}\nonumber
    E * F = (\delta(E) \cdot F)^{\perp} = (E \cdot \delta(F))^{\perp} = (\delta(E) \cdot F +P)^{\perp},
\end{align}
where $P$ can be any element in the Peiffer subspace, $P \in \Pf_n$.\medskip

Furthermore, multiplication in $T_1(V)$ has norm bounded by $1$. Indeed, let $E \in T_1^{\gr{n}}(\V) \cong \Pf_n^\perp$ and $F \in T_1^{\gr{m}}(\V) \cong \Pf_m^\perp$, then we have
\begin{align}
    \|E * F\|^2_{T_1} &= \inf_{P \in \Pf_n}\|\delta(E) \cdot F + P\|^2_{\oT_1} \nonumber \\
    & \leq \inf_{t \in [0,1]} \| (1-t) \delta(E) \cdot F - t E \cdot \delta(F) \|^2_{\oT_1}\nonumber \\
    & \leq \inf_{t \in [0,1]} \| (1-t) \delta(E) \cdot F\|^2_{\oT_1} + \|t E \cdot \delta(F) \|^2_{\oT_1}\nonumber  \\
    & \leq \inf_{t \in [0,1]} \Big( (1-t)^2 + t^2\Big) 2 \|E\|_{\oT_1}^2 \cdot \|F\|_{\oT_1}^2 \nonumber \\
    & \leq \|E\|_{\oT_1}^2 \cdot \|F\|_{\oT_1}^2. \label{eq:levelwise_ba}
\end{align}
Furthermore, the actions of $\V^{\otimes n}$ on $T_1^{\gr{m}}$ are short linear maps by using~\Cref{eq:norm_properties}, for $a^{\gr{n}} \in \V^{\otimes n}$ and $E^{\gr{m}} \in T_1^{\gr{m}}$, we have
\begin{align} \label{eq:levelwise_short_action}
    \|a^{\gr{n}} \cdot E^{\gr{m}}\|_{T_1} = \|(a^{\gr{n}}\cdot E^{\gr{m}})^\perp\|_{\oT_1} \leq \|a^{\gr{n}} \cdot E^{\gr{m}}\|_{\oT_1} \leq \|a^{\gr{n}}\|_{T_0} \cdot \|E^{\gr{m}}\|_{\oT_1} = \|a^{\gr{n}}\|_{T_0} \cdot \|E^{\gr{m}}\|_{T_1}.
\end{align}
The right action can be bound in the same manner. \medskip

\subsubsection{Truncation}
For $m \geq 2$, we define the \emph{level $m$ truncation of $\bT(\V)$} by
\begin{align}\nonumber
    \bT^{\gr{\leq m}}(\V) = (\cmb: T_1^{\gr{\leq m}}(\V) \to T_0^{\gr{\leq m}}(V), \gtd, \ltd), \quad \text{where} \quad T_1^{\gr{\leq m}}(\V) \coloneqq \bigoplus_{n=2}^m T^{\gr{n}}(\V).
\end{align}
The truncation maps $\pi^{\gr{n}} = (\pi_0^{\gr{n}}, \pi_1^{\gr{n}})$ given by $\pi^{\gr{n}}_0 : T_0(\V) \to T_0^{\gr{\leq n}}(\V)$ and $\pi^{\gr{n}}_1 : T_1(\V) \to T_1^{\gr{\leq n}}(\V)$ form a morphism of crossed modules. 
Using the norms from the previous section, we can summarize the results as follows.

\begin{lemma}
    The level $n$ truncation $\bT^{\gr{\leq n}}(\V)$ is a crossed module of Banach algebras. 
\end{lemma}

We can also consider the truncation of Kapranov's crossed module of Lie algebras $\cmg^{\sab}$. Following Kapranov in~\cite[Section 2.2C]{kapranov_membranes_2015}, we define the \emph{$\fg_0$-lower central series of $\fg_1^{\sab}$} as
\begin{align}\nonumber
    \LCS_2(\fg_0, \fg_1^{\sab}) \coloneqq \fg_1^{\sab} \andd \LCS_r(\fg_0, \fg_1^{\sab}) \coloneqq [\fg_0, \LCS_{r-1}(\fg_0, \fg_1^{\sab})].
\end{align}
By definition, the elements of $\LCS_n(\fg_0, \fg_1^{\sab})$ are all the elements of at least level $n$.
Then, we define the $n$-truncated nilpotent Lie algebra $\fg_1^{\gr{\leq n}}$ by
\begin{align}\nonumber
    \fg_1^{\gr{\leq n}} \coloneqq \fg_1^{\sab}/\LCS_{n+1}(\fg_0, \fg_1^{\sab}).
\end{align}
We define the \emph{level $n$ truncation of $\cmg^{\sab}$} to be 
\begin{align} \label{eq:truncated_dcm}
    \cmg^{\gr{\leq n}} \coloneqq (\delta: \fg_1^{\gr{\leq n}} \to \fg_0^{\gr{\leq n}}, \gt),
\end{align}
where $\fg_0^{\gr{\leq n}}$ is the truncation of $\fg_0$ defined in~\Cref{eq:truncated_fg0}.
Then, we can embed this crossed module of Lie algebras into $\Lie(\bT^{\gr{\leq n}})$ by the composition
\begin{align}\nonumber
    \cmg^{\gr{\leq n}} \xhookrightarrow{\iota} \Lie(\bT(\V)) \xrightarrow{\Lie(\pi^{\gr{n}})} \Lie(\bT^{\gr{\leq n}}(\V)),
\end{align}
where $\iota$ is the inclusion from~\Cref{prop:inclusion_kapranov}, and $\pi^{\gr{n}}$ is the truncation morphism. We can now define a crossed module of group-like elements. Let $\hT_1^{\gr{\leq n}, 0}, \hT_1^{\gr{\leq n}, 1} \subset \hT_1^{\gr{\leq n}}$ denote the elements with constant term $0$ and $1$ respectively. We define the truncated exponential map $\exp_*: \hT_1^{\gr{\leq n},0} \to \hT_1^{\gr{\leq n}, 1}$ and truncated logarithm map $\log_*: \hT_1^{\gr{\leq n},1} \to \hT_1^{\gr{\leq n},0}$ by 
\begin{align}\nonumber
    \exp_*(E) \coloneqq \sum_{k=0}^{\lfloor n/2 \rfloor} \frac{E^{*k}}{k!} \andd \log_*(1 + F) \coloneqq \sum_{k=0}^{\lfloor n/2 \rfloor} (-1)^{k+1} \frac{F^{*k}}{k}
\end{align}
respectively. We take the sum up to $\lfloor n/2 \rfloor$ since $E$ and $F$ are of degree at least $2$. 
Note that both the exponential and the logarithm are continuous with respect to the norms on $\hT_1$.

\begin{lemma} \label{eq:cm_truncated_group_like}
    We define the \emph{level $n$ group-like elements} to be 
    \begin{align}\nonumber
        G_1^{\gr{\leq n}} \coloneqq \{ \exp_*(E) \in \hT_1^{\gr{\leq n}} \, : \, E \in \fg_1^{\gr{\leq n}} \} \andd G_0^{\gr{\leq n}} \coloneqq \{ \exp(x) \in T_0^{\gr{\leq n}} \, : \, x \in \fg_0^{\gr{\leq n}} \}.
    \end{align}
    The boundary map of $\bT^{\gr{\leq n}}$ and the action of the invertible elements from~\Cref{prop:cma_invertible} induce a group homomorphism $\delta: G_1^{\gr{\leq n}} \to G_0^{\gr{\leq n}}$ and a group action $\gt$ of $G_0^{\gr{\leq n}}$ on $G_1^{\gr{\leq n}}$ of the form
    \begin{align} \label{eq:grplike_cm_prop}
        \delta(\exp_*(E)) = \exp(\delta E) \andd \exp(x) \gt \exp_*(E) = \exp_*\left(\sum_{k=0}^n \frac{\ad_x^k(E)}{k!}\right)
    \end{align}
    for $x \in \fg_0^{\gr{\leq n}}$ and $E \in \fg_1^{\gr{\leq n}}$, where $\ad_x(E) = [x,E]$.
    Then
    \begin{align} \label{eq:truncated_gcm}
        \cmG^{\gr{\leq n}} = (\delta: G_1^{\gr{\leq n}} \to G_0^{\gr{\leq n}}, \gt)
    \end{align}
    is a crossed module of Lie groups.
\end{lemma}
\begin{proof}
    The identity $\delta(\exp_*(E)) = \exp(\delta E)$ holds since $\delta$ is an algebra morphism. For the action, we have
    \begin{align}\nonumber
        \exp(x) \gt \exp_*(E) = \exp_* \left( \exp(x) \gt E\right)
    \end{align}
    For the action, we note that
    \begin{align}\nonumber
        \ad^k_x(F) = \sum_{i=0}^k (-1)^i \binom{k}{i} x^{\cdot(k-i)} \cdot E \cdot x^{\cdot i}.
    \end{align}
    Then, by direct computation, we have
    \begin{align}\nonumber
        \exp(x) \gt E &= \left(\sum_{k=0}^n \frac{x^{\cdot k}}{k!}\right) \cdot E \cdot \left(\sum_{k=0}^n \frac{x^{\cdot k}}{k!} \right) = \sum_{k=0}^n \sum_{i=0}^k (-1)^i \frac{x^{\cdot(k-i)} \cdot E \cdot x^{\cdot i}}{(k-i)! \, k!} = \sum_{k=0}^n \frac{\ad_x^k(E)}{k!}.
    \end{align}
    Thus, both the boundary map and action are well-defined on group-like elements. Then, $\cmG^{\gr{\leq n}}$ is a crossed module by the same argument as~\Cref{prop:cma_invertible}. 
\end{proof}

\begin{remark} \label{rem:universal_enveloping_group_like}
    A natural question at this point is whether the terminology \emph{group-like elements} for $G_1$ is appropriate. One can see that $T_1(\V)$ is \emph{not} the universal enveloping algebra of $\fg_1^{\sab}$. In particular, using Kapranov's decomposition~\Cref{eq:kapranov_decomposition}, the universal enveloping algebra of $\fg_1^{\sab}$ is
    \begin{align}\nonumber
        U(\fg_1^{\sab}) \cong U([\fg_0, \fg_0]) \oplus S(\fg_1^{\sab,\ker}),
    \end{align}
    where $S$ denotes the symmetric algebra here, since $\fg_1^{\sab,\ker}$ is an abelian Lie algebra. Then, given an element $\ell \in \fg_1^{\sab,\ker}$, one can see that
    \begin{align}\nonumber
        \ell * \ell = 0 \text{ (in $T_1(\V)$) } \andd \ell \bullet \ell \neq 0 \text{ (in $U(\fg_1^{\sab})$)},
    \end{align}
    where $\bullet$ denotes the product in $U(\fg_1^{\sab})$. When we consider the exponential of elements $\ell, \ell'\in \fg_1^{\sab,\ker}$, we note that $\exp_\bullet(\ell)$ is simply the usual commutative exponential in the symmetric algebra $S(\fg_1^{\sab,\ker})$; furthermore, we have
    \begin{align}\nonumber
        \exp_*(\ell) * \exp_*(\ell') &= (1 + \ell) * (1 + \ell') = (1 + (\ell + \ell')) = \exp_*(\ell + \ell')\\
        \exp_\bullet(\ell) \bullet \exp_\bullet(\ell') &= \exp_\bullet(\ell + \ell').
    \end{align}
    Thus, the group-like elements as defined in~\Cref{eq:cm_truncated_group_like} contain the same information as the group-like elements of $U(\fg_1^{\sab})$, and indeed have the same algebraic structure.
\end{remark}

We consider two lemmas that relate the group-like elements with the decomposition of $T_1$.

\begin{lemma}
    Let $s: L \to \fg_1^{\sab}$ be a linear section. Then, $G_1^{\gr{\leq n}}$ is closed and included in the subspace
    \begin{align}\nonumber
        G_1^{\gr{\leq n}} \subset T_1^{s, \gr{\leq n}} \oplus T_1^{\ker, \gr{\leq n}}.
    \end{align}
\end{lemma}
\begin{proof}
    We have $\fg_1^{\gr{\leq n}} \subset T_1^{\gr{\leq n}}$ because it is a finite dimensional linear subspace of $T_1$. Then, because the logarithm $\log_*$ is continuous, $G_1^{\gr{\leq n}}$ is closed. 
    Because $\delta(E) \in G_0^{\gr{\leq n}} \subset U^{\gr{\leq n}}([\fg_0, \fg_0])$. Then, since $E - \ts \circ \delta(E) \in T_1^{\ker, \gr{\leq n}}$, we have our desired result.
\end{proof}

\begin{lemma} \label{lem:section_preserves_grplike}
    Let $s: L \to \fg_1^{\sab}$ be a linear section and $\ts: U([\fg_0, \fg_0]) \to T_1$ be the associated algebra section. Then, we have
    \begin{align}\nonumber
        \ts(G_0^{\gr{\leq n}} \cap U^{\gr{\leq n}}([\fg_0, \fg_0])) \subset G_1^{\gr{\leq n}}.
    \end{align}
\end{lemma}
\begin{proof}
    Let $x \in [\fg_0, \fg_0]$ such that $\exp(x) \in G_0^{\gr{\leq n}} \cap U([\fg_0, \fg_0])$. Then, since $\ts$ is a morphism of algebras, we have $\ts(\exp(x)) = \exp_*(\ts(x)) = \exp_*(\hs(x))$, where $\hs: [\fg_0, \fg_0] \to \fg_1^{\sab}$ is the Lie algebra section induced by $s$.
\end{proof}

By definition, $\cmg^{\gr{\leq n}}$ is the crossed module of Lie algebras associated to $\cmG^{\gr{\leq n}}$. Both of these crossed modules can be embedded into the crossed module of Banach algebras $\bT^{\gr{\leq n}}$; in particular within the subspaces
\begin{align}\nonumber
    \fg_0^{\gr{\leq n}}, G_0^{\gr{\leq n}} \subset U^{\gr{\leq n}}([\fg_0, \fg_0]) \andd \fg_1^{\gr{\leq n}}, G_1^{\gr{\leq n}} \subset T_1^{s, \gr{\leq n}} \oplus T_1^{\ker, \gr{\leq n}}.
\end{align}

\subsubsection{Completion}

In this section, we will consider the completion of the free crossed module $\bT(\V)$. In particular, we define the algebras of formal power series
\begin{align}\nonumber
    T_0\ps{\V} \coloneqq \prod_{n=0}^\infty \V^{\otimes n} \andd T_1\ps{\V} \coloneqq \prod_{n=2}^\infty T_1^{\gr{n}}(\V).
\end{align}
Then, we define the corresponding crossed module of algebras
\begin{align}\nonumber
    \bT\ps{\V} \coloneqq (\cmb: T_1\ps{\V} \to T_0\ps{\V}, \gtd, \ltd),
\end{align}
where the boundary map $\cmb$ and actions $\gtd$ and $\ltd$ are induced from $\bT(\V)$. \medskip

Recall from~\Cref{eq:E0_characterization} and~\Cref{cor:E0_universal} that $E_0(\V)$ is the completion of $T_0(\V)$ equipped with the topology where all extensions $\tcona: T_0(\V) \to A_0$ of continuous linear maps $\cona \in L(\V, A_0)$ into Banach spaces are continuous. We can obtain an analogous locally $m$-convex topology on $\hT_1(\V)$ generated by the norms $\{P_\lambda\}_{\lambda >0}$ defined by
\begin{align} \label{eq:Plambda_norm}
    P_\lambda(E) \coloneqq \|E^{\gr{0}}\| + \sum_{k=2}^\infty \lambda^k \|E^{\gr{k}}\|.
\end{align}
We define the completion of $\hT_1(\V)$ with respect to this topology by
\begin{align}\nonumber
    E_1(\V) = \left\{ E \in \hT_1\ps{\V} \, : \, \sum_{k=2}^\infty \lambda^k \|E^{\gr{k}}\| < \infty \, \text{ for all } \, \lambda > 0 \right\}.
    \end{align}
In particular, we can define a crossed module of locally $m$-convex algebras.

\begin{lemma}
    The structure $\bE(\V) \coloneqq (\delta: E_1(\V) \to E_0(\V), \gtd, \ltd)$ is a crossed module of locally $m$-convex algebras.
\end{lemma}

Furthermore, the universal 2-connection $(\zeta, Z)$ from~\Cref{prop:univ_2con_alg} can be embedded into a \emph{continuous} universal $2$-connection $\zeta \in L(\V, E_0(\V))$ and $Z \in L(\Lambda^2 \V, E_1(\V))$ valued in $\bE(\V)$. 
It satisfies a similar universal property such that the unique morphism (from~\Cref{prop:univ_2con_alg}) induced by a continuous $2$-connection $(\cona, \conc)$ into a crossed module of Banach algebras is continuous. (Compare this with~\Cref{cor:E0_universal}.) Further details on the construction of $E_1(\V)$, and the proof of the previous and following result are in~\Cref{apxsec:universal_locally_convex}.

\begin{proposition} \label{prop:E0_con_universal}
    Let $\cmA = (\delta^A: A_1 \to A_0, \gtd, \ltd)$ be a crossed module of Banach algebras. Let $(\cona, \conc)$ be a continuous connection valued in $\cmA$. Then, there exist unique continuous morphisms of algebras $\tcona: E_0(\V) \to A_0$ and $\tconc: E_1(\V) \to A_1$ such that
    \begin{align}\nonumber
        \cona = \tcona \circ \zeta \andd \conc = \tconc \circ Z.
    \end{align}
\end{proposition}

Similar to the case of paths in~\Cref{eq:G0_def}, we will define the \emph{group-like elements} to be 
\begin{align}\nonumber
    G_1 \coloneqq \left\{ E \in E_1(\V) \, : \, \pi_1^{\gr{\leq n}}(E) \in G_1^{\gr{\leq n}} \right\}.
\end{align}
In particular, we obtain a crossed module of groups,
\begin{align}\nonumber
    \cmG = (\cmb: G_1 \to G_0, \gt),
\end{align}
since all properties hold level-wise by definition.
\medskip

\subsection{The Surface Signature}

In order to simplify notation, we omit $\V$ from the notation for all variants of the crossed modules $\bT$ and $\bE$.
Let $\ucona \in L(\V, E_0)$ and $\uconc \in L(\Lambda^2 \V, E_1)$ be the universal 2-connection from~\Cref{prop:E0_con_universal}. Recall that the path holonomy with respect to $\ucona$ is the path signature
\begin{align}\nonumber
    \sig : C^\infty([0,1], \V) \to G_0.
\end{align}

\subsubsection{Definition of the Surface Signature}
In this section, we consider surface holonomy (\Cref{def:sh}) of a smooth surface $X \in \surfaces^\infty(\V)$ with respect to $Z$. Recall that $Z : \Lambda^2 \V \xhookrightarrow{\cong} T_1^{\gr{2}} \subset E_1$ is the natural inclusion map. Then, evaluating the universal 2-connection $Z \in L(\Lambda^2 \V, e_1)$ on the partial derivatives of $X = (X^1, \ldots, X^d)$ embeds the Jacobian into $E_1$,
\begin{align}\nonumber
    Z\left(\frac{\partial X_{s,t}}{\partial s}, \frac{\partial X_{s,t}}{\partial t}\right) = J_{s,t}(X) \in \Lambda^2\V \subset E_1,
\end{align}
where $J(X) \in \surfaces^\infty(\Lambda^2 \V)$ consists of the Jacobian minors of $X$, defined by
\begin{align}\nonumber
    J_{s,t}(X) = \sum_{i < j} \left( \frac{\partial X^i_{s,t}}{\partial s}\cdot  \frac{\partial X^j_{s,t}}{\partial t} - \frac{\partial X^j_{s,t}}{\partial s}\cdot  \frac{\partial X^i_{s,t}}{\partial t} \right) \, e_i \wedge e_j.
\end{align}

Note that by~\Cref{prop:inclusion_kapranov}, the universal 2-connection is valued in the crossed module of Lie algebras $\cmg$ or in $\cmg^{\gr{\leq n}}$ in the truncated setting, which are embedded in $\bE$ and $\bT^{\gr{\leq n}}$ respectively. In the truncated case, the surface holonomy will be valued in $\cmG^{\gr{\leq n}}$ since this is the associated Lie group of $\cmg^{\gr{\leq n}}$~\cite{martins_surface_2011}.

\begin{definition}
    Let $n \geq 2$. The \emph{$n$-truncated surface signature} is the surface holonomy map with respect to the universal 2-connection valued in $\cmg^{\gr{\leq n}}$. In particular, it is a map
    \begin{align}   \nonumber
        \hssig^{\gr{\leq n}}: \surfaces(\V) \to G_1^{\gr{\leq n}}
    \end{align}
    defined for $X \in C^\infty(Q, \V)$ on a rectangle $Q = [\smp]\times [\tmp] \subset\R^2$ as the solution of the differential equation for (recall that $x^{s,t}$ is the tail path from~\Cref{eq:tail_path})
    \begin{align} \label{eq:truncated_ssig_def}
        \hssig_{s,t}^{\gr{\leq n}}(X): Q \to G_1^{\gr{\leq n}} \quad \text{where} \quad \frac{\partial \hssig_{s,t}^{\gr{\leq n}}(X)}{\partial t} = \hssig_{s,t}^{\gr{\leq n}}(X) * \int_{\smm}^s S(x^{s',t}) \gt J_{s',t}(X) \, ds',
    \end{align}
    with initial condition $\hssig_{s,\tmm}(X) = 1$.
    The $n$-truncated surface signature of $X$ is $\hssig^{\gr{\leq n}}(X) \coloneqq \hssig_{\spp,\tpp}^{\gr{\leq n}}(X)$.
\end{definition}

For $n < m$, we note that $\pi^{\gr{\leq n}}_1(\hssig^{\gr{\leq m}}(X)) = \hssig^{\gr{\leq n}}(X)$ because they satisfy the same differential equation. Then, we can define the untruncated surface signature in the same way. 
\begin{definition}
    The \emph{surface signature} is a map
    \begin{align}   \nonumber
        \hssig:\surfaces(\V) \to G_1 \subset T_1\ps{\V}
    \end{align}
    defined for $X \in C^\infty(Q, \V)$ on a rectangle $Q = [\smp]\times[\tmp]$ as the solution of the differential equation
    \begin{align} \label{eq:ssig_def}
        \hssig_{s,t}(X): Q \to G_1 \quad \text{where} \quad \frac{\partial \hssig_{s,t}(X)}{\partial t} = \hssig_{s,t}(X) * \int_{\smm}^s S(x^{s',t}) \gt J_{s',t}(X) \, ds',
    \end{align}
    with initial condition $\hssig_{s,\tmm}(X) = 1_1$.
    The $n$-truncated surface signature of $X$ is $\hssig(X) \coloneqq \hssig_{\spp,\tpp}(X)$.
\end{definition}
Note that the surface signature is valued in $G_1$, since $\pi^{\gr{\leq n}}_1(\hssig(X)) = \hssig^{\gr{\leq n}}(X)$ for all $n \geq 2$.
It is not immediate that $\hssig(X) \in E_1$, but we will show this in~\Cref{cor:smooth_extension_surface_signature}, and simply state it is valued in $G_1 \subset E_1$ for now. 
Furthermore, since $J_{s,t}(X) \in T_1$, we also have $S(x^{s',t}) \gt J_{s',t}[X] \in T_1$, so it does not contain any unital component. Thus, the only unital component comes from the initial condition, and surface signature has the form $\hssig_{s,t}(X) = (1, \ssig_{s,t}(X))$ for all $(s,t) \in [0,1]^2$, where $\ssig_{s,t}(X) \in T_1$ is the non-unital component of the surface signature. 
Now, we can rewrite this differential equation as an integral equation of the form
\begin{align}\nonumber
    \hssig_{s,t}(X) = 1 + \int_0^t \int_0^s \hssig_{s,t'}(X) *\Big(S(x^{s',t'}) \gt J_{s',t'}[X]\Big) \, ds' dt'
\end{align}
Applying a Picard iteration scheme, we obtain
\begin{align} \label{eq:ssig_picard}
    \hssig_{s,t}(X) = 1 + \sum_{n=1}^\infty \int_{\square^n(\smp)} \int_{\Delta^n(\tmp)} \Big(S(x^{s_1,t_1}) \gt J_{s_1,t_1}[X]\Big) * \ldots * \Big(S(x^{s_n,t_n}) \gt J_{s_n,t_n}[X]\Big) \, d\bs \, d\bt,
\end{align}
where $(s_1, \ldots, s_n) \in \square^n(\smp)$ and $(t_1 < \ldots < t_n) \in \Delta^n(\tmp)$. In contrast to the path signature, the depth of the Picard iteration does not correspond to the level in $T_1$. Indeed, consider the first term of the Picard series for $\hssig(X)$,
\begin{align} \label{eq:first_picard}
    \int_0^1 \int_0^1 S(x^{s,t}) \gt J_{s,t}(X) \, ds \, dt.
\end{align}
While $J_{s,t}(X)$ is a level $2$ element, $S(x^{s,t})$ has nontrivial components at each level. 

\begin{remark}
    We can introduce a separate grading for the Picard iterates by defining for $p \geq 1$
    \begin{align} \label{eq:hssig_picard_grading}
        \hssig^{\pgr{p}}(X) \coloneqq \int_{\square^p(\smp)} \int_{\Delta^p(\tmp)} \Big(S(x^{s_1,t_1}) \gt J_{s_1,t_1}[X]\Big) * \ldots * \Big(S(x^{s_p,t_p}) \gt J_{s_p,t_p}[X]\Big) \, d\bs \, d\bt.
    \end{align}
    In fact, this can be done formally by considering surface holonomy valued in
    \begin{align}\nonumber
        T_1^h \coloneqq \left\{\sum_{p=0}^\infty E^{\pgr{p}} h^p \, : \, E^{\pgr{p}} \in T_1\right\},
    \end{align}
    formal power series in an auxiliary variable $h$ with coefficients in $T_1$, where we embed the Lie algebra in $h$-degree $1$, and treat the action of $T_0$ on $T_1^h$ as an $h$-degree $0$ operation. This is similar to the construction of group-like elements in~\cite[Section 2.4.4]{faria_martins_crossed_2016}, though we will just leave this as a remark and not consider such power series in $h$ here.
\end{remark}

Despite the complexity at higher levels (partially due to the quotienting by the Peiffer ideal $\Pf$), the first two nontrivial levels of the surface signature are fairly straightforward.

\begin{example}[Level 2] \label{ex:level2}
    Level 2 is the lowest nontrivial level of $T_1$, and we note that $T_1^{\gr{2}} \cong \Lambda^2 \V$. For a surface $X \in C^\infty([0,1]^2, \V)$, the surface signature is
    \begin{align}\nonumber
        \hssig^{\gr{2}}(X) = \int_0^1 \int_0^1 J_{s,t}(X) \, ds \,dt,
    \end{align}
    which is the signed area (L\'evy area) of the boundary path. In particular, if we denote the boundary path of $X$ by $\partial X$, this can be computed by the path signature,
    \begin{align}\nonumber
        \hssig^{\gr{2}}(X) = S^{\gr{2}}(\partial X) \in \Lambda^2 \V,
    \end{align}
    which is valued in $\Lambda^2 \V$ since $\partial X$ is a loop.
\end{example}

\begin{example}[Level 3] \label{ex:level3}
    At level 3, we have
    \begin{align}\nonumber
        T_1^{\gr{3}} \cong (V \otimes \Lambda^2 \V) \oplus (\Lambda^2 \V \otimes V),
    \end{align}
    since elements of the Peiffer ideal $\Pf$ are of degree at least $4$, so the quotient does not affect level $3$. For $X \in C^\infty([0,1]^2, \V)$, only the the first Picard iterate in~\Cref{eq:first_picard} contributes to the level $3$ signature. In particular, we note that $S^{\gr{1}}(x^{s,t}) = X_{s,t} - X_{0,0}$, so the level $3$ signature can be explicitly expressed as 
    \begin{align}\nonumber
        \hssig^{\gr{3}}(X) = \int_0^1 \int_0^1 (X_{s,t} - X_{0,0}) \gt J_{s,t}(X) \, ds \, dt \in T_1^{\gr{3}}.
    \end{align}
\end{example}

In general, computation of the surface signature (or even surface holonomy) is more complicated than the path signature. In fact, even for a linear surface, it is difficult to obtain an explicit expression for the full signature. See~\cite[Section 6]{lee_random_2023} for some methods to compute the surface holonomy of a linear surface. Similar to the notion of the surface holonomy \emph{functor} in~\Cref{def:sh_functor}, we define the surface signature functor. 

\begin{definition}
    The \emph{surface signature functor} $\bS: \surfaces^\infty(\V) \to \dg(\cmG)$ is defined for a surface $X \in \surfaces^\infty(\V)$ with boundary paths $x,y,z,w \in \paths^\infty(\V)$ (see the figure after~\Cref{eq:intro_sh_functor}) by
    \begin{align}\nonumber
        \bS(X) \coloneqq \left(S(x), \, S(y), \, S(z), \, S(w), \, \hssig(X)\right).
    \end{align}
\end{definition}
Furthermore, the surface signature is invariant with respect to thin homotopy (as already discussed by Kapranov~\cite{kapranov_membranes_2015}), but we will not discuss this further here (see~\Cref{rem:thin_homotopy}).

As a direct corollary of~\Cref{prop:sh_multiplicative}, the surface signature functor is \emph{multiplicative} (since this holds for all truncations): for appropriately composable $X,Y,Z \in \surfaces^\infty(\V)$, we have
\begin{align}\nonumber
    \bS(X \concat_h Y) = \bS(X) \hmult \bS(Y) \andd \bS(X \concat_v Z) = \bS(X) \vmult \bS(Z).
\end{align}

\medskip
\subsubsection{Preliminary Analytic Properties of the Surface Signature}
We begin our analytic study of the surface signature with a preliminary coarse bound for the level $n$ surface signature. We will use
\begin{align} \label{eq:1d_increment}
    \inc{s_1}{s_2} \coloneqq |s_2 - s_1| \andd \inc{t_1}{t_2} \coloneqq |t_2 - t_1|
\end{align}
to denote 1D increments of the parametrization. 

\begin{proposition} \label{prop:ssig_preliminary_n_bound}
    Let $Q = [\smp]\times [\tmp] \subset \R^2$ be a rectangle and $X \in C^\infty(Q, \V)$. Then,
    \begin{align}\nonumber
        \left\| \hssig^{\gr{n}}(X) \right\| \leq C_n (\inc{\smm}{\spp} + \inc{\tmm}{\tpp})^n
    \end{align}
    where $C_n > 0$ is a constant which depends on $X$ and $n$.
\end{proposition}
\begin{proof}
    We start by considering the level $n$ component of the $p^{th}$ Picard iterate from~\Cref{eq:hssig_picard_grading}. Note that this is only nontrivial when $n \geq 2p$. By expanding the definition, $\hssig^{\pgr{p}, \gr{n}}(X)$ is 
    \begin{align*}
        \sum \int_{\substack{\Delta^n(\tmp)\\ \square^n(\smp)}} \Big(S^{\gr{i_1}}(x^{s_1,t_1}) \cdot J_{s_1,t_1}[X] \cdot S^{-1, \gr{j_1}}(x^{s_1, t_1})\Big) * \ldots * \Big(S^{\gr{i_p}}(x^{s_p,t_p})\cdot J_{s_n,t_n}[X] \cdot S^{-1, \gr{j_p}}(x^{s_p, t_p})\Big) \, d\bs \, d\bt,
    \end{align*}
    where the sum is taken over all indices such that $i_1 + \ldots + i_p + j_1 + \ldots j_p = n - 2p$. 
    Then, using the bounds
    \begin{align}\nonumber
        \|S^{\gr{n}}(x^{s,t})\|, \|S^{-1, \gr{n}}(x^{s,t})\| \leq \frac{L^n(\inc{\smm}{\spp} + \inc{\tmm}{\tpp})^n}{n!} \andd \|J_{s,t}(X)\| \leq 2L^2,
    \end{align}
    and by applying the multinomial theorem, we get
    \begin{align}\nonumber
        \left\| \hssig^{\pgr{p}, \gr{n}}(X) \right\| &\leq \frac{(2pL)^{n-2p}(\inc{\smm}{\spp} + \inc{\tmm}{\tpp})^{n-2p}}{(n-2p)!} \int_{\substack{\Delta^n(\tmp)\\ \square^n(\smp)}} (2L^2)^p d\bs d\bt \\
        & \leq(2L)^n p^{n-2p} \frac{(\inc{\smm}{\spp} + \inc{\tmm}{\tpp})^{n-2p}(\inc{\smm}{\spp} \cdot \inc{\tmm}{\tpp})^p}{(n-2p)! p!} \nonumber
    \end{align}
    Then, summing all non-trivial Picard iterates for level $n$, we get
    \begin{align}\nonumber
        \left\| \hssig^{\gr{n}}(X) \right\| &\leq (2L)^n  \sum_{p=1}^{\lfloor n/2 \rfloor}p^{n-2p} \frac{(\inc{\smm}{\spp} + \inc{\tmm}{\tpp})^{n-2p}(\inc{\smm}{\spp} \cdot \inc{\tmm}{\tpp})^p}{(n-2p)! p!} \\
        & \leq C_n (\inc{\smm}{\spp} + \inc{\tmm}{\tpp})^n\nonumber
    \end{align}
    where $C_n> 0$ is a constant which depends on $L$ and $n$. 
\end{proof}

We emphasize that the constant $C_n$ is not sharp. In fact, $C_n \to \infty$ as $n \to \infty$, and this cannot be used to show that $\hssig(X) \in E_1$ (or even to show it is bounded). One issue is that our naive estimates do not take advantage of cancellations and continuity properties of path signature. We will return to this in~\Cref{cor:smooth_extension_surface_signature}, where we use the uniqueness of extensions to provide a more precise estimate. However, we can still directly show that $\hssig(X)$ is bounded by considering the factorial decay of the Picard iterates.

\begin{proposition}
    Let $Q = [\smp]\times [\tmp] \subset \R^2$ be a rectangle and $X \in C^\infty(Q, \V)$. Then for $p \in \N$, we have
    \begin{align}\nonumber
        \|\hssig^{\pgr{p}}(X)\| \leq \frac{(C\inc{\smm}{\spp} \cdot \inc{\tmm}{\tpp})^p}{p!}
    \end{align}
    for a constant $C> 0$ which depends on $X$.
\end{proposition}
\begin{proof}
    First, the Jacobian $J_{s,t}(X)$ can be bound by
    \begin{align}\nonumber
        \|J_{s,t}(X)\|^2 &  \leq 2 \left\|\frac{\partial X_{s,t}}{\partial s} \right\|^2 \cdot \left\|\frac{\partial X_{s,t}}{\partial t} \right\|^2 \leq 2 L^4
    \end{align}
    Furthermore, the signature has the bound
    \begin{align}\nonumber
        \|S(x^{s,t})\|, \|S^{-1}(x^{s,t})\| \leq \exp((\inc{\smm}{\spp} + \inc{\tmm}{\tpp})L)
    \end{align}
    for any tail path $x^{s,t}$. Then, by the fact that the action in $\bT$ is short, we have
    \begin{align}\nonumber
        \|S(x^{s,t}) \gt J_{s,t}(X)\| \leq \|S(x^{s,t})\| \cdot \|J_{s,t}(X)\| \cdot \|S^{-1}(x^{s,t})\| \leq 2L^2 \exp((\inc{\smm}{\spp} + \inc{\tmm}{\tpp})L)^2.
    \end{align}
    Let $C = 2L^2 \exp((\inc{\smm}{\spp} + \inc{\tmm}{\tpp})L)^2$. Then by the definition of the $p^{th}$ Picard iterate in~\Cref{eq:ssig_picard}, we have
    \begin{align}\nonumber
        \|\hssig^{\pgr{p}}(X)\| &\leq \int_{\square^n(\smp)} \int_{\Delta^n(\tmp)} \left\|\Big(S(x^{s_1,t_1}) \gt J_{s_1,t_1}[X]\Big) * \ldots * \Big(S(x^{s_n,t_n}) \gt J_{s_n,t_n}[X]\Big)\right\| \, d\bs \, d\bt \\
        & \leq \int_{\square^n(\smp)} \int_{\Delta^n(\tmp)}  C^n \, d\bs \, d\bt \nonumber\\
        & \leq \frac{(C\inc{\smm}{\spp} \cdot \inc{\tmm}{\tpp})^n}{n!}.\nonumber
    \end{align}
\end{proof}

\subsection{Universal Property}
Next, we will consider the universal property of the surface signature, which reflects the corresponding property of the path signature in~\Cref{thm:universal_path_signature}. Suppose $(\cona, \conc)$ is a continuous 2-connection valued in a crossed module of Banach algebras $\cmA$. Now let $\tcona: T_0(\V) \to A_0$ and $\tconc : T_1(\V) \to A_1$ be the unique maps from~\Cref{prop:univ_2con_alg} which form a morphism between crossed modules $(\tcona, \tconc): \bT(\V) \to \cmA$ such that $\cona = \tcona \circ \ucona$ and $\conc = \tconc \circ \uconc$.
Suppose $v_1, \ldots, v_k \in T_0$ and $E_1, \ldots, E_k \in T_1$. Then, since $(\tcona, \tconc)$ is a morphism of crossed modules of algebras, we have
\begin{align} \label{eq:tconc_factoring_property}
    \tconc\Big( (v_1 \gt E_1) *_T \ldots  *_T (v_k \gt E_k)\Big) &= \tconc(v_1 \gt E_1) *_A \ldots *_A \tconc(v_k \gt E_k)\\
    & = \Big(\tcona(v_1) \gt \tconc(E_1)\Big) *_A \ldots *_A \Big(\tcona(v_k) \gt \tconc(E_k)\Big).
\end{align}

\begin{theorem}\label{thm:surface_signature_universal}
    Let $(\cona, \conc)$ be a continuous 2-connection valued in a crossed module of Banach algebras $\cmA$. Let $(\tcona, \tconc): \bE \to \cmA$ be the unique continuous morphism of crossed modules from~\Cref{prop:E0_con_universal}. Then, the surface holonomy map $\hH^{\cona, \conc}: \surfaces^\infty(\V) \to A_1$ from~\Cref{def:sh} satisfies
    \begin{align}\nonumber
        \hH^{\cona, \conc}(X) = \tconc\left(\hssig(X)\right).
    \end{align}
\end{theorem}
\begin{proof}
    Similar to the case of the surface signature, we can use the Picard iteration to define the surface holonomy $\hH^{\cona, \conc}$ by
    \begin{align}\nonumber
        \hH^{\cona, \conc}(X) &= 1 + \sum_{n=1}^\infty \int_{\substack{\Delta^n(\tmp)\\ \square^n(\smp)}} \Big( F^\cona(x^{s_1, t_1}) \gt \gamma(J_{s_1, t_1}(X))\Big) * \ldots * \Big( F^\cona(x^{s_n, t_n}) \gt \gamma(J_{s_n, t_n}(X))\Big) \, d\bs \, d\bt.
    \end{align}
    Then, using the universal property of the path signature in~\Cref{thm:universal_path_signature}, the fact that $\tconc = \conc$ restricted to $\Lambda^2 \V \subset T_1$, and the property in~\Cref{eq:tconc_factoring_property} that $\tconc$ forms a morphism of crossed modules, we obtain
    \begin{align}\nonumber
        \hH^{\cona, \conc}(X) &=1+ \sum_{n=1}^\infty \int_{\substack{\Delta^n(\tmp)\\ \square^n(\smp)}} \Big( \tcona(S(x^{s_1, t_1})) \gt \tconc(J_{s_1, t_1}(X))\Big) * \ldots * \Big( \tcona(S(x^{s_n, t_n})) \gt \tconc(J_{s_n, t_n}(X))\Big) \, d\bs \, d\bt \\
        & = 1 + \sum_{n=1}^\infty \tconc\left(\int_{\substack{\Delta^n(\tmp)\\ \square^n(\smp)}}  \Big(S(x^{s_1,t_1}) \gt J_{s_1,t_1}[X]\Big) * \ldots * \Big(S(x^{s_n,t_n}) \gt J_{s_n,t_n}[X]\Big) \, d\bs \, d\bt \right)\nonumber\\
        & = \tconc(\hssig(X)).\nonumber
    \end{align}
\end{proof}

\section{Rough Surfaces}

We will now go beyond the smooth setting, and consider surface holonomy of $\rho$-H\"older surfaces for $\rho \in (0,1]$. In particular, we introduce the notion of a \emph{rough surface}. Analogous to the case of paths, when surfaces become sufficiently irregular, we no longer have an appropriate notion of integration~\cite{zust_integration_2011}, so surface holonomy, including the surface signature, is no longer well-defined. However, a key insight in~\cite{lyons_differential_1998} is that one can still compute the path signature of an irregular path, provided that some lower level signature terms are postulated as additional data. This enriched path is called a \emph{rough path}, and can be viewed as a path in the truncated tensor algebra. One can then consider path holonomy of irregular paths via the universal property of the path signature. Our goal is to generalize these results to the case of surfaces. We define rough surfaces as surfaces valued in the truncated crossed module of group like elements, $\cmG^{\gr{\leq n}}$, where some lower level surface signature terms are postulated. We generalize the extension theorem from~\cite{lyons_differential_1998} to show that the surface signature of rough surfaces is well-defined, and can be used to compute the surface holonomy of rough surfaces with respect to continuous $2$-connections. \medskip

This section will use fractional factorials and binomial coefficients, defined by
\begin{align}\nonumber
    \ffact{n}{\rho} \coloneqq \left( n\rho\right)! = \Gamma\left(n\rho + 1\right) \andd \fbinom{n}{k}{\rho} \coloneqq \frac{\ffact{n}{\rho}}{\ffact{k}{\rho} \cdot \ffact{n-k}{\rho}},
\end{align}
where $\Gamma$ is the Gamma function.
Furthermore, we will regularly use the following neoclassical inequality, which was introduced in the development of rough paths~\cite{lyons_differential_1998} and further improved in~\cite{hara_fractional_2010}.
\begin{lemma}{\cite{hara_fractional_2010,lyons_differential_1998}}
    For any $p \in (0,1]$, $n \in \N$ and $s, t \geq 0$, 
    \begin{align} \label{eq:neoclassical}
        \rho \sum_{i=0}^n \frac{s^{i\rho} t^{(n-i)\rho}}{\ffact{i}{\rho} \ffact{n-i}{\rho}} \leq \frac{(s+t)^{n\rho}}{\ffact{n}{\rho}}.
    \end{align}
\end{lemma}
As a direct corollary, we have the following generalization of the sums of binomial coefficients. 
\begin{corollary} \label{eq:binomial_sum_bound}
    For any $p \in (0,1]$, $n \in \N$, we have
    \begin{align} \label{eq:genbinom_sum}
        \sum_{i=0}^n \fbinom{n}{i}{\rho} = \sum_{i=0}^n \frac{\ffact{n}{\rho}}{\ffact{i}{\rho} \ffact{n-i}{\rho}} \leq \frac{2^{n\rho}}{\rho}.
    \end{align}
\end{corollary}
Finally, we primarily work with paths and surfaces parametrized on $[0,1]$ and $[0,1]^2$ respectively to simplify notation. However, all arguments hold for arbitrary rectangular domains.

\begin{remark}
    Throughout this section, we will only consider the case of rough paths and surfaces valued in the truncated group like elements $G_0^{\gr{\leq n}}$ and $\cmG^{\gr{\leq n}}$; in other words, the \emph{weakly geometric} setting. We will always assume that these objects are embedded into their corresponding (crossed module of) Banach algebras,
    \begin{align}\nonumber
        G_0^{\gr{\leq n}} \subset T_0^{\gr{\leq n}} \andd \cmG^{\gr{\leq n}} \subset \bT^{\gr{\leq n}},
    \end{align}
    and are thus equipped with the corresponding norms. 
\end{remark}

\subsection{Rough Paths on Surfaces}
In this section, we will primarily consider $\rho$-H\"older paths (\Cref{eq:intro_holder_paths}) and surfaces (\Cref{eq:intro_holder_surfaces}).
In the case of paths, there is an intermediate regime of $\rho \in (\frac{1}{2}, 1]$, where Young integration~\cite{young_inequality_1936} can be used to compute the path signature of $\rho$-H\"older paths. In order to go beyond this, we must define the notion of a rough path~\cite{lyons_differential_1998}. We only provide a brief introduction here, and refer the reader to the standard references~\cite{lyons_system_2007,lyons_differential_2007,friz_multidimensional_2010,friz_course_2020} for further details. We use $\omega : \Delta^2 \to \R$, defined by 
\begin{align}\nonumber
    \ctr{}{\smp} \coloneqq C_\omega \inc{\smm}{\spp}
\end{align}
for some constant $C_\omega > 0$, as a H\"older control function.

\begin{definition}
    Let $n \in \N$. A \emph{multiplicative functional} is a map $\rx: \Delta^2 \to G_0^{\gr{\leq n}}$ such that for all $s_1 < s_2 < s_3$, 
    \begin{align}\nonumber
        \rx_{s_1, s_2} \cdot \rx_{s_2, s_3} = \rx_{s_1, s_3}.
    \end{align}
    Let $\rho \in (0, 1]$. A functional $\rx: \Delta^2 \to G_0^{\gr{\leq n}}$ is \emph{$\rho$-H\"older} if
    \begin{align} \label{eq:rp_regularity}
        \|\rx_{s_1, s_2}^{\gr{k}}\| \leq \frac{\ctr{\rho}{\smp}}{\beta \ffact{k}{\rho}}
    \end{align}
    for all $(s_1, s_2) \in \Delta^2$ and $k \in [n]$, where $\beta > 0$ is a sufficiently large constant. 
\end{definition}

Here, the constant $\beta > 0$ is a fixed constant which depends only on the regularity $\rho$. For the rough path extension, we require
\begin{align} \label{eq:rp_beta}
    \beta > \beta_{\RP} \coloneqq \frac{2}{\rho^2} \left( 1 + \sum_{r=3}^\infty \left(\frac{2}{r-2}\right)^{\rho(\lfloor \rho^{-1} \rfloor + 1)}\right).
\end{align}

The Lyons' extension theorem~\cite{lyons_differential_1998} shows that given a multiplicative functional $\rx : \Delta^2 \to T_0^{\gr{\leq n}}$ with the regularity bounds~\Cref{eq:rp_regularity} for all levels $k \in [n]$, we can extend it to a multiplicative function valued in $T_0^{\gr{\leq n+1}}$. Furthermore, if the multiplicative functional $\rx$ is valued in $G_0^{\gr{\leq n}}$, then the extension is also valued in $G_0^{\gr{\leq n+1}}$~\cite[Corollary 3.9]{cass_integration_2016}.

\begin{theorem}[Lyons' Extension Theorem~\cite{lyons_differential_1998,cass_integration_2016}] \label{thm:path_extension}
    Let $\rho \in (0,1]$ and let $n \geq \left\lfloor\frac{1}{\rho}\right\rfloor$. Suppose $\rx: \Delta \to G_0^{\gr{\leq n}}$ is a multiplicative functional which satisfies~\Cref{eq:rp_regularity} for each $k \in [n]$.
    Then, there exists a unique multiplicative function $\trx: \Delta \to G_0^{\gr{\leq n+1}}$ such that $\trx^{\gr{k}} = \rx^{\gr{k}}$ for all $k \in [n]$ and satisfies~\Cref{eq:rp_regularity} for each $k \in [n+1]$. 
\end{theorem}

Furthermore, this extension is continuous in the following sense.

\begin{theorem}[Continuity of Extension Theorem~\cite{lyons_differential_1998}] \label{thm:path_ext_orig_cont}
    Let $\rho \in (0,1]$ and let $n \geq \left\lfloor\frac{1}{\rho}\right\rfloor$. Suppose $\rx, \ry : \Delta \to T_0^{\gr{\leq n}}$ are two multiplicative functionals which satisfy~\Cref{eq:rp_regularity} for each $k \in [n]$ with respect to a common control $\omega$. If there exists some $\epsilon > 0$ such that 
    \begin{align}\nonumber
        \|\rx^{\gr{k}}_{s_1, s_2} - \ry^{\gr{k}}_{s_1, s_2}\| \leq \epsilon \frac{\ctr{k\rho}{s_1, s_2}}{\beta \ffact{k}{\rho}}
    \end{align}
    for all $k \in [n]$ and $(s_1, s_2) \in \Delta^2$, then this condition also holds for the extensions $\trx$ and $\try$ at level $k = n+1$. 
\end{theorem}

Rough paths are then defined to be multiplicative functionals defined up to the level required by the extension theorem.

\begin{definition}
    Let $\rho \in (0,1]$ and $n = \lfloor 1/\rho\rfloor$.  A \emph{(weakly geometric) $\rho$-rough path} is a $\rho$-H\"older multiplicative functional $\rx: \Delta^2 \to G_0^{\gr{\leq n}}$ valued in the group-like elements. The space of (weakly geometric) $\rho$-rough paths is denoted $\RP^\rho$. The $\rho$-H\"older metric on $\rho$-rough paths $\rx, \ry \in \RP^\rho$ is defined by
    \begin{align}\nonumber
        d_\rho(\rx, \ry) = \max_{k \in [n]} \sup_{s_1, s_2 \in \Delta^2} \frac{\|\rx^{\gr{k}}_{s_1, s_2} - \ry^{\gr{k}}_{t_1, t_2}\|}{\inc{s_1}{s_2}^{k\rho}}
    \end{align}
\end{definition}

The extension theorem allows us to compute the signature of $\rho$-rough paths by extending the multiplicative functional one level at a time, and leads to the following corollary.

\begin{corollary}
    The path signature $S: \RP^\rho \to G_0$ defined by $S(\rx) = \trx_{0,1}$, where $\trx: \Delta^2 \to G_0$ is the infinite extension, is continuous. 
\end{corollary}
\begin{proof}
    First, we can show that the signature is bounded and valued in $E_0$. Let $\lambda > 0$, and by the extension theorem in~\Cref{thm:path_extension}, we have
    \begin{align}\nonumber
        p_\lambda(S(\rx)) = \sum_{k=0}^\infty \frac{\lambda^k \ctr{k\rho}{0,1}}{\beta \ffact{k}{\rho}} < \infty,
    \end{align}
    where $p_\lambda$ is the norm in~\Cref{eq:plambda_norm}.
    Then, by the continuity of extensions in~\Cref{thm:path_ext_orig_cont}, this is continuous.
\end{proof}

\subsection{H\"older Double Group Functionals} \label{ssec:holder_dgf}

Consider a smooth path $x \in C^\infty([0,1], \V)$. By computing the path signature of $x$ over all subintervals $[\smp] \subset [0,1]$, we obtain a multiplicative functional $\rx : \Delta^2 \to G_0$ defined by
\begin{align}\nonumber
    \rx_{\smp} \coloneqq S(x|_{[\smp]}).
\end{align}
Now, consider a smooth surface $X \in C^\infty([0,1]^2, \V)$. We wish to define an analogous notion of a 2-dimensional multiplicative functional which encodes the surface signature of $X$ over all subrectangles $Q = [\smp] \times [\tmp] \subset [0,1]^2$, which can be parametrized by $(\smp; \tmp) \in \Delta^2 \times \Delta^2$. In particular, we can define a map $\rX: \Delta^2 \times \Delta^2 \to G_1$ by
\begin{align}\nonumber
    \rX_{\smp; \tmp} \coloneqq \hssig(X|_{Q}).
\end{align}
However, in order to define horizontal and vertical compositions, we must retain the information about the path signature of the boundary edges of this subrectangle. We can achieve this by instead defining a map $\rrX : \Delta^2 \times \Delta^2 \to \dg(\cmG)$ into the double group $\dg(\cmG)$ by the surface signature functor in~\Cref{def:sh_functor},
\begin{align}\nonumber
    \rrX_{\smp; \tmp} \coloneqq \bS(X|_{Q}) =  \left( S(\bdy_b X|_E), \, S(\bdy_r X|_E), \, S(\bdy_u X|_E), \, S(\bdy_l X|_E), \, \hssig(X|_{E})\right).
\end{align}
Then, we note that for all $s_1 < s_2 < s_3$ and $t_1 < t_2 < t_3$, the horizontal and vertical compositions,
\begin{align}\nonumber
    \rrX_{s_1, s_2; t_1, t_2} \hmult \rrX_{s_2, s_3; t_1, t_2} \andd \rrX_{s_1, s_2; t_1, t_2} \vmult \rrX_{s_1, s_2; t_2, t_3}
\end{align}
are well-defined (since the appropriate boundaries coincide by definition) using~\Cref{eq:dg_hmult} and~\Cref{eq:dg_vmult}.
In practice, it is convenient to the horizontal paths, vertical paths, and the surface elements separately, and this leads to our notion of a double group functional. 

\begin{definition}
    Let $n \in \N$. Let $\cmG^{\gr{\leq n}}$ be the crossed module of truncated group-like elements (\Cref{eq:cm_truncated_group_like}).  We define the sets of \emph{horizontal path functionals} and \emph{vertical path functionals} valued in $\cmG^{\gr{\leq n}}$ by
    \begin{align}\nonumber
        \HPF(G_0^{\gr{\leq n}}) \coloneqq \{ \rx^h : \Delta^2 \times [0,1] \to G_0^{\gr{\leq n}} \} \andd \VPF(G_0^{\gr{\leq n}}) \coloneqq \{ \rx^v : [0,1] \times \Delta^2 \to G_0^{\gr{\leq n}}\},
    \end{align}
    and the set of \emph{surface functionals} valued in $G_1^{\gr{\leq n}}$ by
    \begin{align}\nonumber
        \SF(G_1^{\gr{\leq n}}) \coloneqq \{ \rX: \Delta^2 \times \Delta^2 \to G_1^{\gr{\leq n}}\}.
    \end{align}
    A \emph{double group functional} $\rrX = (\rx^h, \rx^v, \rX)$ valued in $\cmG^{\gr{\leq n}}$ consists of three maps
    \begin{align}\nonumber
        \rx^h \in \HPF(G_1^{\gr{\leq n}}) \quad \rx^v \in \VPF(G_0^{\gr{\leq n}}), \quad \rX \in \SF(G_1^{\gr{\leq n}})
    \end{align}
    such that
    \begin{enumerate}
        \item the path functionals are \emph{path-multiplicative}
    \begin{align}\nonumber
        \rx^h_{s_1, s_2;t} \cdot \rx^h_{s_2, s_3; t} = \rx^h_{s_1, s_3; t} \andd \rx^v_{s; t_1, t_2} \cdot \rx^v_{s; t_2, t_3} = \rx^v_{s;t_1, t_3}
    \end{align}
    for all $s,t \in [0,1]$, $s_1 < s_2 < s_3$, and $t_1 < t_2 < t_3$; and
    \item $\rrX$ satisfies the \emph{boundary condition}
    \begin{align} \label{eq:rrX_boundary_condition}
        \cmb(\rX_{s_1,s_2; t_1, t_2}) = \rx^h_{s_1, s_2;t_1} \cdot \rx^v_{s_2; t_1, t_2} \cdot \rx^{-h}_{s_1; t_1, t_2} \cdot \rx^{-v}_{s_1;t_1, t_2}
    \end{align}
    for all $s_1 < s_2$ and $t_1 < t_2$. 
    \end{enumerate}
    We denote the space of double group functionals valued in $\cmG^{\gr{\leq n}}$ by $\DGF(\cmG^{\gr{\leq n}})$. 
    We use the notation $\rx^{-h}$ and $\rx^{-v}$ for group inverses in $H_0$ to simplify notation; in particular,
    \begin{align}\nonumber
        \rx^{-h}_{s_1, s_2;t} \coloneqq (\rx^h_{s_1, s_2; t})^{-1} \andd \rx^{-v}_{t;s_1, s_2} \coloneqq (\rx^v_{s; t_1, t_2})^{-1}.
    \end{align}
    In particular, this structure is equivalent to defining $\rrX: \Delta^2 \times \Delta^2 \to \dg(\cmG^{\gr{\leq n}})$ as
    \begin{align} \label{eq:rrX_in_double_group}
        \rrX_{s_1, s_2;t_1, t_2} \coloneqq \left(\rx^h_{s_1, s_2;t_1},\, \rx^v_{s_2; t_1, t_2}, \, \rx^{h}_{s_1, s_2; t_2}, \, \rx^{-v}_{s_1;t_1, t_2}, \, \rX_{s_1,s_2; t_1, t_2}\right).
    \end{align}
    Finally, $\rrX$ is \emph{horizontally} and \emph{vertically multiplicative} if 
    \begin{align} \label{eq:dgf_def_multiplicative}
    \rrX_{s_1, s_2; t_1, t_2} \mult_h \rrX_{s_2, s_3; t_1, t_2} = \rrX_{s_1, s_3; t_1, t_2} \andd \rrX_{s_1, s_2; t_1, t_2} \mult_v \rrX_{s_1, s_2; t_2, t_3} = \rrX_{s_1, s_2; t_1, t_3}. 
    \end{align}
\end{definition}

\begin{figure}[!h]
    \includegraphics[width=\linewidth]{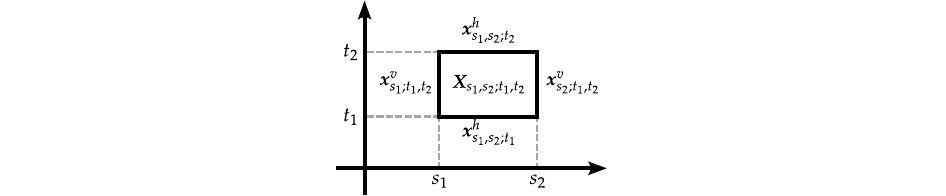}
\end{figure}  

\begin{remark}
    Let $\rrX$ be a double group functional. While horizontal and vertical compositions are only well-defined for squares in the double group, we will use an abuse of notation and write
    \begin{align}\nonumber
        \rX_{s_1, s_2; t_1, t_2} \hmult \rX_{s_2, s_3; t_1, t_2} &\coloneqq (\rx^h_{s_1, s_2; t_1} \gt \rX_{s_2, s_3; t_1, t_2}) * \rX_{s_1, s_2; t_1, t_2} \\
        \rX_{s_1, s_2; t_1, t_2} \vmult \rX_{s_1, s_2; t_2, t_3} &\coloneqq \rX_{s_1, s_2; t_1, t_2} * (\rx^v_{s_1; t_1, t_2} \gt \rX_{s_1, s_2; t_2, t_3})\nonumber
    \end{align}
    to denote the surface component of horizontal and vertical compositions, where the underlying squares are implicit.
\end{remark}

We emphasize that double group functionals (not necessarily multiplicative) satisfy the interchange law by definition since they satisfy the boundary condition in~\Cref{eq:rrX_boundary_condition}, and are thus valued in the double group $\dg(\cmG^{\gr{\leq n}})$ as in~\Cref{eq:rrX_in_double_group}.

\begin{lemma} \label{lem:dgf_interchange}
    A double group functional $\rrX \in \DGF(\cmG^{\gr{\leq n}})$ satisfies the interchange law
    \begin{align*}
        (\rrX_{s_1, s_2; t_1, t_2} \hmult \rrX_{s_2, s_3; t_1, t_2}) \vmult (\rrX_{s_1, s_2; t_2, t_3} \hmult \rrX_{s_2, s_3; t_2, t_3}) = (\rrX_{s_1, s_2; t_1, t_2} \vmult \rrX_{s_1, s_2; t_2, t_3}) \hmult (\rrX_{s_2, s_3; t_1, t_2} \vmult \rrX_{s_2, s_3; t_2, t_3}).
    \end{align*}
\end{lemma}

Next, we will define a notion of $\rho$-H\"older regularity for a double group functional $\rrX \in \DGF(\cmG^{\gr{\leq n}})$, where we take $n > \frac{2}{\rho}$, which is required for the surface extension theorem in~\Cref{thm:surface_extension}. The motivating idea is that a $\rho$-H\"older double group functional $\rrX$ should represent the surface signature of a $\rho$-H\"older surface $X \in C^\rho([0,1]^2, \V)$ restricted to subrectangles $Q \subset [0,1]^2$. We will consider conditions on the path components, $\rx^h$ and $\rx^v$, and the surface component $\rX$ separately. Furthermore, we will motivate these regularity conditions by taking $\rho \in (\frac{2}{3}, 1]$. This is the regime where Z\"ust integration of differential forms is still well-defined~\cite{zust_integration_2011} (see~\Cref{thm:zust}), which corresponds to the Young regime of rough paths.\medskip 

We consider the horizontal path component $\rx^h$, but our discussion also applies to the vertical path component $\rx^v$. For a fixed $t \in [0,1]$, the path component $\rx^h_{\cdot, \cdot; t} : \Delta^2 \to G_0$ is a multiplicative functional, which would represent the path signature of the horizontal path $X_{\cdot, t} \in C^\rho([0,1], \V)$ through the surface. Therefore, we impose the \emph{path regularity conditions}
\begin{align} \label{eq:motivation_path_regularity}
    \|\rx^{h, \gr{k}}_{s_1, s_2; t}\| \leq \frac{\ctr{k\rho}{s_1, s_2}}{\beta \ffact{k}{\rho}} \andd \|\rx^{v, \gr{k}}_{s; t_1, t_2}\| \leq \frac{\ctr{k\rho}{t_1, t_2}}{\beta \ffact{k}{\rho}},
\end{align}
which coincides with the regularity condition for $\rho$-H\"older multiplicative functionals. However, this does not impose any regularity conditions between horizontal paths at different $t$ values. We obtain such a condition by considering the $\rho$-H\"older regularity of the underlying surface $X$. We define the \emph{2D rectangular increment} of $X$ by
\begin{align} \label{eq:2d_increment}
    \square_{s_1, s_2; t_1, t_2}[X] \coloneqq X_{s_1, t_1} - X_{s_1, t_2} - X_{s_2, t_1} + X_{s_2, t_2}.
\end{align}

\begin{lemma} \label{lem:holder_2d_increment}
    Let $X \in C^\rho([0,1]^2, \V)$, and let $(s_1, s_2), (t_1, t_2) \in \Delta^2$. Then,
    \begin{align}
        \|\square_{s_1, s_2; t_1, t_2}[X]\| \leq 2\|X\|_{\rho} \inc{s_1}{s_2}^{\rho/2} \inc{t_1}{t_2}^{\rho/2}.
    \end{align}
\end{lemma}
\begin{proof}
    Without loss of generality, suppose $|s_1 - s_2| \leq |t_1 - t_2|$. Then,
    \begin{align}\nonumber
        \|\square_{s_1, s_2; t_1, t_2}[X]\| & \leq \|X_{s_1, t_1} - X_{s_2, t_1}\| + \|X_{s_2, t_2} - X_{s_1, t_2}\| \leq 2\|X\|_{\rho} \inc{s_1}{s_2}^\rho \leq  2\|X\|_{\rho} \inc{s_1}{s_2}^{\rho/2} \inc{t_1}{t_2}^{\rho/2}.
    \end{align}
\end{proof}

Throughout the remainder of this article, we will extensively use $\orho \coloneqq \rho/2$ to simplify notation. This suggests the regularity condition
\begin{align} \label{eq:level1_path_cont_condition}
    \|\rx^{h, \gr{1}}_{s_1, s_2; t_1} - \rx^{h, \gr{1}}_{s_1, s_2; t_2}\| \leq \frac{\ctr{\orho}{s_1, s_2}\ctr{\orho}{t_1, t_2}}{\beta \ffact{1}{\orho} \ffact{1}{\orho}}.
\end{align}
The next step would be to use the continuity of extensions from~\Cref{thm:path_ext_orig_cont} to motivate a regularity condition for higher levels. However, the $\orho$-regularity of the $s$ dimension in~\Cref{eq:level1_path_cont_condition} does not correspond $\rho$-regularity of the horizontal paths in~\Cref{eq:motivation_path_regularity}, and thus we cannot directly apply~\Cref{thm:path_ext_orig_cont}. Despite this, we prove a modified continuity theorem which is more suitable for paths through $\rho$-H\"older surfaces. The proof is similar to the original continuity theorem~\cite{lyons_differential_1998}, and is provided in~\Cref{apxsec:additional_surface_extension}.

\begin{theorem} \label{thm:path_extension_cont}
    Let $\rho \in (0,1]$ and let $n + \frac{1}{2} > \frac{1}{\rho}$ (or equivalently $2n +1 > \frac{1}{\orho}$). Suppose $\bx, \by: \Delta \to G_0^{\gr{\leq n}}$ are multiplicative functionals which satisfies for each $(\smp) \in \Delta$ and $k \in [n]$
    \begin{align}\nonumber
        \|\bx^{\gr{k}}_{\smp}\|, \|\by^{\gr{k}}_{\smp}\| \leq \frac{\ctr{k\rho}{\smp}}{\beta\ffact{k}{\rho}}.
    \end{align}
    Furthermore, we assume that for all $k \in [n]$, we have
    \begin{align}\nonumber
        \|\bx^{\gr{k}}_{\smp} - \by^{\gr{k}}_{\smp}\| \leq \epsilon \frac{\ctr{(2k-1)\orho}{\smp}}{\beta\ffact{2k-1}{\orho}}.
    \end{align}
    Then, the multiplicative extension $\tbx, \tby : \Delta \to G_0^{\gr{\leq n+1}}$ from~\Cref{thm:path_extension} satisfy
    \begin{align}\nonumber
        \|\tbx^{\gr{n+1}}_{\smp} - \tby^{\gr{n+1}}_{\smp}\|\leq \epsilon \frac{\ctr{(2n+1)\orho}{\smp}}{\beta\ffact{2n+1}{\orho}}.
    \end{align}
\end{theorem}

Because we have assumed that $\rho > \frac{2}{3}$, we have $\frac{1}{\rho} < \frac{3}{2}$, and thus we can apply~\Cref{thm:path_extension_cont} to the extensions of $\rx^{h, \gr{1}}_{s_1, s_2; t_1}$ and $\rx^{h, \gr{1}}_{s_1, s_2; t_2}$, where $\epsilon = \ctr{\orho}{t_1, t_2}$. In particular, we will impose the \emph{path continuity conditions}
\begin{align} \label{eq:motivation_path_continuity}
    \|\rx^{h, \gr{k}}_{s_1, s_2; t_1} - \rx^{h, \gr{k}}_{s_1, s_2; t_2}\| \leq \frac{\ctr{(2k-1)\orho}{s_1, s_2}\ctr{\orho}{t_1, t_2}}{\beta \ffact{2k-1}{\orho} \ffact{1}{\orho}} \andd \|\rx^{v, \gr{k}}_{s_1; t_1, t_2} - \rx^{v, \gr{k}}_{s_2; t_1, t_2}\| \leq \frac{\ctr{\orho}{s_1, s_2}\ctr{(2k-1)\orho}{t_1, t_2}}{\beta  \ffact{1}{\orho} \ffact{2k-1}{\orho}}
\end{align}

\begin{definition} \label{def:holder_path_functionals}
    Let $\rho \in (0,1]$, $\orho = \frac{\rho}{2}$ and $n \in \N$. We define \emph{$\rho$-H\"older horizontal and vertical path functionals} to be path functionals $\rx^h \in \HPF(G_0^{\gr{\leq n}})$ and $\rx^v \in \HPF(G_0^{\gr{\leq n}})$ which satisfy for all $k \in [n]$ the \emph{path regularity conditions},
    \begin{align} \label{eq:path_regularity}
        \left\|\rx^{h, \gr{k}}_{s_1, s_2; t}\right\| \leq \frac{\ctr{k\rho}{s_1, s_2}}{\beta \ffact{k}{\rho}} \andd \left\|\rx^{v, \gr{k}}_{s; t_1, t_2}\right\| \leq \frac{\ctr{k\rho}{t_1, t_2}}{\beta \ffact{k}{\rho}},
    \end{align}
    and the \emph{path continuity conditions},
    \begin{align} \label{eq:path_continuity}
        \left\|\rx^{h, \gr{k}}_{s_1, s_2; t_1} - \rx^{h, \gr{k}}_{s_1, s_2; t_2}\right\| \leq \frac{\ctr{(2k-1)\orho}{s_1, s_2}\ctr{\orho}{t_1, t_2}}{\beta \ffact{2k-1}{\orho} \ffact{1}{\orho}} \andd \left\|\rx^{v, \gr{k}}_{s_1; t_1, t_2} - \rx^{v, \gr{k}}_{s_2; t_1, t_2}\right\| \leq \frac{\ctr{\orho}{s_1, s_2}\ctr{(2k-1)\orho}{t_1, t_2}}{\beta  \ffact{1}{\orho} \ffact{2k-1}{\orho}}.
    \end{align}
    The sets of $\rho$-H\"older horizontal and vertical path functionals are denoted $\HPF^\rho(G_0^{\gr{\leq n}})$ and $\VPF^\rho(G_0^{\gr{\leq n}})$ respectively.
\end{definition}

Next, we will move on to the surface component $\rX$. The level $2$ surface signature of $X = (X^1, \ldots, X^d)$ restricted to a rectangle $Q \subset [0,1]^2$ can be defined as an integral of a differential 2-form, in order to avoid Jacobians and derivatives. In particular, the $e_i \wedge e_j$ component of $\hssig^{\gr{2}}(X|_Q)$
\begin{align}\nonumber
    (\hssig^{\gr{2}}(X))^{i,j} = \int_Q dX^i_{s,t} \wedge dX^j_{s,t},
\end{align}
which is the signed area of the surface projected onto the $(e_i, e_j)$ plane.
Thus, we can apply the Z\"ust integral in~\Cref{thm:zust} to obtain the preliminary bound
\begin{align}\nonumber
    \|\hssig^{\gr{2}}(X|_Q)\| \leq C \diam(Q)^{2\rho},
\end{align}
where $C> 0$ is a constant which depends on $X$ and $\rho$. However, this bound does not take into account the $s$ and $t$ parameters independently. In particular, while $\|\hssig^{\gr{2}}(X|_Q)\| \to 0$ as $|\tpp - \tmm| \to 0$, since $\hssig^{\gr{2}}(X|_Q)$ represents the signed area, $\diam(Q) \not\to 0$ as $|\tpp - \tmm| \to 0$. In order to obtain a more precise bound, we use the fact that $\hssig^{\gr{2}}(X|_Q) = S^{\gr{2}}(\partial X|_Q)$ is exactly the level 2 signature of the boundary, since $\partial X|_Q$ is a loop. We will use a combination of the path regularity and path continuity conditions of $\rx^h$ and $\rx^v$ to obtain a bound for the signature of the boundary path $\partial X|_Q$ in general. 

\begin{lemma} \label{lem:boundary_signature_bound}
    Suppose $\rx^h \in \HPF^\rho(G_0^{\gr{\leq n}})$ and $\rx^v \in \VPF^\rho(G_0^{\gr{\leq n}})$. Then, for any rectangle $[s_1, s_2]\times[t_1, t_2] \subset [0,1]^2$ and any $k \in [n]$, we have
    \begin{align} \label{eq:boundary_signature_bound}
        \left\|(\rx^h_{s_1, s_2; t_1} \cdot \rx^v_{s_2; t_1, t_2} \cdot \rx^{-h}_{s_1, s_2; t_2} \cdot \rx^{-v}_{s_1; t_1, t_2})^{\gr{k}}\right\|  \leq \frac{2^{(2k-1)\orho}}{\beta} \sum_{q=1}^{2k-1} \frac{\ctr{q\orho}{s_1, s_2}\, \ctr{(2k-q)\orho}{t_1, t_2}}{\ffact{q}{\orho} \ffact{2k-q}{\orho}} = \frac{1}{2\beta} \CTR{k}{\orho}{s_1, s_2; t_1, t_2}.
    \end{align}
\end{lemma}
\begin{proof}
    By multiplying, we obtain
    \begin{align} \label{eq:boundary_reg_main_decomp}
        \Big( \rx^h_{s_1, s_2;t_1} \cdot \,&\rx^v_{s_2; t_1, t_2} \cdot \rx^{-h}_{s_1, s_2; t_2} \cdot \rx^{-v}_{s_1;t_1, t_2} \Big)^{\gr{k}}\\
        &= \left(\rx^h_{s_1, s_2; t_1} \cdot \rx^{-h}_{s_1, s_2; t_1}\right)^{\gr{k}} + \left( \rx^v_{s_2; t_1, t_2} \cdot \rx^{-v}_{s_1; t_1, t_2}\right)^{\gr{k}} + \sum_{\substack{i+j+m+\ell = k\\ i + m \geq 1, \, j + \ell \geq 1}} \rx^{h, \gr{i}}_{s_1, s_2;t_1} \cdot \rx^{v, \gr{j}}_{s_2; t_1, t_2} \cdot \rx^{-h, \gr{m}}_{s_1, s_2; t_2} \cdot \rx^{-v, \gr{\ell}}_{s_1;t_1, t_2}. \nonumber
    \end{align}
    We begin by bounding the first term. By~\cite[Lemma 7.48]{friz_multidimensional_2010}, we have
    \begin{align} \label{eq:boundary_reg_horizontal_decomp}
        \left\| \left(\rx^h_{s_1, s_2; t_1} \cdot \rx^{-h}_{s_1, s_2; t_1}\right)^{\gr{k}}\right\| & \leq \sum_{q=1}^{k-1} \left\|\rx^{h, \gr{q}}_{s_1, s_2; t_1} - \rx^{h, \gr{q}}_{s_1, s_2; t_1} \right\| \cdot \left\|\rx^{-h, \gr{k-q}}_{s_1, s_2; t_1}\right\| + \left\|\rx^{h, \gr{k}}_{s_1, s_2; t_1} - \rx^{h, \gr{k}}_{s_1, s_2; t_1} \right\|
    \end{align}
    By applying the path regularity and continuity conditions, we obtain 
    \begin{align}\nonumber
        \left\| \left(\rx^h_{s_1, s_2; t_1} \cdot \rx^{-h}_{s_1, s_2; t_1}\right)^{\gr{k}}\right\| &\leq \frac{1}{\beta^2} \sum_{q=1}^{k} \frac{\ctr{(2q-1)\orho}{s_1, s_2} \ctr{\orho}{t_1, t_2}}{\ffact{2q-1}{\orho} \ffact{1}{\orho}} \frac{\ctr{(2k-2q)\orho}{s_1, s_2}}{\ffact{2k-2q}{\orho}} \\
        & \leq \frac{2^{(2k-1)\orho}}{\orho\beta^2}\frac{\ctr{(2k-1)\orho}{s_1, s_2} \ctr{\orho}{t_1, t_2}}{\ffact{2k-1}{\orho} \ffact{1}{\orho}}, \label{eq:boundary_reg_bound1}
    \end{align}
    where we use~\Cref{eq:binomial_sum_bound}.
    We can repeat the same argument for the second term in~\Cref{eq:boundary_reg_main_decomp} to get
    \begin{align}\label{eq:boundary_reg_bound2}
        \left\| \left( \rx^v_{s_2; t_1, t_2} \cdot \rx^{-v}_{s_1; t_1, t_2}\right)^{\gr{k}}\right\| \leq \frac{2^{(2k-1)\orho}}{\orho\beta^2}\frac{\ctr{\orho}{s_1, s_2} \ctr{(2k-1)\orho}{t_1, t_2}}{ \ffact{1}{\orho}\ffact{2k-1}{\orho} }.
    \end{align}
    Now, we reindex the final sum in~\Cref{eq:boundary_reg_main_decomp} and use the path regularity conditions to get
    \begin{align}\nonumber
        \Bigg\|\sum_{q=1}^{k-1} \sum_{i=0}^q \sum_{j=0}^{k-q} \rx^{h, \gr{i}}_{s_1, s_2;t_1} \cdot \rx^{v, \gr{j}}_{s_2; t_1, t_2} \cdot \rx^{-h, \gr{q-i}}_{s_1, s_2; t_2} \cdot \rx^{-v, \gr{k-q-j}}_{s_1;t_1, t_2}\Bigg\|  &\leq \frac{1}{\beta^4}\sum_{q=1}^{k-1} \sum_{i=0}^q \sum_{j=0}^{k-q} \frac{\ctr{2q\orho}{s_1, s_2} \, \ctr{(2k-2q)\orho}{t_1, t_2}}{\ffact{2i}{\orho} \, \ffact{2q-2i}{\orho}\, \ffact{2j}{\orho} \,\ffact{2k-2q-2j}{\orho}} \\
        & \leq \frac{2^{2k\orho}}{\orho^2\beta^4} \sum_{q=1}^{k-1} \frac{\ctr{2q\orho}{s_1, s_2} \, \ctr{(2k-2q)\orho}{t_1, t_2}}{\ffact{2q}{\orho} \ffact{2k-2q}{\orho}} \nonumber\\
        & \leq \frac{2^{(2k-1)\orho}}{\beta} \sum_{q=1}^{k-1} \frac{\ctr{2q\orho}{s_1, s_2} \, \ctr{(2k-2q)\orho}{t_1, t_2}}{\ffact{2q}{\orho} \ffact{2k-2q}{\orho}}, \label{eq:boundary_reg_bound3}
    \end{align}
    where we use~\Cref{eq:binomial_sum_bound} twice (summing over $i$ and $j$), as well as the fact that $\beta$ satisfies~\Cref{eq:rs_beta} ($\beta > \beta_{\RP}$ suffices here). Then, combining~\Cref{eq:boundary_reg_main_decomp},~\Cref{eq:boundary_reg_bound1},~\Cref{eq:boundary_reg_bound2},~\Cref{eq:boundary_reg_bound3}, we obtain
    \begin{align}\nonumber
         \left\| \left( \rx^h_{s_1, s_2;t_1} \cdot \rx^v_{s_2; t_1, t_2} \cdot \rx^{-h}_{s_1, s_2; t_2} \cdot \rx^{-v}_{s_1;t_1, t_2} \right)^{\gr{k}}\right\| \leq \frac{2^{(2k-1)\orho}}{\beta} \sum_{q=1}^{2k-1} \frac{\ctr{q\orho}{s_1, s_2}\, \ctr{(2k-q)\orho}{t_1, t_2}}{\ffact{q}{\orho} \ffact{2m-q}{\orho}}.
    \end{align}
\end{proof}

This implies that for a surface $X \in C^\rho([0,1]^2, \V)$, the level $2$ surface component $\rX^{\gr{2}}$ satisfies the bound in~\Cref{eq:boundary_signature_bound} (aside from a constant factor due to the difference in norms between $\Lambda^2 \V$ and the antisymmetric elements in $\V^{\otimes 2}$). For higher levels, a double group functional must satisfy the boundary condition in~\Cref{eq:rrX_boundary_condition} and because $\delta$ is a bounded graded linear map, we expect the surface component $\rX$ to satisfy a similar regularity condition. We denote this polynomial-type regularity by
\begin{align}\nonumber
    \CTR{k}{\orho}{s_1, s_2; t_1, t_2} \coloneqq 2^{2k\orho}\sum_{q=1}^{2k-1} \frac{\ctr{q\orho}{s_1, s_2}\, \ctr{(2k-q)\orho}{t_1, t_2}}{\ffact{q}{\orho} \ffact{2k-q}{\orho}}.
\end{align}
Note that the sum is taken from $q=1$ to $q=2k-1$, so $\CTR{k}{\orho}{s_1, s_2; t_1, t_2} \to 0$ as either $|s_1 - s_2| \to 0$ or $|t_1 - t_2| \to 0$. 
Using this regularity condition for the surface component, we define the notion of a $\rho$-H\"older double group functional.

\begin{definition} \label{def:holder_dgf}
    Let $\rho \in (0,1]$, $\orho = \frac{\rho}{2}$ and $n \in \N$. We define a \emph{$\rho$-H\"older surface functional} to be a surface functional $\rX \in \SF(G_1^{\gr{\leq n}})$ which satisfy for all $k \in [n]$ the \emph{surface regularity condition},
    \begin{align} \label{eq:surface_regularity}
        \left\|\rX^{\gr{k}}_{s_1, s_2; t_1, t_2}\right\| \leq \frac{2^{2k\orho}}{\beta} \sum_{q=1}^{2k-1} \frac{\ctr{q\orho}{s_1, s_2}\, \ctr{(2k-q)\orho}{t_1, t_2}}{\ffact{q}{\orho} \ffact{2k-q}{\orho}} = \frac{1}{\beta} \CTRL^k_\orho(s_1, s_2; t_1, t_2).
    \end{align}
    The set of $\rho$-H\"older surface functionals is denoted $\SF(G_1^{\gr{\leq n}})$.
    A \emph{$\rho$-H\"older double group functional} is a double group functional $\rrX = (\rx^h, \rx^v, \rX) \in \DGF(\cmG^{\gr{\leq n}})$ such that 
    \begin{align}\nonumber
    \rx \in \HPF^\rho(G_0^{\gr{\leq n}}), \quad \rx^v \in\VPF^\rho(G_0^{\gr{\leq n}})\andd\rX \in \SF(G_1^{\gr{\leq n}})
    \end{align}
    We denote the set of $\rho$-H\"older double group functionals valued in $\cmG^{\gr{\leq n}}$ by $\DGF^\rho(\cmG^{\gr{\leq n}})$.
\end{definition}

Through our discussion, we have shown that for a surface $X \in C^\rho([0,1]^2, \V)$ with $\rho \in (\frac{2}{3},1]$, we can directly obtain a level $2$ $\rho$-H\"older double group functional by computing path signatures. 

\begin{proposition} \label{prop:dgf_zust_case}
    Let $X \in C^\rho([0,1]^2, \V)$ with $\rho \in (\frac{2}{3},1]$. Define $\rrX = (\rx^h, \rx^v, \rX) \in \DGF(\cmG^{\gr{\leq 2}})$ with path components
    \begin{align}\nonumber
        \rx^{h, \gr{1}}_{s_1, s_2; t} \coloneqq X_{s_2,t} - X_{s_1, t}, &\quad \rx^{h, \gr{2}}_{s_1, s_2; t} \coloneqq S^{\gr{2}}(X|_{[s_1, s_2]\times\{t\}}), \\
        \rx^{v, \gr{1}}_{s;t_1, t_2} \coloneqq X_{s,t_2} - X_{s, t_1}, &\quad \rx^{v, \gr{2}}_{s; t_1, t_2} \coloneqq S^{\gr{2}}(X|_{\{s\} \times [t_1, t_2]}),\nonumber
    \end{align}
    and surface component
    \begin{align}\nonumber
        \rX^{\gr{2}}_{s_1, s_2; t_1, t_2} \coloneqq (\rx^h_{s_1, s_2; t_1} \cdot \rx^v_{s_2; t_1, t_2} \cdot \rx^{-h}_{s_1, s_2; t_2} \cdot \rx^{-v}_{s_1; t_1, t_2})^{\gr{2}}.
    \end{align}
    Then $\rrX$ is a $\rho$-H\"older multiplicative double group functional. 
\end{proposition}
\begin{proof}
    The fact that $\rrX$ is $\rho$-H\"older is given by the above discussion, and it remains to show that $\rrX$ is multiplicative. Indeed, since $\rX$ is defined as the $2$-truncated path signature of the boundary loop, this follows by multiplicativity of the path signature,
    \begin{align*}
        \left(\rX_{s_1, s_2; t_1, t_2} \hmult \rX_{s_2, s_3; t_1, t_2}\right)^{\gr{2}} &= \left((\rx^h_{s_1, s_2; t_1} \gt \rX_{s_2, s_3; t_1, t_2}) * \rX_{s_1, s_2; t_1, t_2}\right)^{\gr{2}} \\
        & = \left( \rx^h_{s_1, s_2; t_1} \cdot \rx^h_{s_2, s_3; t_1} \cdot \rx^v_{s_3; t_1, t_2} \cdot \rx^{-h}_{s_2, s_3; t_2}  \cdot \rx^{-h}_{s_1, s_2; t_2} \cdot \rx^{-v}_{s_1; t_1, t_2} \right)^{\gr{2}} \\
        & = \rX_{s_1, s_3; t_1, t_2}^{\gr{2}},
    \end{align*}
    since the $*$ product of $G_1^{\gr{2}}$ corresponds to the truncated tensor product for level $\leq 2$. The same argument can be made for vertical compositions, and thus $\rrX$ is multiplicative. 
\end{proof}

In the remainder of this section, we will prove an extension theorem which will show that this multiplicative functional can be extended to a multiplicative functional $\trrX \in \DGF(\cmG)$; in other words, we can compute the surface signature for $\rho$-H\"older surfaces when $ \rho \in (\frac{2}{3}, 1]$.

\subsection{Surface Extension Theorem and the Young /Z\"ust Regime} \label{ssec:surface_extension_young_zust}
Our main goal is to show the following generalization to the extension theorem for rough paths~\cite{lyons_differential_1998}.

\begin{theorem} \label{thm:surface_extension}
    Let $\rho \in (0,1]$ and $n \geq \left\lfloor \frac{2}{\rho} \right\rfloor$. Suppose $\rrX \in \DGF^\rho(\cmG^{\gr{\leq n}})$ is a multiplicative double group functional. Then, there exists a multiplicative double group functional $\trrX \in \DGF^\rho(\cmG^{\gr{\leq n+1}})$ such that $\trrX^{\gr{k}} = \rrX^{\gr{k}}$ for all $k \in [n]$. Furthermore, $\trrX$ is unique in the sense that if $\rrY = (\ry^h, \ry^v, \rY) \in \DGF(\cmG^{\gr{\leq n+1}})$ is another double group functional such that $\rrY^{\gr{k}} = \rrX^{\gr{k}}$ for all $k \in [n]$, and 
    \begin{align}\nonumber
        \|\rY^{\gr{n+1}}_{s_1, s_2; t_1, t_2}\| \leq C (\inc{s_1}{s_2} + \inc{t_1}{t_2})^{\theta}
    \end{align}
    for a constant $C>0$ and any $\theta > 1$, then $\trrX = \rrY$. 
\end{theorem}

Analogous to the case of rough paths, we require the constant $\beta > 0$ in the double group functional regularity conditions to be fixed (depending only on the regularity $\rho$). In particular, we define
\begin{align}\label{eq:rs_beta}
    \beta_{\RS} \coloneqq \frac{3}{2} \left( \sum_{m=0}^\infty 4^{m(1 - (\lfloor 2\rho^{-1}\rfloor+1)\rho)}\right),
\end{align}
and require
\begin{align} \label{eq:beta}
    \beta > \max\{\beta_{\RP}, \beta_{\RS}\}.
\end{align}
where $\beta_{\RP}$ is defined in~\Cref{eq:rp_beta}. \medskip

\subsubsection{Extension in the Smooth and Z\"ust Regime} 
Before we prove the extension theorem, we discuss some immediate consequences. 

\begin{corollary} \label{cor:zust_extension}
    Let $X \in C^\rho([0,1]^2, \V)$ with $\rho \in (\frac{2}{3},1]$. Define $\rrX = (\rx^h, \rx^v, \rX) \in \DGF^\rho(\cmG^{\gr{\leq 2}})$ as in~\Cref{prop:dgf_zust_case}. Then, there exists a unique extension to a multiplicative double group functional $\trrX \in \DGF^\rho(\cmG)$.
\end{corollary} 

In the smooth setting, we show that this extension is given by the surface signature. 
Note that the condition imposed on $\rrY$ is weaker than the $\rho$-H\"older surface regularity condition in~\Cref{eq:surface_regularity}. Indeed, suppose $\trrX \in \DGF(\cmG^{\gr{\leq n+1}})$ is the proposed extension from the theorem. Then, the $\rho$-H\"older surface regularity condition can be bound by using the neoclassical inequality in~\Cref{eq:neoclassical},
\begin{align} \label{eq:surface_regularity_neoclassical}
    \|\trrX^{\gr{n+1}}_{s_1, s_2; t_1, t_2}\| \leq \frac{2^{2(n+1)\orho}}{\beta} \sum_{q=1}^{2n+1} \frac{\ctr{q\orho}{s_1, s_2}\, \ctr{(2n+1-q)\orho}{t_1, t_2}}{\ffact{q}{\orho} \ffact{2n+1-q}{\orho}} \leq \frac{2^{2(n+1)\orho}}{\orho \beta} \frac{(\ctr{}{s_1, s_2} + \ctr{}{t_1, t_1})^{2(n+1)\orho}}{\ffact{2(n+1)}{\orho}},
\end{align}
where $2(n+1)\orho = (n+1)\rho > 1$ due to the hypothesis that $n \geq \lfloor \frac{2}{\rho} \rfloor$. This allows us to show that the extension recovers the surface signature in the smooth setting.

\begin{corollary} \label{cor:smooth_extension_surface_signature}
    Let $X \in C^\infty([0,1]^2,\V)$ be a smooth surface. Define a double group functional $\rrY = (\ry^h, \ry^v, \rY) \in \DGF(\cmG)$ by 
    \begin{align}\nonumber
        \ry^h_{s_1, s_2; t} \coloneqq S(X|_{[s_1, s_2] \times\{t\}}), \quad \ry^v_{s; t_1, t_2} \coloneqq S(X|_{\{s\} \times [t_1, t_2]}) \andd \rY_{s_1, s_2; t_1, t_2} = \hssig(X|_{[s_1, s_2] \times [t_1, t_2]}).
    \end{align}
    Then, $\rrY = \trrX$, where $\trrX \in \DGF^\rho(\cmG)$ is the unique extension to $X$ from~\Cref{cor:zust_extension}. In particular, this implies that the surface signature satisfies the bound
    \begin{align}\nonumber
        \left\|\hssig^{\gr{k}}(X|_{[s_1, s_2]\times[t_1, t_2]})\right\| \leq \frac{(\ctr{}{s_1, s_2} + \ctr{}{t_1, t_2})^{k}}{\beta k!}
    \end{align}
    so $\hssig(X) \in E_1$ for all $X \in \surfaces^\infty(\V)$. 
\end{corollary}
\begin{proof}
    Let $\rrX \in \DGF^\rho(\cmG^{\gr{\leq 2}})$ be the double group functional from~\Cref{prop:dgf_zust_case}. Then, by definition, we have $\trrX^{\gr{\leq 2}} = \rrX = \rrY^{\gr{\leq 2}}$. Now, we continue by induction and suppose that $\trrX^{\gr{\leq k}} = \rrY^{\gr{\leq k}}$ for all $k \in [n]$. In particular, this implies that $\rrY^{\gr{\leq k}} \in \DGF^\rho(\cmG^{\gr{\leq n}})$. Furthermore, $\rrY^{\gr{\leq n+1}}$ is multiplicative by~\Cref{prop:sh_multiplicative}. Then, by~\Cref{prop:ssig_preliminary_n_bound}, there exists a constant $C> 0$ such that
    \begin{align}\nonumber
        \left\|\rY^{\gr{n+1}}_{s_1, s_2; t_1, t_2}\right\| = \left\|\hssig^{\gr{n+1}}(X|_{[s_1, s_2] \times [t_1, t_2]})\right\| \leq C (\inc{s_1}{s_2} + \inc{t_1}{t_2})^{n+1}.
    \end{align}
    Thus, by the uniqueness statement of~\Cref{thm:surface_extension}, $\trrX^{\gr{n+1}} = \rrY^{\gr{n+1}}$. Because $\rrY$ satisfies the surface regularity condition, we have
    \begin{align}\nonumber
        \left\|\hssig^{\gr{k}}(X|_{[s_1, s_2]\times[t_1, t_2]})\right\| & \leq \frac{2^{k}}{\beta} \sum_{q=1}^{2k-1} \frac{\ctr{q/2}{s_1, s_2}\, \ctr{(2k-q)/2}{t_1, t_2}}{(\frac{q}{2})! (\frac{2k-q}{2})!} \leq \frac{2^{k+1}(\ctr{}{s_1, s_2} + \ctr{}{t_1, t_2})^{k}}{\beta k!}.
    \end{align}
\end{proof}

These results imply that for all $\rho \in (\frac{2}{3}, 1]$, we can directly compute the surface signature of a surface $X \in C^\rho([0,1]^2, \V)$ without any additional data. In fact, this boundary of $\rho > \frac{2}{3}$ is sharp when considering $\rho$-H\"older surfaces. Suppose $\rho \in (\frac{1}{2}, \frac{2}{3}]$. In this regime, we can still define a double group functional $\rrX \in \DGF^\rho(\cmG^{\gr{\leq 2}})$ as in~\Cref{prop:dgf_zust_case}, as this is defined purely via path signatures, which can be computed for $\rho > \frac{1}{2}$. However, following~\Cref{ex:level3} the level 3 surface signature of $X = (X^1, \ldots, X^d)$ is defined coordinate-wise for the $e_i \cdot (e_j \wedge e_k)$ component as
\begin{align}\nonumber
    \langle e_i \cdot (e_j \wedge e_k), \hssig^{\gr{3}}(X) \rangle = \int_{[0,1]^2} (X^i_{s,t} - X^i_{0,0}) \, dX^j_{s,t} \wedge dX^k_{s,t}.
\end{align}
We note that $X^i_{s,t} - X^i_{0,0}$ is $\rho$-H\"older, and therefore this integral does not satisfy the conditions of the Z\"ust integral from~\Cref{thm:zust}, since $3\rho < 2$. Furthermore, Z\"ust shows that there does not exist a continuous extension of the integration map to this regime~\cite[Section 3.2]{zust_integration_2011}. \medskip

Thus far, we have not considered surfaces which are equipped with additional regularity conditions on 2D rectangular increments. In particular, we say that $X \in C^\rho([0,1]^2, \V)$ is \emph{2D rectangular $\rho$-H\"older} if
\begin{align} \label{eq:2d_rectangular_holder_bound}
    \|X\|_{\square, \rho} \coloneqq \sup_{(s_1, s_2; t_1, t_2) \in \Delta^2 \times \Delta^2} \frac{\|\square_{s_1, s_2; t_1, t_2}[X]\|}{|s_1 - s_2|^\rho \, |t_1 - t_2|^{\rho}} < \infty,
\end{align}
where $\square[X]$ is the 2D increment defined in~\Cref{eq:2d_increment}.
We denote the space of 2D rectangular $\rho$-H\"older surfaces on a rectangle $T = [\smp] \times [\tmp] \subset \R^2$ by
\begin{align} \label{eq:2d_rectangular_holder_space}
    C^{\square, \rho}(T, \V) \coloneqq \left\{ X \in C(T, \V) \, : \, \|X\|_\rho, \|X\|_{\square, \rho} < \infty\right\}.
\end{align}
This is a strictly stronger condition than the rectangular $\rho/2$-H\"older condition in~\Cref{lem:holder_2d_increment} which is implied by classical H\"older regularity. 
Considering this additional 2D regularity is common when studying surfaces, for instance in the generalization of the Young integral~\cite{towghi_multidimensional_2002} to 2D and the multi-dimensional sewing lemma in~\cite{harang_extension_2021}, where it is shown that integrals of the type
\begin{align} \label{eq:increment_integral}
    \int_{[0,1]^2} f_{s,t} dg_{s,t} = \lim_{|\ppart| \to 0}\sum f_{s_i, t_j} (\square_{s_i, s_{i+1}; t_j, t_{j+1}}[g]),
\end{align}
which we call \emph{2D increment integrals}, where $\ppart$ is a partition of $[0,1]^2$, are well-defined for $f,g \in C^{\square,\rho}([0,1]^2, \R)$. 

\begin{theorem}[\cite{harang_extension_2021}] \label{thm:towghi_young}
    Let $T = [\smp] \times[\tmp] \subset \R^2$ be a rectangle. Suppose $f \in C^{\square, \rho_f}(T, L(V, W))$ and $g \in C^{\square, \rho_g}(T, V)$ where $\overline{\rho} = \rho_f + \rho_g > 1$. Then, the integral $\int_T f \, dg \in W$ is well defined and satisfies
    \begin{align}\nonumber
        \left|\int_T f_{s,t} \, dg_{s,t}\right| \leq C \left(\inc{\smm}{\spp}^{2\rho} \inc{\tmm}{\tpp}^\rho + \inc{\smm}{\spp}^{\rho} \inc{\tmm}{\tpp}^{2\rho}\right),
    \end{align}
    where $C> 0$ is a constant that depends on $f$ and $g$.
\end{theorem}

Then, by making a slight modification to the path continuity condition, we can define the extension of a surface in the Young regime. We use the method developed in~\cite[Section 7.1]{lee_random_2023} to reformulate surface holonomy in terms of 2D increment integrals in order to extend~\Cref{prop:dgf_zust_case} to obtain a level $3$ \emph{rectangular} $\rho$-H\"older double group functional for a surface $X\in C^{\square, \rho}([0,1]^2, \V)$ when $\rho \in (\frac{1}{2}, 1]$. We discuss the modifications and prove the following result in~\Cref{apxsec:surface_extension_rectangular}.

\begin{proposition} \label{prop:young_lifting}
    Let $X \in C^{\square, \rho}([0,1]^2, \V)$ with $\rho \in (\frac{1}{2}, 1]$. The \emph{area process of $X$} is the function $A(X): [0,1]^2 \to \Lambda^2 \V$ defined by
    \begin{align}\nonumber
        A_{s,t}(X) = S^{\gr{2}}\left(\partial (X|_{[0,s]\times[0,t]})\right),
    \end{align}
    where $\partial (X|_{[0,s]\times[0,t]})$ is the boundary path of $X$ restricted to $[0,s] \times[0,t]$. Note that this is valued in $\Lambda^2 \V$ since this path is a loop. Define $\rrX = (\rx^h, \rx^v, \rrX) \in \DGF(\cmG^{\gr{\leq 3}})$ where levels $1$ and $2$ are defined as in~\Cref{prop:dgf_zust_case}, level $3$ path components are defined by the path signature,
    \begin{align}\nonumber
        \rx^{h, \gr{3}}_{s_1, s_2; t} \coloneqq S^{\gr{3}}(X|_{[s_1, s_2] \times\{t\}}) \andd \rx^{v, \gr{3}}_{s; t_1, t_2} \coloneqq S^{\gr{3}}(X|_{\{s\} \times[t_1, t_2]}),
    \end{align}
    and the level $3$ surface component is given by
    \begin{align}\nonumber
       \rX^{\gr{3}}_{s_1, s_2; t_1, t_2} &= \int_{[0,1]^2} (X_{s,t} - X_{0,0}) \gt dA_{s,t}(X) 
    \end{align}
    Then, $\rrX$ is a rectangular $\rho$-H\"older multiplicative double group functional, $\rrX \in \DGF^{\square, \rho}(\cmG^{\gr{ \leq 3}})$. In particular, there exists a unique extension from $\rrX$ to a multiplicative double group funtional $\trrX \in \DGF^{\square,\rho}(\cmG)$. 
\end{proposition}

\begin{remark} \label{rem:sewing_with_rectangular_regularity}
    We emphasize that this result does not imply that the level $3$ functional $\rrX$ is uniquely determined by the underlying surface $X$. The result only states that the lift from $\rrX$ to $\trrX$ is unique. The construction of the level $3$ functional here is natural in the sense that this reduces to the surface signature in the case of smooth surfaces. The reason is that the extension theorem used to prove this proposition (which is essentially the same as~\Cref{thm:surface_extension}) does not use the additional rectangular regularity. It remains open whether there is an analogue of~\Cref{thm:surface_extension} which directly takes advantage of this regularity. A similar problem was noticed in~\cite{chevyrev_multiplicative_2024}.
\end{remark}

\subsection{Proof of Surface Extension Theorem} \label{ssec:surface_extension_proof}

In this section, we prove~\Cref{thm:surface_extension} and begin with an outline of the proof. Suppose $\rrX \in \DGF^\rho(\cmG^{\leq n})$ is the original multiplicative double group functional.

\begin{enumerate}
    \item \textbf{Pathwise Extension.} We begin by defining a \emph{pathwise extension} $\orrX  = (\orx^h, \orx^v, \orX) \in \DGF^\rho(\cmG^{\gr{\leq n+1}})$ by using the same idea as the lift from level $1$ to $2$ in the Z\"ust regime (\Cref{prop:dgf_zust_case}). In particular, we define $\orx^h$ and $\orx^v$ via the path extension theorem (\Cref{thm:path_extension}), which combined with~\Cref{thm:path_extension_cont} provides the path regularity and continuity condition. The surface component $\orX$ is defined by applying the algebra section $\ts: U([\fg_0, \fg_0]) \to T_1$ to the signature of boundary paths.
    \item \textbf{Almost Multiplicativity.} While $\orrX \in \DGF^\rho(\cmG^{\gr{\leq n+1}})$, it is not a \emph{multiplicative} double group functional. However, we can show that it is \emph{almost multiplicative} in the sense that
    \begin{align}\nonumber
        \|\orX_{s_1, s_2; t_1, t_2} \hmult \orX_{s_2, s_3; t_1, t_1} - \orX_{s_1, s_3; t_1, t_2}\| &\leq \frac{1}{\beta}\CTR{n+1}{\orho}{s_1, s_3; t_1, t_2} \\
        \|\orX_{s_1, s_2; t_1, t_2} \vmult \orX_{s_1, s_2; t_2, t_3} - \orX_{s_1, s_2; t_1, t_3}\| &\leq \frac{1}{\beta}\CTR{n+1}{\orho}{s_1, s_2; t_1, t_3}.\nonumber
    \end{align}
    \item \textbf{Dyadic Limit.} Denote the \emph{dyadic rationals in $[0,1]$} by 
    \begin{align} \label{eq:dyadic}
        \dya \coloneqq \left\{ \frac{i}{2^m} \in [0,1] \, : \, i \in 0, \ldots, 2^m, m \in \N \right\},
    \end{align}
    and the \emph{dyadic 2-simplex} by
    \begin{align} \label{eq:dyadic_simplex}
        \Delta^2_\dya \coloneqq \{ (s,t) \in \Delta^2 \, : \, s, t \in \dya\}. 
    \end{align}
    Given a dyadic rectangle $Q = [s_1, s_2] \times [t_1, t_2]$ where $s_i, t_i \in \dya$, we consider 2D dyadic subdivisions of $Q$. We horizontally and vertically multiply together the squares in $\orX$ corresponding to these dyadic subdivisions. By leveraging the almost multiplicativity property, we show that the limit exists as subdivisions become finer to define $\trrX$. This dyadic limit is similar to the construction in~\cite{yekutieli_nonabelian_2015}.
    \item \textbf{Multiplicativity and Continuous Extension.} We then show that the limit $\trrX$ is both multiplicative and uniformly continuous on dyadic rectangles. We can then define the continuous extension of $\trrX$ to all (nondyadic) subrectangles of $[0,1]^2$, and show that multiplicativity is preserved.
    \item \textbf{Uniqueness.} Finally, we show the uniqueness property by once again considering grid multiplication of $\trrX$ with respect to dyadic subdivisions, iteratively replacing $\trrX$ with $\rrY$, and then showing that the difference tends to $0$ as subdivisions become finer. 
\end{enumerate}

Throughout this section, we will fix a section $s: L \to \fg_1^{\sab}$ such that its induced algebra section $\ts: U([\fg_0, \fg_0]) \to T_1$ has the property $\|\ts(x)\| \leq \|x\|$. An explicit section with the property is discussed in~\Cref{apxsec:algebra_section}.  \medskip

\subsubsection{Pathwise Extension}
We begin by performing an initial \emph{pathwise extension} of $\rrX$ by using the extension theorem for rough paths. We define a double group functional $\orrX \in \dg(\cmG^{\gr{\leq n+1}})$. For a fixed $t \in [0,1]$, we define $\orx^{h}_{\cdot, \cdot; t} : \Delta^2 \to G_0^{\gr{\leq n+1}}$ to be the unique extension of $\rx^h_{\cdot, \cdot; t} : \Delta^2 \to G_0^{\gr{\leq n}}$ from~\Cref{thm:path_extension}. Similarly, for fixed $s \in [0,1]$, we define $\orx^v_{s; \cdot,\cdot}: \Delta^2 \to G_0^{\gr{\leq n+1}}$ to be the unique extension of $\rx^v_{s; \cdot, \cdot}: \Delta^2 \to G_0^{\gr{\leq n}}$. Then, we define $\orX : \Delta^2 \times \Delta^2 \to G_1^{\gr{\leq n+1}}$ by setting $\orX^{\gr{k}} = \rX^{\gr{k}}$ for $k \in [n]$, and for level $n+1$, we define
\begin{align} \label{eq:rX_pathwise_extension}
    \orX^{\gr{n+1}}_{s_1, s_2; t_1, t_2} \coloneqq \ts\left( \orx^h_{s_1, s_2;t_1} \cdot \orx^v_{s_2; t_1, t_2} \cdot \orx^{-h}_{s_1, s_2; t_2} \cdot \orx^{-v}_{s_1;t_1, t_2} \right)^{\gr{n+1}} \in T_1^{s, \gr{\leq n+1}} \subset T_1^{\gr{\leq n+1}}.
\end{align}

\begin{lemma} \label{lem:pathwise_ext_dg_regularity}
    The pathwise extension $\orrX$ is $\rho$-H\"older, so $\orrX \in \DGF^\rho(\cmG^{\gr{\leq n+1}})$. Furthermore, we have
    \begin{align} \label{eq:pathwise_ext_interior_bound}
        \|\orX^{\gr{n+1}}_{s_1, s_2; t_1, t_2}\|  \leq \frac{1}{2\beta} \CTR{n+1}{\orho}{s_1, s_2; t_1, t_2}.
    \end{align}
\end{lemma}
\begin{proof}
    By the path extension theroem (\Cref{thm:path_extension}), and the modified continuity of extension (\Cref{thm:path_extension_cont}), the path components $\orx^h$ and $\orx^v$ satisfy the path regularity and continuity conditions. Next, by~\Cref{lem:boundary_signature_bound} and the fact that $\|\ts\|_{\op} \leq 1$, the surface component $\orX^{\gr{n+1}}$ satisfies~\Cref{eq:pathwise_ext_interior_bound}, and in particular the surface regularity condition. It remains to show that $\orX$ is valued in $\cmG^{\gr{\leq n+1}}$. Because we assume that the original path components $\rx^h$ and $\rx^v$ are valued in $\cmG^{\gr{\leq n}}$, the path extensions $\orx^h$ and $\orx^v$ are valued in $\cmG^{\gr{\leq n+1}}$~\cite[Corollary 3.9]{cass_integration_2016}. For the surface component, we note that
    \begin{align}\nonumber
        \delta(\orX_{s_1, s_2; t_1, t_2}) = \orx^h_{s_1, s_2;t_1} \cdot \orx^v_{s_2; t_1, t_2} \cdot \orx^{-h}_{s_1, s_2; t_2} \cdot \orx^{-v}_{s_1;t_1, t_2} \in \cmG^{\gr{\leq n+1}}.
    \end{align}
    Then, considering the difference between $\orX$ and $\ts \circ \delta(\orX)$, we get
    \begin{align}\nonumber
        \orX_{s_1, s_2; t_1, t_2} - \ts\circ \delta(\orX_{s_1, s_2; t_1, t_2}) = B \in T_1^{\ker, \gr{\leq n}}.
    \end{align}
    Because the difference $B$ is at most degree $n$, and is in the center of $T_1^{\gr{\leq n+1}}$, we have $\widehat{B} = (1, B) \in \cmG^{\gr{\leq n+1}}$. Furthermore, by~\Cref{lem:section_preserves_grplike}, $\ts\circ \delta(\orX_{s_1, s_2; t_1, t_2}) \in \cmG^{\gr{\leq n+1}}$. Then, we have
    \begin{align}\nonumber
        \orX_{s_1, s_2; t_1, t_2} = B + \ts\circ \delta(\orX_{s_1, s_2; t_1, t_2}) = \widehat{B} * (\ts\circ \delta(\orX_{s_1, s_2; t_1, t_2})) \in \cmG^{\gr{\leq n+1}}. 
    \end{align}
\end{proof}

\subsubsection{Almost Multiplicativity}
While the pathwise extension is a double group functional, it is not multiplicative. Let $s_1 < s_2 < s_3$ and $t_1 < t_2$, and consider
\begin{align}\nonumber
    \orX_{s_1, s_2; t_1, t_2} \hmult \orX_{s_2, s_3; t_1, t_2} = (\orx^h_{s_1, s_2; t_1} \gt \orX_{s_2, s_3; t_1, t_2}) * \orX_{s_1, s_2; t_1, t_2}.
\end{align}
We will express $(\orx^h_{s_1, s_2; t_1} \gt \orX_{s_2, s_3; t_1, t_2})$ and $\orX_{s_1, s_2; t_1, t_2}$ in terms of the decomposition $T_1^s \oplus T_1^{\ker}$ as
\begin{align}\nonumber
    (\orx^h_{s_1, s_2; t_1} \gt \orX_{s_2, s_3; t_1, t_2})^{\gr{k}} = (F^{s,\gr{k}}, F^{\ker, \gr{k}}) \andd \orX^{\gr{k}}_{s_1, s_2; t_1, t_2}  = (E^{s, \gr{k}}, E^{\ker, \gr{k}}).
\end{align}
In particular, by definition of the decomposition, we have
\begin{align} \label{eq:almost_mult_EF}
    F^{s,\gr{k}} = \ts \circ \delta(\orx^h_{s_1, s_2; t_1} \gt \orX_{s_2, s_3; t_1, t_2})^{\gr{k}} \andd E^{s,\gr{k}} = \ts \circ\delta(\orX^{\gr{k}}_{s_1, s_2; t_1, t_2}).
\end{align}
Furthermore, while $E^{\ker, \gr{n+1}} = 0$ by definition of the pathwise extension, $F^{\ker, \gr{n+1}} \neq 0$ in general, since the action does not preserve the $T_1^s$ subspace. Then, by the multiplicative properties from~\Cref{prop:T1_decomposition}, we have
\begin{align}\nonumber
    \left(\orX_{s_1, s_2; t_1, t_2} \hmult \orX_{s_2, s_3; t_1, t_2}\right)^{\gr{n+1}} = \left( F^{s,\gr{n+1}} + E^{s, \gr{n+1}} + \sum_{q=2}^{n-1} F^{s, \gr{q}} * E^{s, \gr{n+1-q}}, \, F^{\ker, \gr{n+1}}\right).
\end{align}
Then, since $\ts$ is a morphism of algebras, we have
\begin{align}\nonumber
    F^{s,\gr{n+1}} &+ E^{s, \gr{n+1}} + \sum_{q=2}^{n-1} F^{s, \gr{q}} * E^{s, \gr{n+1-q}} \\
    &= \ts\left( \delta(\orx^h_{s_1, s_2; t_1} \gt \orX_{s_2, s_3; t_1, t_2})^{\gr{n+1}} + \delta(\orX^{\gr{n+1}}_{s_1, s_2; t_1, t_2}) + \sum_{q=2}^{n-1}\delta(\orx^h_{s_1, s_2; t_1} \gt \orX_{s_2, s_3; t_1, t_2})^{\gr{q}} \cdot \delta(\orX^{\gr{n+1-q}}_{s_1, s_2; t_1, t_2}) \right) \nonumber\\
    & = \ts\left( \orx^h_{s_1, s_2; t_1} \cdot  \delta(\orX_{s_2, s_3; t_1, t_2}) \cdot \orx^{-h}_{s_1, s_2; t_1} \cdot \delta(\orX_{s_1, s_2; t_1, t_2}) \right)^{\gr{n+1}}\nonumber\\
    & = \ts\left( \orx^{h}_{s_1, s_2; t_1} \cdot(\orx^h_{s_2, s_3;t_1} \cdot \orx^v_{s_3; t_1, t_2} \cdot \orx^{-h}_{s_2, s_3; t_2} \cdot \orx^{-v}_{s_2;t_1, t_2}) \cdot \orx^{-h}_{s_1, s_2; t_1} \cdot (\orx^h_{s_1, s_2;t_1} \cdot \orx^v_{s_2; t_1, t_2} \cdot \orx^{-h}_{s_1, s_2; t_2} \cdot \orx^{-v}_{s_1;t_1, t_2}) \right)^{\gr{n+1}} \nonumber\\
    & = \ts\left( \orx^{h}_{s_1, s_2; t_1} \cdot\orx^h_{s_2, s_3;t_1}\cdot \orx^v_{s_3; t_1, t_2} \cdot \orx^{-h}_{s_2, s_3; t_2}\cdot \orx^{-h}_{s_1, s_2; t_2} \cdot \orx^{-v}_{s_1;t_1, t_2} \right)^{\gr{n+1}} \nonumber\\
    & = \orX^{\gr{n+1}}_{s_1, s_3; t_1, t_2}.\nonumber
\end{align}
Here, we use the first Peiffer identity in the second equality, the fact that $\orX$ is a double group function in the third identity, we cancel out the inverses in the fourth line, and finally the last line is by definition of $\orX$. In particular, we have proved the following, where the vertical case can be proved using analogous computations.

\begin{lemma} \label{lem:pathwise_ext_delta_multiplicative}
    Suppose $\orrX$ is the pathwise extension of $\rrX$. Then $\orrX$ is multiplicative in the $T_1^s$ component. In particular, we have
    \begin{align}\nonumber
        \delta (\orX_{s_1, s_2; t_1, t_2} \hmult \orX_{s_2, s_3; t_1, t_2}) = \delta(\orX_{s_1, s_3; t_1, t_2}) \andd \delta (\orX_{s_1, s_2; t_1, t_2} \vmult \orX_{s_1, s_2; t_2, t_3}) = \delta(\orX_{s_1, s_2; t_1, t_3}).
    \end{align}
\end{lemma}

Thus, this implies that the obstruction to multiplicativity is entirely contained in $T_1^{\ker}$,
\begin{align} \label{eq:almost_hmult_element}
    \orX_{s_1, s_2; t_1, t_2} \hmult \orX_{s_2, s_3; t_1, t_2} - \orX_{s_1, s_3; t_1, t_2} = F^{\ker, \gr{n+1}} = (\orx^h_{s_1, s_2; t_1} \gt \orX_{s_2, s_3; t_1, t_2})^{\gr{n+1}} -\ts \circ \delta(\orx^h_{s_1, s_2; t_1} \gt \orX_{s_2, s_3; t_1, t_2})^{\gr{n+1}}.
\end{align}
We can perform the same argument to obtain the analogous identity for vertical composition,
\begin{align} \label{eq:almost_vmult_element}
    \orX_{s_1, s_2; t_1, t_2} \vmult \orX_{s_1, s_2; t_2, t_3} - \orX_{s_1, s_2; t_1, t_3} = (\orx^v_{s_1; t_1, t_2} \gt \orX_{s_1, s_2; t_2, t_3})^{\gr{n+1}} -\ts \circ \delta(\orx^v_{s_1; t_1, t_2} \gt \orX_{s_1, s_2; t_2, t_3})^{\gr{n+1}}.
\end{align}
In fact, we can show that $\orX$ is \emph{almost multiplicative} in the following sense.

\begin{proposition} \label{prop:almost_multiplicative}
    Suppose $\orrX$ is the pathwise extension of $\rrX$. Then, $\orrX$ is \emph{almost horizontally multiplicative}, such that
    \begin{align}\nonumber
        \left\|\left(\orX_{s_1, s_2; t_1, t_2} \hmult \orX_{s_2, s_3; t_1, t_2} - \orX_{s_1, s_3; t_1, t_2}\right)^{\gr{n+1}}\right\| \leq  \frac{1}{2\beta} \CTR{n+1}{\orho}{s_1, s_3; t_1, t_2}
    \end{align}
    for all $s_1 < s_2 < s_3$ and $t_1 < t_2$; and is \emph{almost vertically multiplicative}, such that
    \begin{align}\nonumber
        \left\|\left(\orX_{s_1, s_2; t_1, t_2} \vmult \orX_{s_1, s_2; t_2, t_3} - \orX_{s_1, s_2; t_1, t_3}\right)^{\gr{n+1}}\right\|\leq  \frac{1}{2\beta} \CTR{n+1}{\orho}{s_1, s_2; t_1, t_3}.
    \end{align}
\end{proposition}
\begin{proof}
    We will prove the horizontal case, and the vertical case follows using similar arguments. We use the above computation in~\Cref{eq:almost_hmult_element} for the obstruction to horizontal multiplicativity. We note that $\|\ts \circ \delta\|_{\op} \leq \sqrt{2}$, so that
    \begin{align} \label{eq:almost_mult_initial_bound_short_ts}
        \left\|\left(\orX_{s_1, s_2; t_1, t_2} \hmult \orX_{s_2, s_3; t_1, t_2} - \orX_{s_1, s_3; t_1, t_2}\right)^{\gr{n+1}}\right\| \leq 3 \left\| (\orx^h_{s_1, s_2; t_1} \gt \orX_{s_2, s_3; t_1, t_2})^{\gr{n+1}}\right\|.
    \end{align}
    Then, by the definition of the action, we have
    \begin{align}\nonumber
        (\orx^h_{s_1, s_2; t_1} \gt \orX_{s_2, s_3; t_1, t_2})^{\gr{n+1}} = \sum_{k=0}^{n-1} \sum_{i=0}^k \orx^{h, \gr{i}}_{s_1, s_2; t_1} \cdot \orX^{\gr{n+1-k}}_{s_2, s_3; t_1, t_2} \cdot \orx^{-h, \gr{k-i}}_{s_1, s_2; t_1},
    \end{align}
    where we note that $(\orx^{h}_{s_1, s_2; t_1} \cdot\orx^{-h}_{s_1, s_2; t_1})^{\gr{n+1}} = 0$, so the sum in $k$ only goes up to $k = n-1$. Then, using the path regularity conditions from~\Cref{eq:path_regularity}, 
    \begin{align}\nonumber
        \left\| (\orx^h_{s_1, s_2; t_1} \gt \orX_{s_2, s_3; t_1, t_2})^{\gr{n+1}} \right\| &\leq  \sum_{k=0}^{n-1} \sum_{i=0}^k \|\orx^{h, \gr{i}}_{s_1, s_2; t_1}\| \cdot \|\orX^{\gr{n+1-k}}_{s_2, s_3; t_1, t_2}\| \cdot \|\orx^{-h, \gr{k-i}}_{s_1, s_2; t_1}\| \\
        & \leq \frac{1}{\orho \beta^2}\sum_{k=0}^{n-1} \sum_{i=0}^k \frac{\ctr{2k\orho}{s_1, s_2}}{\ffact{2i}{\orho}{\ffact{2k-2i}{\orho}}} \|\orX^{\gr{n+1-k}}_{s_2, s_3; t_1, t_2}\| \nonumber\\
        & \leq \frac{1}{\orho\beta^2} \sum_{k=0}^{n-1} 2^{2k\orho} \frac{\ctr{2k\orho}{s_1, s_2}}{\ffact{2k}{\orho}} \|\orX^{\gr{n+1-k}}_{s_2, s_3; t_1, t_2}\|, \label{eq:am_prelim_bound}
    \end{align}
    where we use~\Cref{eq:binomial_sum_bound} with respect to the sum over $i$ in the final step. Then, using the surface regularity in~\Cref{eq:surface_regularity} including at level $n+1$ by~\Cref{lem:pathwise_ext_dg_regularity}, we get
    \begin{align}\nonumber
        \left\| (\orx^h_{s_1, s_2; t_1} \gt \orX_{s_2, s_3; t_1, t_2})^{\gr{n+1}} \right\| & \leq \frac{1}{\orho \beta^3} \sum_{k=0}^{n-1} \sum_{q=1}^{2n-2k +1} 2^{2k\orho} \frac{\ctr{2k\orho}{s_1, s_2}}{\ffact{2k}{\orho}} \, \cdot \,  2^{2(n+1-k)\orho}  \frac{\ctr{(2(n-k+1)-q)\orho}{s_2, s_3} \ctr{q\orho}{t_1, t_2}}{\ffact{2(n-k+1)-q}{\orho} \ffact{q}{\orho}}  \\
        & = \frac{2^{2(n+1)\orho}}{\orho\beta^3} \sum_{q=1}^{2n+1} \sum_{k=0}^{k_{n,q}} \frac{\ctr{2k\orho}{s_1, s_2} \ctr{(2(n-k+1)-q)\orho}{s_2, s_3}}{\ffact{2k}{\orho}\ffact{2(n-k+1)-q}{\orho}} \, \cdot \, \frac{\ctr{q\orho}{t_1, t_2}}{\ffact{q}{\orho}} \nonumber\\
        & = \frac{2^{2(n+1)\orho}}{\orho^2\beta^3} \sum_{q=1}^{2n+1} \frac{\ctr{(2n+2-q)\orho}{s_1, s_3}}{\ffact{2n+2-q}{\orho}}\, \cdot \, \frac{\ctr{q\orho}{t_1, t_2}}{\ffact{q}{\orho}}\nonumber
    \end{align}
    where we change the order of summation in the second line, and the summation over $k$ is taken up to $k_{n,q} = \min\{ n-1, \lfloor n - \frac{q-1}{2}\rfloor\}$. We use the neoclassical inequality~\Cref{eq:neoclassical} in the last line, and for $\beta$ sufficiently large (\Cref{eq:beta}), we have our desired result. 
\end{proof}

\begin{remark} \label{rem:almost_mult_algebra_section}
    This is the step in which defining the initial pathwise extension via the algebra section $\ts$ is crucial. In particular, it allows us to simplify the the obstruction to multiplicativity in~\Cref{eq:almost_mult_initial_bound_short_ts}. An alternative method would be to simply define a \emph{trivial} initial extension $\rX^{\gr{n+1}} = 0$, as is done in the case of rough paths. However, here the obstruction to multiplicativity would be 
    \begin{align}\nonumber
        \left(\orX_{s_1, s_2; t_1, t_2} \hmult \orX_{s_2, s_3; t_1, t_2} - \orX_{s_1, s_3; t_1, t_2}\right)^{\gr{n+1}}\leq \left((\orx^h_{s_1, s_2; t_1} \gt \orX_{s_2, s_3; t_1, t_2}) * \rX_{s_1, s_2; t_1, t_2}\right)^{\gr{n+1}}.
    \end{align}
    In this case, one can still prove an almost multiplicativity bound
    \begin{align}\nonumber
        \left\|\left(\orX_{s_1, s_2; t_1, t_2} \hmult \orX_{s_2, s_3; t_1, t_2} - \orX_{s_1, s_3; t_1, t_2}\right)^{\gr{n+1}}\right\| \leq C_{n+1} \CTR{n+1}{\orho}(s_1, s_3; t_1, t_2),
    \end{align}
    but the constant $C_{n+1} > 0$ grows with respect to $n$, and we would not be able to show that the lift to infinite levels is bounded as in~\Cref{cor:surface_sig_rs_cont}. 
\end{remark}

\subsubsection{Dyadic Partitions and Grid Multiplication}
Recall the dyadic rationals $\dya$ and the dyadic simplex $\Delta_\dya^2$ from~\Cref{eq:dyadic,eq:dyadic_simplex} respectively.

\begin{definition}
    A \emph{dyadic double group functional valued in $\cmG^{\gr{\leq n}}$} $\rrY = (\ry^h, \ry^v, \rY)$ consists of three maps
    \begin{align}\nonumber
        \ry^h : \Delta^2_\dya \times \dya \to G_0^{\gr{\leq n}}, \quad \ry^v : \dya \times \Delta^2_\dya \to G_0^{\gr{\leq n}}, \andd \rY : \Delta^2_\dya \times \Delta^2_\dya \to G_1^{\gr{\leq n}}
    \end{align}
    which is path-multiplicative and satisfies the boundary condition in~\Cref{eq:rrX_boundary_condition}. We say it is \emph{multiplicative} if~\Cref{eq:dgf_def_multiplicative} holds for all dyadic rectangles. We denote the set of dyadic double group functionals by $\DGF_\dya(\cmH)$.
\end{definition}

For $m \in \N$ and $(\smm, \spp; \tmm, \tpp) \in \Delta_\dya \times \Delta_\dya$, define the \emph{$m^{th}$ dyadic partition of $[\smp] \times [\tmp]$} by
\begin{align} \label{eq:regular_dyadic_partition}
    \dpart_m(\smm, \spp; \tmm, \tpp) \coloneqq \left\{ (s_i, t_j) \in \dya^2 \, : \, s_i = \smm + \frac{(\spp - \smm)i}{2^m}, \, t_j = \tmm + \frac{(\tpp - \tmm)j}{2^m}, i, j \in [2^m]_0 \right\}.
\end{align}
Fix a dyadic rectangle $(\smm, \spp; \tmm, \tpp) \in \Delta_\dya \times \Delta_\dya$, and we define a grid multiplication operation recursively as follows. This recursive multiplication is similar to the one in~\cite{yekutieli_nonabelian_2015}.

\begin{definition}
    Let $(\smm, \spp; \tmm, \tpp) \in \Delta_\dya \times \Delta_\dya$ be a dyadic rectangle, and $\rrX$ be a double group functional. We define \emph{grid multiplication} operations on $\rrX$ recursively as follows. The $0^{th}$ grid multiplication operator is given by 
    \begin{align}\nonumber
        \gridmult^0_{\smm, \spp; \tmm, \tpp}[\rrX] = \orrX_{\smm, \spp; \tmm, \tpp}.
    \end{align}
    Then, we recursively define the grid multiplication by multiplying horizontally first, and then vertically. Let $\sss = \frac{\smm + \spp}{2}$ and $\tss = \frac{\tmm + \tpp}{2}$. Then, the $m^{th}$ grid multiplication operation is
    \begin{align} \label{eq:gridmult_def_recursive}
        \gridmult^m_{\smm, \spp; \tmm, \tpp}[\rrX] = (\gridmult^{m-1}_{\smm, \sss; \tmm, \tss}[\rrX] \hmult \gridmult^{m-1}_{\sss, \spp; \tmm, \tss}[\rrX]) \vmult (\gridmult^{m-1}_{\smm, \sss; \tss, \tpp}[\rrX] \hmult \gridmult^{m-1}_{\sss, \spp; \tss, \tpp}[\rrX]).
    \end{align}
    In particular, we note that $\gridmult^m[\rrX]$ is a double group functional for all $m \in \N_0$. 
\end{definition}
The following figure shows the order of horizontal and vertical multiplications in the case of $\gridmult^2[\orrX]$, where red lines denote the boundaries on which multiplication is occuring.
\begin{figure}[!h]
    \includegraphics[width=\linewidth]{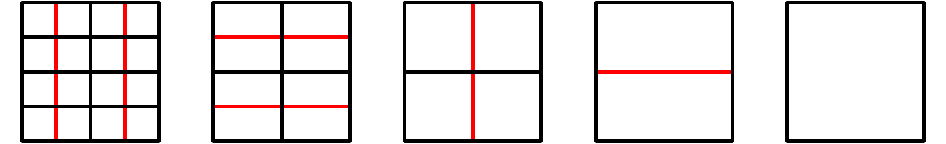}
\end{figure}  

Note that because double group functionals satisfy the boundary condition in~\Cref{eq:rrX_boundary_condition}, the interchange law holds. Thus, these grid multiplication operators are equivalent to multiplying vertically first and then horizontally (or a completely different ordering that doesn't come from this recursive procedure). However, it is convenient to have a specific ordering, and we will use the explicit ordering specified in~\Cref{eq:gridmult_def_recursive} as a way to multiply together elements from the dyadic partition $\dpart_m(\smp; \tmp)$.
Furthermore, we will require intermediate grid multiplications in between $\gridmult^{m-1}_{\smm, \spp; \tmm, \tpp}[\orrX]$ and $\gridmult^m_{\smm, \spp; \tmm, \tpp}[\orrX]$. We define a more general operation which will be used in the proof of~\Cref{lem:uniqueness} to show uniqueness.

\begin{definition} \label{def:intermediate_subdivision}
    Let $\rrX$ and $\rrY$ be double group functionals with the same path components, $\rx^h = \ry^h$ and $\rx^v = \ry^v$. Let
    \begin{align}\nonumber
        \gridmult^{0,0}_{\smm, \spp; \tmm, \tpp}[\rrX,\rrY] \coloneqq \gridmult^0_{\smm, \spp; \tmm, \tpp}[\rrX] \andd \gridmult^{0,1}_{\smm, \spp; \tmm, \tpp}[\rrX, \rrY] = \gridmult^1_{\smm, \spp; \tmm, \tpp}[\rrY].
    \end{align}
    Then, for $m \in \N$, $a \in 0,1,2,3$, and $u \in [4^m]_0$, define the function
    \begin{align} \label{eq:intermediate_index}
        \intind_{m,a}(i) \coloneqq \begin{cases} 
            0 & : u \leq 4^{m-1}a \\
            u - 4^{m-1}a & : 4^{m-1}a < u < 4^{m-1}(a+1) \\
            4^{m-1} & :  i \geq 4^{m-1}(a+1).
        \end{cases}
    \end{align}
    Then, we define
    \begin{align} \label{eq:intermediate_subdivision}
        \gridmult^{m,u}_{\smm, \spp; \tmm, \tpp}[\rrX, \rrY] \coloneqq (\gridmult^{m-1, \intind_{m,0}(u)}_{\smm, \sss; \tmm, \tss}[\rrX, \rrY] \hmult \gridmult^{m-1, \intind_{m,1}(u)}_{\sss, \spp; \tmm, \tss}[\rrX, \rrY]) \vmult (\gridmult^{m-1, \intind_{m,2}(u)}_{\smm, \sss; \tss, \tpp}[\rrX, \rrY] \hmult \gridmult^{m-1, \intind_{m,3}(u)}_{\sss, \spp; \tss, \tpp}[\rrX, \rrY]).
    \end{align}
\end{definition}
This operation has the property that
\begin{align}\nonumber
    \gridmult^{m,0}_{\smp, \tmp}[\rrX, \rrY] = \gridmult^m_{\smp; \tmp}[\rrX] \andd \gridmult^{m,4^m}_{\smp, \tmp}[\rrX, \rrY] = \gridmult^m_{\smp; \tmp}[\rrY]
\end{align}
Returning to the setting of the path extension, we set $\rrX = \orrX$ and $\rrY = \gridmult^1[\orrX]$, and denote
\begin{align}\nonumber
    \gridmult^{m,u}_{\smm, \spp; \tmm, \tpp}[\orrX] \coloneqq \gridmult^{m,u}_{\smm, \spp; \tmm, \tpp}\left[\orrX, \gridmult^1[\orrX]\right]
\end{align}
to simplify notation. In this case, we have
\begin{align}\nonumber
    \gridmult^{m,0}_{\smm, \spp; \tmm, \tpp}[\orrX] = \gridmult^{m}_{\smm, \spp; \tmm, \tpp}[\orrX] \andd \gridmult^{m,4^m}_{\smm, \spp; \tmm, \tpp}[\orrX] = \gridmult^{m}_{\smm, \spp; \tmm, \tpp}[\gridmult^1[\orrX]] = \gridmult^{m+1}_{\smm, \spp; \tmm, \tpp}[\orrX].
\end{align}
The general idea is that $\gridmult^{m,0}_{\smm, \spp; \tmm, \tpp}[\orrX]$ corresponds to the usual dyadic partition of the rectangle in to $4^m$ squares. Then, we consecutively subdivide each of the $4^m$ squares in $\dpart_m$, and multiply them in horizontal-then-vertical ordering.

\begin{figure}[!h]
    \includegraphics[width=\linewidth]{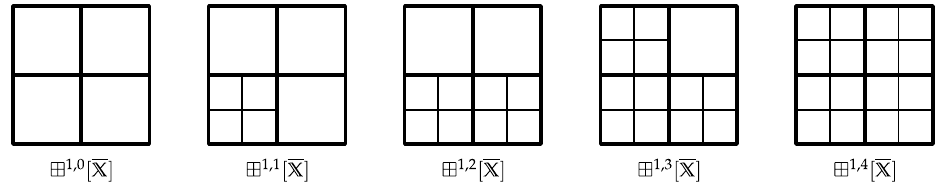}
\end{figure}

\subsubsection{Dyadic Limit}
In this section, we show that $\lim_{m \to \infty} \gridmult^m_{\smm, \spp; \tmm, \tpp}[\orrX]$ converges. First, we denote the components of $\gridmult^m[\orrX] \in \DGF(\cmG^{\gr{\leq n+1}})$ for $(\smp;\tmp) \in \Delta^2 \times \Delta^2$ by
\begin{align}\nonumber
    \gridmult^m_{\smp;\tmp}[\orrX] = \left( \orx^h_{\smp; \tmm}, \orx^v_{\spp; \tmp},\orx^h_{\smp; \tpp}, \orx^v_{\smm; \tmp}, \gridmult^m_{\smp;\tmp}[\orX]\right).
\end{align}
We note the path components are the same as $\orrX$ since it is path-multiplicative.
In particular, $\gridmult^m[\orrX] : \Delta^2 \times \Delta^2 \to \dg(\bT)$ is a double group function, where the boundary condition in~\Cref{eq:rrX_boundary_condition} holds by~\Cref{lem:pathwise_ext_delta_multiplicative}.
We will need the preliminary lemma.

\begin{lemma} \label{lem:mult_bound_reduction}
    Let $\orrX, \orrY$ be two different extensions of $\rrX$ such that the path functions coincide; in particular $\orrX^{\gr{k}} = \orrY^{\gr{k}} = \rrX^{\gr{k}}$ for all $k \in [n]$, $\orx^{h, \gr{k}} = \ory^{h, \gr{k}}$ for all $k \in [n+1]$ and $\orx^{v, \gr{k}} = \ory^{v, \gr{k}}$ for all $k \in [n+1]$. Then,
    \begin{align}\nonumber
        \|\orX_{s_1, s_2; t_1, t_2} \hmult \orX_{s_2, s_3; t_1, t_2} - \orY_{s_1, s_2; t_1, t_2} \hmult \orX_{s_2, s_3; t_1, t_2} \| \leq \|(\orX_{s_1, s_2; t_1, t_2} - \orY_{s_1, s_2; t_1, t_2})^{\gr{n+1}}\|.
    \end{align}
    The analogous bounds hold for the opposite ordering and for vertical compositions.
\end{lemma}
\begin{proof}
    By definition of the horizontal product, and the fact that the path functions coincide, we have
    \begin{align}\nonumber
        \orX_{s_1, s_2; t_1, t_2} \hmult \orX_{s_2, s_3; t_1, t_2} &- \orY_{s_1, s_2; t_1, t_2} \hmult \orX_{s_2, s_3; t_1, t_2} \\
        &= \Big(\orx^h_{s_1, s_2; t_1, t_2} \cdot \orX_{s_2, s_3; t_1, t_2} \cdot \orx^{-h}_{s_1, s_2; t_1, t_2} \cdot (\orX_{s_1, s_2; t_1, t_2} - \orY_{s_1, s_2; t_1, t_2})\Big)^{\gr{n+1}}\nonumber\\
        & = (\orX_{s_1, s_2; t_1, t_2} - \orY_{s_1, s_2; t_1, t_2})^{\gr{n+1}}.\nonumber
    \end{align}
\end{proof}

Now, we will show that the limit of grid multiplications is well defined. 
\begin{proposition} \label{prop:first_sewing}
    Fix a dyadic rectangle $(\smm, \spp; \tmm, \tpp) \in \Delta_\dya \times \Delta_\dya$.  Then,
    \begin{align}\nonumber
        \trX_{\smm, \spp; \tmm, \tpp} \coloneqq \lim_{m \to \infty} \gridmult^m_{\smp; \tmp}[\orX]
    \end{align}
    is well defined and satisfies $\trX^{\gr{k}} = \rX^{\gr{k}}$ for $k \in [n]$, $\trx^h = \orx^h$, and $\trx^v = \orx^v$. Furthermore, $\trrX \in \DGF^{\rho}_\dya(\cmG^{\gr{\leq n+1}})$.
\end{proposition}
\begin{proof}
    In this proof, we will fix a dyadic rectangle $(\smp; \tmp) \in \Delta_\dya^2 \times \Delta_\dya^2$ and omit the subscripts in the grid multiplication operation to simplify notation, ie. $\gridmult^m = \gridmult^m_{\smp; \tmp}$. First, since $\orrX^{(k)} = \rrX^{(k)}$ for all $k \in [n]$, and the fact that $\rrX$ is a multiplicative double group function, we have
    \begin{align}\nonumber
        \trrX^{\gr{k}}_{\smp; \tmp} = \lim_{m \to \infty}\left(\gridmult^m[\orrX]\right)^{\gr{k}} = \lim_{m \to \infty} \left(\gridmult^m[\rrX]\right)^{\gr{k}} = \rrX^{\gr{k}}_{\smp, \tmp}.
    \end{align} 
    \medskip

    Thus, it remains to consider level $n+1$. 
    We will show that for any $m \in \N$, we have
    \begin{align} \label{eq:maximal_inequality}
        \|(\gridmult^m[\orX])^{\gr{n+1}}\| \leq \frac{1}{\beta} \CTR{n+1}{\orho}{\smp;\tmp}.
    \end{align}
    This will be done by bounding consecutive subdivisions, and using the intermediate subivisions from~\Cref{eq:intermediate_subdivision}, we have
    \begin{align} \label{eq:max_ineq_intermediate_subdivisions}
        \|(\gridmult^{m+1}[\orX])^{\gr{n+1}} - (\gridmult^{m}[\orX])^{\gr{n+1}}\| \leq \sum_{u = 1}^{4^k} \left\|(\gridmult^{m, u}[\orX])^{\gr{n+1}} - (\gridmult^{m, u-1}[\orX])^{\gr{n+1}} \right\|.
    \end{align}
    Because $\gridmult^{m, u}[\orX]$ and $\gridmult^{m, u-1}[\orX]$ differ by further subdividing a certain subsquare $[s_{i-1}, s_i] \times [t_{j-1}, t_j] \in \dpart_m$, we have 
    \begin{align} \label{eq:consec_intermediate_subdivision}
        \left\|(\gridmult^{m, u}[\orX])^{\gr{n+1}} - (\gridmult^{m, u-1}[\orX])^{\gr{n+1}} \right\| = \left\|(\gridmult^1_{s_{i-1}, s_i; t_{j-1}, t_j}[\orX])^{\gr{n+1}} - \orX^{\gr{n+1}}_{s_{i-1}, s_i; t_{j-1}, t_j}\right\|.
    \end{align}
    Let $s = \frac{s_{i-1} + s_i}{2}$ and $t = \frac{t_{j-1} + t_j}{2}$. Then, using the definition of $(\gridmult^1_{s_{i-1}, s_i; t_{j-1}, t_j}[\orX])^{\gr{n+1}}$, we get
    \begin{align}\nonumber
        \Big\|(&\gridmult^1_{s_{i-1}, s_i; t_{j-1}, t_j}[\orX])^{\gr{n+1}} - \orX^{\gr{n+1}}_{s_{i-1}, s_i; t_{j-1}, t_j}\Big\| \\
        &= \left\|(\orX_{s_{i-1}, s; t_{j-1}, t} \hmult \orX_{s, s_i; t_{j-1}, t}) \vmult (\orX_{s_{i-1}, s; t, t_j} \hmult \orX_{s, s_i; t, t_j}) - \orX_{s_{i-1}, s_i; t_{j-1}, t}\vmult (\orX_{s_{i-1}, s; t, t_j} \hmult \orX_{s, s_i; t, t_j})\right\|\\
        & \hspace{10pt} + \left\|\orX_{s_{i-1}, s_i; t_{j-1}, t}\vmult (\orX_{s_{i-1}, s; t, t_j} \hmult \orX_{s, s_i; t, t_j}) - \orX_{s_{i-1}, s_i; t_{j-1}, t}\vmult \orX_{s_{i-1}, s_i; t, t_j} \right\|\nonumber\\
        & \hspace{10pt} + \left\|\orX_{s_{i-1}, s_i; t_{j-1}, t}\vmult \orX_{s_{i-1}, s_i; t, t_j} - \orX^{\gr{n+1}}_{s_{i-1}, s_i; t_{j-1}, t_j}\right\|\nonumber
    \end{align}
    Then, applying~\Cref{lem:mult_bound_reduction} to each of these three terms, we obtain
    \begin{align} \label{eq:consec_intermediate_subdivision2}
        \left\|(\gridmult^1_{s_{i-1}, s_i; t_{j-1}, t_j}[\orX])^{\gr{n+1}} - \orX^{\gr{n+1}}_{s_{i-1}, s_i; t_{j-1}, t_j}\right\| \leq & \left\|\left(\orX_{s_{i-1}, s; t_{j-1}, t} \hmult \orX_{s, s_i; t_{j-1}, t} - \orX_{s_{i-1}, s_i; t_{j-1}, t}\right)^{\gr{n+1}}\right\| \\
        & + \left\|\left(\orX_{s_{i-1}, s; t, t_j} \hmult \orX_{s, s_i; t, t_j} - \orX_{s_{i-1}, s_i; t, t_j}\right)^{\gr{n+1}}\right\| \nonumber\\
        & + \left\|\left(\orX_{s_1, s_2; t_1, t_2} \vmult \orX_{s_1, s_2; t_2, t_3} - \orX_{s_1, s_2; t_1, t_3}\right)^{\gr{n+1}}\right\|\nonumber
    \end{align}
    Here, the first two terms on the right correspond to the almost multiplicative error terms from the two horizontal products, and the third term corresponds to the almost multiplicative error term from the vertical product. Next, we apply the almost multiplicativity bounds from~\Cref{prop:almost_multiplicative} to all three terms, and using $\ctr{}{t_{j-1}, t}, \,\ctr{}{t, t_{j}} \leq \ctr{}{t_{j-1}, t_j}$ for the first two, we have
    \begin{align}\nonumber
        \left\|(\gridmult^1_{s_{i-1}, s_i; t_{j-1}, t_j}[\orX])^{\gr{n+1}} - \orX^{\gr{n+1}}_{s_{i-1}, s_i; t_{j-1}, t_j}\right\| \leq \frac{3\cdot 2^{2(n+1)\orho}}{\beta^2} \sum_{q=1}^{2n+1} \frac{\ctr{q\orho}{s_{i-1}, s_i} \ctr{(2n+2-q)\orho}{t_{j-1}, t_j}}{\ffact{q}{\orho} \ffact{2n+2-q}{\orho}}
    \end{align}
    Next, because the $s_i, t_j \in \cD_m$ are in the $m^{th}$ dyadic partition, we have
    \begin{align*}
        \ctr{}{s_{i-1}, s_i} = 2^{-m} \ctr{}{\smm, \spp} \andd \ctr{}{t_{j-1}, t_j} = 2^{-m} \ctr{}{\tmm, \tpp}.
    \end{align*}
    Then, substituting this into the above, we obtain
    \begin{align}\nonumber
        \left\|(\gridmult^1_{s_{i-1}, s_i; t_{j-1}, t_j}[\orX])^{\gr{n+1}} - \orX^{\gr{n+1}}_{s_{i-1}, s_i; t_{j-1}, t_j}\right\| \leq \left(2^{-2(n+1)m\orho}\right)\frac{3\cdot 2^{2(n+1)\orho}}{\beta^2} \sum_{q=1}^{2n+1} \frac{\ctr{q\orho}{\smp} \ctr{(2n+2-q)\orho}{\tmp}}{\ffact{q}{\orho} \ffact{2n+2-q}{\orho}}
    \end{align}
    Summing over all the intermediate subdivisions in~\Cref{eq:max_ineq_intermediate_subdivisions}, we get
    \begin{align} \label{eq:max_ineq_consecutive}
        \left\|(\gridmult^{m+1}[\orX])^{\gr{n+1}} - (\gridmult^{m}[\orX])^{\gr{n+1}}\right\| \leq \left(2^{2m(1-(n+1)\orho)}\right)\frac{3\cdot 2^{2(n+1)\orho}}{\beta^2} \sum_{q=1}^{2n+1} \frac{\ctr{q\orho}{\smp} \ctr{(2n+2-q)\orho}{\tmp}}{\ffact{q}{\orho} \ffact{2n+2-q}{\orho}}.
    \end{align}
    Using a telescoping sum, we have 
    \begin{align}\nonumber
        \left\|(\gridmult^{m+1}[\orX])^{\gr{n+1}} - \orX^{\gr{n+1}}_{\smp;\tmp}\right\| \leq \left(\sum_{m = 0}^\infty 2^{2m(1-(n+1)\orho)}\right) \frac{3\cdot 2^{2(n+1)\orho}}{\beta^2} \sum_{q=1}^{2n+1} \frac{\ctr{q\orho}{\smp} \ctr{(2n+2-q)\orho}{\tmp}}{\ffact{q}{\orho} \ffact{2n+2-q}{\orho}}.
    \end{align}
    Now, using the fact that $(n+1)\orho > 1$ by definition, the infinite sum above converges. If we take $\beta > 0$ to be such that
    \begin{align}\nonumber
        \beta \geq \left(\sum_{m = 0}^\infty 2^{2m(1-(n+1)\orho)}\right) \frac{3}{2}, 
    \end{align}
    we get
    \begin{align} \label{eq:max_ineq_subdivision_max_bound}
        \left\|(\gridmult^{m+1}[\orX])^{\gr{n+1}} - \orX^{\gr{n+1}}_{\smp;\tmp}\right\| \leq \frac{2^{(2n+1)\orho}}{\beta}  \sum_{q=1}^{2n+1} \frac{\ctr{q\orho}{\smp} \ctr{(2n+2-q)\orho}{\tmp}}{\ffact{q}{\orho} \ffact{2n+2-q}{\orho}} = \frac{1}{2\beta}\CTR{n+1}{\orho}{\smp;\tmp}
    \end{align}
    for any $m \in \N$. Finally, combining this with the bound for $\|\orX^{\gr{n+1}}_{\smp;\tmp}\|$ from~\Cref{eq:pathwise_ext_interior_bound}, we get our desired maximal inequality,
    \begin{align} \label{eq:max_ineq_final}
        \left\|(\gridmult^{m+1}[\orX])^{\gr{n+1}}\right\| \leq \frac{2^{(2n+2)\orho}}{\beta}  \sum_{q=1}^{2n+1} \frac{\ctr{q\orho}{\smp} \ctr{(2n+2-q)\orho}{\tmp}}{\ffact{q}{\orho} \ffact{2n+2-q}{\orho}} = \frac{1}{\beta}\CTR{n+1}{\orho}{\smp;\tmp}
    \end{align}

    Next, we note that this is Cauchy. Indeed, for any $m \in \N$ and any $M > m$, we can sum the bound in~\Cref{eq:max_ineq_consecutive} to obtain
    \begin{align}\nonumber
        \|(\gridmult^{M}[\orX])^{\gr{n+1}} - (\gridmult^{m}[\orX])^{\gr{n+1}}\| \leq \left(\sum_{m' = m}^\infty 2^{2m(1-(n+1)\orho)}\right) \frac{3\cdot 2^{2(n+1)\orho}}{\beta^2} \sum_{q=1}^{2n+1} \frac{\ctr{q\orho}{\smp} \ctr{(2n+2-q)\orho}{\tmp}}{\ffact{q}{\orho} \ffact{2n+2-q}{\orho}}.
    \end{align}
    which converges to $0$ as $m \to \infty$. Thus, the sequence $(\gridmult^m[\orX])^{\gr{n+1}}$ is Cauchy in $m$, and since the the underlying space $G_1^{\gr{\leq n+1}}$ is a closed subset of a complete space, it converges to $\trX^{\gr{n+1}}_{\smp;\tmp} \coloneqq \lim_{m \to \infty}(\gridmult^m[\orX])^{\gr{n+1}}$. Finally, by~\Cref{eq:max_ineq_final}, this satisfies the surface regularity condition
    \begin{align}\nonumber
        \|\trX^{\gr{n+1}}_{\smp;\tmp}\| \leq  \frac{1}{\beta}\CTR{n+1}{\orho}{\smp;\tmp}.
    \end{align}

    Thus, we define the extension $\trrX = (\trx^h, \trx^v, \trX)$ on the dyadic rationals 
    \begin{align}\nonumber
        \trx^h: \Delta^2_\dya \times \dya \to T_0^{\gr{\leq n+1}}, \quad \trx^v : \dya \times \Delta^2_\dya\to T_0^{\gr{\leq n+1}}, \andd \trX: \Delta^2_\dya \times \Delta^2_\dya \to T_1^{\gr{\leq n+1}}
    \end{align}
    In particular, since the boundary map $\delta$ is already compatible with horizontal and vertical compositions of $\orrX$ by~\Cref{lem:pathwise_ext_delta_multiplicative}, the limit $\trrX$ is a double group function.
\end{proof}

\subsubsection{Multiplicativity}

In this section, we will show that $\trrX$ is multiplicative on the dyadic rectangles where it is defined. We begin by showing that multiplicativity with respect to dyadic partitions is immediate.

\begin{lemma} \label{lem:grid_multiplicative}
    For any $(\smp; \tmp) \in \Delta_\dya \times \Delta_\dya$, we have
    \begin{align}\nonumber
        \gridmult^m_{\smp;\tmp}[\trrX] = \trrX_{\smp; \tmp}.
    \end{align}
\end{lemma}
\begin{proof}
    By definition of $\trrX$, we have
    \begin{align}\nonumber
        \gridmult^m_{\smp;\tmp}[\trrX] = \lim_{k \to \infty} \gridmult^m_{\smp;\tmp}[\gridmult^k[\orrX]] = \lim_{k \to \infty} \gridmult^{k+m}_{\smp;\tmp}[\orrX] = \trrX_{\smp;\tmp}.
    \end{align}
\end{proof}

We will require a preliminary lemma a bound for the grid multiplication with respect to an arbitrary partition (still made up of dyadics). Let $(\smp; \tmp) \in \Delta_\dya \times \Delta_\dya$ be a dyadic rectangle. Let
\begin{align} \label{eq:arb_dpart}
    \ppart = \{\smm = s_0 < \ldots < s_a = \spp\} \times \{\tmm = t_0 < \ldots < t_b = \tpp\} \subset \dya^2
\end{align}
be an arbitrary partition made up of dyadic rationals. Then, we let $\gridmult_{\ppart}[\trrX]$ denote the grid multiplication with respect to this partition. This is well-defined since $\trrX$ is a double group function, and is thus independent of the ordering due to the interchange law. For instance, we have $\gridmult_{\dpart_m(\smp;\tmp)}[\trrX] = \gridmult^m_{\smp;\tmp}[\trrX]$, where $\dpart_m(\smp;\tmp)$ is the regular dyadic partition defined in~\Cref{eq:regular_dyadic_partition}. 
We will only require a loose bound here, and proof is provided in~\Cref{apxsec:additional_surface_extension}.

\begin{lemma} \label{lem:arbitrary_partition_bound}
    Let $\ppart$ be an arbitrary partition of the dyadic rectangle $(\smp;\tmp) \in \Delta_\dya \times \Delta_\dya$ made up of dyadic rationals, of the form given in~\Cref{eq:arb_dpart}. Then
    \begin{align}\nonumber
        \left\| \gridmult_{\ppart}[\trX] - \trX_{\smp;\tmp}\right\| \leq \frac{ab}{\beta} \CTR{n+1}{\orho}{\smp;\tmp}.
    \end{align}
\end{lemma}

Finally, we are ready to prove that $\trrX$ is multiplicative.

\begin{proposition} \label{prop:dya_multiplicative}
    The double group functional $\trrX: \Delta_\dya \times \Delta_\dya \to \dg(\cmG^{\leq n+1})$ is multiplicative. In particular, for $\smm < \sss< \spp\in \dya$ and $\tmm < \tss < \tpp \in \dya$, we have
    \begin{align}\nonumber
        \trrX_{\smm, \sss; \tmp} \hmult \trrX_{\sss, \spp; \tmp} = \trrX_{\smp; \tmp} \andd \trrX_{\smp; \tmm, \tss} \vmult \trrX_{\smp; \tss, \tpp} = \trrX_{\smp; \tmp}.
    \end{align}
\end{proposition}
\begin{proof}
    We consider the horizontal case and the vertical case follows by the same argument.

    \begin{figure}[!h]
        \includegraphics[width=\linewidth]{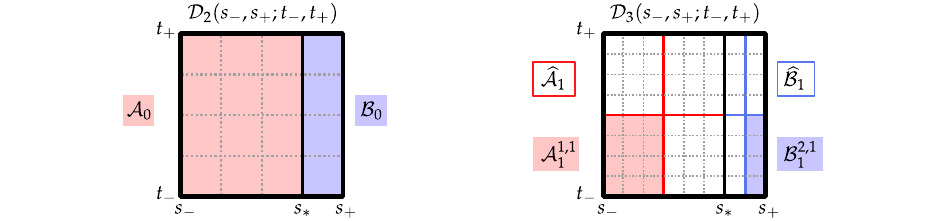}
    \end{figure}  
    
    \textbf{Step 1: Preliminary case.}
    First, since $\smm < \sss < \spp\in \dya$, there exists some $m \in \N$ such that $\sss \in \dpart_m(\smp)$ and we fix $m$ to be the smallest such integer. We decompose the 2D partition $\dpart_m(\smp;\tmp)$ into two 2D partitions 
    \begin{align} \label{eq:dya_mult_initial_partition}
        \cA_0 \coloneqq \dpart_m(\smp, \tmp) \cap [\smm, \sss] \times [\tmp] \andd \cB_0 \coloneqq \dpart_m(\smp; \tmp) \cap [\sss, \spp]\times[\tmp].
    \end{align}
    Then, we note that
    \begin{align}\nonumber
        \trX_{\smp;\tmp} = \gridmult^m_{\smp;\tmp}[\trX] = \gridmult_{\cA_0}[\trX] \hmult \gridmult_{\cB_0}[\trX],
    \end{align}
    where the first equality is due to~\Cref{lem:grid_multiplicative}, and the second equality is due to the interchange law. Note that we cannot immediately conclude that $\gridmult_{\cA_0}[\trX] = \trX_{\smm, s; \tmp}$ and $\tgridmult_{\cB_0}[\trX] = \trX_{s, \spp; \tmp}$ since $\cA_0$ and $\cB_0$ are not dyadic partitions (so we cannot apply~\Cref{lem:grid_multiplicative}). However, we can apply~\Cref{lem:mult_bound_reduction} twice and then~\Cref{lem:arbitrary_partition_bound} to obtain
    \begin{align}\nonumber
        \| \tgridmult_{\cA_0}[\trX] \hmult \tgridmult_{\cB_0}[\trX] - \trX_{\smm, s;\tmp} \hmult \trX_{s, \spp; \tmp}\| &\leq \|\tgridmult_{\cA_0}[\trX] - \trX_{\smm, s;\tmp}\| + \| \tgridmult_{\cB_0}[\trX] - \trX_{s, \spp; \tmp}\|\\
        & \leq 4^m \frac{2^{2(n+1)\orho}}{\beta} \sum_{q=1}^{2n+1} \frac{\ctr{q\orho}{\smp} \ctr{(2n+2-q)\orho}{\tmp}}{\ffact{q}{\orho} \ffact{2n+2-q}{\orho}},\nonumber
    \end{align}
    where $4^m$ is the total number of rectangles in the partition $\dpart_m(\smp;\tmp)$. \medskip

    \textbf{Step 2: Further subdivisions.}
    Now, for any $r \in \N$, define the partitions
    \begin{align}\nonumber
        \hcA_r \coloneqq \dpart_r(\smm, \sss; \tmp) \andd \hcB_r \coloneqq \dpart_r(\sss, \spp; \tmp). 
    \end{align}
    We note that $\hcA_r, \hcB_r \subset \dpart_{m+r}(\smp;\tmp)$. We denote the elements of $\hcA_r$ and $\hcB_r$ by
    \begin{align}\nonumber
        \hcA_r &= \{\smm = a_0 < \ldots < a_{2^r} = \sss\} \times \{\tmm = u_0 < \ldots < u_{2^r} = \tpp\}\\
        \hcB_r &= \{\sss = b_0 < \ldots < b_{2^r} = \spp\} \times \{\tmm = u_0 < \ldots < u_{2^r} = \tpp\}\nonumber
    \end{align}
    Next, we define
    \begin{align}\nonumber
        \cA_r \coloneqq \dpart_{m+r}(\smp, \tmp) \cap [\smm, \sss] \times [\tmp] \andd \cB_r \coloneqq \dpart_{m+r}(\smp; \tmp) \cap [\sss, \spp]\times[\tmp].
    \end{align}
    and for $i, j \in [2^r]$, we define further partitions
    \begin{align}\nonumber
        \cA^{i,j}_r \coloneqq \dpart_{m+r}(\smp, \tmp) \cap [a_{i-1}, a_i] \times [u_{j-1}, u_j] \andd \cB^{i,j}_r \coloneqq \dpart_{m+r}(\smp; \tmp) \cap [b_{i-1}, b_i]\times[u_{j-1}, u_j].
    \end{align}
    By a slight abuse of notation, we let
    \begin{align}\nonumber
        \gridmult_{\hcA_r}\left[ \gridmult_{\cA_r^{i,j}}[\trX]\right]
    \end{align}
    denote a grid multiplication with respect to $\hcA_r$, where the element corresponding to the rectangle $[a_{i-1}, a_i] \times [u_{j-1}, u_j]$ is $ \tgridmult_{\cA_r^{i,j}}[\trX]$. Note that by the interchange law, we have
    \begin{align} \label{eq:dya_mult_subsubpartitions}
        \gridmult_{\cA_r}[\trX] = \gridmult_{\hcA_r}\left[ \gridmult_{\cA_r^{i,j}}[\trX]\right].
    \end{align}
    The analogous statement holds for the $\cB$ partitions. Now, we note that
    \begin{align} \label{eq:dya_mult_decompose_AB}
        \trX_{\smp;\tmp} = \tgridmult^{m+r}_{\smp;\tmp}[\trX] = \tgridmult_{\cA_r}[\trX] \hmult \tgridmult_{\cB_r}[\trX].
    \end{align}
    Thus, we aim to bound
    \begin{align}\nonumber
        \| \gridmult_{\cA_r}[\trX] \hmult \gridmult_{\cB_r}[\trX] - \trX_{\smm, \sss;\tmp} \hmult \trX_{\sss, \spp; \tmp}\| &\leq \|\gridmult_{\cA_r}[\trX] - \trX_{\smm, \sss;\tmp}\| + \| \gridmult_{\cB_r}[\trX] - \trX_{\sss, \spp; \tmp}\|.
    \end{align}
    We begin with the first term on the right hand side. Using~\Cref{eq:dya_mult_subsubpartitions} and the fact that $\hcA_r$ is a regular dyadic partition (and thus we can use~\Cref{lem:grid_multiplicative}), we obtain
    \begin{align}\nonumber
        \left\|\gridmult_{\cA_r}[\trX] - \trX_{\smm, s;\tmp}\right\| = \left\| \gridmult_{\hcA_r}\left[ \gridmult_{\cA_r^{i,j}}[\trrX]\right] - \gridmult_{\hcA_r}[\trX]\right\|.
    \end{align}
    We use the inequality from~\Cref{lem:arbitrary_partition_bound} on each subrectangle $\gridmult_{\cA_r^{i,j}}[\trrX]$, and the fact that $\omega(a_{i-1}, a_i) = 2^{-r} \omega(\smm, \sss)$ and $\omega(u_{j-1}, u_j) = 2^{-r} \omega(\tmp)$ for all $i, j \in [2^r]$, to obtain
    \begin{align}\nonumber
        \left\|\gridmult_{\cA_r}[\trX] - \trX_{\smm, s;\tmp}\right\| &\leq \sum_{i,j=1}^{2^r} 4^m \frac{2^{2(n+1)\orho}}{\beta} \sum_{q=1}^{2n+1} \frac{\ctr{q\orho}{\smm, \sss} \ctr{(2n+2-q)\orho}{\tmp}}{\ffact{q}{\orho} \ffact{2n+2-q}{\orho}}\\
        & \leq 2^{2r(1 - (n+1)\orho)} \left( \frac{4^m}{\beta} \sum_{q=1}^{2n+1} \frac{\ctr{q\orho}{\smm, \sss} \ctr{(2n+2-q)\orho}{\tmp}}{\ffact{q}{\orho} \ffact{2n+2-q}{\orho}} \right).\nonumber
    \end{align}
    Note that the $4^m$ appears in the first line, since this bounds the number of rectangles in each $\cA_r^{i,j}$. 
    Repeating the same argument for the $\cB$ partitions, we obtain
    \begin{align}\nonumber
        \| \gridmult_{\cB_r}[\trX] - \trX_{\sss, \spp; \tmp}\| \leq 2^{2r(1 - (n+1)\orho)} \left( \frac{4^m}{\beta} \sum_{q=1}^{2n+1} \frac{\ctr{q\orho}{\sss,\spp} \ctr{(2n+2-q)\orho}{\tmp}}{\ffact{q}{\orho} \ffact{2n+2-q}{\orho}} \right).
    \end{align}
    Putting this together, we obtain
    \begin{align}\nonumber
        \| \gridmult_{\cA_r}[\trX] \hmult \gridmult_{\cB_r}[\trX] - \trX_{\smm, \sss;\tmp} \hmult \trX_{\sss, \spp; \tmp}\| \leq 2^{2r(1 - (n+1)\orho)}\left( \frac{2\cdot 4^m}{\beta} \sum_{q=1}^{2n+1} \frac{\ctr{q\orho}{\smp} \ctr{(2n+2-q)\orho}{\tmp}}{\ffact{q}{\orho} \ffact{2n+2-q}{\orho}} \right).
    \end{align}
    We note that the term in the parentheses does not depend on $r$, and thus since $(n+1)\orho > 1$, we have
    \begin{align}\nonumber
        \| \gridmult_{\cA_r}[\trX] \hmult \gridmult_{\cB_r}[\trX] - \trX_{\smm, \sss;\tmp} \hmult \trX_{\sss, \spp; \tmp}\| \to 0
    \end{align}
    as $r \to \infty$. Finally, by~\Cref{eq:dya_mult_decompose_AB}, this implies that $\trrX$ is horizontally multiplicative. 
\end{proof}

\subsubsection{Continuity}
Next, we will show that each component of the multiplicative double group function $\trrX: \Delta_\dya \times \Delta_\dya \to \dg(\cmG^{\gr{\leq n+1}})$ is uniformly continuous. This is an immediate corollary of multiplicativity.

\begin{proposition} \label{prop:surface_extension_unif_cont}
    Each component of $\trrX$ is uniformly continuous.
\end{proposition}
\begin{proof}

    First, we consider the path component $\trx^h$. For $t \in \dya$ and $s_1 < s_2 < s_3 \in \dya$, we have
    \begin{align*}
        \|\trx^h_{s_1, s_3; t} - \trx^h_{s_1, s_2; t} \| & = \|\trx^h_{s_1, s_2; t} \cdot \trx^h_{s_2, s_3; t} - \trx^h_{s_1, s_2; t}\| \\
        &\leq \sum_{q=1}^{n+1} \|\trx^{h, \gr{n+1-q}}_{s_1, s_2; t}\| \cdot \|\trx^{h, \gr{q}}_{s_2, s_3;t}\| \\
        & \leq \frac{\ctr{\rho}{s_2, s_3}}{\beta^2} \sum_{q=1}^{n+1} \frac{\ctr{(n+1-q)\rho}{s_1, s_2} \ctr{(q-1)\rho}{s_2, s_3}}{\ffact{n+1-q}{\rho}\ffact{q}{\rho}} \\
        & \leq \ctr{\rho}{s_2, s_3} \left( \frac{1}{\beta^2} \sum_{q=1}^{n+1} \frac{\ctr{n\rho}{0,1} }{\ffact{n+1-q}{\rho}\ffact{q}{\rho}} \right).
    \end{align*}
    A similar argument can be made for the first $s$ variable in $\trx^h$. For the $t$ variable, we note that the path continuity condition implies uniform continuity. Therefore, $\trx^h$ is uniformly continuous in all variables, and the same argument can be made for $\trx^v$. \medskip

    Next, we consider the surface component $\trX: \Delta_\dya \times \Delta_\dya \to G_1^{\gr{\leq n+1}}$. This takes advantage of multiplicativity, like in the proof of continuity for the $s$ variables for the horizontal path component. Let $s_1 < s_2 < s_3 \in \dya$ and fix $t_1 < t_2 \in \dya$. Then, because $\hrrX$ is multiplicative by~\Cref{prop:dya_multiplicative}, we have
    \begin{align}\nonumber
        (\trX_{s_1, s_2; t_1, t_2} \hmult \trX_{s_2, s_3; t_1, t_2} - \trX_{s_1, s_2; t_1, t_2})^{\gr{n+1}} = \sum_{q = 2}^{n+1}(\trx^h_{s_1, s_2; t_1} \gt \trX_{s_2, s_3; t_1, t_2})^{\gr{q}} \cdot \trX^{\gr{n+1-q}}_{s_1, s_2; t_1, t_2},
    \end{align}
    where the sum starts at $q=2$ since the $q=0$ term is cancelled by the $- \trX_{s_1, s_2; t_1, t_2}$ term on the left hand side, and level $q=1$ terms are trivial in $G_1^{\gr{\leq n+1}}$.
    Then, we use the bound in~\Cref{eq:am_prelim_bound} from the proof of almost multiplicativity (\Cref{prop:almost_multiplicative}) to get
    \begin{align}\nonumber
        \| (\trx^h_{s_1, s_2; t_1} \gt \trX_{s_2, s_3; t_1, t_2})^{\gr{q}} \| \leq \frac{1}{\orho\beta^2} \sum_{k=0}^{q-2} 2^{2k\orho} \frac{\ctr{2k\orho}{s_1, s_2}}{\ffact{2k}{\orho}} \|\trX^{\gr{q-k}}_{s_2, s_3; t_1, t_2}\|.
    \end{align}
    Putting the above two equations together, we obtain
    \begin{align}\nonumber
        \| (\trX_{s_1, s_2; t_1, t_2} \hmult \trX_{s_2, s_3; t_1, t_2} - \trX_{s_1, s_2; t_1, t_2})^{\gr{n+1}} \| \leq \frac{2^{2(n+1)\orho}}{\orho\beta^2} \sum_{q = 2}^{n+1} \sum_{k=0}^{q-2}  \frac{\ctr{2k\orho}{s_1, s_2}}{\ffact{2k}{\orho}} \|\trX^{\gr{q-k}}_{s_2, s_3; t_1, t_2}\|\cdot \|\trX^{\gr{n+1-q}}_{s_1, s_2; t_1, t_2}\|.
    \end{align}
    The key observation here is that the degree of $\trX_{s_2, s_3; t_1, t_2}^{\gr{q-k}}$ is always $ q- k \geq 2$. Thus, we can factor out a $\ctr{\orho}{s_2, s_3}$ term from each $\|\trX_{s_2, s_3; t_1, t_2}^{\gr{q-k}}\|$, and bound the remainder using a constant $C > 0$ depending only on $\beta, \rho, n$, and $\omega$,
    \begin{align}\nonumber
        \| (\trX_{s_1, s_2; t_1, t_2} \hmult \trX_{s_2, s_3; t_1, t_2} - \trX_{s_1, s_2; t_1, t_2})^{\gr{n+1}} \| \leq C \ctr{\orho}{s_2, s_3}.
    \end{align}
    Therefore, $\trX^{\gr{n+1}}_{s_1, s_2; t_1, t_2}$ is uniformly continuous in the $s_2$ variable. The other variables can be proven in the same way. 
\end{proof}

\subsubsection{Continuous Extension}
In the previous sections, we have proved that $\trrX$ is a uniformly continuous, multiplicative $\rho$-H\"older double group functional. Because $\Delta^2_\dya$ and $\dya$ are dense in $\Delta^2$ and $[0,1]$ respectively, we can consider the continuous extension of $\trx^h, \trx^v$, and $\trX$ to
\begin{align}\nonumber
    \trx^h : \Delta \times [0,1] \to G_0^{\gr{\leq n+1}}, \quad \trx^v : [0,1] \times \Delta \to G_0^{\gr{\leq n+1}}, \andd \trX: \Delta \times \Delta \to G_1^{\gr{\leq n+1}}
\end{align}
respectively. We will show that this is still a double group function and multiplicative.

\begin{proposition}
    The continuous extension of $\trrX = (\trx^h, \trx^v, \trX)$ is a multiplicative double group function. 
\end{proposition}
\begin{proof}
    We must show that the continuous extension still satisfies the boundary condition in~\Cref{eq:rrX_boundary_condition}, the regularity and continuity conditions, and multiplicative property. All of these properties follow by continuity. In particular, let $(\smp; \tmp) \in \Delta \times \Delta$ and suppose $(\smpi{k}; \tmpi{k}) \in \Delta_\dya \times \Delta_\dya$ such that $(\smpi{k}; \tmpi{k}) \to (\smp; \tmp)$ as $k \to \infty$. Then, since $\delta: T_1^{\gr{\leq n+1}} \to T_0^{\gr{\leq n+1}}$ is a continuous function,
    \begin{align}\nonumber
        \delta(\trX_{\smp; \tmp}) &= \lim_{k \to \infty} \delta(\trX_{\smpi{k}; \tmpi{k}})\\
        & = \lim_{k \to \infty} \trx^h_{\smpi{k}; \tmm^k} \cdot \trx^{v}_{\spp^k; \tmpi{k}} \cdot \trx^{-h}_{\smpi{k}; \tpp^k}  \cdot \trx^{-v}_{\smm^k; \tmpi{k}} \nonumber\\
        & = \trx^h_{\smp; \tmm} \cdot \trx^{v}_{\spp; \tmp} \cdot \trx^{-h}_{\smp; \tpp}  \cdot \trx^{-v}_{\smm; \tmp}.\nonumber
    \end{align}
    Then, for multiplicativity, let $s_1< s_2 < s_3 \in [0,1]$ and $(\tmp) \in \Delta$. Once again, let $(s_1^k < s_2^k < s_3^k) \in \dya$ and $(\tmpi{k}) \in \Delta_\dya$ such that $s_i^k \to s_i$ and $(\tmpi{k}) \to (\tmp)$ as $k \to \infty$. Then, since horizontal multiplication is continuous, we have
    \begin{align}\nonumber
        \trX_{s_1, s_2; \tmp} \hmult \trX_{s_2, s_3; \tmp} & = \lim_{k \to \infty} \trX_{s_1^k, s_2^k; \tmpi{k}} \hmult \trX_{s_2^k, s_3^k; \tmpi{k}} \\
        & = \lim_{k \to \infty} \trX_{s_1^k, s_3^k;\tmpi{k}} \nonumber\\
        & = \trX_{s_1, s_3; \tmp}.  \nonumber
    \end{align}
    The same argument holds for the vertical case. 
\end{proof}

\subsubsection{Uniqueness}
The only remaining step to prove~\Cref{thm:surface_extension} is uniqueness.
\begin{lemma} \label{lem:uniqueness}
    Let $\rho \in (0,1]$ and $n \geq \left\lfloor \frac{2}{\rho} \right\rfloor$. Suppose $\rrX \in \DGF^\rho(\cmG^{\gr{\leq n}})$ is a multiplicative double group functional. Suppose $\trrX = (\trx^h, \trx^v, \trX), \trrY = (\try^h, \try^v, \trY) \in \DGF(\cmG^{\gr{\leq n+1}})$ are multiplicative double group functionals which extend $\rrX$, such that $\trrX^{\gr{k}} = \trrY^{\gr{k}} = \rrX^{\gr{k}}$ for all $k \in [n]$. If $\trrX$ and $\trrY$ satisfy the path regularity condition for $k = n+1$ and
    \begin{align}\nonumber
        \|\trX^{\gr{n+1}}_{s_1, s_2; t_1, t_2}\|, \|\trY^{\gr{n+1}}_{s_1, s_2; t_1, t_2}\| \leq C (\inc{s_1}{s_2} + \inc{t_1}{t_2})^{\theta}
    \end{align}
    for some $C> 0$ and $\theta > 2$, then $\trrX = \trrY$. 
\end{lemma}
\begin{proof}
    Because both $\trrX$ and $\trrY$ satisfy the path regularity condition, we have $\trx^h = \try^h$ and $\trx^v = \try^v$ by the uniqueness of path extensions (\Cref{thm:path_extension}). Now, we consider the surface component. Fix a rectangle $(\smp; \tmp) \in \Delta^2 \times \Delta^2$. Because both $\trrX$ and $\trrY$ are multiplicative, we have
    \begin{align} \label{eq:uniqueness_mult}
        \trrX_{\smp;\tmp} = \gridmult^m_{\smp;\tmp}[\trrX] \andd \trrY_{\smp; \tmp} = \gridmult^m_{\smp;\tmp}[\trrY].
    \end{align}
    Next, we use the intermediate grid multiplication operations from~\Cref{def:intermediate_subdivision}, where we have
    \begin{align}\nonumber
        \gridmult^{m,0}_{\smp; \tmp}[\trrX, \trrY] = \gridmult^m_{\smp; \tmp}[\trrX] \andd \gridmult^{m,4^m}_{\smp; \tmp}[\trrX, \trrY] = \gridmult^m_{\smp; \tmp}[\trrY].
    \end{align}
    Next, note that we have
    \begin{align}\nonumber
        \|\trX^{\gr{n+1}}_{\smp; \tmp} - \trY^{\gr{n+1}}_{\smp;\tmp}\| \leq 2C (\inc{\smm}{\spp} + \inc{\tmm}{\tpp})^{\theta}
    \end{align}
    Then, denoting the vertices of the partition by $(s_i, t_j) \in \cP_m(\smp;\tmp)$, we note that
    \begin{align}\nonumber
        \inc{s_{i-1}}{s_i} + \inc{t_{j-1}}{t_j} = 2^{-m} (\inc{\smm}{\spp} + \inc{\tmm}{\tpp})
    \end{align}
    Then, we obtain
    \begin{align}\nonumber
        \|\gridmult^m_{\smp; \tmp}[\trX] - \gridmult^m_{\smp; \tmp}[\trY]\| &\leq \sum_{u=1}^{4^m} \| \gridmult^{m, u}_{\smp; \tmp}[\trX, \trY] - \gridmult^{m, u-1}_{\smp; \tmp}[\trX, \trY]\| \\
        & \leq \sum_{i=1}^{2^m} \sum_{j=1}^{2^m} \|\trX^{\gr{n+1}}_{s_{i-1}, s_i; t_{j-1}, t_j} - \trY^{\gr{n+1}}_{s_{i-1}, s_i; t_{j-1}, t_j}\|\nonumber\\
        & \leq \sum_{i=1}^{2^m} \sum_{j=1}^{2^m} 2C (\inc{s_{i-1}}{s_i} + \inc{t_{j-1}}{t_j})^{\theta}\nonumber\\
        & \leq 2^{m(2- \theta)}  2C (\inc{\smm}{\spp} + \inc{\tmm}{\tpp})^{\theta}\nonumber
    \end{align}
    Because $\theta > 2$, we have
    \begin{align}\nonumber
        \|\gridmult^m_{\smp; \tmp}[\trX] - \gridmult^m_{\smp; \tmp}[\trY]\| \to 0
    \end{align}
    as $m \to \infty$, and thus by~\Cref{eq:uniqueness_mult}, we have
    \begin{align}\nonumber
        \|\trX_{\smp; \tmp} - \trY_{\smp; \tmp}\| = 0.
    \end{align}
\end{proof}

\subsection{Rough Surfaces}

Motivated by the surface extension theorem, we will define a rough surface in the same way as rough paths.

\begin{definition}
    Let $\rho \in (0,1]$. A \emph{$\rho$-rough surface}\footnote{As this is valued in the group-like elements $\cmG$, this would be the analogue of a \emph{weakly geometric} rough path. However, because we only work with such objects, we will not make this distinction here.} is a $\rho$-H\"older multiplicative double group function $\rrX \in \DGF^\rho(\bT^{\gr{\leq \lfloor2/\rho\rfloor}})$. The space of $\rho$-rough surfaces will be denoted $\RS^\rho$.
\end{definition}

In the remainder of this section, we discuss metrics for the space of rough paths, and show that we can define the surface signature (and surface holonomy) of rough surfaces as a continuous map. \medskip

\subsubsection{Metric Spaces of Rough Surfaces}
We begin by showing that the surface extension is continuous in the following sense, which mimics the statement of~\Cref{thm:path_ext_orig_cont}. The proof largely follows the structure of the surface extension theorem, and we defer this to~\Cref{apxsec:additional_surface_extension}.

\begin{theorem} \label{thm:cont_extension}
    Let $\rrX$ and $\rrY$ be $\rho$-rough surfaces, and let $\trrX$ and $\trrY$ be the unique extension to a double group function on $\bT$. Let $\epsilon > 0$ and suppose for $k \in \lfloor\frac{1}{\orho}\rfloor$, we have
    \begin{align}
        \|\rx^{h, \gr{k}}_{s_1, s_2; t} - \ry^{h, \gr{k}}_{s_1, s_2; t}\| \leq \epsilon \frac{\ctr{2k\orho}{s_1, s_2}}{\beta \ffact{2k}{\orho}}, &\quad \|\rx^{v, \gr{k}}_{s; t_1, t_2} - \ry^{v, \gr{k}}_{s;t_1, t_2}\| \leq \epsilon \frac{\ctr{2k\orho}{t_1, t_2}}{\beta \ffact{2k}{\orho}} \label{eq:ext_cont_path}\\
        \left\|\rx^{h, \gr{k}}_{s_1, s_2; t_2} - \rx^{h, \gr{k}}_{s_1, s_2; t_1} - \ry^{h, \gr{k}}_{s_1, s_2; t_2} - \ry^{h, \gr{k}}_{s_1, s_2; t_1}\right\| &\leq \epsilon \, \frac{\ctr{(2k-1)\orho}{s_1, s_2} \ctr{\orho}{t_1, t_2}}{\beta \ffact{2k-1}{\orho} \ffact{1}{\orho}} \label{eq:ext_cont_path_hmixed} \\
        \left\|\rx^{v, \gr{k}}_{s_2; t_1, t_2} - \rx^{v, \gr{k}}_{s_1; t_1, t_2} - \ry^{v, \gr{k}}_{s_2; t_1, t_2} - \ry^{v, \gr{k}}_{s_1; t_1, t_2}\right\| &\leq \epsilon \, \frac{\ctr{\orho}{s_1, s_2} \ctr{(2k-1)\orho}{t_1, t_2}}{ \beta\ffact{1}{\orho}\ffact{2k-1}{\orho}} \label{eq:ext_cont_path_vmixed}\\
        \andd \|\rX^{\gr{k}}_{s_1, s_2; t_1, t_2} - \rY^{\gr{k}}_{s_1, s_2; t_1, t_2}\| &\leq \epsilon \CTR{k}{\orho}{s_1, s_2; t_1, t_2}.\label{eq:ext_cont_surface}
    \end{align}
    Then, these inequalities hold for the extensions $\trrX$ and $\trrY$ for all $k \in \N$. 
\end{theorem}

\begin{remark}
    The H\"older metric on $C^\rho_0([0,1]^2, \V)$ motivates the continuity conditions used in~\Cref{thm:cont_extension}. In particular, suppose $X, Y \in C^\rho([0,1]^2, \V)$ such that $\|X\|_{\rho}, \|Y\|_{\rho} < L$ and $\|X - Y\|_{\rho} < \epsilon L$. Consider the control $\omega(s_1, s_2) = L^{1/\rho}\inc{s_1}{s_2}$. Then,
    \begin{align}\nonumber
        \|\rx^{h, \gr{1}}_{s_1, s_2; t} - \ry^{h, \gr{1}}_{s_1, s_2;t}\| &= \|X_{s_1, t} - X_{s_2, t} - Y_{s_1, t} + Y_{s_2, t}\| \leq \epsilon L \inc{s_1}{s_2}^{\rho} = \epsilon \ctr{\rho}{s_1, s_2}.
    \end{align}
    Furthermore, by using~\Cref{lem:holder_2d_increment}, we have
    \begin{align}\nonumber
        \|\rx^{h, \gr{1}}_{s_1, s_2; t_2} - \rx^{h, \gr{1}}_{s_1, s_2; t_1} - \ry^{h, \gr{1}}_{s_1, s_2; t_2} + \ry^{h, \gr{1}}_{s_1, s_2; t_1}\| & \leq \|\square_{s_1, s_2; t_1, t_2}[X - Y]\| \leq \epsilon L \inc{s_1}{s_2}^{\rho/2} \inc{t_1}{t_2}^{\rho/2} = \epsilon \ctr{\rho/2}{s_1, s_2} \ctr{\rho/2}{t_1, t_2}.
    \end{align}

\end{remark}

We will now define metrics for $\rho$-H\"older double group functionals. We used the 1D H\"older controls $\ctr{}{}$, the polynomial of H\"older 1D controls $W_\orho$ and included factorials in our definition of $\rho$-H\"older double group functionals in order to conveniently apply the neoclassical inequality (\Cref{eq:neoclassical}) in the proof of the surface extension theorem. In accordance with the standard rough path metrics~\cite{lyons_differential_2007,friz_multidimensional_2010,friz_course_2020}, we will define metrics for functionals in terms of increments. In particular, we will need to introduce the 2D mixed polynomial increments,
\begin{align}\nonumber
    \INC{k}{\orho}{s_1, s_2; t_1, t_2} \coloneqq \sum_{q=1}^{2k-1} \inc{s_1}{s_2}^{q\rho}\, \inc{t_1}{t_2}^{(2k-q)\orho}.
\end{align}

\begin{definition}
Let $\rrX, \rrY \in \DGF^\rho(\cmG^{\gr{\leq n}})$. We begin by defining metrics on the path and surface components individually. In particular,
\begin{align*}
    d^h_\rho(\rx^h, \ry^h) &\coloneqq \max_{k\in [n]} \sup_{\Delta^2 \times [0,1]} \frac{\|\rx^{h, \gr{k}}_{s_1, s_2; t} - \ry^{h, \gr{k}}_{s_1, s_2; t}\|}{\inc{s_1}{s_2}^{k\rho}} + \max_{k\in [n]} \sup_{\Delta^2 \times \Delta^2} \frac{\|\rx^{h, \gr{k}}_{s_1, s_2; t_1} - \rx^{h, \gr{k}}_{s_1, s_2; t_2} - \ry^{h, \gr{k}}_{s_1, s_2; t_1} + \ry^{h, \gr{k}}_{s_1, s_2; t_2}\|}{\inc{s_1}{s_2}^{(2k-1)\orho} \inc{t_1}{t_2}^{\orho}} \\
    d^v_\rho(\rx^v, \ry^v) &\coloneqq \max_{k\in [n]} \sup_{\Delta^2 \times [0,1]} \frac{\|\rx^{v, \gr{k}}_{s; t_1, t_2} - \ry^{v, \gr{k}}_{s; t_1, t_2}\|}{\inc{t_1}{t_2}^{k\rho}} + \max_{k\in [n]} \sup_{\Delta \times \Delta^2 } \frac{\|\rx^{v, \gr{k}}_{s_1; t_1, t_2} - \rx^{v, \gr{k}}_{s_2; t_1, t_2} - \ry^{v, \gr{k}}_{s_1; t_1, t_2} + \ry^{v, \gr{k}}_{s_2; t_1, t_2}\|}{\inc{s_1}{s_2}^{(2k-1)\orho} \inc{t_1}{t_2}^{\orho}} \\
    d^s_\rho(\rX, \rY) & \coloneqq \max_{k\in [n]} \sup_{\Delta^2 \times \Delta^2 } \frac{\|\rX^{\gr{k}}_{s_1, s_2; t_1, t_2} - \rY^{\gr{k}}_{s_1, s_2; t_1, t_2}\|}{W^k_{\orho}(s_1, s_2; t_1, t_2)}.
\end{align*}
Then, we \emph{$\rho$-H\"older metric} to be
\begin{align}\nonumber
    d_\rho(\rrX, \rrY) \coloneqq d^h(\rx^h, \ry^h) + d^v(\rx^v, \ry^v) + d^s(\rX, \rY).
\end{align}
\end{definition}
The following result is proved using standard methods and is deferred to~\Cref{apxsec:additional_surface_extension}.

\begin{proposition} \label{prop:rough_surface_complete}
    The metric space $(\RS^\rho, d_\rho)$ is complete.
\end{proposition}

\subsubsection{Surface Holonomy of Rough Surfaces}

Finally, we will show that we can compute surface holonomy of rough surfaces by defining it via the universal property of the surface signature in~\Cref{thm:surface_signature_universal}. 

\begin{corollary} \label{cor:surface_sig_rs_cont}
    Let $\rho \in (0,1]$. We define the \emph{surface signature} $\hssig: \RS^\rho \to G_1$ of a rough surface by
    \begin{align}\nonumber
        \hssig(\rrX) \coloneqq \trrX_{0,1;0,1},
    \end{align}
    where $\trrX = (\trx^h, \trx^v, \trX) \in \DGF^\rho(\cmG)$ is the unique extension to $\cmG$ given by~\Cref{thm:surface_extension}. Furthermore, $\hssig$ is continuous.
\end{corollary}
\begin{proof}
    First, we show that the surface signature is indeed bounded, and thus valued in $E_1$ from~\Cref{eq:E0_characterization}. Let $\rrX \in \RS^\rho$, and by applying the neoclassical inequality to the surface regularity given by the extension theorem as is done in~\Cref{eq:surface_regularity_neoclassical}, we have
    \begin{align}\nonumber
        \left\|\hssig^{\gr{k}}(\rrX) \right\| \leq \frac{1}{\orho \beta} \frac{2^{k\rho}(\ctr{}{s_1, s_2} + \ctr{}{t_1, t_1})^{k\rho}}{\ffact{k}{\rho}}.
    \end{align}
    Then, for any $\lambda > 0$, we have
    \begin{align}\nonumber
        P_\lambda(\hssig(\rrX)) \leq 1 + \frac{1}{\orho \beta}\sum_{k=0}^\infty \frac{\lambda^k 2^{k\rho}(\ctr{}{s_1, s_2} + \ctr{}{t_1, t_1})^{k\rho}}{\ffact{k}{\rho}} < \infty,
    \end{align}
    so $\hssig(\rrX) \in E_1$. Then, continuity is immediate from~\Cref{thm:cont_extension}.
\end{proof}

Finally, we can use the surface signature of rough surfaces to define surface holonomy for rough surfaces, where we note that this coincides with surface holonomy of smooth surfaces by~\Cref{thm:surface_signature_universal}.

\begin{theorem} \label{thm:surface_holonomy_rough}
    Let $\rho \in (0,1]$. Let $(\cona, \conc)$ be a continuous $2$-connection valued in a crossed module of Banach algebras $\cmA$. Then, we define the \emph{surface holonomy with respect to $(\cona, \conc)$} for rough surfaces $\hH^{\cona, \conc}: \RS^\rho \to A_1$ by
    \begin{align}\nonumber
        \hH^{\cona, \conc}(\rrX) \coloneqq \tconc \left( \hssig(\rrX)\right),
    \end{align}
    where $(\tcona, \tconc): \bE \to \cmA$ is the unique continous morphism from~\Cref{prop:E0_con_universal}. The map $\hH^{\cona, \conc}$ is continuous. 
\end{theorem}
\begin{proof}
    Because both the surface signature (\Cref{cor:surface_sig_rs_cont}) and $\tconc$ from the universal property in~\Cref{prop:E0_con_universal} are continuous, the composition is continuous. 
\end{proof}

\section{Conclusion and Outlook}
In this article, we have built upon the work of Kapranov~\cite{kapranov_membranes_2015} to explicitly define the surface signature for smooth surfaces in terms of a free crossed module of associative algebras, which preserves the horizontal and vertical concatenation structures of the space of surfaces. We show that the surface signature is the universal surface holonomy map: surface holonomy with respect to any continuous 2-connections factors through the signature. Going beyond the smooth setting, we introduce the notion of rough surfaces, and prove an extension theorem for rough surfaces, which allows us to define the surface signature (and surface holonomy with respect to continuous 2-connections) for rough surfaces. This work opens up an avenue towards higher dimensional rough analysis, and we highlight some immediate questions which stem from this article. \medskip

\begin{itemize}
    \item \textbf{Sewing with Rectangular Increments.} As discussed in~\Cref{rem:sewing_with_rectangular_regularity}, additional regularity assumptions on rectangular increments are not explicitly used in the surface extension theorem (\Cref{thm:surface_extension}). Is there an analogue of the extension theorem which takes this into account, and allows us to define a unique lift of a surface $X \in C^{\square, \rho}([0,1]^2, \V)$ in the Young regime ($\rho \in (\frac{1}{2},1]$)?
    \item \textbf{Surfaces in Banach Spaces.} This article focuses on the case of surfaces valued in a finite dimensional vector space $\V$. A natural generalization is to consider the analogous constructions for surfaces valued in a Banach space $\V$ (like for rough paths). There are two main places where we explicitly use the finite dimensional or Hilbert structure. First, we use the projection onto the orthogonal complement to define norms for $T_1(\V)$ in~\Cref{sssec:norms}; however, this can be easily rectified by simply considering quotient norms in a Banach space. Second, we use finite-dimensionality to construct an explicit section which induces a short algebra section $\ts$ in~\Cref{apxsec:algebra_section}. While one approach is to consider this construction in the Banach space setting, another is to avoid the algebra section altogether. In principle, one can still perform sewing for Banach valued surfaces though without boundedness guarantees by following~\Cref{rem:almost_mult_algebra_section}.
    \item \textbf{Universality and Characteristicness.} One aspect we did not discuss is the analogue of the shuffle product. From~\Cref{rem:universal_enveloping_group_like}, we see that $T_1(\V)$ is not the universal enveloping algebra of $\fg_1^{\sab}$, but one can embed our notion of group-like elements $G_1$ into the group-like elements of $U(\fg_1^{\sab})$; this amounts to appending all polynomials of elements in the kernel $\fg_1^{\sab, \ker}$. This would provide a Hopf algebra structure, and the analogue of the shuffle can be derived from this. Then, the notions of analytic universality and characteristicness should follow. Note that an analogous modification is considered in~\cite[Section 9]{lee_random_2023} to obtain an algebra structure to prove an analytic universality result, where complex exponentials were used rather than polynomials. 
    \item \textbf{Characterization of Thin Homotopy Classes.} It is a well-known property that the path signature characterizes (rough) paths up to tree-like equivalence~\cite{chen_integration_1958,hambly_uniqueness_2010,boedihardjo_signature_2016}. Surface holonomy enjoys invariance with respect to thin homotopy (see~\Cref{rem:thin_homotopy}). A natural question is whether the surface signature can characterize (rough) surfaces up to thin homotopy equivalence? Note that this question was already asked by Kapranov in the smooth setting~\cite[Question 2.5.6]{kapranov_membranes_2015}.
    \item \textbf{Rough Integration of Differential Forms.} An immediate consequence of rough paths theory is the deterministic rough integration theory. What is the analogue of this for rough surfaces? One direction is the extension of the Z\"ust integral~\cite{zust_integration_2011}, in line with~\cite{stepanov_towards_2021,alberti_integration_2023} and the recent paper~\cite{chandra_rough_2024}. In particular, given a smooth 2-form $\omega$ on $\V$, can we make sense of
    \begin{align}
        \int \rrX^*\omega,
    \end{align}
    where $\rrX^*\omega$ denotes some notion of a pullback along a rough surface $\rrX$?
    \item \textbf{Applications in Machine Learning.} Signature methods are a powerful new class of tools for sequential data in machine learning~\cite{chevyrev_primer_2016,lyons_signature_2022}, and higher dimensional signatures have already been applied to image and texture classification~\cite{zhang_two-dimensional_2022}. In particular, can we develop a surface signature kernel in the same way as the path signature kernel~\cite{kiraly_kernels_2019,salvi_signature_2021,lee_signature_2023}? Can we use surface holonomy in matrix groups~\cite{lee_random_2023} (see~\Cref{apxsec:matrix_sh}) for image data in the same way that path development has been used for time series~\cite{lou_path_2022,lou_pcf-gan_2023-1}?
    \item \textbf{Towards Higher Dimensions.} Kapranov~\cite{kapranov_membranes_2015} also considers universal higher holonomy for higher dimensional maps, and one can consider an extension of the results of this paper in the higher dimensional setting.
\end{itemize}
\clearpage

\appendix

\section{Categorical Structure of Paths and Surfaces} \label{apxsec:categorical_structure}

\subsection{Paths}
We begin by considering the space of all piecewise smooth paths valued in a finite dimensional vector space $\V$ parametrized over all intervals in $\R$, defined by
\begin{align}\nonumber
    \paths(\V) \coloneqq \bigcup_{s_1 \leq s_2} C^\infty([s_1, s_2], \V).
\end{align}
This space of parametrized paths is equipped with a partially defined concatenation operation.
Given two paths $x \in C^\infty([s_1, s_2], \V)$ and $y \in C^\infty([s_2, s_3], \V)$ such that $x_{s_2} = y_{s_2}$, we define the concatenation
\begin{align} \label{eq:path_concat}
    x \concat y \in C^\infty([s_1, s_3], \V) \quad \text{by} \quad (x \concat y)_s \coloneqq \left\{
        \begin{array}{ll}
            x_{s} & : s \in [s_1, s_2] \\
            y_{s} & : s \in [s_2, s_3].
        \end{array}
    \right.
\end{align}
The space of parametrized paths fits into the structure of a \emph{category}.
\begin{definition}
    A \emph{category} $\sC$ consists of 
    \begin{itemize}
        \item a set of \emph{objects} $\sC_0$, and 
        \item a set of \emph{morphisms} (or directed edges) $\sC_1$, equipped with \emph{source} and \emph{target} maps $\bdy_s, \bdy_t: \sC_1 \to \sC_0$. Given a morphism $\bx \in \sC_1$ with source $\bdy_s \bx = a$ and target $\bdy_t \bx = b$, we write $\bx: a \to b$.
    \end{itemize}
    These sets are equipped with additional structure. 
    \begin{itemize}
        \item \textbf{Composition.} For any $a,b,c \in \sC_0$ and morphisms $\bx,\by \in \sC_1$ such that $\bdy_t \bx = \bdy_s \by$, there exists an associative composition operator $\bx \concat \by$ with $\bdy_s (\bx \concat \by) = \bdy_s \bx$ and $\bdy_t (\bx \concat \by) = \bdy_t \by$. 
        \item \textbf{Identity Edges.} For any $a \in \sC_0$, there exists an identity edge $1_a$ which acts as a unit under composition: for $\bx: a \to b$ and $\by : b \to a$ we have $1_a \concat \bx = \bx$ and $\by \concat 1_b = \by$.
    \end{itemize}
\end{definition}

The \emph{category of parametrized smooth paths} $\Sigma^\infty$ is defined by objects and edges
\begin{align}\nonumber
    \Sigma^\infty_0 \coloneqq \R \times \V \andd \Sigma^\infty_1 \coloneqq \paths(\V)
\end{align}
respectively. An object $(s,v) \in \Sigma^\infty_0$ consists of a point $v \in \V$ parametrized by $s \in \R$. The source and target maps are defined for a path $x \in C^\infty([s_1, s_2], \V)$ by
\begin{align}\nonumber
    \bdy_s(x) = (s_1, x_{s_1}) \andd \bdy_t(x) = (s_2, x_{s_2}).
\end{align}
Now, we note that in the definition of parametrized paths in~\Cref{eq:paths}, we allow for paths defined on $[s,s] = \{s\}$, which is simply a parametrized point. These degenerate paths are the identity edges in the category $\Sigma^\infty$.

\subsection{Surfaces}
Next, we consider the space of all piecewise smooth surfaces valued in a finite dimensional vector space $\V$ parametrized over all rectangular domains in $\R^2$, denoted by
\begin{align} \nonumber
    \surfaces(\V) \coloneqq \bigcup_{s_1 \leq s_2, \, t_1 \leq t_2} C^\infty([s_1, s_2] \times [t_1, t_2], \V).
\end{align}
This space of surfaces is equipped with a partially defined horizontal concatenation operation. If $X \in C^\infty([s_1, s_2] \times [t_1, t_2], \V)$ and $Y \in C^\infty([s_2, s_3] \times [t_1, t_2], \V)$ such that $X_{s_2, t} = Y_{s_2, t}$ for all $t \in [t_1, t_2]$, we say that $X$ and $Y$ are \emph{horizontally composable} and we define 
\begin{align} \label{eq:surf_hconcat}
    X \concat_h Y \in C^\infty([s_1, s_3]\times[t_1, t_2], \V) \quad \text{by} \quad (X \concat_h Y)_{s,t} \coloneqq \left\{
        \begin{array}{ll}
            X_{s,t} & : s \in [s_1, s_2] \\
            Y_{s,t} & : s \in [s_2, s_2].
        \end{array}
    \right.
\end{align}
Similarly, if $Z \in C^\infty([s_1, s_2] \times[t_2, t_3])$ such that $X_{s, t_2} = Z_{s, t_2}$ for all $s \in [s_1, s_2]$, we say that $X$ and $Z$ are \emph{vertically composable} and we define
\begin{align}\label{eq:surf_vconcat}
    X \concat_v Z \in C^\infty([s_1, s_2] \times[t_1, t_3], \V) \quad \text{by} \quad (X \concat_v Z)_{s,t} \coloneqq \left\{
        \begin{array}{ll}
            X_{s,t} & : t \in [t_1, t_2] \\
            Z_{s,t} & : t \in [t_2, t_3].
        \end{array}
    \right.
\end{align}
Both horizontal and vertical concatenation are associative by definition when the operations are well-defined. Furthermore, these two operations are compatible via the \emph{interchange law}: for $X, Y, Z, W \in \surfaces(\V)$ which are appropriately composable, in particular with $X,Y,Z$ defined as above, and $W \in C^\infty([s_2, s_3] \times [t_2, t_3], \V)$ such that $W_{s, t_2} = Y_{s,t_2}$ for all $s \in [s_2, s_3]$ and $W_{s_2,t} = Z_{s_2, t}$ for all $t \in [t_2, t_3]$, 
\begin{align}\nonumber
    (X \concat_h Y ) \concat_v (Z \concat_h W) = (X \concat_v Z) \concat_h (Y \concat_v W).
\end{align}

In fact, this structure can be encoded as a double category. 

\begin{definition} \label{def:double_cat}
    A \emph{strict double category} $\bsC$ consists of
    \begin{itemize}
        \item \textbf{a set of objects} $\bsC_0$,
        \item \textbf{a set of horizontal edges} $\bsC_{1,h}$, equipped with \emph{source} and \emph{target} maps $\bdy^h_s, \bdy^h_t : \bsC_{1,h} \to \bsC_0$, and
        \item \textbf{a set of veritcal edges} $\bsC_{1,v}$, equipped with \emph{source} and \emph{target} maps $\bdy^v_s, \bdy^v_t : \bsC_{1,v} \to \bsC_0$, and
        \item \textbf{a set of squares} $\bsC_2$, equipped with left and right boundary maps $\bdy_l, \bdy_r: \bsC_2 \to \bsC_{1,v}$ and bottom and upper boundary maps $\bdy_b, \bdy_u : \bsC_2 \to \bsC_{1,h}$, which satisfy the compatability relations in the figure below.
    \end{itemize}
    The objects, edges and squares satisfy the following conditions.
    \begin{enumerate}
    \item The objects $\bsC_0$ and horizontal edges $\bsC_{1,h}$ have the structure of a category.
    \item The objects $\bsC_0$ and vertical edges $\bsC_{1,h}$ have the structure of a category.
    \item The squares in $\bsC_2$ are equipped with two partial compositions.
    \begin{itemize}
        \item \textbf{Horizontal Composition.} For any $X, X' \in \bsC_2$ such that $\bdy_r X = \bdy_l X'$, there exists a composite $X \concat_h X'$.
        \item \textbf{Vertical Composition.} For any $X,X' \in \bsC_2$ such that $\bdy_u X = \bdy_b X'$, there exists a composite $X \concat_v X'$.
        
        \item \textbf{Interchange Law.} For any $X,Y,Z,W \in \bsC_2$, we have
        \begin{align}\nonumber
            (X \concat_h Y) \concat_v (W \concat_h Z) = (X \concat_v W) \concat_h (Y \concat_v Z),
        \end{align}
        whenever these compositions are well-defined.
    \end{itemize}
    \item Finally, there exist units for both compositions.
    \begin{itemize}
        \item \textbf{Identity Squares.} For any $x \in \bsC_{1,v}$, there exists a \emph{horizontal identity} $1^h_x \in \bsC_2$ and for any $y \in \bsC_{1,h}$, there exists a \emph{vertical identity} $1^v_x \in \bsC_2$ which acts as a unit under composition for horizontal and vertical composition respectively. In particular, for $X,Y \in \bsC_2$ such that $\bdy_l X = \bdy_r Y = x$, we have $1^h_x \concat_h X = X$ and $Y \concat_h 1^h_x = Y$, and same for vertical composition.
    \end{itemize}
\end{enumerate}
\end{definition}

In particular, we define the \emph{double category of parametrized smooth surfaces} $\bSigma^\infty$ with objects, horizontal edges, vertical edges, and squares by
\begin{align}\nonumber
    \bSigma^\infty_0 \coloneqq  \R^2 \times \V , \quad \bSigma^\infty_{1,h} \coloneqq \R \times \paths(\V),\quad \bSigma^\infty_{1,v} \coloneqq \R \times \paths(\V), \andd \bSigma^\infty_2 \coloneqq \surfaces(\V)
\end{align}
respectively.
An object is a parametrized point $(s,t,v) \in \bSigma^\infty_0$ consisting of a point $v \in \V$ in the state space along with its underlying parametrization $(s,t) \in \R^2$.
A horizontal path $(t, x) \in \bSigma^\infty_{1,h}$ consists of a path $x \in C^\infty([s_1, s_2], \V)$, and a number $t \in \R$ which specifies the horizontal line $[s_1, s_2] \times \{t\} \subset \R^2$ on which $x$ is parametrized. Similarly, a vertical path $(s,y) \in \bSigma^\infty_{1,v}$ consists of a path $y \in C^\infty([t_1, t_2], \V)$ and a number $s \in \R$ which specifies the vertical line $\{s\} \times [t_1, t_2]$ on which $y$ is parametrized. We define the horizontal and vertical source and target maps by 
\begin{align}\nonumber
    \bdy^h_s(t,x) \coloneqq (s_1, t, x_{s_1}), \quad \bdy^h_t(x) \coloneqq (s_2, t, x_{s_2}), \quad \bdy^v_s(s,y) \coloneqq (s, t_1, y_{t_1}) \andd \bdy^v_t(s,y) \coloneqq (s,t_2, y_{t_2}).
\end{align}
Furthermore, for a surface $X \in C^\infty([s_1, s_2] \times [t_1, t_2], \V)$, we define the boundary maps by
\begin{align}\nonumber
    \bdy_l(X) \coloneqq (s_1, X_{s_1, \cdot}), \quad \bdy_r(X) \coloneqq (s_2, X_{s_2, \cdot}), \quad \bdy_b(X) \coloneqq (t_1, X_{\cdot, t_1}), \andd \bdy_u(X) \coloneqq (t_2, X_{\cdot, t_2}),
\end{align}
where $\bdy_l(X), \bdy_r(X) \in \bSigma^\infty_{1,v}$ and $\bdy_b(X), \bdy_u(X) \in \bSigma^\infty_{1,h}$. We define composition of horizontal paths
\begin{align}\nonumber
    (t, x^1) \in \R \times C^\infty([s_1, s_2], \V) \andd (t, x^2) \in \R \times C^\infty([s_2, s_3], \V) \quad \text{by} \quad (t, x^1) \concat (t, x^2) \coloneqq (t, x^1 \concat x^2),
\end{align}
where $x^1_{s_2} = x^2_{s_2}$, and path concatenation is defined as in~\Cref{eq:path_concat}. Similarly, we define composition of vertical paths 
\begin{align}\nonumber
    (s, y^1) \in \R \times C^\infty([t_1, t_2], \V) \andd (s, y^2) \in \R \times C^\infty([t_2, t_3], \V) \quad \text{by} \quad (s, y^1) \concat (s, y^2) \coloneqq (s, y^1 \concat y^2),
\end{align}
where $y^1_{t_2} = y^2_{t_2}$. Furthermore, horizontal and vertical composition for surfaces in $\bSigma^\infty_2$ is defined by~\Cref{eq:surf_hconcat} and~\Cref{eq:surf_vconcat}. We note that the conditions for both path and surface composability are exactly those specified by source, target and boundary maps in~\Cref{def:double_cat}. Finally, similar to the case of paths, our definition of parametrized surfaces in~\Cref{eq:surface} allows for degenerate surfaces parametrized on $[s_1, s_2] \times \{t\}$ and $\{s\} \times [t_1, t_2]$, and these degenerate surfaces are the identity squares in $\bSigma^\infty$, and thus this is a double category. Next, our desired notion of a structure-preserving map for surfaces is a functor between double categories.

\begin{definition}
    Let $\bsC$ and $\bsD$ be double categories. A \emph{double functor} $\bF: \bsC \to \bsD$ consists of maps
    \begin{align}\nonumber
        \bF_0: \bsC_0 \to \bsD_0, \quad \bF_{1,h}: \bsC_{1,h} \to \bsD_{1,h}, \quad \bF_{1,v}: \bsC_{1,v} \to \bsD_{1,v}, \andd \bF_2 : \bsC_2 \to \bsD_2,
    \end{align}
    where preserve the source, target, and boundary maps for $\bsC$ and $\bsD$. Furthermore, for appropriately composable $x^1, x^2 \in \bsC_{1,h}$ and $y^1, y^2 \in \bsC_{1,v}$, we have
    \begin{align}\nonumber
        \bF_{1,h}(x^1 \concat x^2) = \bF_{1,h}(x^1) \concat \bF_{1,h}(x^2) \andd \bF_{1,v}(y^1 \concat y^2) = \bF_{1,v}(y^1) \concat \bF_{1,v}(y^2),
    \end{align}
    and for appropriately composable $X,Y, Z \in \bsC_2$, we have
    \begin{align}\nonumber
        \bF_2(X \concat_h Y) = \bF_2(X) \concat_h \bF_2(Y) \andd \bF_2(X \concat_v Z) = \bF_2(X) \concat_v \bF_2(Z).
    \end{align}
\end{definition}

The following is a more formal definition of a double group.
\begin{definition}
    A \emph{double group} is a double groupoid with a single object. Given a crossed module of groups $\cmG = (\delta: G_1 \to G_0, \gt)$, the \emph{double group associated to $\cmG$}, denoted $\dg(\cmG)$, has a single object $\dg_0(\cmG)$, edges defined by $\dg_1(\cmG) = G_0$, and squares defined by
    \begin{align}\nonumber
        \dg_2(\cmG) \coloneqq \{(x,y,z,w, E) \in G_0^4 \times G_1 \, : \, \delta(E) = x\cdot y \cdot z^{-1} \cdot w^{-1}\}.
    \end{align}
\end{definition}
Thus, in the main text, the $\dg(\cmG)$ is considered to be the \emph{squares} of the double group.

\section{Explicit Section of \texorpdfstring{$[\fg_0, \fg_0]$}{[g_0,g_0]}.} \label{apxsec:algebra_section}
Here, we describe an explicit linear section of Lie generators $s: L \to \fg_1^{\sab}$, and study the properties of its induced algebra section $\ts: U([\fg_0, \fg_0]) \to T_1$. Recall that $\fg_0$ is the free Lie algebra over a finite dimensional Hilbert space $\V$. Let $e_1, \ldots, e_d$ denote an orthonormal basis of $\V$. By the Shirshov-Witt theorem~\cite[Theorem 2.5]{reutenauer_free_1993}, $[\fg_0, \fg_0]$ is free. We begin by providing an explicit set of Lie generators $L$ for $[\fg_0, \fg_0]$ based on the Hall basis of $\fg_0$ which is compatible with the derived series of $\fg_0$~\cite[Section 5.3]{reutenauer_free_1993}. \medskip

\textbf{Hall Basis for Derived Series.} We recall the construction from~\cite[Section 5.3]{reutenauer_free_1993}. Let $A$ be a totally ordered set, and let $\fh = \FL(A)$. We define a sequence of Lie ideals of $\fh$, called the \emph{derived series}, by
\begin{align}\nonumber
    D_0(\fh) \coloneqq \fh \andd D_{k+1}(\fh) \coloneqq [D_{k}(\fh), D_{k}(\fh)].
\end{align}
We define a linear basis of $\fh$ as follows. Define the totally ordered set $H_0(A) = A \subset \fh$. Next, we define recursively $H_{n+1}(A)$ as the set of elements
\begin{align}\nonumber
    a = [a_1, \ldots, [a_{m-2}, [a_{m-1}, a_m]] \ldots] \in \fh
\end{align}
where $m \geq 2$ and where $a_1, \ldots, a_m \in H_n(A)$ such that
\begin{align}\nonumber
    a_1 \leq \ldots \leq a_{m-2} \leq a_{m-1} > a_m.
\end{align}
We order $H_{n+1}(A)$ totally, and let $H(A) = \bigcup_{n \geq 0} H_n(A)$ with the total ordering $H_n(A) > H_{n+1}(A)$.

\begin{theorem}{\cite[Theorem 5.7]{reutenauer_free_1993}}
    For each $n \geq 0$, $\bigcup_{n \geq k} H_n(A)$ is a basis of $D_{k}(\fh)$.
\end{theorem}

\textbf{Lie Generators of $[\fg_0, \fg_0]$.} Now, we will use this basis to obtain an explicit set of Lie generators for $[\fg_0, \fg_0]$. 

\begin{proposition} \label{prop:commutator_lie_generator}
    Let $\V$ be a finite dimensional Hilbert space, and let $A = \{e_1 \leq \ldots \leq e_d\}$ be a totally ordered basis for $\V$. Let $H_1(A)$ be defined as above. Then, $[\fg_0, \fg_0] \cong \FL(H_1(A))$.
\end{proposition}
\begin{proof}
    The proof of this statement will follow the proof of the Shirshov-Witt theorem in~\cite[Theorem 2.5]{reutenauer_free_1993}. For each $n$, we define the subspace
    \begin{align}\nonumber
        E^n \coloneqq \{ P \in [\fg_0, \fg_0] \, : \, \deg(P) \leq n\}.
    \end{align}
     Then, we let $\langle E \rangle$ denote the subalgebra generated by $E\subset \fg_0$, and let $F^n \subset E^n$ be defined by
     \begin{align}\nonumber
        F^n \coloneqq E^n \cap \langle E^{n-1} \rangle,
     \end{align}
     the elements in $E^n$ which are generated by elements in $E^{n-1}$. Let $H_1^n$ denote the degree $n$ elements in $H_1(A)$. We note that $D_1(\fg_0) = [\fg_0, \fg_0]$, and note that $\langle E^{n} \rangle \subset D_2(\fg_0)$ for any $n$. Therefore, $H^n_1(A)$ is a basis for $E^n/F^n$, the subspace of elements in $E^{n}$, which \emph{are not} generated by elements in $E^{n-1}$. Furthermore, $H_1(A) = \bigcup_{n \geq 1} H_1^n(A)$, and by the proof of~\cite[Theorem 2.5]{reutenauer_free_1993}, $[\fg_0, \fg_0]$ is freely generated by $H_1(A)$.
\end{proof}

\textbf{Defining the Linear Section.}
From~\Cref{prop:commutator_lie_generator}, the Lie generators $L = H_1(A)$ are of the form
\begin{align}\nonumber
    [e_1, \ldots, [e_{m-2}, [e_m, e_{m-1}]] \ldots],
\end{align}
with $e_1 \leq \ldots \leq e_{m-2} \leq e_{m-1} > e_m$. Note that we have flipped the innermost elements, so that $e_m < e_{m-1}$. Now, we define the linear section $s: L \to \fg_1^{\sab} \subset T_1$ by
\begin{align} \label{eq:short_linear_section}
    s\Big([e_1, \ldots, [e_{m-2}, [e_m, e_{m-1}]] \ldots]\Big) \coloneqq \Big([e_1, \ldots, [e_{m-2}, e_{m, m-1}]_1 \ldots]_1\Big)^{\perp},
\end{align}
where $[\cdot, \cdot]_1$ denotes the commutator bracket in $T_1$ and $e_{i,j} \coloneqq e_i \wedge e_j \in \Lambda^2 \V$ for $i < j$ forms a basis of $\Lambda^2 \V$. Now, we embed $L$ into $T_0$ via the commutator bracket for $T_0$, and note that we have
\begin{align}\nonumber
    \Big\| [e_1, \ldots, [e_{m-2}, e_{m, m-1}]_1 \ldots]_1\Big\|_{T_1} \leq \Big\| [e_1, \ldots, [e_{m-2}, e_{m, m-1}]_1 \ldots]_1\Big\|_{\oT_1} < \Big\| [e_1, \ldots, [e_{m-2}, [e_m, e_{m-1}]_0]_0 \ldots ]_0\Big\|_{T_0}
\end{align}
Therefore, $\|s(\ell)\|_{T_1} \leq \|\ell\|_{T_0}$ for all $\ell \in L$. \medskip

Now, we consider the algebra section $\ts: U([\fg_0, \fg_0])\subset T_0 \to T_1$ induced by $s$. In particular, using the fact that $T_1$ is a Banach algebra, and the bound on generators, we have
\begin{align}\nonumber
    \Big\|\ts(\ell_1 \cdots \ell_n)\Big\|_{T_1} = \Big\|\ell_1 \cdots \ell_{n-1} \cdot \ts(\ell_n)\Big\|_{T_1} \leq \Big\|\ell_1 \cdots \ell_{n-1}\Big\|_{T_0} \cdot \Big\|\ts(\ell_n)\Big\|_{\oT_1} < \Big\| \ell_1 \cdots \ell_n\Big\|_{T_0}.
\end{align}
Thus, we have proved the following.

\begin{lemma} \label{lem:short_linear_section}
    The linear section $s: L \to \fg_1^{\sab} \subset T_1$ defined in~\Cref{eq:short_linear_section} induces an algebra section $\ts: U([\fg_0, \fg_0]) \subset T_0 \to T_1$ such that
    \begin{align}\nonumber
        \|\ts(\ell)\|_{T_1} \leq \|\ell\|_{T_0}
    \end{align}
    for all $\ell \in U([\fg_0, \fg_0])$.
\end{lemma}

\section{Matrix Surface Holonomy} \label{apxsec:matrix_sh}
We begin with the example of surface holonomy valued in matrix double groups, which is analogous to the case of path holonomy valued in the classical matrix groups $\GL^n$. This was previous studied in detail in~\cite{lee_random_2023}; however, we will present this from the perspective of crossed modules of algebras, rather than Lie groups and Lie algebras. \medskip

\subsection{Crossed Module of Chain Maps and Homotopies}
A \emph{(Baez-Crans) 2-vector space}~\cite{baez_higher-dimensional_algebras_2004} is a 2-term chain complex; in other words, it consists of two vector spaces $W_0$ and $W_1$, and a linear map (the \emph{boundary map}) $\phi: W_1 \to W_0$. Given dimensions $n,m,p \geq 0$, we define the 2-vector space $\cW^{n,m,p}$ with $W_1 = \R^{n+p}$ and $W_0 = \R^{n+m}$, 
\begin{align}\nonumber
    \cW^{n,m,p} = (W_1 = \R^{n+p} \xrightarrow{\phi} W_0 =\R^{n+m}), \quad \text{where} \quad \phi = \pmat{I_n & 0 \\ 0 & 0},
\end{align}
and $I_n \in \Mat^n$ is the $n\times n$ identity matrix.\medskip

We will define a crossed module of algebras made up of chain maps and chain homotopies of $\cW^{n,m,p}$, and it is convenient to express these as linear automorphisms of the direct sum of $W_0$ and $W_1$, which we also denote by $\cW^{n,m,p} =  W_0  \oplus W_1 = \R^{n+m} \oplus \R^{n+p}$. Here, the boundary map is expressed as $\Phi : \cW^{n,m,p} \to \cW^{n,m,p}$ by
\begin{align}\nonumber
    \Phi = \pmat{0 & \phi \\ 0 & 0}.
\end{align}

First, we define the algebra of chain maps $F = (f,g) \in \Ch_0^{n,m,p}$, which consist of linear maps $f \in L(W_1, W_1)$ and $g \in L(W_0, W_0)$, arranged in a single matrix $F \in L(\cW^{n,m,p}, \cW^{n,m,p})$ as
\begin{align}\nonumber
    F = \pmat{g & 0 \\ 0 & f}
\end{align}
such that $\Phi F = F \Phi$ or equivalently $\phi f = g \phi$. Then, $\Ch_0^{n,m,p}$ simply the matrix multiplication structure (denoted $\cdot$), and the identity matrix is the unit. \medskip

Next, we consider the algebra of chain homotopies $H \in \Ch_1^{n,m,p}$, which consists of a linear map $h \in L(W_0, W_1)$, also represented as a matrix $H \in L(\cW^{n,m,p}, \cW^{n,m,p})$ as
\begin{align}\nonumber
    H = \pmat{ 0 & 0 \\ h & 0}.
\end{align}
We define a linear map $d: \Ch_1^{n,m,p} \to \Ch_0^{n,m,p}$ by
\begin{align}\nonumber
    d(H) = \Phi H + H \Phi = \pmat{0 & \phi \\ 0 & 0} \pmat{ 0 & 0 \\ h & 0 }  + \pmat{ 0 & 0 \\ h & 0 }\pmat{0 & \phi \\ 0 & 0} = \pmat{ \phi h & 0 \\ 0 &h \phi}.
\end{align}
We define the algebra structure on $\Ch_1^{n,m,p}$ by
\begin{align}\nonumber
    H * H' \coloneqq d(H) \cdot H' = H \cdot d(H') = H \cdot \Phi \cdot H' = \pmat{ 0 & 0 \\ h\phi h' & 0},
\end{align}
and therefore $d(H * H') = d(H) \cdot d(H')$, and $d$ is a morphism of algebras. \medskip

Finally, we define left and right actions, $\gtd$ and $\ltd$ of $\Ch_0^{n,m,p}$ on $\Ch_1^{n,m,p}$ by matrix multiplication,
\begin{align}\nonumber
    F \gtd H = F \cdot H = \pmat{0 & 0 \\ f\cdot h & 0 } \andd H \ltd F = H \cdot F = \pmat{0 & 0 \\ h \cdot g & 0}.
\end{align}

\begin{proposition}
    The structure $\cmCh^{n,m,p} = (d: \Ch_1^{n,m,p} \to \Ch_0^{n,m,p}, \gtd, \ltd)$ is a crossed module of algebras. 
\end{proposition}

\subsection{Matrix Surface Holonomy}
Here, we will fix a 2-vector space $\cW^{n,m,p}$ and supress the dimensions from the notation.
We begin by considering 2-connections valued in $\cmCh$. This consists of linear maps
\begin{align}\nonumber
    A = \pmat{\conb & 0 \\ 0 & \cona} \in L(\V, \Ch_0) \andd \Gamma = \pmat{ 0 & 0 \\ \conc & 0} \in L(\Lambda^2 \V, \Ch_1).
\end{align}
These has to satisfy the fake-flatness condition from~\Cref{eq:fake_flatness}, so that
\begin{align}\nonumber
     \pmat{\phi \gamma & 0 \\ 0 & \gamma \phi} = \pmat{ [\beta, \beta] & 0 \\ 0 & [\alpha, \alpha]}.
\end{align}
Parallel transport of a path $x \in C^\infty([0,1], \V)$ with respect to $(\cona, \conb)$ expressed in block form as
\begin{align}\nonumber
    F^{\cona, \conb}_t(x) = \pmat{ g^\conb_t(x) & 0 \\ 0 & f^\cona_t(x)},
\end{align}
where $f^{\cona}_t(x)$ and $g^{\conb}_t(x)$ are the usual parallel transport of paths with respect to $\cona$ and $\conb$ respectively. \medskip

Recall from~\Cref{def:sh} that surface holonomy is defined on the unital algebra $\hCh_1$. For 
\[
(\lambda, H), (\lambda', H') \in \hCh_1,
\]
multiplication is defined by
\begin{align}\nonumber
    (\lambda, H) * (\lambda', H') = (\lambda\cdot \lambda', \lambda H' + \lambda H + H\Phi H).
\end{align}
Furthermore, the boundary map is extended to $d: \hCh_1 \to \Ch_0$ by defining $d(1,0) = I$ to be the identity matrix. Thus, in general,
\begin{align}\nonumber
    d(1, H) = \Phi H + H \Phi + I = \pmat{\phi h + I & 0 \\ 0 & h \phi + I}.
\end{align}
Then, expressing the surface holonomy equation in this case, we obtain for a surface $X \in C^\infty([0,1]^2, \V)$,
\begin{align}\nonumber
    \frac{\partial \hH^\con_{s,t}(X)}{\partial t} = \hH^{\con}_{s,t}(X) * \int_0^s F^{\cona, \conb}(x^{s',t}) \cdot \Gamma\left( \frac{\partial X_{s',t}}{\partial s'}, \frac{\partial X_{s',t}}{\partial t}\right) \cdot F^{\cona, \conb}(x^{s',t})^{-1} \, ds', \quad \hH^\con_{s,0}(X) = (1,0).
\end{align}
We note that the integral has no unit component, and is thus valued in $\Ch_1$, and furthermore we can write $\hH^\con_{s,t}(X) = (1, H^\con_{s,t}(X))$, where $H^\con_{s,t}(X) \in \Ch_1$ is the non-unital component. Further decomposing
\begin{align}\nonumber
    H^{\con}_{s,t}(X) = \pmat{0 & 0 \\ h^{\con}_{s,t}(X) & 0},
\end{align}
we can explicitly express the differential equation for $h^{\con}_{s,t}(X)$ as
\begin{align}\nonumber
    \frac{\partial h^{\con}_{s,t}(X)}{\partial t} = (I + h^{\con}_{s,t}(X) \phi) \cdot \int_0^s f^\cona(x^{s',t}) \cdot \gamma\left( \frac{\partial X_{s',t}}{\partial s'}, \frac{\partial X_{s',t}}{\partial t}\right) \cdot g^{\conb}(x^{s',t})^{-1} \, ds', \quad h^{\con}_{s,0}(X) = 0,
\end{align}
which coincides with the matrix surface holonomy equation from~\cite{lee_random_2023}.

\section{Free Crossed Modules of Algebras} \label{apx:fxa}

\begin{proof}{(\Cref{prop:fxa_construction})}
    We begin by showing the $\cmFXA(\delta_0)$ is a crossed module. Let $E = a_1 \otimes w \otimes a_2, F = b_1 \otimes v \otimes b_2 \in A_0 \otimes W \otimes A_0$. Then, 
    \begin{align} \label{eq:fxa_delta_comp}
        \delta(E) \gtd F = (a_1 \cdot \delta_0(w) \cdot a_2 \cdot b_1) \otimes v \otimes b_2 \andd E \ltd \delta(F) = a_1 \otimes w \otimes (a_2 \cdot b_1 \cdot \delta_0(v) \cdot b_2);
    \end{align}
    thus $\delta(\delta(E) \gtd F - E \ltd \delta(F)) = 0$. Then, $\delta: \FXA_1(\delta_0) \to A_0$ is well-defined since $\delta(\Pf) = 0$. Furthermore, $\delta$ is a morphism of algebras when $\FXA_1(\delta_0)$ is equipped with the product in~\Cref{eq:fxa_product},
    \begin{align}\nonumber
        \delta(E*F) =\delta(\delta(E) \gtd F) = a_1 \cdot \delta_0(w) \cdot a_2 \cdot b_1 \cdot \delta_0(w) \cdot b_2= \delta(E) \cdot \delta(F),
    \end{align}
    using the computation above.
    Next, we note that $\Pf$ is closed under the $A_0$ action and is indeed an ideal. In particular, let $E, F \in A_0 \otimes W \otimes A_0$ as before and $c \in A_0$. Then,
    \begin{align}\nonumber
        c \gtd (\delta(E) \gtd F - E \ltd \delta(F)) &= \delta(c \gtd E) \gtd F - (c \gtd E) \ltd \delta(F)\\
        (\delta(E) \gtd F - E \ltd \delta(F)) \ltd c &= \delta(E) \gtd (F \ltd c) - E \ltd \delta(F \ltd c),
    \end{align}
    so $c \gtd \Pf,\, \Pf \ltd c \subset \Pf$ for any $c \in A_0$. Thus, the actions $\gtd$ and $\ltd$ are well-defined on $\FXA_1(\delta_0)$. The first Peiffer holds by definition of the actions, and the second Peiffer identity holds by definition of the multiplication in~\Cref{eq:fxa_product}. Finally, let $\iota: W \to \FXA_1(\delta_0)$ be the inclusion $\iota(w) = 1 \otimes w \otimes 1$. \medskip

    Now, we will show that this is free. Suppose $\cmA = (\delta^A : A_1 \to A_0)$ is a crossed module, and $\gamma: W \to A_1$ is a linear map such that $\delta_0 = \delta^A \circ \gamma$. We begin by defining a $A_0$-bimodule map $\tconc: A_0 \otimes W \otimes A_0 \to A_1$ by
    \begin{align}\nonumber
        \tconc(a_1 \otimes w \otimes a_2) = a_1 \gtd_A \gamma(w) \ltd_A a_2.
    \end{align}
    We note that this is the unique map of $A_0$-bimodules such that
    \[
        \begin{tikzcd}
            W \ar[r, "\iota"] \ar[dr, swap, "\conc"] & A_0 \otimes W \otimes A_0 \ar[r, "\delta"] \ar[d, dashed, "\tconc"] & A_0 \ar[d, equal] \\
            & A_1 \ar[r, "\delta^A"] & A_0
        \end{tikzcd}
    \]
    Now using the same notation for $E, F \in A_0 \otimes W \otimes A_0$ as above, we have
    \begin{align}\nonumber
        \tconc(\delta(E) \gtd F - E \ltd \delta(F)) & = (a_1 \cdot \delta_0(w) \cdot a_2 \cdot b_1) \gtd  \gamma(v) \ltd b_2 - a_1 \gtd \gamma(v) \ltd (a_2 \cdot b_1 \cdot \delta_0(v) \cdot b_2) \\
        & = a_1 \gtd \Big(\delta^A(\gamma(w) \ltd (a_2\cdot b_1)) \gtd \gamma(v) - \gamma(w) \ltd \delta^A((a_2 \cdot b_2) \gtd \gamma(v))\Big) \ltd b_2 \\
        & = a_2 \gtd \Big( (\gamma(w) \ltd (a_2\cdot b_1)) \gtd \gamma(v) - \gamma(w) \ltd ((a_2 \cdot b_2) \gtd \gamma(v)) \Big) \\
        & = 0,
    \end{align}
    where we use~\Cref{eq:fxa_delta_comp} in the first line, the fact that $\delta_0 = \delta^A\circ \gamma$ and the first Peiffer identity for $\cmA$ in the second line, and the second Peiffer identity for $\cmA$ in the third line. Thus, $\tconc$ is well-defined on $\FXA_1(\delta_0)$. Finally, $\tconc$ is an algebra morphism,
    \begin{align}\nonumber
        \tconc (P*Q) = \tconc( \delta(P) \gtd Q) = \delta(P) \gtd \tconc(Q) = \delta^A(\tconc(P)) \gtd \tconc(Q) = \tconc(P) * \tconc(Q).
    \end{align}
\end{proof}

\section{Additional Proofs for Surface Extension Theorem} \label{apxsec:additional_surface_extension}

\textbf{Bound on Arbitrary Partitions.}
\begin{proof}[Proof of~\Cref{lem:arbitrary_partition_bound}]
    We choose a specific ordering for $\gridmult_{\ppart}$ by first multiplying the grid horizontally, and then vertically. In particular, we set
    \begin{align}\nonumber
        \gridmult_{\ppart}[\trX] = \bigvmult{j=1}{b} \left(\bighmult{i=1}{a} \trX_{s_{i-1}, s_i; t_{j-1}, t_j}\right).
    \end{align}
    We note that $\trrX$ is still \emph{almost multiplicative}, since the same analysis will yield~\Cref{eq:almost_hmult_element} and~\Cref{eq:almost_vmult_element}. We note that $\|\trX^{\gr{n+1}}\|$ satisfies the standard surface regularity bound from~\Cref{eq:surface_regularity}, and since the proof of~\Cref{prop:almost_multiplicative} only depends on this bound, we obtain the same almost multiplicativity bounds for $\trX$ from~\Cref{prop:almost_multiplicative}

    We begin by bounding the horizontal compositions for a fixed row $[t_{j-1}, t_j]$ by removing one horizontal partition at a time and using~\Cref{lem:mult_bound_reduction},
    \begin{align}\nonumber
        \left\|\bighmult{i=1}{a} \trX_{s_{i-1}, s_i; t_{j-1}, t_j} - \trX_{\smp; t_{j-1}, t_j}\right\| &\leq \sum_{i'=1}^{a-1} \left\|\trX_{\smm, s_{i'}; t_{j-1}, t_j} \hmult \trX_{s_{i'}, s_{i'+1}; t_{j-1}, t_j} - \trX_{\smm, s_{i'+1}; t_{j-1}, t_j}\right\| \\
        & \leq \sum_{i'=1}^{a-1} \frac{2^{2(n+1)\orho}}{\beta^2} \sum_{q=1}^{2n+1} \frac{\ctr{q\orho}{\smm, s_{i'+1}} \ctr{(2n+2-q)\orho}{t_{j-1}, t_j}}{\ffact{q}{\orho} \ffact{2n+2-q}{\orho}} \nonumber\\
        & \leq (a-1) \frac{2^{2(n+1)\orho}}{\beta^2} \sum_{q=1}^{2n+1} \frac{\ctr{q\orho}{\smp} \ctr{(2n+2-q)\orho}{\tmp}}{\ffact{q}{\orho} \ffact{2n+2-q}{\orho}},\nonumber
    \end{align}
    where we use the almost multiplicativity bound for $\trX$ in the second line, and the fact that $\omega(\smm, s_{i'+1}) \leq \omega(\smp)$ in the last line. Then, replacing the horizontal compositions one by one, we obtain
    \begin{align}\nonumber
        \left\| \tgridmult_{\ppart}[\trX] - \bigvmult{j=1}{b}  \trX_{\smp; t_{j-1}, t_j}\right\| \leq b(a-1) \frac{2^{2(n+1)\orho}}{\beta^2} \sum_{q=1}^{2n+1} \frac{\ctr{q\orho}{\smp} \ctr{(2n+2-q)\orho}{\tmp}}{\ffact{q}{\orho} \ffact{2n+2-q}{\orho}},
    \end{align}
    Now, using the same method as the horizontal case, we obtain
    \begin{align}\nonumber
        \left\| \bigvmult{j=1}{b}  \trX_{\smp; t_{j-1}, t_j} - \trX_{\smp;\tmp}\right\| \leq (b-1) \frac{2^{2(n+1)\orho}}{\beta^2} \sum_{q=1}^{2n+1} \frac{\ctr{q\orho}{\smp} \ctr{(2n+2-q)\orho}{\tmp}}{\ffact{q}{\orho} \ffact{2n+2-q}{\orho}}.
    \end{align}
    Putting this together, we obtain
    \begin{align}\nonumber
        \left\| \tgridmult_{\ppart}[\trX] - \trX_{\smp;\tmp}\right\| \leq ab \frac{2^{2(n+1)\orho}}{\beta^2} \sum_{q=1}^{2n+1} \frac{\ctr{q\orho}{\smp} \ctr{(2n+2-q)\orho}{\tmp}}{\ffact{q}{\orho} \ffact{2n+2-q}{\orho}}.
    \end{align}
\end{proof}

\textbf{Proof of Modified Continuity for Path Extensions.}
\begin{proof}[Proof of~\Cref{thm:path_extension_cont}]
    Let $s_1 < s_2 < s_3$. We begin by considering a bound for
    \begin{align}\nonumber
        \Bigg\| \sum_{q=1}^n \bx^{\gr{q}}_{s_1, s_2} \cdot \bx^{\gr{n+1-q}}_{s_2, s_3} &- \by^{\gr{q}}_{s_1, s_2} \cdot \by^{\gr{n+1-q}}_{s_2, s_3} \Bigg\| \\
        & \leq \sum_{q=1}^n \left\| \bx^{\gr{q}}_{s_1, s_2} \cdot \bx^{\gr{n+1-q}}_{s_2, s_3}  - \by^{\gr{q}}_{s_1, s_2} \cdot\bx^{\gr{n+1-q}}_{s_2, s_3} + \by^{\gr{q}}_{s_1, s_2} \cdot\bx^{\gr{n+1-q}}_{s_2, s_3} - \by^{\gr{q}}_{s_1, s_2} \cdot \by^{\gr{n+1-q}}_{s_2, s_3} \right\| \nonumber\\
        & \leq \sum_{q=1}^n \left\| \bx^{\gr{q}}_{s_1, s_2} \cdot \bx^{\gr{n+1-q}}_{s_2, s_3}  - \by^{\gr{q}}_{s_1, s_2} \cdot\bx^{\gr{n+1-q}}_{s_2, s_3}\right\| + \left\|\by^{\gr{q}}_{s_1, s_2} \cdot\bx^{\gr{n+1-q}}_{s_2, s_3} - \by^{\gr{q}}_{s_1, s_2} \cdot \by^{\gr{n+1-q}}_{s_2, s_3} \right\|\nonumber\\
        & \leq \sum_{q=1}^n \left\|\bx^{\gr{q}}_{s_1, s_2} - \by^{\gr{q}}_{s_1, s_2}\right\| \cdot \left\|\bx^{\gr{n+1-q}}_{s_2, s_3}\right\| + \left\|\by^{\gr{q}}_{s_1, s_2} \right\|\cdot \left\|\bx^{\gr{n+1-q}}_{s_2, s_3} -  \by^{\gr{n+1-q}}_{s_2, s_3}\right\|\nonumber\\
        & \leq \frac{\epsilon}{\beta^2} \sum_{q=1}^n \frac{\ctr{(2q-1)\orho}{s_1, s_2}}{\ffact{2q-1}{\orho}} \cdot \frac{\ctr{(2n-2q+2)\orho}{s_2, s_3}}{\ffact{2n-2q+2}{\orho}} + \frac{\ctr{2q\orho}{s_1,s_2}}{\ffact{2q}{\orho}} \cdot \frac{\ctr{(2n-2q+1)\rho}{s_2, s_3}}{\ffact{2n-2q+1}{\orho}}\nonumber \\
        & \leq \frac{\epsilon}{\beta^2} \sum_{q'=1}^{2n} \frac{\ctr{q'\orho}{s_1, s_2}}{\ffact{q'}{\orho}} \cdot \frac{\ctr{(2n+1-q')\orho}{s_2, s_3}}{\ffact{2n+1-q'}{\orho}} \nonumber\\
        & \leq \frac{\epsilon}{\orho\beta^2} \frac{\ctr{(2n+1)\orho}{s_1, s_3}}{\ffact{2n+1}{\orho}}\nonumber
    \end{align}

    Now, we can consider the usual Young trick to prove a maximal inequality and conclude; in particular, we follow the proof of~\cite[Theorem 3.7]{lyons_differential_2007}.
\end{proof}

\textbf{Continuity of Modified Path Extensions}

In this section, we provide a modified continuity result for rough path functionals over a rectangular domain.

\begin{theorem} \label{thm:mixed_pathwise_ext_cont}
    Let $\rho \in (0,1]$ and let $n + \frac{1}{2} > \frac{1}{\rho}$ (or equivalently $2n +1 > \frac{1}{\orho}$). Let $(\rx^h, \rx^v)$ and $(\ry^h, \ry^v)$ be path-multiplicative path components of a double group function,
    \begin{align}\nonumber
        \rx^h, \ry^h : \Delta^2 \times [0,1] \to T_0^{\gr{\leq n}} \andd \rx^v, \ry^v : [0,1] \times \Delta^2 \to T_0^{\gr{\leq n}}.
    \end{align}
    Suppose both $(\rx^h, \rx^v)$ and $(\ry^h, \ry^v)$ satisfy path regularity and continuity in~\Cref{eq:path_regularity} and~\Cref{eq:path_continuity} respectively. Furthermore, for all $k \in [n]$ and some $\epsilon > 0$, we assume
    \begin{align}
        \left\| (\rx^h_{s_1, s_2; t} - \ry^h_{s_1, s_2; t})^{\gr{k}}\right\| \leq \epsilon \cdot \frac{\ctr{2k\orho}{s_1, s_2}}{\ffact{2k}{\orho}}, \quad &\left\| (\rx^v_{s; t_1, t_2} - \ry^v_{s; t_1, t_2})^{\gr{k}}\right\| \leq \epsilon \cdot \frac{\ctr{2k\orho}{t_1, t_2}}{\ffact{2k}{\orho}} \label{eq:pathwise_ext_cont_condition}\\
        \left\|\left((\rx^h_{s_1, s_2; t_2} - \rx^h_{s_1, s_2; t_1}) - (\ry^h_{s_1, s_2; t_2} - \ry^h_{s_1, s_2; t_1})\right)^{\gr{k}}\right\| &\leq \epsilon \, \frac{\ctr{(2k-1)\orho}{s_1, s_2} \ctr{\orho}{t_1, t_2}}{\beta \ffact{2k-1}{\orho} \ffact{1}{\orho}} \label{eq:pathwise_ext_hcont_condition}\\
        \left\|\left((\rx^v_{s_2; t_1, t_2} - \rx^v_{s_1; t_1, t_2}) - (\ry^v_{s_2; t_1, t_2} - \ry^v_{s_1; t_1, t_2})\right)^{\gr{k}}\right\| &\leq \epsilon \, \frac{\ctr{\orho}{s_1, s_2} \ctr{(2k-1)\orho}{t_1, t_2}}{ \beta\ffact{1}{\orho}\ffact{2k-1}{\orho}}\label{eq:pathwise_ext_vcont_condition}
    \end{align}
    Then, these continuity conditions hold for the pathwise extensions $(\orx^h, \orx^v)$ and $(\ory^h, \ory^v)$ at level $k=n+1$. 
\end{theorem}
\begin{proof}
    First, we note that the first two conditions in~\Cref{eq:pathwise_ext_cont_condition} follow immediately from the standard continuity of rough path liftings~\cite{lyons_differential_1998,lyons_differential_2007}. The argument for the other two bounds is similar, and we focus on the horizontal case of~\Cref{eq:pathwise_ext_hcont_condition}, where we omit the $h$ superscript to simplify notation (for instance, $\rx \coloneqq \rx^h$). The primary technical component of this proof is to bound
    \begin{align} \label{eq:mixed_pathwise_ext_cont_main_bound1}
        \left\| \sum_{q=1}^n \bx^{\gr{q}}_{s_1, s_2; t_2} \cdot \bx^{\gr{n+1-q}}_{s_2, s_3; t_2} - \bx^{\gr{q}}_{s_1, s_2; t_1} \cdot \bx^{\gr{n+1-q}}_{s_2, s_3; t_1} - \by^{\gr{q}}_{s_1, s_2; t_2} \cdot \by^{\gr{n+1-q}}_{s_2, s_3; t_2} + \by^{\gr{q}}_{s_1, s_2; t_1} \cdot \by^{\gr{n+1-q}}_{s_2, s_3; t_1} \right\|.
    \end{align}
    Then, for a fixed $q$, we have (see~\cite[Lemma E.25]{lee_random_2023})
    \begin{align}\nonumber
        \Big\|&\bx^{\gr{q}}_{s_1, s_2; t_2} \cdot \bx^{\gr{n+1-q}}_{s_2, s_3; t_2} - \bx^{\gr{q}}_{s_1, s_2; t_1} \cdot \bx^{\gr{n+1-q}}_{s_2, s_3; t_1} - \by^{\gr{q}}_{s_1, s_2; t_2} \cdot \by^{\gr{n+1-q}}_{s_2, s_3; t_2} + \by^{\gr{q}}_{s_1, s_2; t_1} \cdot \by^{\gr{n+1-q}}_{s_2, s_3; t_1}\Big\| \\
         &\leq  \left\|\bx^{\gr{q}}_{s_1, s_2; t_2}\right\| \cdot \left\| \bx^{\gr{n+1-q}}_{s_2, s_3; t_2} - \bx^{\gr{n+1-q}}_{s_2, s_3; t_1} - \by^{\gr{n+1-q}}_{s_2, s_3; t_2} + \by^{\gr{n+1-q}}_{s_2, s_3; t_1}\right\| + \left\|\bx^{\gr{q}}_{s_1, s_2; t_2} - \bx^{\gr{q}}_{s_1, s_2; t_1} - \by^{\gr{q}}_{s_1, s_2; t_2} + \by^{\gr{q}}_{s_1, s_2; t_1}\right\| \cdot \left\| \by^{\gr{n+1-q}}_{s_2, s_3; t_2}\right\| \nonumber \\
        & \hspace{5pt} + \left\|\bx^{\gr{q}}_{s_1, s_2; t_2} - \bx^{\gr{q}}_{s_1, s_2; t_1}\right\| \cdot \left\|\bx^{\gr{n+1-q}}_{s_2, s_3; t_1} - \by^{\gr{n+1-q}}_{s_2, s_3; t_1}\right\| + \left\|\bx^{\gr{q}}_{s_1, s_2; t_1} - \by^{\gr{q}}_{s_1, s_2; t_1}\right\| \cdot \left\|\by^{\gr{n+1-q}}_{s_2, s_3; t_2} - \by^{\gr{n+1-q}}_{s_2, s_3; t_1}\right\|.\nonumber
    \end{align}
    We can bound each of these four terms by using a combination of the regularity of $\rx$ and $\ry$ in~\Cref{eq:path_regularity}, the continuity properties of $\rx$ and $\ry$ in~\Cref{eq:path_continuity}, as well as the continuity properties between $\rx$ and $\ry$ in~\Crefrange{eq:pathwise_ext_cont_condition}{eq:pathwise_ext_vcont_condition} which are assumed by induction. In particular, we get 
    \begin{align}\nonumber
        \Big\|&\bx^{\gr{q}}_{s_1, s_2; t_2} \cdot \bx^{\gr{n+1-q}}_{s_2, s_3; t_2} - \bx^{\gr{q}}_{s_1, s_2; t_1} \cdot \bx^{\gr{n+1-q}}_{s_2, s_3; t_1} - \by^{\gr{q}}_{s_1, s_2; t_2} \cdot \by^{\gr{n+1-q}}_{s_2, s_3; t_2} + \by^{\gr{q}}_{s_1, s_2; t_1} \cdot \by^{\gr{n+1-q}}_{s_2, s_3; t_1}\Big\| \\ 
        &\leq  \left(\frac{\ctr{2q\orho}{s_1, s_2}}{\beta \ffact{2q}{\orho}}\right) \cdot \left(\epsilon \, \frac{\ctr{(2n-2q+1)\orho}{s_2, s_3} \ctr{\orho}{t_1, t_2}}{\beta \ffact{2n-2q+1}{\orho} \ffact{1}{\orho}}\right) + \left(\epsilon \, \frac{\ctr{(2q-1)\orho}{s_1, s_2} \ctr{\orho}{t_1, t_2}}{\beta \ffact{2q-1}{\orho} \ffact{1}{\orho}}\right) \cdot \left(\frac{\ctr{2(n+1-q)\orho}{s_2, s_3}}{\beta \ffact{2(n+1-q)}{\orho}}\right) \nonumber\\
        & \hspace{5pt} + \left(\frac{\ctr{(2q-1)\orho}{s_1, s_2} \ctr{\orho}{t_1, t_2}}{\beta \ffact{2q-1}{\orho} \ffact{1}{\orho}}\right) \cdot \left(\epsilon \cdot \frac{\ctr{2(n+1-q)\orho}{s_2, s_3}}{\beta \ffact{2(n+1-q)}{\orho}}\right) + \left(\epsilon \cdot \frac{\ctr{2q\orho}{s_1, s_2}}{\beta \ffact{2q}{\orho}}\right) \left(\frac{\ctr{(2n-2q+1)\orho}{s_2, s_3} \ctr{\orho}{t_1, t_2}}{\beta \ffact{2n-2q+1}{\orho} \ffact{1}{\orho}}\right) \nonumber\\
        & = \frac{2\epsilon}{\beta^2} \left( \frac{\ctr{2q\orho}{s_1, s_2}\ctr{(2n-2q+1)\orho}{s_2, s_3} \ctr{\orho}{t_1, t_2}}{ \ffact{2q}{\orho}\ffact{2n-2q+1}{\orho} \ffact{1}{\orho}} + \frac{\ctr{(2q-1)\orho}{s_1, s_2} \ctr{2(n+1-q)\orho}{s_2, s_3}\ctr{\orho}{t_1, t_2}}{ \ffact{2q-1}{\orho} \ffact{2(n+1-q)}{\orho}\ffact{1}{\orho}}\right)\nonumber
    \end{align}
    Then, taking the sum over $q$ and applying the neoclassical inequality, we obtain
    \begin{align} \label{eq:mixed_pathwise_ext_cont_main_bound2}
        \bigg\| \sum_{q=1}^n &\bx^{\gr{q}}_{s_1, s_2; t_2} \cdot \bx^{\gr{n+1-q}}_{s_2, s_3; t_2} - \bx^{\gr{q}}_{s_1, s_2; t_1} \cdot \bx^{\gr{n+1-q}}_{s_2, s_3; t_1} - \by^{\gr{q}}_{s_1, s_2; t_2} \cdot \by^{\gr{n+1-q}}_{s_2, s_3; t_2} + \by^{\gr{q}}_{s_1, s_2; t_1} \cdot \by^{\gr{n+1-q}}_{s_2, s_3; t_1} \bigg\|  \leq \frac{2\epsilon}{\orho\beta^2}  \frac{\ctr{(2n+1)\orho}{s_1, s_3} \ctr{\orho}{t_1, t_2}}{ \ffact{2n+1}{\orho}\ffact{1}{\orho}}
    \end{align}

    Once again, we consider the usual Young trick to conclude, see the proof of~\cite[Theorem 3.7]{lyons_differential_2007}.

\end{proof}

\begin{lemma} \label{lem:replacing_stuff}
    Let $a_1, \ldots, a_{k-1}, a_{k+1}, \ldots, a_n, b_1, \ldots, b_{k-1}, b_{k+1}, \ldots, b_n \in T_0$ and $a_k, b_k \in T_1$. Then, 
    \begin{align}
        \|a_1 \cdots a_n - b_1 \cdots b_n\| \leq \sum_{i=1}^n \|a_1\| \cdots \|a_{i-1}\| \cdot\|a_i - b_i\| \cdot \|b_{i+1}\| \cdots \|b_n\|.
    \end{align}
\end{lemma}
\begin{proof}
    This is done by first decomposing the left hand side, then using the properties of the norms on $T_0$ and $T_1$ from~\Cref{eq:norm_properties},
    \begin{align}\nonumber
        \|a_1 \cdots a_n - b_1 \cdots b_n\| &\leq \sum_{i=1}^n \|a_1 \cdots a_{i-1} \cdot a_i \cdot b_{i+1} \cdots b_n - a_1 \cdots a_{i-1} \cdot b_i \cdot b_{i+1} \cdots b_n\| \\
        & \leq \sum_{i=1}^n \|a_1 \cdots a_{i-1} \cdot (a_i-b_i) \cdot b_{i+1} \cdots b_n\| \nonumber\\
        & \leq \sum_{i=1}^n \|a_1\| \cdots \|a_{i-1}\| \cdot\|a_i - b_i\| \cdot \|b_{i+1}\| \cdots \|b_n\|.\nonumber
    \end{align}
\end{proof}

\subsection{Continuity of Surface Extension} \label{apxssec:continuity_surface_extension}

Proof of~\Cref{thm:cont_extension}

\begin{proof}
    The path continuity bounds in~\Cref{eq:ext_cont_path} holds by the continuity of rough path extensions~\cite{lyons_differential_1998,lyons_differential_2007}. The path continuity bounds for parallel paths in~\Cref{eq:ext_cont_path_hmixed} and~\Cref{eq:ext_cont_path_vmixed} hold by~\Cref{thm:mixed_pathwise_ext_cont}. For the surface component, we assume by induction that these continuity inequalities hold for the extensions $\trrX$ and $\trrY$ up to some $n \in \N$. With slight abuse of notation, we will denote the extension up to level $n$ by $\rrX$ and $\rrY$, and we will consider the case of $n+1$. \medskip
    
    \textbf{Step 1: Bounding Initial Pathwise Extension.} We begin with the pathwise extensions $\orrX$ and $\orrY$ to level $n+1$. Recall that the pathwise extensions $\orX$ and $\orY$ are defined by applying the algebra section $\ts$ to the extended path boundaries. Thus, we begin by bounding the difference between path boundaries of $\orrX$ and $\orrY$, similar to what was done in~\Cref{lem:pathwise_ext_dg_regularity}. In particular, we have
    \begin{align}
        \|(\orx^{h}_{s_1, s_2; t_1} \cdot \orx^{v}_{s_2; t_1, t_2} &\cdot \orx^{-h}_{s_1, s_2; t_2} \cdot \orx^{-v}_{s_1; t_1, t_2} - \ory^{h}_{s_1, s_2; t_1} \cdot \ory^{v}_{s_2; t_1, t_2} \cdot \ory^{-h}_{s_1, s_2; t_2} \cdot \ory^{-v}_{s_1; t_1, t_2})^{\gr{n+1}}\|  \label{eq:ext_cont_main_decomp}\\
        & \leq \| (\orx^{h}_{s_1, s_2; t_1} \cdot \orx^{-h}_{s_1, s_2; t_2} - \ory^{h}_{s_1, s_2; t_1} \cdot \ory^{-h}_{s_1, s_2; t_2})^{\gr{n+1}}\| + \| (\orx^{v}_{s_2; t_1, t_2} \cdot \orx^{-v}_{s_1;t_1, t_2} - \ory^{v}_{s_2; t_1, t_2} \cdot \ory^{-v}_{s_1;t_1, t_2})^{\gr{n+1}}\| \nonumber\\
        &+  \sum_{\substack{i + j + k + \ell = n+1 \\ i +k \geq 1, \, j + \ell \geq 1}} \|\orx^{h, \gr{i}}_{s_1, s_2; t_1} \cdot \orx^{v, \gr{j}}_{s_2; t_1, t_2} \cdot \orx^{-h,\gr{k}}_{s_1, s_2; t_2} \cdot \orx^{-v, \gr{\ell}}_{s_1; t_1, t_2} - \ory^{h, \gr{i}}_{s_1, s_2; t_1} \cdot \ory^{v, \gr{j}}_{s_2; t_1, t_2} \cdot \ory^{-h, \gr{k}}_{s_1, s_2; t_2} \cdot \ory^{-v, \gr{\ell}}_{s_1; t_1, t_2}\|.\nonumber
    \end{align}

    We start with the first term on the right hand side, where we use~\cite[Lemma 7.48]{friz_multidimensional_2010} to obtain
    \begin{align}
        \| (\orx^{h}_{s_1, s_2; t_1} \cdot \orx^{-h}_{s_1, s_2; t_2} - \ory^{h}_{s_1, s_2; t_1} \cdot \ory^{-h}_{s_1, s_2; t_2})^{\gr{n+1}}\|  \leq & \sum_{q=1}^n \left\|(\orx^{h, \gr{q}}_{s_1, s_2; t_1} - \orx^{h, \gr{q}}_{s_1, s_2; t_2}) \cdot \orx^{-h, \gr{n+1-q}}_{s_1, s_2; t_1} - (\ory^{h, \gr{q}}_{s_1, s_2; t_1} - \ory^{h, \gr{q}}_{s_1, s_2; t_2}) \cdot \ory^{-h, \gr{n+1-q}}_{s_1, s_2; t_1} \right\| \nonumber \\
        & + \|\orx^{h, \gr{n+1}}_{s_1, s_2; t_1} - \orx^{h, \gr{n+1}}_{s_1, s_2; t_2} - \ory^{h, \gr{n+1}}_{s_1, s_2; t_1} - \ory^{h, \gr{n+1}}_{s_1, s_2; t_2}\|. \label{eq:ext_cont_pure_horizontal_comp}
    \end{align}
    The second term on the right can be bound using~\Cref{eq:ext_cont_path_hmixed}. For a fixed $q \in [n]$, the first term on the right can be bound by
    \begin{align}\nonumber
        \bigg\|&(\orx^{h, \gr{q}}_{s_1, s_2; t_1} - \orx^{-h, \gr{q}}_{s_1, s_2; t_2}) \cdot \orx^{-h, \gr{n+1-q}}_{s_1, s_2; t_1} - (\ory^{h, \gr{q}}_{s_1, s_2; t_1} - \ory^{h, \gr{q}}_{s_1, s_2; t_2}) \cdot \ory^{-h, \gr{n+1-q}}_{s_1, s_2; t_1} \bigg\| \\
        & \leq \|\orx^{h, \gr{q}}_{s_1, s_2; t_1} - \orx^{h, \gr{q}}_{s_1, s_2; t_2} - \ory^{h, \gr{q}}_{s_1, s_2; t_1} + \ory^{h, \gr{q}}_{s_1, s_2; t_2} \| \cdot \|\orx^{-h, \gr{n+1-q}}_{s_1, s_2; t_1} \| + \| \ory^{h, \gr{q}}_{s_1, s_2; t_1} - \ory^{h, \gr{q}}_{s_1, s_2; t_2}\| \cdot \|\orx^{-h, \gr{n+1-q}}_{s_1, s_2; t_1} - \ory^{-h, \gr{n+1-q}}_{s_1, s_2; t_1}\| \nonumber\\
        & \leq \left(\epsilon \, \frac{\ctr{(2q-1)\orho}{s_1, s_2} \ctr{\orho}{t_1, t_2}}{\beta \ffact{2q-1}{\orho} \ffact{1}{\orho}} \right) \left( \frac{\ctr{2(n+1-q)\orho}{s_1, s_2}}{\beta \ffact{2(n+1-q)}{\orho}}\right) + \left( \frac{\ctr{(2q-1)\orho}{s_1, s_2} \ctr{\orho}{t_1, t_2}}{\beta \ffact{2q-1}{\orho} \ffact{1}{\orho}}\right) \left(\epsilon \frac{\ctr{2(n+1-q)\orho}{s_1, s_2}}{\beta \ffact{2(n+1-q)}{\orho}} \right) \nonumber\\
        & \leq \frac{2\epsilon}{\beta^2} \left( \frac{\ctr{(2n+1)\orho}{s_1, s_2} \ctr{\orho}{t_1, t_2}}{\ffact{2q-1}{\orho} \ffact{2(n+1-q)}{\orho} \ffact{1}{\orho} } \right).\nonumber
    \end{align}
    Then summing over $q$, including the $q = n+1$ term which already satisfies the bound~\Cref{eq:ext_cont_path_hmixed} due to~\Cref{thm:mixed_pathwise_ext_cont}, and applying~\Cref{eq:binomial_sum_bound}, the bound in~\Cref{eq:ext_cont_pure_horizontal_comp} becomes
    \begin{align}\nonumber
        \| (\orx^{h}_{s_1, s_2; t_1} \cdot \orx^{-h}_{s_1, s_2; t_2} - \ory^{h}_{s_1, s_2; t_1} \cdot \ory^{-h}_{s_1, s_2; t_2})^{\gr{n+1}}\| &\leq \epsilon \left(\frac{ 2^{2(n+1)\orho}}{\orho \beta^2} \right) \left( \frac{\ctr{(2n+1)\orho}{s_1, s_2} \ctr{\orho}{t_1, t_2}}{\ffact{2n+1}{\orho} \ffact{1}{\orho} } \right).
    \end{align}
    By the same argument, we obtain the analogous bound for the second term in~\Cref{eq:ext_cont_main_decomp},
    \begin{align}\nonumber
        \| (\orx^{v}_{s_2; t_1, t_2} \cdot \orx^{-v}_{s_1;t_1, t_2} - \ory^{v}_{s_2; t_1, t_2} \cdot \ory^{-v}_{s_1;t_1, t_2})^{\gr{n+1}}\| \leq \epsilon \left(\frac{ 2^{2(n+1)\orho}}{\orho \beta^2} \right) \left( \frac{\ctr{\orho}{s_1, s_2} \ctr{(2n+1)\orho}{t_1, t_2}}{\ffact{1}{\orho}\ffact{2n+1}{\orho}  } \right).
    \end{align}
    Then, by applying~\Cref{lem:replacing_stuff} and~\Cref{eq:ext_cont_path} for each summand in the third term of~\Cref{eq:ext_cont_main_decomp}, we have 
    \begin{align}\nonumber
        \|\orx^{h, \gr{i}}_{s_1, s_2; t_1} &\cdot \orx^{v, \gr{j}}_{s_2; t_1, t_2} \cdot \orx^{-h,\gr{k}}_{s_1, s_2; t_2} \cdot \orx^{-v, \gr{\ell}}_{s_1; t_1, t_2} - \ory^{h, \gr{i}}_{s_1, s_2; t_1} \cdot \ory^{v, \gr{j}}_{s_2; t_1, t_2} \cdot \ory^{-h, \gr{k}}_{s_1, s_2; t_2} \cdot \ory^{-v, \gr{\ell}}_{s_1; t_1, t_2}\|  \leq \frac{4\epsilon}{\beta^4} \frac{\ctr{2(i+k)\orho}{s_1, s_2} \ctr{2(j + \ell)\orho}{t_1, t_2}}{\ffact{i}{\orho} \ffact{j}{\orho}\ffact{k}{\orho} \ffact{\ell}{\orho}}.
    \end{align}
    After reindexing the sum in~\Cref{eq:ext_cont_main_decomp} as was done in~\Cref{eq:boundary_reg_bound3}, the sum is bound by
    \begin{align}\nonumber
        \frac{4 \epsilon}{\beta^4}\sum_{q=1}^n \sum_{i=1}^q \sum_{j=1}^{n+1-q} \frac{\ctr{2q\orho}{s_1, s_2} \, \ctr{(2n-2q+2)\orho}{t_1, t_2}}{\ffact{2i}{\orho} \, \ffact{2q-2i}{\orho}\, \ffact{2j}{\orho} \,\ffact{2n-2q-2j+2}{\orho}} & \leq \frac{4 \epsilon 2^{2(n+1)\orho}}{\orho^2 \beta^4} \sum_{q=1}^n \frac{\ctr{2q\orho}{s_1, s_2} \, \ctr{(2n-2q+2)\orho}{t_1, t_2}}{\ffact{2q}{\orho}\ffact{2n-2q+2}{\orho}}.
    \end{align}
    Then, putting this together in~\Cref{eq:ext_cont_main_decomp}, we get
    \begin{align}\nonumber
        \|(\orx^{h}_{s_1, s_2; t_1} \cdot \orx^{v}_{s_2; t_1, t_2} &\cdot \orx^{-h}_{s_1, s_2; t_2} \cdot \orx^{-v}_{s_1; t_1, t_2} - \ory^{h}_{s_1, s_2; t_1} \cdot \ory^{v}_{s_2; t_1, t_2} \cdot \ory^{-h}_{s_1, s_2; t_2} \cdot \ory^{-v}_{s_1; t_1, t_2})^{\gr{n+1}}\| \leq \frac{\epsilon}{2\beta} \CTR{n+1}{\orho}{s_1, s_2; t_1, t_2}.
    \end{align}
    Then, by definition of the pathwise extensions $\orX$ and $\orY$ in~\Cref{eq:rX_pathwise_extension}, and the fact that $\ts$ is a short linear map, we have
    \begin{align} \label{eq:cont_ext_pathwise_ext_final_bound}
        \|\orX^{\gr{n+1}}_{s_1, s_2; t_1, t_2} - \orY^{\gr{n+1}}_{s_1, s_2; t_1, t_2}\| \leq \frac{\epsilon}{2\beta} W^{n+1}_\orho(s_1, s_2; t_1, t_2).
    \end{align}

    \textbf{Step 2: Setup Full Extension.} Next, we will consider the bounds between the full extensions $\trX$ and $\trY$, where we will follow the construction and notation of~\Cref{prop:first_sewing}. Let $(\smp; \tmp) \in \Delta^2_\dya \times \Delta^2_\dya$ be a dyadic rectangle, and we will omit the subscripts from the grid multiplication to simplify notation. We will need to bound the difference between successive grid multiplications, which we do by bounding intermediate grid multiplications, as was done in~\Cref{eq:max_ineq_intermediate_subdivisions},
    \begin{align} \label{eq:cont_ext_decompose_into_intermediate}
        \|(\gridmult^{m+1}[\orX] - \gridmult^m[\orX] - \gridmult^{m+1}[\orY] + \gridmult^m[\orY])^{\gr{n+1}}\| \leq \sum_{u = 1}^{4^m} \left\|(\gridmult^{m, u}[\orX] - \gridmult^{m, u-1}[\orX] - \gridmult^{m, u}[\orY] + \gridmult^{m, u-1}[\orY])^{\gr{n+1}}\right\|.
    \end{align}
    Recall from the definition of the intermediate grid multiplication in~\Cref{def:intermediate_subdivision} that for a fixed $u \in [4^m]$, $\gridmult^{m, u}[\orX]$ and $\gridmult^{m, u-1}[\orX]$ differ by subdividing a certain subspace $[s_{i-1}, s_i] \times [t_{j-1}, t_j] \in \dpart_m(\smp; \tmp)$. Then, following the argument from~\Crefrange{eq:consec_intermediate_subdivision}{eq:consec_intermediate_subdivision2}, we get
    \begin{align} \label{eq:cont_ext_consec_intermediate_subdivision}
        \bigg\|(\gridmult^{m, u}[\orX] - &\gridmult^{m, u-1}[\orX] - \gridmult^{m, u}[\orY] + \gridmult^{m, u-1}[\orY])^{\gr{n+1}}\bigg\| \\
         \leq& \left\|\left(\orX_{s_{i-1}, s; t_{j-1}, t} \hmult \orX_{s, s_i; t_{j-1}, t} - \orX_{s_{i-1}, s_i; t_{j-1}, t}\right)^{\gr{n+1}} - \left(\orY_{s_{i-1}, s; t_{j-1}, t} \hmult \orY_{s, s_i; t_{j-1}, t} - \orY_{s_{i-1}, s_i; t_{j-1}, t}\right)^{\gr{n+1}}\right\| \\
        & + \left\|\left(\orX_{s_{i-1}, s; t, t_j} \hmult \orX_{s, s_i; t, t_j} - \orX_{s_{i-1}, s_i; t, t_j}\right)^{\gr{n+1}} - \left(\orY_{s_{i-1}, s; t, t_j} \hmult \orY_{s, s_i; t, t_j} - \orY_{s_{i-1}, s_i; t, t_j}\right)^{\gr{n+1}}\right\| \nonumber\\
        & + \left\|\left(\orX_{s_{i-1}, s_i; t_{j-1}, t} \vmult \orX_{s_{i-1}, s_i; t, t_j} - \orX_{s_{i-1}, s_i; t_{j-1}, t_j}\right)^{\gr{n+1}} - \left(\orY_{s_{i-1}, s_i; t_{j-1}, t} \vmult \orY_{s_{i-1}, s_i; t, t_j} - \orY_{s_{i-1}, s_i; t_{j-1}, t_j}\right)^{\gr{n+1}}\right\|.\nonumber
    \end{align}
    Thus, we must now consider the difference between the obstruction to multiplicativity of $\orX$ and $\orY$. \medskip
    
    \textbf{Step 3: Almost Multiplicativity.}
    We simplify the notation to prove this almost multiplicativity bound, and consider $s_1 < s_2 < s_3$ and $t_1 < t_2$. This part largely follows the proof of~\Cref{prop:almost_multiplicative}. Using the explicit form of this obstruction from~\Cref{eq:almost_hmult_element}, and the bound in~\Cref{eq:almost_mult_initial_bound_short_ts}, we have
    \begin{align} \label{eq:cont_ext_almost_mult_initial_bound}
        \bigg\|\left(\orX_{s_1, s_2; t_1, t_2} \hmult \orX_{s_2, s_3; t_1, t_2} - \orX_{s_1, s_3; t_1, t_2}\right)^{\gr{n+1}} &- \left(\orY_{s_1, s_2; t_1, t_2} \hmult \orY_{s_2, s_3; t_1, t_2} - \orY_{s_1, s_3; t_1, t_2}\right)^{\gr{n+1}}\bigg\| \\
        & \leq  3 \left\|(\orx^h_{s_1, s_2; t_1} \gt \orX_{s_2, s_3; t_1, t_2} - \ory^h_{s_1, s_2; t_1} \gt \orY_{s_2, s_3; t_1, t_2})^{\gr{n+1}}\right\|\nonumber
    \end{align}
    Then, by definitions of the action, we have 
    \begin{align}\nonumber
        \left\|(\orx^h_{s_1, s_2; t_1} \gt \orX_{s_2, s_3; t_1, t_2} - \ory^h_{s_1, s_2; t_1} \gt \orY_{s_2, s_3; t_1, t_2})^{\gr{n+1}}\right\| & \leq \sum_{k=0}^{n-1} \sum_{i=0}^k \left\| \orx^{h, \gr{i}}_{s_1, s_2; t_1} \cdot \orX^{\gr{n+1-k}}_{s_2, s_3; t_1, t_2} \cdot \orx^{-h, \gr{i}}_{s_1, s_2; t_1} - \ory^{h, \gr{i}}_{s_1, s_2; t_1} \cdot \orY^{\gr{n+1-k}}_{s_2, s_3; t_1, t_2} \cdot \ory^{-h, \gr{i}}_{s_1, s_2; t_1} \right\|
    \end{align}
    Then, applying~\Cref{lem:replacing_stuff} and using the induction hypothesis that the continuity properties in the theorem statement are true for levels $k \in [n]$, we get
    \begin{align}\nonumber
        \bigg\|(\orx^h_{s_1, s_2; t_1} \gt \orX_{s_2, s_3; t_1, t_2} - \ory^h_{s_1, s_2; t_1} \gt \orY_{s_2, s_3; t_1, t_2})^{\gr{n+1}}\bigg\| &\leq \frac{3\epsilon}{\beta^3} \sum_{k=0}^{n-1} \sum_{i=0}^k \frac{\ctr{2k\orho}{s_1, s_2} \CTR{n+1-k}{\orho}{s_2, s_3; t_1, t_2}}{\ffact{2i}{\orho}\ffact{2k - 2i}{\orho}} \\
        & \leq \frac{3\epsilon}{\orho\beta^3} \sum_{k=0}^{n-1} 2^{2k\orho} \frac{\ctr{2k\orho}{s_1, s_2} \CTR{n+1-k}{\orho}{s_2, s_3; t_1, t_2}}{\ffact{2k}{\orho}}\nonumber\\
        & \leq \frac{3\epsilon 2^{2(n+1)\orho}}{\orho\beta^3} \sum_{q=0}^{2n+1} \sum_{k=0}^{k_{n,q}} \frac{\ctr{2k\orho}{s_1, s_2}\ctr{(2(n-k+1)-q)\orho}{s_2, s_3} \ctr{q\orho}{t_1, t_2}}{\ffact{2k}{\orho}\ffact{2(n-k+1)-q}{\orho} \ffact{q}{\orho}} \nonumber\\
        & \leq \frac{3\epsilon 2^{2(n+1)\orho}}{\orho^2\beta^3} \sum_{q=0}^{2n+1} \frac{\ctr{(2n+2-q)\orho}{s_1, s_3} \ctr{q\orho}{t_1, t_2}}{\ffact{2n+2-q}{\orho} \ffact{q}{\orho}} \nonumber\\
        & = \frac{3\epsilon}{\orho^2 \beta^3} \CTR{n+1}{\orho}{s_1, s_3; t_1, t_2}.\nonumber
    \end{align}
    where $k_{n,q} = \min\{n-1, \lfloor n - \frac{q-1}{2}\rfloor\}$. Then, substituting this into~\Cref{eq:cont_ext_almost_mult_initial_bound}, and repeating this argument for the vertical case, we get
    \begin{align}\nonumber
        \bigg\|\left(\orX_{s_1, s_2; t_1, t_2} \hmult \orX_{s_2, s_3; t_1, t_2} - \orX_{s_1, s_3; t_1, t_2} - \orY_{s_1, s_2; t_1, t_2} \hmult \orY_{s_2, s_3; t_1, t_2} + \orY_{s_1, s_3; t_1, t_2}\right)^{\gr{n+1}}\bigg\| & \leq \frac{9\epsilon}{\orho^2 \beta^3} \CTR{n+1}{\orho}{s_1, s_3; t_1, t_2}\\
        \bigg\|\left(\orX_{s_1, s_2; t_1, t_2} \vmult \orX_{s_1, s_2; t_2, t_3} - \orX_{s_1, s_2; t_1, t_3} - \orY_{s_1, s_2; t_1, t_2} \vmult \orY_{s_1, s_2; t_2, t_3} + \orY_{s_1, s_2; t_1, t_3}\right)^{\gr{n+1}}\bigg\| & \leq \frac{9\epsilon}{\orho^2 \beta^3} \CTR{n+1}{\orho}{s_1, s_2; t_1, t_3}.\nonumber
    \end{align}
    \medskip

    \textbf{Step 4: Bounding Full Extension}
    Now, we return to bounding the full extension. Applying these almost multiplicativity bounds to~\Cref{eq:cont_ext_consec_intermediate_subdivision}, we obtain
    \begin{align}\nonumber
        \bigg\|(\gridmult^{m, u}[\orX] - &\gridmult^{m, u-1}[\orX] - \gridmult^{m, u}[\orY] + \gridmult^{m, u-1}[\orY])^{\gr{n+1}}\bigg\| \leq \frac{27\epsilon}{\orho^2 \beta^3} \CTR{n+1}{\orho}{s_{i-1}, s_i; t_{j-1}, t_j}.
    \end{align}
    Then, since $\CTR{n+1}{\orho}{s_{i-1}, s_i; t_{j-1}, t_j} = 2^{-2m(n+1)\orho} \CTR{n+1}{\orho}{\smp; \tmp}$, and summing up over all $4^m$ subsquares from~\Cref{eq:cont_ext_decompose_into_intermediate}, we get
    \begin{align}\nonumber
        \|(\gridmult^{m+1}[\orX] - \gridmult^m[\orX] - \gridmult^{m+1}[\orY] + \gridmult^m[\orY])^{\gr{n+1}}\| \leq \frac{27\epsilon}{\orho^2 \beta^3} 2^{2m(1-(n+1)\orho)} W^{n+1}_\orho(\smp; \tmp).
    \end{align}
    Next, summing up the difference between consecutive grid multiplications, we obtain
    \begin{align}\nonumber
        \|(\gridmult^{m+1}[\orX] - \gridmult^{m+1}[\orY]  - \orX_{\smp; \tmp} + \orY_{\smp; \tmp})^{\gr{n+1}}\| &\leq \frac{27\epsilon}{\orho^2 \beta^3} \left(\sum_{m=0}^\infty 2^{2m(1-(n+1)\orho)}\right) \CTR{n+1}{\orho}{\smp; \tmp} \\
        & \leq \frac{\epsilon}{2\beta} W^{n+1}_\orho(\smp; \tmp).
    \end{align}
    Finally, combining this with~\Cref{eq:cont_ext_pathwise_ext_final_bound}, we obtain
    \begin{align}\nonumber
        \|(\gridmult^{m+1}[\orX] - \gridmult^{m+1}[\orY])^{\gr{n+1}}\| \leq \frac{\epsilon}{\beta} \CTR{n+1}{\orho}{\smp; \tmp}.
    \end{align}
    Then, since this bound does not depend on $m$, we have
    \begin{align}\nonumber
        \|\trX^{\gr{n+1}}_{\smp;\tmp} - \trY^{\gr{n+1}}_{\smp;\tmp}\| \leq \frac{\epsilon}{\beta} \CTR{n+1}{\orho}{\smp; \tmp}.
    \end{align}

\end{proof}

\subsection{Completeness of Rough Surface Space}

\begin{proof}[Proof of~\Cref{prop:rough_surface_complete}]
    This proof is fairly standard and follows the structure of the proof that the space of path functionals equipped with the rough path metric is complete. We begin by showing that the spaces of path and surface functionals
    \begin{align}\nonumber
        (\HPF_\rho^{\gr{\leq n}}, d^h_\rho), \quad (\VPF_\rho^{\gr{\leq n}}, d^v_\rho) \andd (\SF_\rho^{\gr{\leq n}}, d^s_\rho)
    \end{align}
    are complete. We begin with the horizontal path components. Suppose $\rx^h(r) \in \HPF_\rho^{\gr{\leq n}}$ is a Cauchy sequence with respect to $d^h_\rho$. In particular, each level $\rx^{h, \gr{k}}(r)$ is Cauchy, and by definition of the $d^h_\rho$ metric, this implies that for $r, r' \in \N$,
    \begin{align}\nonumber
        \|\rx^{h, \gr{k}}_{s_1, s_2; t}(r) - \rx^{h, \gr{k}}_{s_1, s_2; t}(r')\| \leq d^h_\rho\left(\rx^{h, \gr{k}}(r), \rx^{h, \gr{k}}(r')\right) \Delta_{s_1, s_2}^{k\rho} \leq d^h_\rho\left(\rx^{h, \gr{k}}(r), \rx^{h, \gr{k}}(r')\right),
    \end{align}
    so it is Cauchy in $C(\Delta^2 \times [0,1], T_0^{\gr{k}})$ equipped with the uniform topology. Thus, $\rx^{h, \gr{k}}(r)$ converges uniformly to a function $\rx^{h, \gr{k}} \in C(\Delta^2 \times [0,1], T_0^{\gr{k}})$. Next, we must show that $\rx^{h, \gr{k}} \in \HPF_\rho^{\gr{k}}$. \medskip
    
    \textbf{Path Functional: Regularity Condition.} Because $\rx^{h, \gr{k}}(r)$ is Cauchy in the $d^h_\rho$ metric, for a fixed $\epsilon > 0$, there exists some $R \in \N$ such that for $r, r' \geq R$ such that
    \begin{align}\nonumber
        \|\rx^{h, \gr{k}}_{s_1, s_2; t}(r) - \rx^{h, \gr{k}}_{s_1, s_2; t}(r')\| \leq \epsilon \Delta_{s_1, s_2}^{k\rho}.
    \end{align}
    Then, for $r \geq R$, we have
    \begin{align}\nonumber
        \|\rx^{h, \gr{k}}_{s_1, s_2; t}(r)\| & \leq \|\rx^{h, \gr{k}}_{s_1, s_2; t}(r) - \rx^{h, \gr{k}}_{s_1, s_2; t}(R)\| + \|\rx^{h, \gr{k}}_{s_1, s_2; t}(R)\| \\
        & \leq (C_R + \epsilon) \Delta_{s_1, s_2}^{k\rho},
    \end{align}
    where $C_R > 0$ is the constant such that $\|\rx^{h, \gr{k}}_{s_1, s_2; t}(R)\| \leq C_R \Delta_{s_1, s_2}^{k\rho}$. Thus, there exists some fixed $M> 0$ such that
    \begin{align}\nonumber
        \|\rx^{h, \gr{k}}_{s_1, s_2; t}(r)\| \leq M \Delta_{s_1, s_2}^{k\rho}.
    \end{align}
    Fix $\epsilon > 0$. For each $(s_1, s_2; t) \in \Delta^2 \times [0,1]$, there exists some $R_{s_1, s_2; t} \in \N$ such that
    \begin{align}\nonumber
        \| \rx^{h, \gr{k}}_{s_1, s_2; t}(R_{s_1, s_2; t}) - \rx^{h, \gr{k}}_{s_1, s_2; t}\| \leq \epsilon \Delta_{s_1, s_2}^{k\rho}.
    \end{align}
    by uniform convergence.
    Then, we have
    \begin{align}\nonumber
        \|\rx^{h, \gr{k}}_{s_1, s_2; t}\| &\leq \| \rx^{h, \gr{k}}_{s_1, s_2; t}(r) - \rx^{h, \gr{k}}_{s_1, s_2; t}\| + \|\rx^{h, \gr{k}}_{s_1, s_2; t}(r)\| \\
        & \leq (M + \epsilon) \Delta_{s_1, s_2}^{k\rho}.\nonumber
    \end{align}
    Thus, $\rx^{h, \gr{k}}$ satisfies the path regularity condition. \medskip

    \textbf{Path Functional: Continuity Condition.} This argument will be similar to the previous. For a fixed $\epsilon> 0$, there exists some $R \in \N$ such that for $r,r' \geq R$, we have
    \begin{align}\nonumber
        \|\rx^{h, \gr{k}}_{s_1, s_2; t_1}(r) - \rx^{h, \gr{k}}_{s_1, s_2; t_2}(r) - \rx^{h, \gr{k}}_{s_1, s_2; t_1}(r') + \rx^{h, \gr{k}}_{s_1, s_2; t_2}(r')\| \leq \epsilon \inc{s_1}{s_2}^{(2k-1)\orho} \inc{t_1}{t_2}^{\orho}.
    \end{align}
    Then, for $r \geq R$, we have
    \begin{align}\nonumber
        \|\rx^{h, \gr{k}}_{s_1, s_2; t_1}(r) - \rx^{h, \gr{k}}_{s_1, s_2; t_2}(r)\| &\leq \|\rx^{h, \gr{k}}_{s_1, s_2; t_1}(r) - \rx^{h, \gr{k}}_{s_1, s_2; t_2}(r) - \rx^{h, \gr{k}}_{s_1, s_2; t_1}(R) + \rx^{h, \gr{k}}_{s_1, s_2; t_2}(R)\| + \|\rx^{h, \gr{k}}_{s_1, s_2; t_1}(R) - \rx^{h, \gr{k}}_{s_1, s_2; t_2}(R)\| \\
        & \leq (C_R + \epsilon)\inc{s_1}{s_2}^{(2k-1)\orho} \inc{t_1}{t_2}^{\orho},\nonumber
    \end{align}
    where $C_R > 0$ is the constant such that $\|\rx^{h, \gr{k}}_{s_1, s_2; t_1}(R) - \rx^{h, \gr{k}}_{s_1, s_2; t_2}(R)\| \leq C_R \inc{s_1}{s_2}^{(2k-1)\orho} \inc{t_1}{t_2}^{\orho}$. Thus, there exists some fixed $M>0$ such that for all $r \in \N$,
    \begin{align}\nonumber
        \|\rx^{h, \gr{k}}_{s_1, s_2; t_1}(r) - \rx^{h, \gr{k}}_{s_1, s_2; t_2}(r)\| \leq M \inc{s_1}{s_2}^{(2k-1)\orho} \inc{t_1}{t_2}^{\orho}.
    \end{align}
    Fix $\epsilon > 0$, then by the same argument as the path regularity condition, we have 
    \begin{align}\nonumber
        \|\rx^{h, \gr{k}}_{s_1, s_2; t_1} - \rx^{h, \gr{k}}_{s_1, s_2; t_2}\| \leq (M+\epsilon) \inc{s_1}{s_2}^{(2k-1)\orho} \inc{t_1}{t_2}^{\orho}.
    \end{align}
    Thus, we have $\rx^{h, \gr{k}} \in \HPF_\rho^{\gr{k}}$. \medskip

    \textbf{Path Functional: Convergence in $d^h_\rho$.} Now, we need to make sure that $\rx^{h, \gr{k}}(r)$ converges to $\rx^{h, \gr{k}}$ in the $d^h_\rho$ metric. For $\epsilon > 0$, there exists $R\in \N$ such that for $r, r' \geq R$, we have
    \begin{align}\nonumber
        \|\rx^{h, \gr{k}}_{s_1, s_2; t}(r) - \rx^{h, \gr{k}}_{s_1, s_2; t}(r')\| \leq \epsilon \Delta_{s_1, s_2}^{k\rho}.
    \end{align}
    Then, by uniform convergence, we can find some $r' > R$ sufficiently large such that
    \begin{align}\nonumber
        \|\rx^{h, \gr{k}}_{s_1, s_2; t}(r') - \rx^{h, \gr{k}}_{s_1, s_2; t}\| \leq \epsilon \Delta_{s_1, s_2}^{k\rho},
    \end{align}
    so that
    \begin{align}\nonumber
        \|\rx^{h, \gr{k}}_{s_1, s_2; t}(r) - \rx^{h, \gr{k}}_{s_1, s_2; t}\| \leq 2\epsilon \Delta_{s_1, s_2}^{k\rho}.
    \end{align}
    We can make the same argument for the continuity component of $d^h_\rho$, and thus $\rx^{h, \gr{k}}(r)$ converges to $\rx^{h, \gr{k}}$ in the $d^h_\rho$ metric. Thus $(\HPF_\rho^{\gr{\leq n}}, d^h_\rho)$ is complete. Note that we can make the same arguments to show that $(\VPF_\rho^{\gr{\leq n}}, d^v_\rho)$ is complete. The same is true for the surface functional $(\SF_\rho^{\gr{\leq n}}, d^s_\rho)$. \medskip

    \textbf{Multiplicative Double Group Functional.} Then, since all of these sequences converge uniformly, if $\rrX(r) = (\rx^h(r), \rx^v(r), \rX(r))$ is path multiplicative and/or multiplicative, the limit will be as well. Furthermore, since $\delta$ is a continuous map, the boundary condition also holds in the limit. Thus, $(\DGF^\rho(\bT^{\gr{\leq{n}}}), d_\rho)$ is complete.
\end{proof}

\section{Universal Crossed Module of Locally \texorpdfstring{$m$}{m}-Convex Algebras} \label{apxsec:universal_locally_convex}

The universal property of the free associative algebra $T_0(\V)$ allows us to lift a linear map $\cona \in L(\V, A_0)$ into an associative algebra $A_0$, to an algebra morphism $\tcona: T_0(\V) \to A_0$. However, this construction ignores topological considerations, which are required to formulate the universal property of signatures. In particular, if $A_0$ is a Banach algebra, and $\cona \in L(\V, A_0)$ is a continuous linear map, we require a topology on $T_0(\V)$ such that $\tcona$ is continuous.

Here, we will follow the construction in~\cite[Section 2]{chevyrev_characteristic_2016} to construct a universal locally $m$-convex algebra. We refer the reader to~\cite{mallios_topological_2011,treves_topological_2006} for background on locally convex spaces and topological algebras. Here, we fix $\V$ to be a finite dimensional Hilbert space. \medskip

We can simply define a topology on $T_0(\V)$ by taking the coarsest topology such that the extension $\tcona$ is continuous for any continuous linear map $\cona \in L(\V, A_0)$ into a Banach algebra $A_0$. In fact, by~\cite[Corollary 2.5]{chevyrev_characteristic_2016}, $T_0(\V)$ becomes a locally $m$-convex algebra, where the topology is generated by the family $\{p_\lambda\}_{\lambda >0}$ of norms defined in~\Cref{eq:plambda_norm}. The completion of $T_0(\V)$ is $E_0(\V)$ and can be expressed as~\Cref{eq:E0_characterization}, and is equipped with the universal property in~\Cref{cor:E0_universal}. \medskip

Our goal is to develop the crossed module analogue of this construction such that the unique morphisms obtained in the universal property of $\bT(\V)$ in~\Cref{prop:univ_2con_alg} are continuous. We begin with the free bimodule construction. We use $\otimes$ for the algebraic tensor product, $\otimes_\pi$ for the projective tensor product, and $\hotimes_\pi$ for the completion of the projective tensor product.

\begin{lemma} \label{lem:projective_bimodule}
    Let $(A, \cdot, 1)$ be a unital locally $m$-convex algebra, and $W$ a locally convex vector space. The \emph{projective $A$-bimodule over $W$} is $\PBi(A, W) \coloneqq A \hotimes W \hotimes A$, where $\hotimes$ is the completion of the projective tensor product. The left and right actions are continuous and this satisfies the universal property that for any $A$-bimodule $M$ with continuous actions and a continuous linear map $f \in L(W, M)$, there exists a unique continuous morphism of $A$-bimodules $\tf: \PBi(A, W) \to M$ such that the following commutes.
    \[
        \begin{tikzcd}
            W \ar [r, "f"] \ar[d, swap,"\iota"] & M \\
            \PBi(A,W) \ar[ur, swap, dashed, "\tf"]&
        \end{tikzcd}
    \]
\end{lemma}
\begin{proof}
    The left action is continuous as a bilinear map since it factors as
    \begin{align}\nonumber
        A \times (A \hotimes_\pi W \hotimes_\pi A) \rightarrow A \otimes_\pi A \hotimes_\pi W \hotimes_\pi A \xrightarrow{\cdot \otimes \id \otimes \id} A \hotimes_\pi W \hotimes_\pi A,
    \end{align}
    where the first map is continuous by the definition of the projective tensor product~\cite[Definition 43.2]{treves_topological_2006}, and the second is continuous since $A$ is a locally $m$-convex algebra and thus multiplication is continuous~\cite[Proposition 1.6]{mallios_topological_2011}. The same argument holds for the right action. \medskip
    
    By the universality of $\FBi(A, W) = A \otimes W \otimes A$, there exists a unique morphism of $A$-bimodules $\tf: \FBi(A,W) \to M$ with the desired properties. Now, we note that $\tf$ factors as
    \begin{align}\nonumber
        \tf : A \otimes W \otimes A \xrightarrow{\id \otimes f \otimes \id} A \otimes M \otimes A \rightarrow M,
    \end{align}
    which is continuous as a map $\tf: A \otimes_\pi W \otimes_\pi A \to M$, and therefore its completion $\tf: \PBi(A, W) \to M$ is also continuous.
\end{proof}

We will be primarily interested in the case of the projective bimodule
\begin{align}\nonumber
    \oE_1(\V) \coloneqq \PBi(E_0(\V), \Lambda^2\V) = E_0(\V) \hotimes_\pi \Lambda^2 \V \hotimes_\pi E_0(\V).
\end{align}
In fact, we can characterize this as a subset of $\oT_1\ps{\V}$ in the same way as~\Cref{eq:E0_characterization}.

\begin{lemma} \label{lem:Plambda_topology}
    Equip each $\oT_1^{\gr{k}}(\V)$ with the norms discussed in~\Cref{sssec:norms}. Then,
    \begin{align}\nonumber
        \oE_1(\V) \coloneqq \left\{ E \in \oT_1\ps{\V} \, : \, \sum_{n=2}^\infty \lambda^n \|E^{\gr{n}}\| < \infty \, \text{ for all } \, \lambda > 0 \right\}.
    \end{align}
\end{lemma}
\begin{proof}
    We will need to distinguish between different norms. For $\lambda > 0$, let $p_\lambda$ be the norm on $T_0(\V)$ (and $E_0(\V)$) defined by~\Cref{eq:plambda_norm}, $q$ denote the norm on $\Lambda^2 \V$, and $P_\lambda$ denote the norm on $\oT_1(\V)$ given by
    \begin{align}\nonumber
        P_\lambda(E) \coloneqq \sum_{k=2}^\infty \lambda^k \|E^{\gr{k}}\|.
    \end{align}
    Furthermore, for $\lambda, \mu > 0$, let $p_{\lambda, \mu}$ be a norm on $\oT_1(\V)$ given by the projective tensor product of norms $p_{\lambda, \mu} \coloneqq p_\lambda \otimes_\pi q \otimes_\pi p_\mu$.
    Finally, let $Q^{\gr{n,k}}$ be the projective tensor norm on $\oT_1^{\gr{n,k}}(\V) = \V^{\otimes k} \otimes \Lambda^2 \otimes \V^{\otimes n-k-2}$. Now, we note that for $E \in \oT_1(\V)$, we have
    \begin{align}\nonumber
        p_{\lambda, \mu}(E) &= \inf \sum_{i=1}^r p_\lambda(a_i) \cdot q(v) \cdot q_\mu(b_i) \\
        & = \sum_{n=2}^\infty \sum_{k=0}^{n-2} \, \inf \, \sum_{i=1}^n  \lambda^k \mu^{n-k-2} \|a_i^{\gr{k}}\| \cdot \|v_i\| \cdot \|b_i^{\gr{n-k-2}}\| 
    \end{align}
    where the infimum is taken over all representations $\sum_{i=1}^r a_i \otimes v \otimes b_i = E$. Suppose $\lambda \leq \mu$, then,
    \begin{align}\nonumber
        p_{\lambda, \mu}(E)\leq \sum_{n=2}^\infty \mu^{n-2} \sum_{k=0}^{n-2} Q^{\gr{n,k}}(E^{\gr{n,k}}),
    \end{align}
    and since all norms on the finite dimensional spaces $\oT_1^{\gr{n,k}}(\V)$ are equivalent, there exists some constant $C_{\lambda, \mu}>0$ such that
    \begin{align}\nonumber
        p_{\lambda, \mu}(E) \leq C_{\lambda,\mu} P_\mu(E).
    \end{align}
    Similarly, we can find a constant $c_{\lambda,\mu} > 0$ such that
    \begin{align}\nonumber
        c_{\lambda, \mu} P_\lambda(E) \leq p_{\lambda, \mu}(E) \leq C_{\lambda, \mu} P_\mu(E),
    \end{align}
    and thus the families $\{p_{\lambda, \mu}\}_{\lambda, \mu > 0}$ (which defines the topology of $\oE_1(\V)$) and $\{P_\lambda\}_{\lambda >0}$ define the same topology on $\oT_1(\V)$. Thus, we obtain our desired result. 
\end{proof} 

Next, consider the quotient of $\oT_1(\V)$ (equipped with the topology generated by the norms $P_\lambda$) by the Peiffer ideal $\Pf$. This quotient is still a locally convex space, and its topology is generated by the corresponding quotient norms~\cite[Proposition 7.9]{treves_topological_2006}. Then, we denote the completion by $E_1(\V)$, which can be characterized as a subset of $T_1\ps{\V}$ as
\begin{align}\nonumber
    E_1(\V) \coloneqq \left\{ E \in T_1\ps{\V} \, : \, \sum_{k=2}^\infty \lambda^k \|E^{\gr{k}}\| < \infty \, \text{ for all } \, \lambda > 0 \right\},
\end{align}
where $\|E^{\gr{n}}\|$ here denotes the quotient norms on $E^{\gr{n}} \in T_1^{\gr{n}}$ defined in~\Cref{sssec:norms}.

\begin{lemma}
    The structure $\bE(\V) = (\delta: E_1(\V) \to E_0(\V), \gtd, \ltd)$ is a crossed module of locally $m$-convex algebras.
\end{lemma}
\begin{proof}
    First, $\delta$ is continuous because it is continuous levelwise (with respect to the usual norms on $T_1^{\gr{n}}$ and $T_0^{\gr{n}}$). Furthermore, the actions are continuous since the actions at the level of the projective bimodule (before taking quotients) is continuous (\Cref{lem:projective_bimodule}).
\end{proof}

Recall the universal 2-connection $\zeta \in L(\V, T_0(\V))$ and $Z \in L(\Lambda^2 \V, T_1(\V))$ from~\Cref{prop:univ_2con_alg}, which can be embedded into $E_0(\V)$ and $E_1(\V)$ as continuous linear maps $\zeta \in L(\V, E_0(\V))$ and $L(\Lambda^2 \V, E_1(\V))$. Then, this space has the desired universal property.

\begin{proposition}
    Let $\cmA = (\delta^A: A_1 \to A_0, \gtd, \ltd)$ be a crossed module of Banach algebras. Let $\cona \in L(\V, A_0)$ and $\conc \in L(\Lambda^2 \V, A_1)$ be continuous linear maps which define a 2-connection. Then, there exist unique continuous morphisms of algebras $\tcona: E_0(\V) \to A_0$ and $\tconc: E_1(\V) \to A_1$ such that
    \begin{align}\nonumber
        \cona = \tcona \circ \zeta \andd \conc = \tconc \circ Z.
    \end{align}
\end{proposition}
\begin{proof}
    By~\Cref{prop:univ_2con_alg}, there exist algebra morphisms $\tcona : T_0(\V) \to A_0$ and $\tconc: T_1(\V) \to A_1$ such that the desired properties hold. Furthermore, by~\Cref{cor:E0_universal}, we can extend $\tcona: E_0(\V) \to A_0$ to a continuous map. It remains to show that $\tconc$ is continuous when $T_1(\V)$ is equipped with the topology generated by $\{P_\lambda\}_{\lambda > 0}$.

    Recall the construction of $\tconc$ from~\Cref{prop:univ_2con_alg}, where it is defined to be the composition
    \begin{align} \label{eq:apx_tconc_decomposition}
        \tconc: T_1(\V) \xrightarrow{\FXA_1(\tcona)} \FXA_1(\delta_0) \xrightarrow{\overline{\conc}} A_1,
    \end{align}
    where $\delta_0 = \delta_A \circ \conc : \Lambda^2 \V \to A_0$, $\overline{\conc}$ comes from the universal property of the free crossed module of $\FXA_1(\delta_0)$ (\Cref{prop:fxa_construction}) and $\FXA_1(\tcona)$ comes from the functoriality of the free crossed module construction (\Cref{prop:fxa_functorial}). Both of these maps are first constructed at the level of bimodules,
    \begin{align}\nonumber
        T_0(\V) \otimes \Lambda^2 \V \otimes T_0(\V) \xrightarrow{\tcona \otimes \id \otimes \tcona} A_0 \otimes \Lambda^2 \V \otimes A_0 \xrightarrow{\gtd \gamma \ltd} A_1.
    \end{align}
    Then, when we consider the topology on $T_0(\V)$ generated by the $p_\lambda$, and all tensor products to be projective tensor products, this composition is continuous. In particular, this implies the quotient in~\Cref{eq:apx_tconc_decomposition} is continuous when $T_1(\V)$ is equipped with the topology generated by the $P_\lambda$ (\Cref{lem:Plambda_topology}). 
\end{proof}

\section{Surface Extension with Rectangular Increments} \label{apxsec:surface_extension_rectangular}

In~\Cref{ssec:holder_dgf}, the definition of a $\rho$-H\"older double group functional in~\Cref{def:holder_dgf} was motivated by the surface signature of $\rho$-H\"older surfaces (\Cref{eq:intro_holder_surfaces}). These surfaces are not equipped with additional regularity conditions on the rectangular increments $\square[X]$, but we use the natural $\rho/2$-H\"older bounds on the rectangular increments from~\Cref{lem:holder_2d_increment}.
However, we can also consider 2D rectangular $\rho$-H\"older surfaces $C^{\square, \rho}([0,1]^2, \V)$ from \Cref{eq:2d_rectangular_holder_space}, which additionally satisfy the rectangular regularity condition in~\Cref{eq:2d_rectangular_holder_bound}. In particular, the classes of $\rho$-H\"older and rectangular $\rho$-H\"older surfaces are distinct, and have distinct generalizations of Young integration, as discussed in~\Cref{sssec:intro_rough_surfaces}.

In this appendix, we show that we can make slight modifications to the notion of a $\rho$-H\"older double group functional such that the surface extension theorem still holds, and show that we can compute the surface signature of rectangular $\rho$-H\"older surfaces in the Young regime, with $\rho \in (\frac{1}{2},1]$. In particular, the only part we must change is the path continuity condition in the path functionals (\Cref{def:holder_path_functionals}).\medskip

\begin{definition}
    Let $\rho \in (0,1]$ and $n \in \N$. We define \emph{rectangular $\rho$-H\"older horizontal and vertical path functionals} to be path functionals $\rx^h \in \HPF(G_0^{\gr{\leq n}})$ and $\rx^v \in \VPF(G_0^{\gr{\leq n}})$ which satisfy for all $k \in [n]$ the path regularity conditions in~\Cref{eq:path_regularity} and the \emph{rectanglar path continuity conditions},
    \begin{align} \label{eq:rectangular_path_continuity_condition}
        \left\|\rx^{h, \gr{k}}_{s_1, s_2; t_1} - \rx^{h, \gr{k}}_{s_1, s_2; t_2}\right\| \leq \frac{\ctr{k\rho}{s_1, s_2}\ctr{\rho}{t_1, t_2}}{\beta \ffact{k}{\rho} \ffact{1}{\rho}} \andd \left\|\rx^{v, \gr{k}}_{s_1; t_1, t_2} - \rx^{v, \gr{k}}_{s_2; t_1, t_2}\right\| \leq \frac{\ctr{\rho}{s_1, s_2}\ctr{k\rho}{t_1, t_2}}{\beta  \ffact{1}{\orho} \ffact{k}{\orho}}.
    \end{align}
    The sets of rectangular $\rho$-H\"older horizontal and vertical path functionals are denoted $\HPF^{\square, \rho}(G_0^{\gr{\leq n}})$ and $\VPF^{\square, \rho}(G_0^{\gr{\leq n}})$. Furthermore, a double group functional $\rrX \in \DGF(\cmG^{\gr{\leq n}})$ is \emph{rectangular $\rho$-H\"older} if it satisfies the path regularity conditions in~\Cref{eq:path_regularity}, the rectangular path continuity conditions here, and the surface regularity conditions~\Cref{eq:surface_regularity}, and the space of such functionals is denoted $\DGF^{\square, \rho}(\cmG^{\gr{\leq n}})$. 
\end{definition}

The primary way the path continuity condition is used in the surface extension theorem is in showing that the surface component $\orX$ of the initial pathwise extension $\orrX$ satisfies the surface regularity condition by~\Cref{lem:boundary_signature_bound}. In fact, we prove that this bound still holds for rectangular $\rho$-H\"older path functionals.

\begin{lemma} \label{lem:boundary_signature_bound_rectangular}
    Suppose $\rx^h \in \HPF^{\square, \rho}(G_0^{\gr{\leq n}})$ and $\rx^v \in \VPF^{\square, \rho}(G_0^{\gr{\leq n}})$. Then, for any rectangle $[s_1, s_2]\times[t_1, t_2] \subset [0,1]^2$ and any $k \in [n]$, we have
    \begin{align} \label{eq:boundary_signature_bound_rectangular}
        \left\|(\rx^h_{s_1, s_2; t_1} \cdot \rx^v_{s_2; t_1, t_2} \cdot \rx^{-h}_{s_1, s_2; t_2} \cdot \rx^{-v}_{s_1; t_1, t_2})^{\gr{k}}\right\|  \leq \frac{2^{(2k-1)\orho}}{\beta} \sum_{q=1}^{2k-1} \frac{\ctr{q\orho}{s_1, s_2}\, \ctr{(2k-q)\orho}{t_1, t_2}}{\ffact{q}{\orho} \ffact{2k-q}{\orho}} = \frac{1}{2\beta} \CTR{k}{\orho}{s_1, s_2; t_1, t_2}.
    \end{align}
\end{lemma}
\begin{proof}
    This proof follows the same structure as~\Cref{lem:boundary_signature_bound}, except we adjust some bounds. 
    We begin by decomposing the boundary in the same way as in~\Cref{eq:boundary_reg_main_decomp}. However, instead of using the usual path continuity condition in~\Cref{eq:boundary_reg_horizontal_decomp}, we use the rectangular path continuity condition to get
    \begin{align}\nonumber
        \left\| \left(\rx^h_{s_1, s_2; t_1} \cdot \rx^{-h}_{s_1, s_2; t_1}\right)^{\gr{k}}\right\| &\leq \frac{\ctr{\rho}{s_1,s_2}}{\beta^2} \left(\frac{\ffact{k-1}{\rho}}{\ffact{k}{\rho}}\right) \sum_{q=1}^{k} \frac{\ctr{(q-1)\rho}{s_1, s_2} \ctr{\rho}{t_1, t_2}}{\ffact{k-1}{\rho} \ffact{1}{\rho}} \frac{\ctr{(k-q)\rho}{s_1, s_2}}{\ffact{k-q}{\rho}} \\
        & \leq \frac{2^{k\rho}\ctr{\rho}{0,1}}{\orho\beta^2}\left(\frac{\ffact{k-1}{\rho}}{\ffact{k}{\rho}}\right) \frac{\ctr{(k-1)\rho}{s_1, s_2} \ctr{\rho}{t_1, t_2}}{\ffact{k-1}{\rho} \ffact{1}{\rho}}.\nonumber
    \end{align}
    Then, the remainder of the proof remains the same, and we obtain our desired result.
\end{proof}

\begin{remark}
    In fact, this proof provides a slightly stronger surface regularity condition,
    \begin{align}\nonumber
        \left\|(\rx^h_{s_1, s_2; t_1} \cdot \rx^v_{s_2; t_1, t_2} \cdot \rx^{-h}_{s_1, s_2; t_2} \cdot \rx^{-v}_{s_1; t_1, t_2})^{\gr{k}}\right\|  \lesssim \sum_{q=1}^{k-1} \frac{\ctr{q\rho}{s_1, s_2}\, \ctr{(k-q)\rho}{t_1, t_2}}{\ffact{q}{\rho} \ffact{2k-q}{\rho}},
    \end{align}
    which eliminates the need to consider the $\sigma = \rho/2$ exponents.
\end{remark}

This allows us to show the surface extension theorem for rectangular $\rho$-H\"older double group functionals.

\begin{theorem} \label{thm:surface_extension_rectangular}
    Let $\rho \in (0,1]$ and $n \geq \left\lfloor \frac{2}{\rho} \right\rfloor$. Suppose $\rrX \in \DGF^{\square,\rho}(\cmG^{\gr{\leq n}})$ is a multiplicative double group functional. Then, there exists a multiplicative double group functional $\trrX \in \DGF^{\square, \rho}(\cmG^{\gr{\leq n+1}})$ such that $\trrX^{\gr{k}} = \rrX^{\gr{k}}$ for all $k \in [n]$. Furthermore, $\trrX$ is unique in the sense that if $\rrY = (\ry^h, \ry^v, \rY) \in \DGF(\cmG^{\gr{\leq n+1}})$ is another double group functional such that $\rrY^{\gr{k}} = \rrX^{\gr{k}}$ for all $k \in [n]$, and 
    \begin{align}\nonumber
        \|\rY^{\gr{n+1}}_{s_1, s_2; t_1, t_2}\| \leq C (\inc{s_1}{s_2} + \inc{t_1}{t_2})^{\theta}
    \end{align}
    for a constant $C>0$ and any $\theta > 1$, then $\trrX = \rrY$. 
\end{theorem}
\begin{proof}
    The entire proof of~\Cref{thm:cont_extension} as given in~\Cref{ssec:surface_extension_proof} holds in this setting. In particular, the original continuity theorem for path extensions in~\Cref{thm:path_ext_orig_cont} allows us to lift the path functionals with the desired rectangular path continuity conditions in~\Cref{eq:rectangular_path_continuity_condition}. Then,~\Cref{lem:boundary_signature_bound_rectangular} allows us to show that the surface component of the pathwise extension $\orrX$ has the desired surface regularity condition. The only other place where the path continuity condition is used is in proving that the extension $\trrX$ on dyadics is uniformly continuous in~\Cref{prop:surface_extension_unif_cont}, and this still holds with the rectangular continuity conditions. The remainder of the proof only relies on other properties, and the desired result follows.
\end{proof}

We can now prove~\Cref{prop:young_lifting}, which shows that for a surface $X \in C^{\square, \rho}([0,1]^2, \V)$ in the Young regime (with $\rho \in (\frac{1}{2},1]$), we can define a level $3$ rectangular $\rho$-H\"older multiplicative double group functional, which we can extend a multiplicative functional in $\DGF^{\square, \rho}(\cmG)$. 

\begin{proof}[Proof of~\Cref{prop:young_lifting}]
    The fact that the path components satisfy the path regularity and continuity conditions is given by the path extension and the original continuity theorems (\Cref{thm:path_extension} and~\Cref{thm:path_extension_cont} respectively). Now, we will show that the surface component is $\rho$-H\"older. \medskip

    \textbf{Step 1: The Area Process is 2D Rectangular $\rho$-H\"older.} This is proved in the 2D $p$-variation setting in~\cite[Proposition E.16]{lee_random_2023} and we use the same method here. We note that the 2D increment of the area process is
    \begin{align}\nonumber
        \square_{s_1, s_2; t_1, t_2}[A(X)] \coloneqq S^{\gr{2}}\left(\partial (X|_{[s_1, s_2]\times[t_1, t_2]})\right) = \rX^{\gr{2}}_{s_1, s_2; t_1, t_2}.
    \end{align}
    Then, using~\Cref{lem:boundary_signature_bound_rectangular}, we see that $A(X) \in C^{\square, \rho}([0,1], \Lambda^2 \V)$. \medskip

    \textbf{Step 2: Level 3 Surface Component via 2D Increment Integrals.}
    Next, we define 
    \begin{align}\nonumber
        f: [0,1]^2 \to L(\Lambda^2 \V, T^{\gr{3}}_1(\V)) \quad \text{by} \quad f_{s,t}(v) \coloneqq (X_{s,t} - X_{0,0}) \gt v.
    \end{align}
    By using the regularity of $X$, we get $f \in C^{\square, \rho}\left([0,1]^2, L(\Lambda^2 \V, T^{\gr{3}}_1(\V))\right)$. Then, by using the 2D Young integral from~\Cref{thm:towghi_young}, we obtain
    \begin{align}\nonumber
        \|\rX^{\gr{3}}_{s_1, s_2; t_1, t_2}\| \leq C \left(\inc{\smm}{\spp}^{2\rho} \inc{\tmm}{\tpp}^\rho + \inc{\smm}{\spp}^{\rho} \inc{\tmm}{\tpp}^{2\rho}\right) \leq \CTR{3}{\orho}{s_1, s_2; t_1, t_2}.
    \end{align}
    Thus, $\rX$ satisfies the surface regularity condition. \medskip

    \textbf{Step 3: Multiplicative Double Group Functional.}
    Finally, we show that $\rrX$ is a multiplicative double group functional by considering smooth approximations. Let $\rho' \in (\frac{1}{2}, \rho)$, and let
    \begin{align}\nonumber
        C^{0, \square, \rho'}([0,1]^2, \V) \subset C^{\square, \rho'}([0,1]^2, \V)
    \end{align}
    be the closure of smooth surfaces $C^\infty([0,1]^2, \V)$ (with respect to $\|X\|_\infty + \|X\|_{\rho'} + \|X\|_{\square, \rho'}$) in $C^{\square, \rho'}([0,1]^2, \V)$, and note that
    \begin{align}\nonumber
        C^{\square, \rho}([0,1]^2, \V) \subset C^{0, \square, \rho'}([0,1]^2, \V).
    \end{align}
    The analogous statement for 2D $p$-variation is shown in~\cite[Proposition 7.6]{lee_random_2023}, and the H\"older variant can be shown in the same manner. We note that for a smooth surface $Y \in C^\infty([0,1]^2, \V)$, we have
    \begin{align}\nonumber
        \int_{[0,1]^2} (Y_{s,t} - Y_{0,0}) \gt dA_{s,t}(Y) &= \int_{[0,1]^2}(Y_{s,t} - Y_{0,0}) \gt \left( \frac{\partial^2 A_{s,t}(Y)}{\partial s \partial t}\right) ds dt \\
        &= \int_{[0,1]^2}(Y_{s,t} - Y_{0,0}) \gt J_{s,t}(X) ds dt,\nonumber
    \end{align}
    which is the level $3$ surface signature, which is multiplicative and satisfies the boundary condition in~\Cref{eq:rrX_boundary_condition}. Then, since the area process is continuous~\cite[Proposition E.24]{lee_random_2023} and the Young integral is continuous, this also holds in the limit. Thus, $\rrX \in \DGF^{\square,\rho}(\cmG^{\gr{\leq 3}})$ is a rectangular $\rho$-H\"older double group functional, and~\Cref{thm:surface_extension_rectangular} provides the unique extension. 
\end{proof}

\bibliographystyle{plain}
\bibliography{sewing}

\end{document}